\documentclass[reqno]{amsart}

\usepackage{amsthm}   
\usepackage{amsfonts, amsmath, amscd}
\usepackage{amssymb, latexsym}
\usepackage[all,cmtip]{xy}     
\usepackage{enumerate} 
\usepackage{color}
\usepackage[usenames,dvipsnames]{xcolor}
\usepackage{verbatim}
\usepackage{xcolor}

\usepackage{hyperref}
 \hypersetup{colorlinks=false}
 \usepackage{cite}
 \usepackage{ amssymb }

 \usepackage[OT2,T1]{fontenc}
\newcommand\textcyr[1]{{\fontencoding{OT2}\fontfamily{wncyr}\selectfont #1}}

 \usepackage{graphicx, psfrag}
 \usepackage{pstool}

    \usepackage{xspace}   
    \usepackage{tikz-cd}
\usepackage{tikz}
\usetikzlibrary{arrows}

\usepackage[titletoc]{appendix}
\usepackage[normalem]{ulem}     
\usepackage{amsthm}
\usepackage{amsmath}
\usepackage{amscd}
\usepackage[latin2]{inputenc}
\usepackage{t1enc}
\usepackage[mathscr]{eucal}
\usepackage{indentfirst}
\usepackage{graphicx}
\usepackage{graphics}
\usepackage{pict2e}
\usepackage{epic}

\usepackage{amssymb}
\usepackage{amsmath}
\usepackage{amsfonts}
\usepackage{amscd}
\usepackage{amsthm}
\usepackage{graphicx,psfrag}
\usepackage{wrapfig}
\usepackage{subeqnarray}
\usepackage{amssymb}
\usepackage{yfonts}
\usepackage[makeroom]{cancel}
\usepackage{trimclip}
\usepackage{cancel}
\newif\iflclip
\newif\ifbclip
\newif\ifrclip
\newif\iftclip
\def\CLIP{\dimexpr\fboxrule+.2pt\relax}
\def\nulclip{0pt}
\newcommand\partbox[2]{%
  \lclipfalse\bclipfalse\rclipfalse\tclipfalse%
  \let\lkern\relax\let\rkern\relax%
  \let\lclip\nulclip\let\bclip\nulclip\let\rclip\nulclip\let\tclip\nulclip%
  \parseclip#1\relax\relax%
  \iflclip\def\lkern{\kern\CLIP}\def\lclip{\CLIP}\fi
  \ifbclip\def\bclip{\CLIP}\fi
  \ifrclip\def\rkern{\kern\CLIP}\def\rclip{\CLIP}\fi
  \iftclip\def\tclip{\CLIP}\fi
  \lkern\clipbox{\lclip{} \bclip{} \rclip{} \tclip}{\fbox{#2}}\rkern%
}
\def\parseclip#1#2\relax{%
  \ifx l#1\lcliptrue\else
  \ifx b#1\bcliptrue\else
  \ifx r#1\rcliptrue\else
  \ifx t#1\tcliptrue\else
  \fi\fi\fi\fi
  \ifx\relax#2\relax\else\parseclip#2\relax\fi
}
\usepackage{slashed} 

\theoremstyle{definition}

\theoremstyle{remark}

\numberwithin{equation}{section}

\usepackage{amsmath}
\usepackage{amssymb}
\usepackage{amsfonts}
\usepackage{amsthm}
\usepackage{mathtools}
\usepackage{graphicx}
\usepackage{colonequals}
\usepackage{booktabs}
\usepackage{cancel}
\usepackage[numbers]{natbib}
\usepackage{hyperref}
\usepackage{fancyhdr}
\usepackage{multirow}
\usepackage{multirow}
\usepackage{braket}
\usepackage{tikz}
\usetikzlibrary{arrows,shapes,decorations.pathmorphing}
\usepackage{tikz-cd}
\usepackage[margin=2.9cm]{geometry}
\usepackage{mathabx}
\usepackage{ esint }
\usepackage{slashed}
\usepackage{mathrsfs}
\usepackage{amssymb}
\usepackage{blkarray}
\usepackage{xcolor}
\usepackage{pinlabel}   
\usepackage{scalerel,stackengine}
\stackMath
\newcommand\reallywidehat[1]{%
\savestack{\tmpbox}{\stretchto{%
  \scaleto{%
    \scalerel*[\widthof{\ensuremath{#1}}]{\kern-.6pt\bigwedge\kern-.6pt}%
    {\rule[-\textheight/2]{1ex}{\textheight}}
  }{\textheight}%
}{0.5ex}}%
\stackon[1pt]{#1}{\tmpbox}%
}
\newcommand{\R}{\mathbb{R}}

\newcommand{\Z}{\mathbb{Z}}

\newcommand{\N}{\mathbb{N}}

\newcommand{\C}{\mathbb{C}}

\newcommand{\br}{\langle}
\newcommand{\kt}{\rangle}

\usepackage{graphicx}

\theoremstyle{plain}
\newtheorem{thm}{Theorem}[section]
\newtheorem{prop}[thm]{Proposition}
\newtheorem{lm}[thm]{Lemma}
\newtheorem{cor}[thm]{Corollary}

\theoremstyle{definition}
\newtheorem{defn}[thm]{Definition}

\newtheorem{rem}[thm]{Remark}
\newtheorem{eg}[thm]{Example}
\newtheorem{claim}{Claim}[thm]
\newtheorem{notation}[thm]{Notation}

\newtheorem{algorithm}[thm]{Algorithm}

\newtheorem*{hyp1*}{Hypothesis (I)}
\newtheorem*{hyp2*}{Hypothesis (II)}
\newtheorem*{hyp3*}{Hypothesis (III)}
\newtheorem*{impremark}{Conventions}
\newtheorem{hypothesis}{Hypothesis}[section]
\renewcommand{\thehypothesis}{\thesection.\Alph{hypothesis}}

\newtheoremstyle{exercise}{}{}{\itshape}{}{\bfseries}{:}{.5em}{\thmname{#1} \thmnumber{#2}\thmnote{(#3)}}
\theoremstyle{exercise}

\usepackage{amsfonts}
\DeclareFontFamily{U}{wncy}{}
\DeclareFontShape{U}{wncy}{m}{n}{<->wncyr10}{}
\DeclareSymbolFont{mcy}{U}{wncy}{m}{n}
\DeclareMathSymbol{\sha}{\mathord}{mcy}{"58}

\renewcommand{\d}[2]
{\frac{d#1}{d#2}}

\newcommand{\nc}{\newcommand}
\nc {\Prod}{(\underset{I}\ph\, r_I^{n_I})}
\nc {\wlw}{\wedge\ldots\wedge}
\nc {\olo}{\otimes\ldots\otimes}
\nc{\bea}{\begin{eqnarray*}}
\nc{\eea}{\end{eqnarray*}}
\nc {\e}{\varepsilon}
\nc{\de}{\delta}
\nc{\A}{\hat a}
\nc{\Ad}{\hat a^\dag}
\nc{\grad}{\nabla}
\nc{\nlim}{\underset{n\to\infty}\lim}
\nc{\ilim}{\underset{i\to\infty}\lim}
\nc{\isum}{\sum\limits_{i=1}^N}
\nc{\xx}{\underset{x\to x_0}\lim}
\nc{\Bnrm}{\Big |\Big|}
\nc{\sym}{\text{Sym}}
\nc{\inte}{\overset{\circ}}
\nc{\del}{\partial}
\nc{\plp}{+\ldots+}
\nc{\lre}{\longrightarrow}
\nc{\be}{\begin{equation}}
\nc{\ee}{\end{equation}}
\nc{\CP}{\mathbb{CP}}
\nc{\End}{\text{End}}
\nc{\mf}{\mathfrak}
\nc{\Hom}{\text{Hom}}
\nc{\spec}{\text{Spec}}
\nc{\sub}{\subseteq}
\nc{\weakto}{\rightharpoonup}
\nc{\ph}{\varphi}
\nc{\leqc}{\lesssim}
\nc{\delbar}{\overline \del}
\nc{\Tr}{\text{Tr}}
\nc{\Lhe}{\mathcal L_{(\Phi^{h_\e}, A^{h_\e})}}
\nc{\Yminus}{Y\mathord{-}}

\begin{document}


\title{Gluing $\Z_2$-Harmonic Spinors and Seiberg--Witten Monopoles on 3-Manifolds}

\author{Gregory J. Parker}
\address{Department of Mathematics, UC Berkeley}
\email{gjparker@berkeley.edu}

\begin{abstract}
Given a $\Z_2$-harmonic spinor satisfying some genericity assumptions, this article constructs a 1-parameter family of two-spinor Seiberg--Witten monopoles converging to it after renormalization. The proof is a gluing construction beginning with the model solutions from \cite{PartI}. The gluing is complicated by the presence of an infinite-dimensional obstruction bundle for the singular limiting linearized operator. This difficulty is overcome by introducing a generalization of Donaldson's alternating method in which a deformation of the $\Z_2$-harmonic spinor's singular set is chosen at each stage of the alternating iteration to cancel the obstruction components.

\end{abstract}


\maketitle
\tableofcontents

\section{Introduction}
\label{Introduction}

The Uhlenbeck compactification of the moduli space of anti-self-dual (ASD) Yang-Mills instantons on a compact 4-manifold exemplifies a philosophy for constructing natural compactifications of moduli spaces that is ubiquitous in modern differential geometry. The construction has two main steps:  first, K. Uhlenbeck's compactness theorem shows that any sequence $A_n$ of ASD instantons  subconverges either to another instanton, or to a limiting object consisting of a background instanton $A_\infty$ of lower charge and bubbling data $B$ \cite{Uhlenbeckcompactness,uhlenbeckremovable}.  Second, C. Taubes's gluing results reverse this process, showing that each pair $(A_\infty, B)$ is the limit of smooth ASD instantons \cite{taubesgluing}.  Together, these results allow the construction of boundary charts which endow the moduli space with the structure of a smoothly stratified manifold. This moduli space is the basis for the celebrated applications of Yang-Mills theory to 3 and 4-dimensional topology \cite{DK}.

 C. Taubes's more recent extension of  Uhlenbeck's compactness theorem to $\text{PSL}(2, \C)$ connections introduced a new type of non-compactness to gauge theory \cite{Taubes3dSL2C}, which has since been  shown to be quite general in 3 and 4 dimensional gauge theories. It is exhibited by most generalized Seiberg--Witten (SW) equations \cite{DWDeformations, WalpuskiNotes}, a class of equations that includes the Kapustin--Witten equations \cite{TaubesKWNahmPole, TaubesKW2}, the Vafa--Witten equations \cite{TaubesVW}, the complex ASD equations \cite{Taubes4dSL2C}, the Seiberg--Witten equations with multiple spinors \cite{HWCompactness,TaubesU1SW}, and the ADHM Seiberg--Witten equations \cite{WalpuskiZhangCompactness}. For these equations, a sequence of solutions need not converge but, after renormalization, subconverges either to another solution or to limiting data called a {\it Fueter section} -- a solution of a different elliptic PDE that is usually degenerate, and in many cases non-linear \cite{DoanThesis, Haydys2008nonlinear, TaubesNonlinearDirac}.

It is natural to ask whether this more subtle limiting process can be reversed by a gluing construction. An affirmative answer would provide an essential step in constructing compactifications of the moduli spaces of solutions to generalized Seiberg--Witten equations, which are expected to be necessary to study the conjectured
relations of these equations to the geometry of manifolds \cite{VWOriginalPaper, DonaldsonSegal, WittenFivebranesKnots, WittenKhovanovGaugeTheory, HaydysFSCategory, JoyceAssociatives, HaydysG2SW, DWAssociatives}. 
Such a gluing result would produce, from a given Fueter section $\Phi$, a family of solutions to the corresponding generalized Seiberg--Witten equation that converges to $\Phi$ after renormalization.

The purpose of this article is to prove a gluing result of this form in the case of the two-spinor Seiberg--Witten equations on a compact 3-manifold, where the corresponding Fueter sections are $\Z_2$-harmonic spinors. In most cases, $\Z_2$-harmonic spinors possess a codimension two singular set which is stable under perturbations, along which the relevant linearized operator degenerates. This degenerate operator carries an infinite-dimensional obstruction bundle, making the gluing problem considerably more challenging than most gluing problems in the literature. 

Despite the presence of the infinite-dimensional obstruction bundle, the gluing can still be accomplished for a set of parameters with {\it finite} codimension. Geometrically, this infinite-dimensional obstruction arises because the location of the singular set is a degree of freedom that varies during the limiting process. The same freedom plays an important role in previous work of the author \cite{PartII}, R. Takahashi \cite{RyosukeThesis}, and S. Donaldson \cite{DonaldsonMultivalued} on the deformation theory of $\Z_2$-harmonic spinors. 
To account for it here, deformations of the $\Z_2$-harmonic spinor's singular set are included as an infinite-dimensional gluing parameter. This leads to an infinite-dimensional family of Seiberg--Witten equations coupled to embeddings of the singular set; the first-order effect of deforming the singular set may then be calculated by differentiating this family with respect to the embedding. The crucial result that allows the gluing to succeed is that the linearized deformations of the singular set perfectly pair with the infinite-dimensional obstruction, allowing it to be cancelled.

 Once this analytic set-up is in place, the gluing is accomplished by adapting Donaldson's alternating method \cite{DonaldsonAlternating} to the semi-Fredholm setting. The starting point is an approximate solution constructed by splicing in the model solution constructed in \cite{PartI} on a neighborhood of the singular set. The gluing iteration is then a three-step cycle, which alternates (i) correcting the approximate solution near the singular set, (ii) away from it, and (iii) at the start of each new cycle, deforms the singular set to cancel the obstruction. This procedure is complicated by the fact that the linearized deformation operator displays a loss of regularity, which necessitates refining the approach to deforming the singular set in \cite{PartII,RyosukeThesis, DonaldsonMultivalued} by introducing  specially adapted families of smoothing operators. The end result of the iteration is a family of Seiberg--Witten solutions converging to a given $\Z_2$-harmonic spinor after renormalization. 
 The framework developed here may also be useful for addressing other geometric problems requiring the deformation of a singular set \cite{ DonaldsonMultivalued, SiqiSLag, yuanqiatiyahclass, donaldsonfibrations}.

\vspace{.01cm}
\subsection{The Seiberg--Witten Equations}
\label{section1.1}
Let $(Y,g)$ be a closed, oriented, Riemannian 3-manifold, and fix a $\text{Spin}^c$-structure with spinor bundle $S\to Y$. Choose a rank 2 complex vector bundle $E\to Y$ with trivial determinant, and fix a smooth background $SU(2)$-connection $B$ on $E$. The {\bf two-spinor Seiberg--Witten} equations are the following equations for a pair $(\Psi, A)\in \Gamma(S\otimes_\C E)\times \mathcal A_{U(1)}$ of an $E$-valued spinor and a $U(1)$ connection on $\text{det}(S)$: 
\begin{eqnarray}
\slashed D_{A}\Psi&=& 0\label{prelimSW1} \\
\star F_A +\tfrac{1}{2}\mu(\Psi, \Psi)&=&0,  \label{prelimSW2} 
\end{eqnarray}

\noindent where $\slashed D_{A}$ is the twisted Dirac operator formed using $B$ on $E$ and the Spin$^c$ connection induced by $A$ and the spin connection of $g$ on $S$, $F_A$ is the curvature of $A$, and  $\mu: S\otimes E\to \Omega^1(i\R)$ is a pointwise-quadratic map. The equations are invariant under $U(1)$-gauge transformations. General pairs $(\Psi,A)$ are called {\bf configurations}, and solutions of (\ref{prelimSW1}--\ref{prelimSW2}) are called {\bf monopoles}.

Unlike for the standard (one-spinor) Seiberg--Witten equations, sequences of solutions to 
(\ref{prelimSW1}--\ref{prelimSW2}) may lack subsequences where $\|\Psi\|_{L^2}$ remains bounded, thus no subsequences can converge. The renorm-

\noindent alization procedure alluded to above simply normalizes $\Psi$ in $L^2$ by setting $\Phi=\e \Psi$ where $\e =\tfrac{1}{\|\Psi\|_{L^2}}$. 
\noindent The re-normalized equations become
\begin{eqnarray}
\slashed D_{A}\Phi&=& 0 \label{buSW1} \\
 \star \e^2 F_A +\tfrac{1}{2}\mu(\Phi, \Phi)&=&0 \label{buSW2} \\
\|\Phi\|_{L^2}&=&1, \label{buSW3}
\end{eqnarray}
\noindent and diverging sequences are now described by the degenerating family of equations with parameter $\e\to 0$. A theorem of Haydys--Walpuski \cite{HWCompactness} (Theorem \ref{HWcompactness}) shows that sequences of solutions for which $\e\to 0$ must converge, in an appropriate sense, to a  solution of the $\e=0$ equations, i.e. to a pair $(\Phi, A)$ where $\Phi$ is a normalized harmonic spinor with pointwise values in $\mu^{-1}(0)$, which is a 5-dimensional cone. Up to gauge, such solutions are equivalent (via the Haydys Correspondence, Section \ref{section3.2}) to harmonic spinors valued in a vector bundle up to a sign ambiguity. The latter are $\Z_2$-harmonic spinors, which are the simplest non-trivial type of Fueter section.   

A key feature of convergence for a sequence $(\Phi_i, A_i, \e_i)$ is the concentration of curvature, which gives rise to a singular set. As $\e_i \to 0$, curvature may concentrate along a closed subset $\mathcal Z$ of Hausdorff codimension 2, so that the $L^p$-norm of $F_{A_i}$ diverges on any neighborhood of $\mathcal Z$ for $p>1$. In fact, \cite{HaydysBlowupsets} shows that $\mathcal Z$ represents the Poincar\'e dual of $-c_1(S)$ in $H_1(Y;\Z)$, hence it is necessarily non-empty if the $\text{Spin}^c$ structure is non-trivial. Away from $\mathcal Z$, the connections converge to a flat connection with holonomy in $\Z_2$, the data of which is equivalent to that of a real Euclidean line bundle $\ell \to Y-\mathcal Z$. This limiting connection and line bundle may have holonomy around $\mathcal Z$ that is the remnant of the curvature that has ``bubbled'' away. If this holonomy is non-trivial, the Dirac equation twisted by such a limiting connection is singular along $\mathcal Z$;  a $\Z_2$-harmonic spinor is, more accurately, a solution of such a singular equation on the complement of $\mathcal Z$.

\subsection{$\Z_2$-Harmonic spinors}

\label{section1.2}

Let $(Y,g)$ be as in Section \ref{section1.1}, and now fix a spin structure with spinor bundle $S_0\to Y$. Let $B_0$ be a zeroth order, $\R$-linear perturbation to the spin connection that commutes with Clifford multiplication. Given a closed submanifold $\mathcal Z\subset Y$ of codimension 2, choose a real Euclidean line bundle  $\ell \to Y-\mathcal Z$ and let $A$ be the unique flat connection with $\Z_2$-holonomy on $\ell$. The twisted spinor bundle $S_0 \otimes_\R \ell$ carries a Dirac operator  $\slashed D_A$ twisted by $A$ on $\ell$, and perturbed using $B_0$. A $\Z_2$-{\bf harmonic spinor} is a solution $\Phi \in \Gamma(S_0\otimes_\R \ell )$ of the twisted Dirac equation on $Y-\mathcal Z$ satisfying 
\be \slashed D_{A} \Phi =0 \hspace{1cm}\text{and} \hspace{1cm} \nabla_{A} \Phi \in L^2, \label{1.1} \ee

\noindent where the dependence on the metric $g$ and perturbation $B_0$ are kept implicit in the notation. $\mathcal Z$ is called the {\bf singular set} of the $\Z_2$-harmonic spinor. A $\Z_2$-harmonic spinor is denoted by the triple $(\mathcal Z, A, \Phi)$ where $\mathcal Z$ is the singular set, $A$ is the unique flat connection determined by the line bundle $\ell$, and $\Phi$ is the spinor itself.   

When  $\mathcal Z=\emptyset$ is empty, solutions of (\ref{1.1}) are classical harmonic spinors whose study goes back to the work of Lichnerowicz \cite{Lichnerowicz}, Atiyah-Singer \cite{AS1963}, and Hitchin \cite{HitchinHarmonicSpinors}. When the singular set is non-empty, the twisted Dirac operator is a type of degenerate elliptic operator known as an {\bf elliptic edge operator}. This class of operators has little precedent in gauge theory, but extensive tools for their study have been developed in microlocal analysis \cite{MazzeoEdgeOperators, MazzeoEdgeOperatorsII, grieser2001basics}. Doan--Walpuski established the existence and abundance of solutions with  $\mathcal Z\neq \emptyset$ on compact 3-manifolds in \cite{DWExistence}, and a stronger version of their result appears in \cite{GregSiqi}. Additional examples have been constructed in \cite{TaubesWu, SiqiZ3, MazzeoHaydysTakahashiExamples,Dashen1,Dashen2}.

The equations (\ref{1.1}) do not carry an action of $U(1)$-gauge transformations; there is only a residual action of $\Z_2$ by sign. In particular, (\ref{1.1}) is not a gauge theory. On the other hand (\ref{1.1}) is $\R$-linear, so admits a scaling action by $\R^+$ (note that the assumption that the perturbation $B_0$ is $\R$-linear means that the Dirac operator is also only $\R$-linear in general). $\Z_2$-harmonic spinors are considered as equivalence classes modulo these two actions; the scaling action is eliminated by fixing the normalization condition (\ref{buSW3}, leaving a residual $\Z_2$ action. Notice that (\ref{1.1}) defines $\Z_2$-harmonic spinors without reference to the Seiberg--Witten equations; it is therefore not {\it a priori} clear whether an arbitrary $\Z_2$-harmonic spinor (\ref{1.1}) should arise as a limit of (\ref{buSW1})--(\ref{buSW3}).
 
 \subsection{The Gluing Problem}
 \label{section1.3}

The gluing problem may now be stated more precisely. Fix a $\Z_2$-harmonic spinor $(\mathcal Z_0, A_0, \Phi_0)$.  The goal is to construct a family of solutions $(\Phi_\e, A_\e)$ to (\ref{buSW1}--\ref{buSW3}) for sufficiently small $\e>0$ such that 
 \be (\Phi_\e, A_\e)\lre (\mathcal Z_0, A_0, \Phi_0)\label{limitingspinor}\ee
 \noindent  in the sense of Haydys--Walpuski's compactness theorem (Theorem \ref{HWcompactness}). In particular, this requires reconstructing a smooth $E$-valued spinor $\Phi_\e$ (note $\Phi_0$ is a section of a different bundle of real rank 4 over $\Yminus \mathcal Z$), and re-introducing the highly concentrated curvature by smoothing the singular connection $A_0$. The latter implicitly requires recovering the $\text{Spin}^c$-structure, which is lost in the limit as $\e\to 0$.  

In the simplest case, when $\mathcal Z_0=\emptyset$ and standard elliptic theory applies, Doan--Walpuski \cite{DWDeformations} showed that all classical harmonic spinors arise as limits of a family as in  (\ref{limitingspinor}). Reversing the convergence by a gluing in the singular case $\mathcal Z_0\neq \emptyset$ is far more challenging, and requires new analytic tools for elliptic edge operators and their desingularizations. 
 The concentration of curvature along the singular set $\mathcal Z_0$ as $\e\to 0$ manifests by making  the linearization of the $\e=0$ version of (\ref{buSW1})--(\ref{buSW3}) a singular elliptic edge operator with an infinite-dimensional cokernel. This prevents the application of the standard Fredholm approaches that have historically been used in gluing problems \cite{ DK, KM, MSBig, DonaldsonGluing}. 
 
 The presence of the infinite-dimensional obstruction bundle does not mean that the gluing can only be accomplished for a subset of parameters of infinite codimension. Rather, this obstruction is an artifact arising from inadvertently fixing the singular set $\mathcal Z_0$. In fact, the location of the singular set is a degree of freedom that may also vary as $\e\to 0$. Indeed, work of the author \cite{PartII} and R. Takahashi \cite{RyosukeThesis} has shown that constructing families of $\Z_2$-harmonic spinors with respect to families of  metrics $g_s$ for $s\in \mathcal S$ requires allowing the singular set to depend on $s$. Because the gluing problem is a de-singularization of the same situation, one anticipates the same phenomenon will occur. It is therefore necessary to include space of all possible singular sets nearby $\mathcal Z_0$ as a parameter in the gluing construction. This approach has some precedent in the work of Pacard--Ritor\'e   \cite{PacardRitore} on gluing problems in minimal surface theory arising from the Allen-Cahn and Yang-Mills-Higgs equations \cite{YMHGluing}, though the singular nature of the operators in these situations is more tractable.

As explained above, the main idea of this paper is to show that the deformations of the singular set pair with the infinite-dimensional obstruction to create a Fredholm gluing theory.  Key aspects of this approach rely on the theory developed in \cite{PartI}, and \cite{PartII}, and this article is in some sense the sequel to and culmination of these. The first, \cite{PartII} develops the deformation theory for the singular set for $\Z_2$-harmonic spinors alone, without reference to Seiberg--Witten theory. The second, \cite{PartI} constructs model solutions near the singular set $\mathcal Z$, which are the starting point of the gluing construction. To keep this article self-contained, the relevant parts of \cite{PartII} and \cite{PartI} are reviewed in detail in Sections \ref{section4}--\ref{section6} and Section \ref{section7}, respectively.

\subsection{Main Results}
\label{section1.4}

To state the main result, we first describe our assumptions on the starting data. 

Let $Y$ be a compact, oriented three-manifold. The two-spinor Seiberg--Witten equations (\ref{prelimSW1}--\ref{prelimSW2}) depend on a smooth background parameter pair $p=(g,B)$ of a Riemannian metric and an auxiliary $SU(2)$-connection on $E$ in the space 
\be \mathcal P=\Big\{(g,B) \ | \ g\in \text{Met}(Y) \ , \ B\in \mathcal A_{SU(2)}(E)\Big \}\label{mathcalPdef}\ee
of all such choices. The definition of a $\Z_2$-harmonic spinor (\refeq{1.1}) makes reference to a similar parameter pair $(g,B)$, where $B$ is now a perturbation to the spin connection on $S_0$ that is inherited from the $SU(2)$ connection denoted by the same symbol. The gluing construction can be carried out beginning from a $\Z_2$-harmonic spinor satisfying several conditions; these conditions constrain the parameters to lie in the complement of the locus $\mathcal P'\subset \mathcal P$ admitting $\Z_2$-harmonic spinors with worse singular behavior.   

\medskip

\begin{defn} \label{regulardef}A $\Z_2$-harmonic spinor $(\mathcal{Z}_0, A_0, \Phi_0)$ with respect to a parameter pair $p_0=(g_0, B_0)$ is said to be {\bf regular} if it satisfies the following three conditions: 
\medskip 


\begin{enumerate}
\item[(i)] { {\underline{{\it {Smooth}}}} :} the singular set $\mathcal Z_0\subset Y$ is a smooth, embedded link, and the holonomy of $A_0$ is equal to $-1$ around the meridian of each component.
\smallskip 

\item[(ii)] { {\underline{\it{Isolated}}} :} $\Phi_0$ is the unique $\Z_2$-harmonic spinor for the pair $(\mathcal Z_0, A_0)$ with respect to $p_0=(g_0,B_0)$ up to normalization and sign.
\smallskip 

\item[(iii)]{{\underline{\it{Non-degenerate}}} :} $\Phi_0$ has non-vanishing leading-order, i.e. there is a constant $c>0$ such that $$|\Phi_0| \geq c\cdot \text{dist}(-,\mathcal Z_0)^{1/2}.$$
\end{enumerate}


\end{defn}

In Section \ref{section4}, we show that for a fixed singular set $\mathcal Z_0$, the twisted Dirac operator,
\be \slashed D_{A_0}: H^1(S_0\otimes_\R \ell ) \lre L^2(S_0\otimes_\R \ell)\label{prelimsingulardirac}\ee
\noindent is semi-Fredholm, where $H^1$ denotes the Sobolev space of sections whose covariant derivative is $L^2$ with appropriate weights, and $S_0$ is the spinor bundle of the spin structure hosting $\Phi_0$. For weights that imply the integrability requirement in (\ref{1.1}), this operator is left semi-Fredholm and has infinite-dimensional cokernel. This cokernel gives rise to the infinite-dimensional obstruction of the linearized Seiberg--Witten equations at $\e=0$.

As explained above, cancelling the obstruction requires deforming the singular set.  Any singular set $\mathcal Z$ near $\mathcal Z_0$ defines its own flat connection $A_\mathcal Z$ whose holonomy representation is equal to that of $A_0$ after selecting an isomorphism $\pi_1(\Yminus \mathcal Z_0)\simeq \pi_1(\Yminus \mathcal Z)$ induced by a  homotopy equivalence, thus it defines an accompanying real line bundle isomorphic to the original (up to homotopy). Allowing the singular set to vary over the space of embedded links $\mathcal Z$ gives an infinite-dimensional family of Dirac operators parameterized by embeddings, which we combine into a {\bf universal Dirac operator} 
\be \slashed {\mathbb D}(\mathcal Z,\Phi)=\slashed D_{A_{\mathcal Z}}\Phi, \label{universalDiracdef}\ee
\noindent which is (written appropriately) quasi-linear in the embedding and linear in the spinor. The theory of  the universal Dirac operator was developed in \cite{PartII} and is reviewed and refined in Section \ref{section6}.

The key idea is that the derivative of (\ref{universalDiracdef}) with respect to the embedding should cancel the cokernel of twisted Dirac operator. This was investigated in \cite{PartII} (see also the work of Takahashi \cite{RyosukeThesis} and Donaldson \cite{DonaldsonMultivalued}, which take a different approach). The first main result of \cite{PartII} is the following:
\medskip 

\begin{thm} (\cite[Thm 1.3]{PartII})\label{PartIImainb} Let $(\mathcal Z_0, A_0,\Phi_0)$ be a regular $\Z_2$-harmonic spinor, and let $\Pi_0$ denote the projection onto the cokernel of (\ref{prelimsingulardirac}). Then the cokernel component of the linearization of the universal Dirac operator with respect to deformations of $\mathcal Z_0$ 
\be \Pi_0\circ  (\text{d}_{\mathcal Z_0}\slashed{\mathbb D})_{(\mathcal Z_0,\Phi_0)}: L^{2,2}(\mathcal Z_0; N\mathcal Z_0) \lre \text{Coker}(\slashed D_{\mathcal Z_0})\label{mainamap}\ee

\noindent is (after an isomorphism of the codomain) an elliptic pseudo-differential operator and its Fredholm extension has index  $-1$.  \qed
\end{thm}
\medskip

\noindent More specifically, the parenthetical means the following. Section \ref{section5} shows that there is a complex line bundle $\mathcal C_0\to \mathcal Z_0$ and an isomorphism $\text{Coker}(\slashed D_{\mathcal Z_0})\simeq \Gamma(\mathcal Z_0; \mathcal C_0) \oplus \R$ with the space of sections, up to a 1-dimensional subspace. After composing with this isomorphism,  (\ref{mainamap}) is a map between spaces of sections of vectors bundles on $\mathcal Z_0$ (modulo a 1-dimensional subspace), and the assertion that it is an elliptic pseudo-differential operator then has the standard meaning. The order of this operator depends on the conventions in the choice of the isomorphism, and is order 1/2 in the conventions used below (see Remark  \ref{regrem}). In the statement of the theorem, the domain is to be interpreted as the tangent space at $\mathcal Z_0$ to the space of embedded singular sets of
Sobolev regularity $(2,2)$. 

Using the Fredholm property of (\refeq{mainamap}), we define

\begin{defn} A $\Z_2$-harmonic spinor is {\bf unobstructed} if the Fredholm extension of (\ref{mainamap}) has trivial kernel. \label{unobstructeddef}
\end{defn} 

\medskip

Since the Seiberg-Witten equations in dimension 3 have index 0, one does not expect 1-parameter families of solutions $(\Phi_\e, A_\e)$ for the fixed parameter $p_0=(g_0, B_0)$. Rather, one expects that, along a 1-dimensional path of parameters $p_\tau=(g_\tau, B_\tau)$ equal to $p_0$ at $\tau=0$, there is a converging family of solutions $(\Phi_\e, A_\e)$ for which $\tau$ depends on $\e$ (or vice versa). Once such a path is incorporated, the parameterized  version of (\refeq{mainamap}) has Fredholm index 0; the results of Section \ref{section6} imply that being unobstructed in the sense of Definition \ref{unobstructeddef} implies that the parameterized version is unobstructed in the standard sense (i.e. has trivial cokernel).

The final condition necessary for the gluing result is that $(\mathcal Z_0, A_0, \Phi_0)$ arises from a path of parameters $p_\tau$ with transverse spectral crossing. This makes sense because of the following theorem proved in \cite{PartII} using Theorem \ref{PartIImainb} and the Nash-Moser Implicit Function Theorem.

\medskip 
\begin{thm}\label{SpectralCrossing}(\cite{PartII} Corollary 1.5) Suppose that $p_\tau=(g_\tau, B_\tau)$ is a smooth path of parameters, and that $(\mathcal Z_0, A_0, \Phi_0)$ is a regular, unobstructed $\Z_2$-harmonic spinor with respect to $p_0$. Then, there is a $\tau_0>0$ such that for $\tau \in (-\tau_0, \tau_0)$, there exist $\Z_2$-eigenvectors $(\mathcal Z_\tau, A_\tau, \Phi_\tau)$ with eigenvalues $\Lambda_\tau\in \R$ satisfying
\smallskip
 \be \slashed D_{A_\tau}\Phi_\tau=\Lambda_\tau \Phi_\tau, \label{eigenvectors} \ee 

\smallskip 
\noindent where $\slashed D_{A_\tau}$ is the twisted Dirac operator as in (\refeq{prelimsingulardirac}). Moreover, $(\mathcal Z_\tau, A_\tau, \Phi_\tau)$ are regular and unobstructed for each $\tau\in (-\tau_0, \tau_0)$, and $(\mathcal Z_\tau, A_\tau, \Phi_\tau, \Lambda_\tau)$ depends smoothly on $\tau$.  

\end{thm}

\medskip 
\begin{defn} \label{spectralcrossingdef}
The family of Dirac operators $\slashed D_{A_\tau}$ is said to have {\bf transverse spectral crossing} at $\tau=0$ if the family of $\Z_2$-eigenvectors $(\mathcal Z_\tau, A_\tau, \Phi_\tau,\Lambda_\tau)$ has $$\dot \Lambda(0)\neq 0,$$
\noindent where $\dot \  $ denotes the derivative with respect to $\tau$. 
\end{defn}

\medskip

We now state the main result, which establishes the existence of two-spinor Seiberg-Witten solutions converging to a $\Z_2$-harmonic spinor satisfying the above conditions. 
%
Here, $S_0$ denotes the spinor bundle of a Spin structure that hosts a $\Z_2$-harmonic spinor as in (\ref{1.1}), while $S$ is reserved for the spinor bundle of the $\text{Spin}^c$-structure in the Seiberg--Witten equations. 
\medskip    
\begin{thm}\label{maina} Suppose that $(\mathcal Z_0, A_0,\Phi_0)$ is a regular, unobstructed $\Z_2$-harmonic spinor with respect to a parameter $p_0=(g_0,B_0)$, and that $p_\tau=(g_\tau, B_\tau)$ is a path of parameters such that the corresponding family  (\ref{eigenvectors}) has transverse spectral crossing. 

Then, for each orientation of $\mathcal Z_0$, there is a unique $\text{Spin}^c$ structure with spinor bundle $S \to Y$, an $\e_0>0$, and a family of configurations $(\Psi_\e, A_\e)$ for $\e\in (0, \e_0)$ satisfying the following. 

\begin{itemize}
\item[(A)] The $\text{Spin}^c$ structure is characterized by $$c_1(S)=-\text{PD}[\mathcal Z_0] \hspace{1cm }\text{and }\hspace{1cm} S|_{Y-\mathcal Z_0}\simeq S_0\otimes_\R \ell.$$
\item[(B)] The configurations $(\Psi_\e, A_\e)  \in \Gamma(S_E)\times \mathcal A_{U(1)}$ solve the two-spinor Seiberg-Witten equations (\refeq{prelimSW1}--\refeq{prelimSW2})
\bea
\slashed D_{A_\e}\Psi_\e &=& 0 \\
\star F_{A_\e} + \tfrac12{\mu(\Psi_\e, \Psi_\e)}&=& 0 
\eea
\noindent on $Y$ with respect to $(g_\tau, B_\tau)$, where $\tau=\tau(\e)$ is defined implicitly as a function of  $\e\in (0, \e_0)$ and satisfies either $\tau(\e)>0$ or $\tau(\e)<0$.  
\item[(C)] The spinors have $L^2$-norm \be \|\Psi_\e\|_{L^2(Y)} =\frac{1}{\e},\label{normalization}\ee and after renormalizing by setting $\Phi_\e= \e \Psi_\e$, the pairs $(\Phi_\e,A_\e)$ converge to $(\mathcal Z_0, A_0,\Phi_0)$ in the sense of Theorem \ref{HWcompactness}, i.e. 
$$  \Phi_\e   \to  \Phi_0 \hspace{1cm}\text{and} \hspace{1cm}  A_\e \to  A_0 $$
\noindent in $C^\infty_{loc}(Y-\mathcal Z_0)$ after applying gauge transformations defined on $Y-\mathcal Z_0$, and $|\Phi_\e|\to |\Phi_0|$ in $C^{0,\alpha}(Y)$ for some $\alpha>0$. 
\end{itemize}

\end{thm}

\medskip

\begin{rem}
It is expected that all $\Z_2$-harmonic spinors are regular and unobstructed for generic parameters among those admitting $\Z_2$-harmonic spinors (see Section \ref{section1.5}). We do not undertake the task of establishing the genericity results here. A partial result for the genericity of the non-degeneracy condition in Definition \ref{regulardef} is proved in \cite{SiqiZ3} in the situation of $\Z_2$-harmonic 1-forms. It is straightforward to show (see Remark \ref{remtransverse}) that a generic path $(g_\tau, B_\tau)$ has transverse spectral crossing. 
\end{rem}
\medskip 
\begin{rem}
By a simple diagonalization argument, Theorem \ref{maina} may be extended to the case of a $\Z_2$-harmonic spinor that is instead a limit of $\Z_2$-harmonic spinors satisfying the hypotheses of the theorem. In particular, the singular set of such a limiting $\Z_2$-harmonic spinor need not be an embedded submanifold.

The discussion in the upcoming Section \ref{section1.5} suggests it is likely that {\it every} isolated $\Z_2$-harmonic spinor arises as such a limit.  If this were the case, Theorem \ref{maina}  would imply that every isolated $\Z_2$-harmonic spinor on a compact 3-manifold arises as the limit of Seiberg--Witten monopoles (in fact in multiple ways---one for each $\text{Spin}^c$ structure whose first Chern class is Poincar\'e dual to {\it some} orientation of the singular set).

\end{rem}

\smallskip

\subsection{Wall-Crossing Formulas}
\label{section1.5}

The non-compactness of the moduli space $\mathcal M_{SW}$ of solutions to (\ref{buSW1})--(\ref{buSW3}) prevents the (signed) count of two-spinor Seiberg--Witten monopoles from being a topological invariant. In particular, the compactness theorem (Theorem \ref{HWcompactness}) suggests that along a path of parameters $p_\tau$ such that $p_0$ admits a $\Z_2$-harmonic spinor, a family of monopoles may diverge so that the signed count of solutions changes; in fact, Theorem \ref{maina} shows that this {\it necessarily} happens if the $\Z_2$-harmonic spinor is regular and unobstructed. 

Rather than being a topological invariant, it is conjectured that signed count $\# \mathcal M_{SW}$ is a chambered invariant with wall-crossing formulas. That is, it is conjectured that the subset $\mathcal W_{\Z_2}\subseteq \mathcal P$ of parameters admitting $\Z_2$-harmonic spinors has codimension 1, and it divides its complement in $\mathcal P$ into a collection of open chambers inside which the count is invariant, with a well-defined formula for how the count changes as it crosses the ``wall'' $\mathcal W_{\Z_2}$. This chambered invariant is conjectured to fit into a larger scheme of constructing invariants by summing chambered invariants with cancelling wall-crossing formulas (see \cite{DonaldsonSegal,JoyceAssociatives, HaydysG2SW, DWAssociatives} for details and examples).     

The main result of \cite{PartII} provides a step towards confirming this picture by proving that $\mathcal W_{\Z_2}$ indeed forms a ``wall'' near regular, unobstructed $\Z_2$-harmonic spinors.  

\medskip 
\begin{thm}(\cite[Thm 1.4]{PartII}) Suppose that $(\mathcal Z_0, A_0, \Phi_0)$ is a regular, unobstructed $\Z_2$-harmonic spinor with respect to $p_0\in \mathcal P$. Then there is an open neighborhood $\mathcal U$ of $p_0$ such that , $$\mathcal W_{\Z_2}\cap \mathcal U\subseteq \mathcal P$$
\noindent is a smooth Fr\'echet submanifold of codimension 1. 
\end{thm}
\medskip 
\noindent More generally, it is expected that $\mathcal W_{\Z_2}$ is a stratified space with the following global structure. The top stratum $\mathcal W{^\text{reg}}_{\Z_2}$ should consist of a disconnected Fr\'echet submanifold of codimension 1 whose components are labeled by isotopy classes of embedded links in $Y$, and where each parameter $p\in \mathcal W_{\Z_2}^\text{reg}$ admits a (unique) regular, unobstructed $\Z_2$-harmonic spinor. Confirming this expectation would, in particular, require confirming the prediction of Taubes that the singular set of a $\Z_2$-harmonic spinor is smooth for generic choices of smooth parameters \cite[pg. 9]{Taubes3dSL2C}. Deeper strata of higher finite codimension are expected to consist of the locus where the singular set has the structure of an embedded graph with increasingly complicated self-intersections, and the loci where the regular or unobstructed condition fails. There should also be strata of infinite codimension where wilder singular behavior can occur. The work of \cite{TaubesWu, DWExistence,MazzeoHaydysTakahashiExamples,  MazzeoHaydysTakahashi} support this picture.  

Theorem \ref{maina} shows that the count $\#\mathcal M_{SW}$ changes along a path of parameters crossing $\mathcal W^\text{reg}_{\Z_2}$ transversely, which provides a key step in confirming the conjectured wall-crossing formula. In particular, Theorem \ref{maina} constructs a subset of the parameterized moduli space over $p_\tau$ for either $\tau\geq 0$ or $\tau\leq 0$ that is homeomorphic to a half-open interval $[0,\e_0)$. A complete proof of a wall-crossing formula would additionally require investigating orientations to determine the sign of the crossing as in \cite{DWExistence}, and showing that the homeomorphism to $[0,\e_0)$ determines a boundary chart for the moduli space. The latter involves showing the {\it surjectivity of gluing}, i.e. that the family of monopoles in Theorem \ref{maina} is the unique family converging to $(\mathcal Z_0, A_0, \Phi_0)$. This problem will be addressed in future work.

\begin{rem}\label{remtransverse}
A path $p_\tau=(g_\tau, B_\tau)$ has transverse spectral crossing (Definition  \ref{spectralcrossingdef}) at a regular, unobstructed $\Z_2$-harmonic spinor $(\mathcal Z_0, A_0, \Phi_0)$ if and only if $p_\tau$ intersects $\mathcal W_{\Z_2}^\text{reg}$ transversally. The genericity of this condition follows easily from this characterization as transversality.  
\end{rem}

\subsection{Outline}
\label{section1.6}

The article has 13 sections, culminating in the proof of Theorem \ref{maina}. The proof is accomplished by a gluing iteration that extends Donaldson's alternating method to the semi-Fredholm setting. Briefly, the alternating method decomposes the manifold into two regions $$Y= Y^+\cup Y^-$$
\noindent each of which admits a model solution. These model solutions are spliced together to form a global approximate solution, which is then corrected to a true solution by alternating making corrections localized on $Y^+$ and $Y^-$ using the linearized equations. We take $Y^+=N(\mathcal Z_0)$ to be a tubular neighborhood of the singular set (the ``inner'' region), and $Y^-$ to be the complement of a slightly smaller tubular neighborhood (the ``outer'' region).  

Section \ref{section2} gives a self-contained introduction to the alternating method and explains this generalization. It also lays out the main steps of the proof of Theorem \ref{maina}. Section \ref{losssection}, in particular, explains the main technical challenges. Most important among these is that the operator \eqref{mainamap}  giving the linearized deformation of the singular set displays a {\it {loss of regularity}} \cite{HamiltonNashMoser}. Such a loss typically necessitates the use of Nash-Moser theory, the correct application of which can require delicate setup. Somewhat surprisingly, unlike in \cite{PartII, DonaldsonMultivalued}, the need for the full strength of Nash-Moser theory can be eliminated here, as explained in Section \ref{losssection}.

Section \ref{section3} covers background material on the two-spinor Seiberg--Witten equations, beginning with the compactness theorem for (\ref{buSW1})--(\ref{buSW3}) and the Haydys Correspondence. Sections \ref{section4}--\ref{section6} review and extend the analytic setup for the singular Dirac operator and the deformations of its singular studied in \cite{PartII}. This serves as the linear analysis on outer the region $Y^-$. Section \ref{section6.3} introduces the {``tangential smoothing gauge''}, an infinite-dimensional gauge choice relating to the deformations of the singular set that is the key to eliminating Nash-Moser theory. Section \ref{section7} reviews and extends the model solutions and analysis from \cite{PartI}, which become the linear analysis for the inner region $Y^+$.

Sections \ref{section8}--\ref{section11} then set up and carry out the alternating gluing iteration. Section \ref{section8} constructs the infinite-dimensional family of model solutions parameterized by deformations $\mathcal Z_0$ alluded to in the introduction. This family is used to define a universal version of the Seiberg--Witten equations analogous to Eq. (\ref{universalDiracdef}). Section \ref{section9} calculates the derivative of this family with respect to the deformations of the singular set. Sections \ref{section10}--\ref{section11} then combine this setup with the analysis of Sections \ref{section4}--\ref{section7} in the two regions to complete the gluing. The gluing iteration converges to a smooth two-parameter family of Seiberg--Witten ``eigenvectors'' solving
\be \text{SW}(\Phi, A)=\Big(\Lambda(\tau)+\mu({\e,\tau})\Big)\chi \frac{\Phi_\tau}{\e}\label{1dimobprelim}\ee
\noindent for every pair $(\e,\tau)$, where $\mu(\e,\tau)\in \R$, $\Phi_\tau$ and $\Lambda(\tau)$ are as in Theorem \ref{SpectralCrossing}, and $\chi$ is a smooth function on $Y$. Thus (\ref{1dimobprelim}) are solutions of the Seiberg--Witten equations precisely on the set of pairs $(\e,\tau)$ where $\Lambda(\tau)+\mu(\e,\tau)=0$. Finally, Section \ref{section12} uses a trick due to T. Walpuski \cite{WalpuskiG2Gluing} to show that this condition defines $\tau(\e)$ implicitly as a function of $\e$, completing the proof of Theorem \ref{maina}. 
\subsection*{Acknowledgements}
The author gratefully acknowledges the support and guidance of his Ph.D. advisor, Clifford Taubes throughout this project. This work also benefitted from the interest and expertise of many other people, including Aleksander Doan, Siqi He, Rafe Mazzeo, Tom Mrowka, and Thomas Walpuski. The author is supported by an NSF Mathematical Sciences Postdoctoral Research Fellowship (Award No. 2303102). 


\section{Gluing by the Alternating Method}
 \label{section2}

This section reviews Donaldson's Alternating method for gluing \cite[Sec. 4]{DonaldsonAlternating} and introduces the semi-Fredholm generalization used in the proof of Theorem \ref{maina}. 
These methods are introduced in a reasonably general framework to make clear the structure of the argument and to isolate the principal ingredients in the proof of Theorem~\ref{maina}. 
In order to avoid obfuscating the main ideas, this section defers certain technical details and presents mainly proof sketches. The full technical details in the case of $\Z_2$-harmonic spinors will be filled in by the remainder of the article, with the present section serving as a guide.

\subsection{The Structure of Gluing Problems}
\label{section2.1}

Let $Y$ be a compact manifold, and let $\mathfrak F: \mathcal X\to \mathcal Y$ be a non-linear elliptic PDE viewed as a continuous map between Banach spaces $\mathcal X, \mathcal Y$ of sections of vector bundles. Suppose that $Y$ is the union of two overlapping open regions $$Y= Y^+ \cup Y^-,$$

\noindent each of which hosts a solution $\Phi^\pm$ (called {\it model solutions}) of $\mathfrak F(\Phi^\pm)=0$ on $Y^\pm$ (or more generally, a near solution, $\mathfrak F(\Phi^\pm)\approx 0$).   

The associated {\it gluing problem} is to produce a global solution $\Phi$ of $\frak F(\Phi)=0$ on $Y$ beginning with the model solutions. This is done by splicing the model solutions together to form a global model solution
\be \Phi^{(1)}_\epsilon:= \Phi^+ \  \#_\epsilon  \ \Phi^-\label{initialmodel}\ee
 \noindent on $Y$, which is then corrected to a true solution. Here, $\#_\e$ denotes a splicing procedure, usually performed using a cutoff function, which is allowed to depend on a parameter $\epsilon$ called the {\it gluing parameter(s)}.  The true solution is obtained by a particular {\it gluing procedure}, which is method for correcting approximate solutions,  
 \vspace{-.5cm}
\begin{center}
\be
\label{picture1}
\tikzset{node distance=3.7cm, auto}
\begin{tikzpicture}[decoration=snake]
\node(A){$\Phi^{(N)}_\epsilon \ \ $};
\node(C)[right of=A]{$ \ \ \Phi^{(N+1)}_\epsilon:= \Phi^{(N)}_\epsilon + \delta \Phi^{(N)}_\epsilon$};
\draw[->, decorate] (A) to node {$ \text{correct}$} (C);
\end{tikzpicture}
\ee
\end{center}

\noindent that is applied iteratively to construct a sequence of successively improving approximations  $\Phi^{(1)}_\epsilon, \Phi^{(2)}_\epsilon,...$. The correction (\refeq{picture1}) is most often done by solving a version of the linearized equations, so that the gluing procedure is a variation of Newton's method. 

If the resulting sequence $\Phi^{(N)}_\epsilon \to \Phi_\epsilon$ converges in a sufficient function space for appropriate choices of the gluing parameter $\epsilon$, the limit is a global solution of $\frak F(\Phi_\epsilon)=0$ (or more generally, a family of global solutions parameterized by appropriate choices of $\epsilon$). Often, this procedure is packaged into a suitable version of the Implicit Function Theorem or as a contraction mapping. 

This framework  has been applied in dozens of other well-known gluing problems in geometric analysis. We refer the reader to \cite{DonaldsonGluing} for additional exposition of general approaches to gluing, and to \cite{KM, DK, MSBig, WalpuskiG2Gluing} for related gluing results. A prototypical example is Taubes's method for constructing ASD instantons on 4-manifolds.
\begin{eg}\label{egASD}
(Gluing ASD Instantons, \cite{taubesgluing}) On a compact $4$-manifold $X^4$, one seeks to construct an ASD instanton of charge $k+1$ by gluing two connections $A^\pm$. In this case, one takes $A^+$ to be the standard $k=1$ instanton with dilation parameter $\lambda$ on a ball $X^+=B_{\sqrt{\lambda}}(x_0)$ of radius $\sqrt{\lambda}$ around a bubbling point $x_0 \in X$, and $A^-$ to be an instanton of charge $k$ on the complement $X^-=X^4 - B_{\lambda}(x_0)$ of a smaller ball. Here, the dilation $\epsilon=\lambda$ of the standard instanton is the gluing parameter.  See  \cite{DK,FrU12, TomsNotes} for details. 
\end{eg}

\noindent Notice the following two features of Example \ref{egASD} that will be pertinent for the upcoming case of $\Z_2$-harmonic spinors. (1) the decomposition $X^4=X^+ \cup_\lambda X^-$ depends on the gluing parameter $\lambda$, with $X^+_\circ$ shrinking as $\lambda \to 0$. (2) While the problem has a natural ``invariant scale'' of radius $\lambda$, the two regions $X^\pm$ overlap in a ``neck region'' which extends from radius $\lambda$ to $\sqrt{\lambda}$, and the much of the gluing analysis occurs there.  

Returning now to the case of $\Z_2$-harmonic spinors, let $(Y,g_0)$ be a compact 3-manifold and $(\mathcal Z_0, A_0, \Phi_0)$ a $\Z_2$-harmonic spinor satisfying the hypotheses of Theorem \ref{maina}. We wish to construct a solution of equations (\refeq{buSW1}--\refeq{buSW3}) on $Y$. In this case, the spinor's $L^2$-norm (prior to renormalization) denoted by $\e$ in Eqns. (\refeq{buSW1}--\refeq{buSW3}), becomes one of the gluing parameters, and we take the regions $Y^\pm$ to be a tubular neighborhood of the singular set $\mathcal Z_0$ whose radius depends on the parameter $\e$, and its complement. More specifically, let $\lambda^\pm =\lambda^\pm(\e)$ be radii to be specified shortly (see Appendix  \ref{appendixofnotation}, Definition \ref{insideoutsidedef}), and set
\be Y^+=N_{\lambda^+}(\mathcal Z_0) \hspace{2cm} Y^-=Y-N_{\lambda^{-}}(\mathcal Z_0). \label{insideoutside}\ee
\noindent In this case, we tacitly refer to the region $Y^+$ including the singular set as the  {\bf inside} region, and $Y^-$ as the {\bf outside} region; the overlap $Y^+ \cap Y^-$ is referred to as the {\bf neck} region as in Example \ref{egASD}. The model solutions $(\Phi^+_\e, A^+_\e)$ on $Y^+$ are the concentrating local family of model solutions constructed in \cite{PartI} and reviewed in Section \ref{section7}, and the model solutions on $Y^-$ are simply the limiting $\Z_2$-harmonic eigenvectors $(\Phi_\tau, A_\tau)$ from Theorem \ref{SpectralCrossing}. The full gluing parameter in our situation, after adding deformations of the singular set, consists of triples $\epsilon=(\e, \tau, \mathcal Z)$ of the $L^2$-norm prior to renormalization, the parameter along the path $p_\tau$, and a deformation of the singular set, and there is, more precisely, a decomposition (\refeq{insideoutside}) for each such triple, defined precisely in Section \ref{section8}.  


 \subsection{The Alternating Method}
 \label{section2.2}
 The gluing procedure employed to prove Theorem \ref{maina} is a generalization of the ``alternating method''. This method iteratively corrects approximate solutions by alternating making corrections localized to the two regions $Y^\pm$. This method was first used by S. Donaldson \cite{DonaldsonAlternating} to give a different perspective on C. Taubes's gluing theorem for ASD instantons (Example \ref{egASD}), and is a non-linear analogue of Schwartz's original alternating method \cite{SchwartzAlternating}.

 In the alternating method, the successively approximations (\refeq{picture1}) at the $N^{\text{th}}$ stage have the form \be \Phi^{(N)}_\epsilon= \Phi^{(1)} _\epsilon+ \chi^+\ph^{(N)}_\epsilon + \chi^-\psi^{(N)}_\e\label{alternatingapproximation}\ee 

\noindent where $\ph^{(N)}_\epsilon, \psi^{(N)}_\epsilon$ are corrections supported in the regions $Y^+, Y^-$ respectively, and $\chi^\pm$ are cutoff function restricting them to their respective regions. One then passes from the $N^{\text{th}}$ stage to in two steps

\be \Phi^{(N)}_\epsilon \mapsto \Psi^{(N)}_\epsilon \mapsto \Phi^{(N+1)}_\epsilon \label{standardalternating1}\ee 

\noindent done roughly as follows. \begin{enumerate}
\item[(i)]  The $N^{\text{th}}$ approximation  solves the equation with an error of ${\mathfrak e}_N={\mathcal F}(\Phi^{(N)}_\epsilon)$.  Find a perturbation $\psi$ such that $\Phi^{(N)}_\epsilon+\psi$ solves the equation on $Y^-$.  Set 
$$
\Psi^{(N)}_\epsilon =\Phi^{(N)}_\epsilon+\chi^-\psi.
$$

\item[(ii)] Find a perturbation $\ph$ such that $\Psi^{(N)}_\epsilon+\ph$ solves the equation on $Y^+$, and set $\Phi^{(N+1)}_\epsilon=\Psi^{(N)}+\chi^+\ph$.

\end{enumerate}
The method is so named because with each successive correction, the support of the error term alternates between the regions $Y^+$ and $Y^-$.  The iteration converges to a solution if the errors
${\mathfrak e}_N\to 0$. In the notation of Eq. (\ref{picture1}), one has $\delta \Phi^{(N)}_\epsilon = \chi^+ \delta \ph^{(N)}_\epsilon + \chi^-\delta \psi^{(N)}_\epsilon$ where $\delta \ph^{(N)}_\epsilon=\ph^{(N)}_\epsilon-\ph^{(N-1)}_\epsilon$ and likewise for $\psi$.

 To explain further, let us give a more precise description of the steps of correcting the solution on $Y^\pm$. The equations at a small perturbation $\Phi + \ph$ of an approximate solution $\Phi$ may be written
\be \mathfrak F(\Phi + \ph) = \mathfrak F(\Phi) + \mathcal L_\Phi(\ph) + Q(\ph)\label{FLQ}\ee
\noindent where $\mathcal L_{\Phi}=\text{d}\mathfrak F_{\Phi}$ is the linearization of $\mathfrak F$ at $\Phi$, and $Q$ the higher-order terms. In order order for the alternating method to work, several hypotheses are required. 
\smallskip 

\begin{hypothesis}  \label{hyp1}There are extensions $Y^\pm\subseteq Y^\pm_\circ$ of the two regions, and extensions $\mathcal L^{\pm}_{\Phi}$ of the linearized operators that are uniformly invertible in the gluing parameter $\epsilon$ on appropriate function spaces. 
\end{hypothesis}

\smallskip

\noindent Which extensions and function spaces are appropriate depends strongly on the context. Most often, the extensions $Y^\pm_\circ$ of the two regions are obtained either by attaching a tubular ends to $\del Y^\pm$ or imposing boundary conditions.

In order for the iterates (\refeq{alternatingapproximation}) to converge to a solution, the following second hypothesis is also required. 

\begin{hypothesis} There is a $0<\delta<1$ such that if $\text{supp}(g^\pm)\subseteq \text{supp}(d\chi^\mp)$, then there is a constant $C$ such that the solutions of $\mathcal L^+_{\Phi} \ph=g^+$ and $\mathcal L^-_{\Phi}\psi=g^-$ obey 

 \be \| d\chi^+.\ph \| \leq C\delta \|g^+ \| \hspace{2cm} \| d\chi^-. \psi \| \leq C\delta \|g^- \| \label{deltadef1}\ee
 
 \smallskip 
 
\noindent where $d\chi^\pm. $ is shorthand for the application of the principal symbol $\sigma_{\mathcal L}(d\chi^\pm)$ of $\mathcal L_{\Phi}$.   
\label{hyp2}
\end{hypothesis}

\noindent This hypothesis ensures, provided $\delta$ is sufficiently small, that each successive error term in the alternation is smaller than the previous one. In Section \ref{section4.2}, we will use weighted Sobolev spaces to define the norm in (\refeq{deltadef1}). The inequalities (\refeq{deltadef1}) imply that solutions decay away from their support across the neck region as illustrated in Figure \ref{Fig1} below. 

Finally, as with all non-linear problems, one must assume appropriate control of the non-linearity: 

\begin{hypothesis} \label{hyp3} The non-linear term $Q$ in Eq. (\refeq{FLQ}) is sufficiently mild that the Implicit Function Theorem guarantees there are unique solutions to the non-linear equations 
$$(\mathcal L^{+}_{\Phi} + Q)\ph=\frak e_N|_{Y^+_\circ} \hspace{2cm} (\mathcal L^{-}_{\Phi} + Q)\psi=\frak e_N|_{Y^-_\circ}  $$
\noindent  where $\frak e_N:=\mathfrak F(\Phi^{(N)}_\epsilon) $ is the approximation at the $N^{\text{th}}$ stage and $\frak e_N|_{Y^\pm}$ its restriction to the two regions via appropriate cutoff functions.   Hypothesis \ref{hyp1} guarantees that Implicit Function Theorem can be applied. 

In addition, we assume that there is an $r_0$ such that $\|Q(\psi)\|\leq C\delta \|\psi\|$ for all $\|\psi\|\leq r_0$ where $\delta,C$ are the constant in Hypothesis \ref{hyp2}. 
\end{hypothesis}

\begin{figure}[h!]
\begin{center}
\begin{picture}(200,160)
\put(-60,0){\includegraphics[scale=.37]{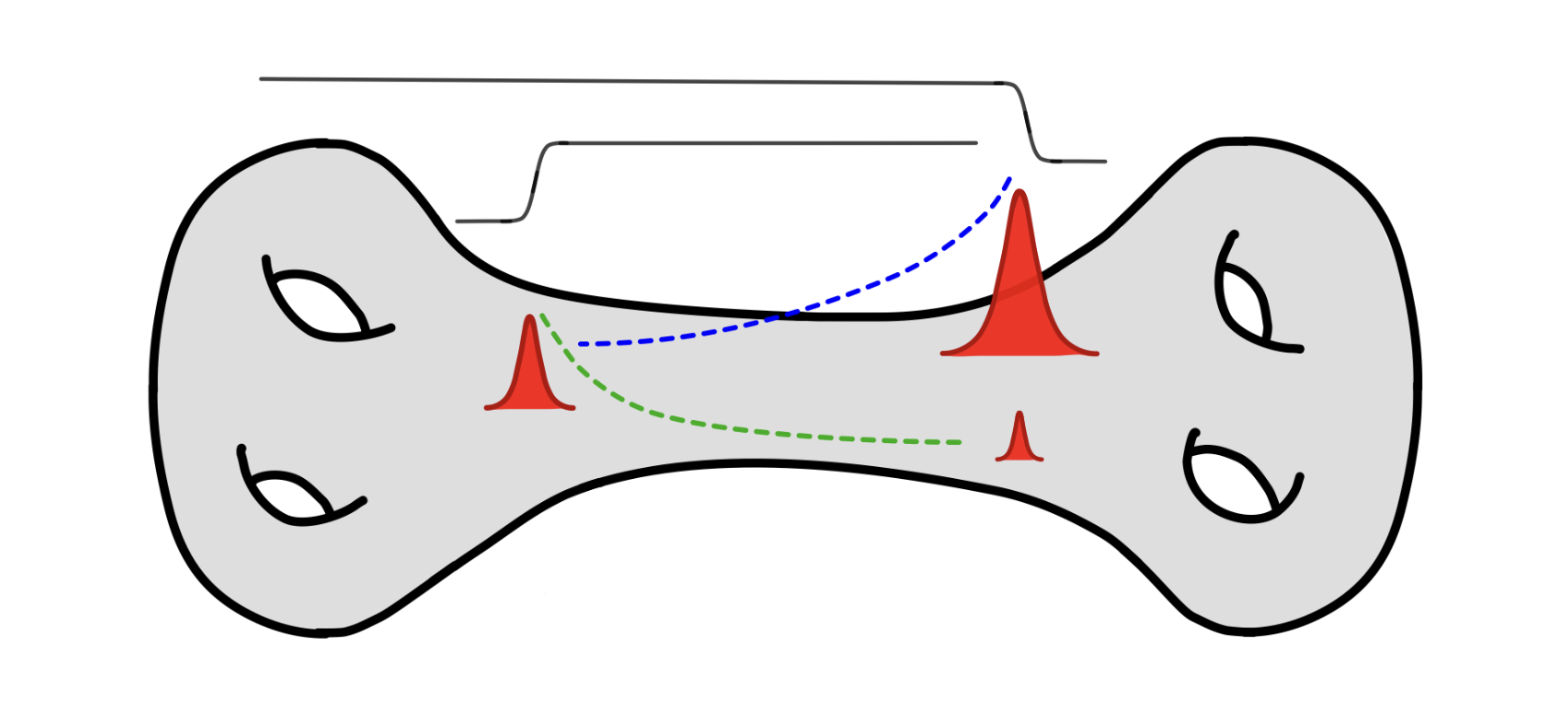}}
\put(133,132){ $\chi^+$}
\put(35,115){ $\chi^-$}
\put(-45,105){ \large $Y^+$}
\put(220,105){ \large $Y^-$}
\put(155,80){ \color{red}$\frak e_N$}
\put(24,68){ \color{red}$\frak e'_{N}$}
\put(148,55){ \color{red}$\frak e_{N+1}$}
\label{Fig1}
\end{picture}
\caption{
An illustration of the alternating iteration in Algorithm \ref{alternation}. (Top) The cutoff functions $\chi^\pm$, (red) the error terms with alternating support and decreasing norm, (blue/green) the pointwise decay of solutions across the neck region.}
\end{center}
\end{figure}

With these hypotheses, the alternating method is the following algorithm. 
 
 \medskip 
  \begin{algorithm} \label{alternation}Suppose that Hypotheses \ref{hyp1}--\ref{hyp3} holds, and that $(C\delta)<1$ in (\refeq{deltadef1}). Then the following iteration scheme converges to a global solution $\Phi$ of \be \mathfrak F(\Phi)=0\label{theeq1}\ee
 
 \noindent on $Y$. 
 
 Let $\frak e_1$ denote the initial error term in Hypothesis \ref{hyp3}, and assume $\text{supp}(\frak e_1)\subseteq \text{supp}(d\chi^+)$. The alternating method constructs the sequence in Eq. (\refeq{alternatingapproximation}) are inductively as follows, beginning from the approximate solution $\Phi^{(1)}_\epsilon$. 
 
 \medskip 
\begin{enumerate}
\item[(1)] Given an error term $\frak e_N$ with $\text{supp}(\frak e_N)\subseteq \text{supp}(d\chi^+)$, let $\psi$ be the unique solution of \be(\mathcal L^-_{\Phi^{(1)}} +Q) \psi = -\frak e_N\label{mathcalL-}\ee
\noindent on $Y^-$ whose existence is guaranteed by Hypotheses \ref{hyp1}, \ref{hyp3} and the Implicit Function Theorem. 
\medskip 
\item[(2)] Set $\psi^{(N+1)}_\epsilon=\psi^{(N)}_\epsilon+\psi$, and $\Psi^{(N)}_\epsilon=\Phi^{(1)}_\epsilon+ \chi^+\ph^{(N)}_\epsilon + \chi^-\psi^{(N+1)}_\epsilon$. This intermediate approximation satisfies $$\mathfrak F(\Psi^{(N)}_\epsilon)=\cancel{ \frak e_N - \frak e_N }+ d\chi^-. \psi + g_N  =:\frak e_N'$$
\noindent where $g_N= \chi^-Q(\psi) -Q(\chi^-\psi)$.
\medskip 
  
\item[(3)]  $\frak e_N'$ is supported where $d\chi^-\neq 0$, and $$\|\frak e_N'\| \leq  \|d\chi^-.\psi\| \ + \ \|g_N\| \leq C \delta \| \frak e_N\|,$$

\noindent by applying Hypotheses \ref{hyp2} and \ref{hyp3} on the two terms respectively. 
\medskip 
\item[(4)] Repeat steps (1)--(3) on $Y^+$ to obtain a correction $\ph$ so $\ph^{(N+1)}_\epsilon=\ph^{(N)}_\epsilon + \ph$, then set $$\Phi^{(N+1)}_\epsilon=\Phi^{(1)} _\epsilon+ \chi^+\ph^{(N+1)}_\epsilon +\chi^-\psi^{(N+1)}_\epsilon. $$
The resulting error, $\frak e_{N+1}$, then satisfies $\|\frak e_{N+1}\|\leq (C\delta)^2 \|\frak e_{N}\|$. 
\end{enumerate}

\noindent Since we assumme $(C\delta)<1$, the iterates $\Phi^{(N+1)}_\epsilon$ converge by geometric series, and the limit obeys (\refeq{theeq1}) by continuity. 
 \end{algorithm}

 The main advantage of the alternating method over other gluing procedures, and the reason it is suitable in our setting, is that it can effectively treat  asymmetry between the two regions $Y^\pm$: it only requires analysis of $\mathcal L^\pm$ in the two distinct regions separately,  and never the analysis of a global linearization uniformly invertible on $Y$. This allows the asymmetric character of the equation in the two regions to be isolated and analyzed separately. \cite{PartI} and \cite{PartII} should be viewed as manuals for the Seiberg--Witten theory on $Y^+$ and $Y^-$ respectively.

\begin{rem}\label{nonlinearsolvingremark}
A slight variation on Steps (1)--(4) in Algorithm \ref{alternation} above is to solve the strictly linear equation at each step. Thus Step (1) is replaced by solving $$\mathcal L^-_{\Phi^{(1)}} \psi = - \frak e_N + Q(\psi^{(N)}_\epsilon+ \ph^{(N)}_\epsilon)$$
\noindent where $Q$ denotes the higher-order errors from the correction at the previous stage. This formulation is equivalent, though it comes at the cost of disrupting the fact that the error terms are entirely supported where $d\chi^\pm\neq 0$.  
In the case of $\Z_2$-harmonic spinors, the higher-order terms are sufficiently mild that this variation of Algorithm \ref{alternation} is simpler.  \end{rem}

  \subsection{The Semi-Fredholm Alternating Method}
 \label{section2.4}
 
 The alternating method as implemented in Algorithm \ref{alternation} is not sufficient for the gluing problem for $\Z_2$-harmonic spinors. As explained in the introduction, the singular Dirac operator in Eq. (\refeq{1.1}) is an elliptic edge operator with infinite dimensional cokernel, thus Hypothesis \ref{hyp1} fails badly in the outside region $Y^-$. In order to accomplish the gluing required for Theorem \ref{maina}, we adapted the alternating method to the semi-Fredholm setting, with an additional infinite-dimensional gluing parameter (denoted by $\xi$ below) corresponding to deformations of the singular set, which is used to cancel the obstruction. 
 
 We replace Hypothesis \ref{hyp1} with the following, weaker version. 
 
 \begingroup
\renewcommand{\thehypothesis}{\thesection.A$'$}
\begin{hypothesis}\label{hyp1prime}
There are appropriate extensions $Y^\pm_\circ\subseteq Y^\pm$ of the two regions, and appropriate extensions $\mathcal L^{\pm}_{\Phi}$ of the linearized operators such that the following hold on appropriate function spaces: 
 
 \begin{enumerate}
 \item[(A)] $\mathcal L_{\Phi^{(1)}}^+$ is uniformly invertible in the gluing parameter $\epsilon$ (i.e. Hypothesis \ref{hyp1} holds on $Y^+$). 
 \smallskip 
 \item[(B)] $\mathcal L_{\Phi^{(1)}}^-$ is left semi-Fredholm, with left-inverses uniformly bounded in the gluing parameter $\epsilon$. 
 \end{enumerate}
 
 \noindent Here, the linearizations $\mathcal L_{\Phi^{(1)}}^\pm$ are taken at the initial approximate solution $\Phi^{(1)}_\epsilon$. 
 
\end{hypothesis}
\addtocounter{hypothesis}{-1} 
\endgroup

In the case of $\Z_2$-harmonic sinors, $\mathcal L_{\Phi}^+$ is the linearization at the model solution on $Y^+$, analyzed in detail in Section \ref{section7}, and $\mathcal L_{\Phi}^-$ is given in Eq. (\refeq{blockdiagonaldecomp}), and includes the singular Dirac operator (\refeq{1.1}) as a direct summand.  The latter is analyzed in detail in Sections \ref{section4} and Section \ref{section7.3}.

We assume in this semi-Fredholm case that Hypothesis \ref{hyp2}-\ref{hyp3} hold as before, where $g^-$ in Hypothesis \ref{hyp2} lies in the range of the semi-Fredholm linearization. Finally, we require the final hypothesis about the existence of an infinite-dimensional gluing parameter canceling the obstruction. 

\begin{hypothesis}\label{hyp4}There is an infinite-dimensional gluing parameter parameterized by an open neighborhood of the origin in a Banach space $\frak W$, and an accompanying family of operators $F_\xi$ for $\xi \in \frak W$ such that the ``universal'' PDE 
\be \mathbb F(\xi, \ph):= \mathfrak F(F_\xi (\Phi^{(1)}_\epsilon + \ph))\label{universalF}\ee

\noindent obeys the following.  
\begin{enumerate}
\item[(A)] The projection $T_{\Phi}:=\Pi_0\circ \text{d}_\xi \mathbb F_{\Phi}$ of the partial derivative with respect to the parameter $\xi$ extends to a bounded linear isomorphism 
\be T_{\Phi}: \frak W \lre \text{Coker}(\mathcal L^-_{\Phi}), \label{abstractdeformation}\ee
\noindent with uniformly bounded inverse (in both $\xi,\epsilon$) for any $\Phi$ sufficiently close to $\Phi^{(1)}_\epsilon$,  where $\Pi_0$ denotes the projection to the cokernel.  

\item[(B)] The component of the same derivative along the range of $\mathcal L_{\Phi}^-$ obeys
\be \|(1-\Pi_0)\circ\text{d}_\xi\mathbb F_{ \Phi} (\xi)\| \leq C_\epsilon\|\xi\|_{\frak W}, \label{Cepsilondef}  \ee

\noindent uniformly in $\xi$, again for any $\Phi$ sufficiently close to $\Phi^{(1)}_\epsilon$.   
\end{enumerate}

\noindent Finally, we require that all the corrections of the iteration process are sufficiently close to the original approximation that (A),(B) hold for all. 
\end{hypothesis}

 Note we do not require this bound in (B) to be uniform in $\epsilon$, and indeed it is not in the case of $\Z_2$-harmonic spinors. In the setting of $\Z_2$-harmonic spinors, the parameters $\xi$ are appropriately chosen sections of the normal bundle $N\mathcal Z_0$, and $F_\xi$ is a diffeomorphism on $Y$ that deforms the singular set to the graph of $\xi$. The analysis of the linearization with respect to this parameter, and of the projection map (\refeq{abstractdeformation}) occupy Sections \ref{section5}--\ref{section6.3}. The correct manifestation of Part (B) appears later, in Section \ref{section11pre}.

Given that the hypotheses in this semi-Fredholm setting are satisfied, the iteration now proceeds as follows. The approximate solutions in the sequence (\refeq{alternatingapproximation}) now also depends on the parameter $\xi$ that is adjusted in each step of the iteration, so that
 \be  F_{\xi^{(N)}_\epsilon} \left( \Phi^{(N)}_\epsilon\right)= F_{\xi^{(N)}_\epsilon}\Big(\Phi^{(1)}_\epsilon + \chi^+\ph^{(N)}_\epsilon + \chi^-\psi^{(N)}_\epsilon\Big),  \label{Fetadiffeo}\ee

\noindent are a sequence of successively improving approximate solutions. The alternating iteration accordingly becomes a three-stage cycle, with the $N^\text{th}$ stage updating both $\xi^{(N)}_\epsilon$ and $\Phi^{(N)}_\epsilon$ as follows. 

\bigskip
\medskip

 \begin{enumerate}
\item[{\bf (Deform)}]  The $N^{\text{th}}$ approximation  solves the equation with an error of ${\mathfrak e}_N=\mathbb F(, \xi^{(N)}_\epsilon,\Phi^{(N)}_\epsilon)$.  Let $\xi$ be such that $$\mathbb F(\xi^{(N)}_\epsilon+\xi , \Phi^{(N)}_\epsilon) \in \text{Range}(\mathcal L^{-}_{\Phi^{(1)}})$$

\noindent is in the range of the semi-Fredholm inverse on $Y^-$, i.e. update $\xi$ so that the cokernel component of $\frak e_N$ is cancelled. Set $\xi^{(N+1)}_\epsilon=\xi^{(N)}_\epsilon+\xi$. 

\bigskip
\item[{\bf (Outside)}] Let $\psi$ be such that $\Phi^{(N)}_\epsilon + \psi$ cancels the error in the outside region $Y^-$, as following Eq. \refeq{standardalternating1}). 
\bigskip
\item[{\bf (Inside)}]Let $\ph$ be such that $\Psi^{(N)}_\epsilon + \ph$ cancels the error in the inside region $Y^+$, as following Eq. \refeq{standardalternating1}). 

\end{enumerate}

\medskip

We combine these into the following semi-Fredholm gluing iteration, extending Algorithm \ref{alternation}. We once again only provide a proof sketch in this abstract setting.

 \begin{algorithm} \label{alternation2}Suppose that Hypotheses \ref{hyp1prime}, \ref{hyp2}, \ref{hyp3}, and \ref{hyp4} hold, and that $(C_\epsilon C\delta^2 )<1$, for the constants $C,\delta$ of (\refeq{deltadef1}) and $C_\epsilon$ of (\refeq{Cepsilondef}). Then the following iteration scheme converges to a global solution $\Phi$ of \be \mathfrak F(\Phi)=0\label{theeq}\ee
 
 \noindent on $Y$. 
 
 Let $\frak e_1$ denote the initial error term in Hypothesis \ref{hyp3}, and assume $\text{supp}(\frak e_1)\subseteq \text{supp}(d\chi^+)$. The alternating method constructs the sequence in Eq. (\refeq{alternatingapproximation}) are inductively as follows, beginning from the approximate solution $(0,\Phi^{(1)}_\epsilon)$.

 \medskip 
\begin{enumerate}
\item[(0)] By Hypothesis \ref{hyp1prime}, let $\xi$ denote the unique solution of \be T_{\Phi^{(1)}}(\xi)=- \Pi_0(\frak e_N),\label{solvingeta_N}\ee
where $\Pi_0$ is the projection to the cokernel of $\mathcal L^{-}_{\Phi^{(1)}}$, and set $\xi^{(N+1)}_\epsilon=\xi^{(N)}_\epsilon+\xi$. 

\item[(0')] Let $$\frak e_N':= (1-\Pi_0)\frak e_N + (1-\Pi_0)\text d\mathbb F_{\Phi^{(1)}}(\xi,0),$$

\noindent which obeys $\|\frak e_N'\|\leq C_\e \|\frak e_N\|$ by Part (B) of Hypothesis \ref{hyp4}. 
\item[(1)] By Hypotheses \ref{hyp1prime} and \ref{hyp3}, let $\psi$ be the unique solution of \be(\mathcal L^-_{\Phi^{(1)}} +\Pi_0\circ Q) \psi = -\frak e'_N\label{mathcalL-}\ee
\noindent on $Y^-$ given by the Inverse Function Theorem. Absorb $(1-\Pi_0)Q(\psi)$ into $\frak e_{N+1}$. 
\medskip 
\item[(2--4)] The remaining steps proceed precisely as in Algorithm \ref{alternation} and result in approximate solutions$$\Phi^{(N+1)}_\epsilon=F_{\xi^{(N+1)}_\epsilon}\Big(\Phi^{(1)}_\epsilon + \chi^+\ph^{(N+1)}_\epsilon +\chi^-\psi^{(N+1)}_\epsilon\Big). $$
The resulting error, $\frak e_{N+1}$, then satisfies $\|\frak e_{N+1}\|\leq (C_\epsilon C\delta)^2 \|\frak e_{N}\|$. 
\medskip 
  
\end{enumerate}

\noindent By the assumption that $(C_\epsilon C\delta^2)<1$, the iterates $\Phi^{(N+1)}_\epsilon$ and $\xi^{(N+1)}_\epsilon$ both converge by geometric series, and the limit obeys $\mathbb F(\xi_\epsilon, \Phi_\epsilon)=0$ by continuity. By the definition (\refeq{universalF}), this gives a corresponding solution $\frak F(F_{\xi_{\epsilon}}(\Phi_\epsilon))=0$ of (\refeq{theeq}). 
 \end{algorithm}

We make several remarks on the implementation of this iteration scheme to our gluing problem for $\Z_2$-harmonic spinors. First, in our case the gluing depends on two one-dimensional parameters $\epsilon=(\e,\tau)$ given by the $L^2$-norm in (\refeq{buSW1}--\refeq{buSW3}) and the parameter $\tau$ along the path in Theorem \ref{maina}, as well as the infinite-dimensional parameter $\xi$. Second, we use the equivalent but less succinctly phrased alternative scheme described in Remark \ref{nonlinearsolvingremark}, where only the linearization is solved. Third, in our setting, there is a slight asymmetry between the linear analysis of the regions $Y^\pm$, resulting in Hypothesis \ref{hyp2} being satisfied for two different decay rates $\delta^\pm$ respectively. Thus the convergence condition becomes \be (C C_\epsilon \delta^+ \delta^-) <1. \label{convergence}\ee

\noindent Finally, we observe that the assumptions of uniformity in $\e$ in Hypotheses \ref{hyp1prime} can be weakened, provided any non-uniformity can be lumped into constants $C_\e$ so that (\refeq{convergence}) still holds. Ultimately, we will find that for some $\gamma<<1$,

\be  C_\e = \e^{-1/12-\gamma} \hspace{1.5cm}\delta^+=\e^{1/24-\gamma} \hspace{1.5cm}\delta^-=\e^{1/12-\gamma},\label{deltapm}\ee

\noindent Thus (\refeq{convergence}) holds for $\e$ sufficiently small. We indicate lemmas and propositions that fill in pieces of the above algorithm as the later sections proceed.

\subsubsection{Alternation as a Contraction Mapping} 
\label{contractionsubspacessection}The alternating method can be rephrased using the language of parametrices and contraction mappings. This perspective is useful for establishing the uniqueness of the solution in given function spaces, and its smooth dependence on the gluing parameter $\epsilon$. It is not immediately clear in which function space the alternating method becomes a contraction. Indeed, there is an apparent infinite-dimensional ambiguity in the construction, coming from the fact that the function spaces on the extensions $Y^\pm$ overlap across the neck region. This ambiguity can be resolved by viewing the gluing as a non-linear analogue of the excision principle for elliptic operators, in which one also solves the ``virtual'' gluing problem on the neck region. This perspective on gluing is described in \cite{TomsNotes}. 

Careful scrutiny of Algorithm \ref{alternation2} shows that the correction terms all lie in the following function space. Let $1=\bold 1^+ + \bold 1^-$ be indicator functions of disjoint regions whose union is $Y$ such that $\text{supp}(d\chi^\pm)\subseteq \text{supp}(\bold 1^\pm)$. Fix Hilbert spaces $H^\pm(Y)$ so that for $\Phi=\Phi^{(1)}_\epsilon$, the operator $\mathcal L^+_{\Phi^{(1)}}: H^+(Y^+)\to L^2(Y^+)$ is invertible with inverse $P^+$, and $\mathcal L^-_{\Phi^{(1)}}: H^-(Y^-)\to L^2(Y^+)$ has left-inverse $P^-$ as in Hypothesis \ref{hyp4}. Set 
\smallskip
\be   \mathcal H=\left\{ \begin{pmatrix} P^+ (g\bold 1^+) \ \ \ \   \ \ \ \ \ \ \   \\ 
 P^-((1-\Pi_0)g\bold 1^-)\end{pmatrix}  \ \Big | \  g \in L^2(Y)\right\} \subseteq H^{+}(Y^+) \oplus H^-(Y^-). \ee 
\smallskip

\noindent (One might alternatively replace $L^2(Y)$ with a higher-regularity space and the indicator functions $\bold 1^\pm$ with a smooth partition of unity via bootstrapping). Sections in this space may be viewed as section on the closed manifold $Y$ via (\refeq{Fetadiffeo}) so that there is a $\xi$-parameterized family of maps 
\bea  \frak W \oplus \mathcal H & \to& H(Y) \\ 
(\xi, \ph, \psi) &\mapsto&  F_{\xi}\Big(\Phi^{(1)}_\epsilon + \chi^+\ph + \chi^-\psi \Big)
\eea

\noindent into a suitable global function space $H(Y)$ on $Y$. It is on the domain $\frak W\oplus \mathcal H$ of this map that the alternating iteration scheme constitutes a contraction mapping. 

This contraction may be written more precisely as follows. Let $P_\xi$ denote the inverse of the operator (\refeq{abstractdeformation}) guaranteed by Hypothesis \ref{hyp4}, and continue to denote the inverses above by $P^\pm$. The updates to the triples in the three stages of the cycle following Eq. (\refeq{Fetadiffeo}) can then be written 
\bea
(\xi^{(N+1)}_\epsilon, \ph^{(N)}_\epsilon, \psi^{(N)}_\epsilon)&=& \left(\text{Id} -  P_\xi \circ \mathbb F\right)\left(  \xi^{(N)}_\epsilon, \ph^{(N)}_\epsilon, \psi^{(N)}_\epsilon \right) \\
(\xi^{(N+1)}_\epsilon, \ph^{(N)}_\epsilon, \psi^{(N+1)}_\epsilon)&=& \left(\text{Id} -  P^- \circ \mathbb F\right) \left(\text{Id} -  P_\xi \circ \mathbb F\right)\left(  \xi^{(N)}_\epsilon, \ph^{(N)}_\epsilon, \psi^{(N)}_\epsilon \right) \\ 
(\xi^{(N+1)}_\epsilon, \ph^{(N+1)}_\epsilon, \psi^{(N+1)}_\epsilon)&=& \left(\text{Id} -  P^+ \circ  \mathbb F\right)  \left(\text{Id} -  P^- \circ  \mathbb F\right) \left(\text{Id} -  P_\xi \circ \mathbb  F \right)\left(  \xi^{(N)}_\epsilon, \ph^{(N)}_\epsilon, \psi^{(N)}_\epsilon \right)
\eea
\noindent where $ \mathbb F(\xi, \ph, \psi)=\mathfrak F( F_\xi(\Phi^{(1)}_\e + \chi^+ \ph + \chi^- \psi))$ is shorthand for evaluation on (\refeq{Fetadiffeo}). Therefore, the alternation constructs an approximate inverse  $\mathbb A: L^2(Y)\to \frak W\oplus \mathcal H$ defined by
 \begin{eqnarray} \mathbb A &= &P_\xi\  +  \  P^- (\text{Id} - \    \mathbb F\circ  P_\xi)  \  \ + \ \  P^+ \left(\text{Id} -  \  \mathbb F \circ   P_\xi  \ - \   \mathbb F\circ  P^- (\text{Id} \  - \  \mathbb F\circ   P_\xi) \right) \label{boldAprelimdef}
 \end{eqnarray} 
 
 \noindent so that a single application of
\begin{eqnarray}\mathbb T \ &:=& \  \text{Id}-\mathbb A\circ  \mathbb F \nonumber \\ & \   = &    \left(\text{Id} -  P^+ \circ \mathbb F\right)  \left(\text{Id} -  P^- \circ  \mathbb F\right) \left(\text{Id} -  P_\xi \circ   \mathbb  F\right) \label{nestedT}
\end{eqnarray}
as a map $\mathbb T: \frak W\oplus \mathcal H\to \frak W\oplus \mathcal H$ constitutes a full cycle of the alternating iteration. The proof of Theorem \ref{maina} shows that the appropriate version of (\refeq{nestedT}) defined in Section \ref{section11} is indeed a contraction. 
 
\subsection{Eliminating Nash-Moser Theory}
\label{losssection}

\label{section2.5}
As explained in the introduction, the deformation operator in Theorem \ref{PartIImainb} displays a loss of regularity. This phenomenon forces the proof of Theorem \ref{SpectralCrossing} in \cite{PartII} (and the related proof in \cite{DonaldsonMultivalued}) to proceed in the tame Fr\'echet category, using Nash-Moser theory. Because the operator $T_{\Phi}$ in Hypothesis \ref{hyp4} is closely related to (\refeq{mainamap}), one may at first fear that the full technical set-up of Nash-Moser theory is necessary in this setting as well. This would, in particular, complicate the alternating iteration scheme by requiring additional steps applying smoothing operators. This subsection explains how the gluing problem for $\Z_2$-harmonic spinors has hidden regularizing properties that eliminate the need for Nash-Moser theory. Nevertheless, the loss of regularity rears its head in other ways, and dealing with its remnant on the Sobolev spaces chosen still constitutes the main technical challenge in the proof of Theorem \ref{maina}.

\subsubsection{Smoothing Projection Operators} \label{smoothingprojections}The singular deformation problem for $\Z_2$-harmonic spinors (Theorem \ref{SpectralCrossing}) is, in some sense, the $\e=0$ limit of the gluing problem. In this limit, as explained in the introduction, there is an infinite-dimensional obstruction to solving the Dirac equation, and the derivative in Theorem \ref{PartIImainb} can be used to cancel this obstruction. More precisely, because of the loss of regularity, it can be used to cancel a dense subset of the obstruction given by the intersection with a higher regularity space, and smoothing operators are required to smooth each successive error term back into this dense subset. The primary reason the loss of regularity can be eliminated in the present setting is that in the gluing problem for $\e>0$, {\it the infinite-dimensional obstruction --- i.e. $\text{coker}(\slashed D_{A})$ --- is replaced by a large finite-dimensional obstruction whose dimension grows with $\e$}. In fact, with our choices, it has $\text{dim}=\e^{-1/2}$. This finite-dimensional subspace consists entirely of smooth elements, thus no smoothing operations on error terms are needed to return them to the dense, higher-regularity subspace which can be cancelled.

This reduction to a finite-dimensional space comes from a fundamental geometric property of semi-Fredholm edge operators, which was first described in \cite[Sec. 6]{PartII}. It is most easily understood for the Dirac operator (\refeq{1.1}) in the model setting that $Y=S^1\times \R^2$ with its flat product metric, $\mathcal Z=S^1 \times\{(0,0)\}$, and $\ell \to \Yminus\mathcal Z$ is the mobius bundle. In this product case, the infinite-dimensional obstruction is spanned in $L^2$ by linear combinations of the singular harmonic spinors 
\be \Psi_\ell^\circ= \sqrt{|\ell|}e^{i\ell t}e^{-|\ell| |z|}\begin{pmatrix}\tfrac{1} {\sqrt{z}} \\ \tfrac{\text{sgn}(\ell)}{ \sqrt{\overline z}}\end{pmatrix}\label{precokernel}\ee

\noindent where $(t,x,y)$ are cylindrical coordinates with $z=x+iy$, and $\ell \in \Z$ indexes the Fourier modes tangential to $\mathcal Z$. There is an isomorphism \be \text{ob}:L^2(S^1;\C)\to \text{coker}(\slashed D_{A_0})\label{2.22a}\ee given by the linear extension of $e^{i\ell t}\mapsto \Psi_\ell^\circ$. Elements of the obstruction have the fundamental geometric property that {\it the radial decay away from the singular set is linked to the tangential Fourier mode}, i.e. $\Psi_{\ell}^\circ$ decays exponentially with rate $|\ell|$. This is a general phenomena for elliptic edge operators, not unique to our setting.

The above decay implies that for error terms supported away from $\mathcal Z$, the projection to the obstruction (composed with $\text{ob}^{-1}$) is a high-order smoothing operator into $H^s(S^1;\C)$. The $L^2$-orthogonal projection $\Pi $ to $\text{coker}(\slashed D_A)$ may be written in Fourier modes via the isomorphism (\refeq{2.22a}) as $$\text{ob}^{-1}\circ \Pi(\frak e)=\sum_{\ell\in \Z}\br \frak e  , \  \Psi_\ell^\circ\kt_{L^2(Y)} \cdot e^{i\ell t}.$$
 
 \noindent Thus, if an error term $\frak e$ is compactly supported where $|z|\geq R_0$ for some $R_0>0$, the Fourier modes are exponentially suppressed for $|\ell| \geq R_0^{-1}$ as a result of the exponential decay in (\refeq{precokernel}), hence $\Pi(\frak e)\in C^\infty(S^1;\C)$, see Lemma \ref{fourierregimescokernel} below.  We emphasize that this is the case even if $\frak e$ has no square-integrable weak derivatives. In the language of pseudo-differential operators, the projection restricted to spinors compactly supported away from $\mathcal Z$ is an infinite-order smoothing operator. Section \ref{section5} shows that this discussion carries over to the case of a general Riemannian 3-manifold, with appropriate modifications.

Given an error term $\frak e$ as above with obstruction $\Psi=\text{ob}^{-1}\circ\Pi(\frak e)$,  we may split $\Psi=\Psi^\text{low} + \Psi^\text{high}$ into its high and low Fourier modes depending on whether $|\ell| \geq R_0^{-1}$ or not. The first error term lies in a finite-dimensional subspace of $C^\infty(S^1;\C)$, and the latter is exponentially small i.e. $O(e^{-R_0})$. Since the main error terms in the alternating iteration are compactly supported where $d\chi^+\neq 0$, the above applies with $R_0=O(\e^{-1/2})$, splitting the obstruction component into its low and high mode pieces (for technical reasons, a middle range is also required, see Definition \ref{Fourierregimes}).  These two components of the obstruction are cancelled in two different ways, which we now describe.
\subsubsection{Singular Spinors} The exponentially small tail-end of the obstruction, $\Psi^\text{high}$ in the above notation, can be cancelled without relying on deformations of the singular set. When the integrability condition in Eq. (\refeq{1.1}) is relaxed to allow spinors that become singular along $\mathcal Z$, the singular Dirac operator becomes surjective. In Section \ref{section5.2}, we define a subspace $\mathcal X \subseteq \mathcal D^\text{max}$ of the $L^2$-maximal domain such that the adjoint problem 

$$\slashed D^\star_{A_0}: \mathcal X\to \text{Range}(\slashed D_{A_0})^\perp$$

\noindent is surjective. This provides a second way of cancelling the obstruction, in addition to deforming the singular set. This method cannot be used to fully replace the deformations of the singular set to cancel the full obstruction, because spinors $u\in \mathcal X$ blow-up rather than decay across the neck region, which prevents the alternating iteration from converging --- specifically, Hypothesis \ref{hyp2} fails. Nevertheless, for the tail-end of the obstruction, the exponential suppression described above sufficiently overcomes the (polynomial) blow-up of the alternation in these modes. 

To cancel the remaining low Fourier modes $\Psi^\text{low}$, we introduce a space of deformations of the singular as explained in the introduction (see Section \ref{section6}). Then, in Section \ref{section10}, we combine these two ways of canceling the obstruction into an infinite-dimensional gluing parameter given by a subspace  $$\frak W\subseteq L^{2}(\mathcal Z_\tau; N\mathcal Z_\tau) \oplus \mathcal X$$

\noindent consisting of linearized deformations of $\mathcal Z_\tau$ to cancel the ``low'' Fourier modes of the obstruction, and singular spinors to cancel the ``high'' Fourier modes. By the elliptic regularity of (\refeq{mainamap}), a linearized deformation cancelling a (smooth) low Fourier mode obstruction is also smooth, thus in fact $\frak W\subseteq C^\infty(\mathcal Z_\tau; N\mathcal Z_\tau) \oplus \mathcal X$, and the gluing argument only ever needs to consider smooth deformations. 

\subsubsection{Where Does the Loss of Regularity Go?} A loss of regularity means that certain terms require bounds in Sobolev norms of regularity $s$ than is {\it higher} than the regularity $s'$ obtained by inverting the operator (\refeq{mainamap}). While the (effective) reduction of the obstruction to a finite-dimensional space of smooth obstructions spanned by low Fourier modes means that any two Sobolev norms on this subspace are equivalent, they are not {\it uniformly} equivalent in $\e$. For instance, on the space spanned by the lowest $\e^{-\alpha}$ Fourier modes, one has  \be \|\xi\|_{s+k}\leq C \e^{-k\alpha}\|\xi\|_{s}.\label{badpower}\ee

 \noindent {\it This leads to unfavorable powers of $\e^{-1}$ accumulating in many crucial bounds, which disrupt the convergence of the gluing iteration without careful estimates}. With our setup specifically, certain estimates during the gluing iteration require control of the $H^2$ norm of a linearized deformation $\xi\in \Gamma(\mathcal Z_\tau; N\mathcal Z_\tau)$, while the inverse of the deformation operator $T_{\Phi_0}=\Pi_0\circ \text{d}\slashed{\mathbb D}_{\Phi_0}(\_, 0)$ from (\refeq{mainamap}) with the loss of regularity\footnote{This is also he relevant manifestation of the operator in Part(A) of Hypothesis \ref{hyp4}} only controls the $H^{1/2}$ norm. Because the range of Fourier modes from the obstruction $\Psi^\text{low}$ (i.e. modes up to $|\ell|=O(\e^{-1/2})$) effectively carries over to solution of $T_{\Phi_0}(\xi)=\Psi^\text{low}$ (see Section \ref{section6.2} for a precise statement), this results in an unfavorable factor of $\e^{-3/4}$.

Regardless of choices, this exponent is larger than the decay factors in Eq. (\refeq{deltapm}) picked up in each cycle of the iteration; thus it disrupts the convergence of the iteration.  The main offending term is the off-diagonal component in the linearization of the universal Dirac operator (\refeq{universalDiracdef}) indicated by a box below.   
 \be 
 \text{d}\slashed{\mathbb D}_{\Phi_0}(\xi,\psi)=
 \label{boxedtermmatrix}
 \begin{pmatrix} T_{\Phi_0} & 0  \\ \boxed{(1-\Pi_0) \text{d}\slashed{\mathbb D}_{\Phi_0}}& \slashed D_{A_0} \end{pmatrix}\begin{pmatrix} \xi \\ \psi \end{pmatrix},
 \ee 

\noindent The block-diagonal decomposition is explained precisely in Section \ref{section6.2}, and the boxed term is precisely our version of the term in Hypothesis \ref{hyp4}(B). In particular, this term must be bounded by a constant $C_\epsilon =C(\e)$ as in that statement, such that $C_\e$ contains a sufficiently mild power of $\e^{-1}$ that (\refeq{convergence}) holds. Thus, for the proof of Theorem \ref{maina}, controlling the loss of regularity for the linearization amounts to showing a sufficiently strong version of  Hypothesis \ref{hyp4}(B), which is done in Proposition \ref{effectivesupportdD}.

To achieve this, we leverage a specific (infinite-dimensional) gauge freedom in the gluing construction. With a judicious choice of gauge, the unfavorable powers of $\e$ appearing in the bounds (\refeq{badpower}) can be tamed without Nash-Moser theory.

\subsubsection{The Tangential Smoothing Gauge} The singular set $\mathcal Z_0$ of a $\Z_2$-harmonic spinor may be deformed in the direction of a vector field $\xi \in \Gamma(\mathcal Z_0; N\mathcal Z_0)$ by choosing family of diffeomorphisms $F_\xi:Y\to Y$ such that $F_\xi(\mathcal Z_0)=\text{graph}(\xi)$ for each $\xi$.  One natural choice --- the one taken in \cite{PartII, RyosukeThesis} --- is to let $F_\xi$ be a constant translation by $\xi$ in the normal directions to $\mathcal Z_0$ and zero away from a neighborhood of $\mathcal Z_0$. This is not the only choice, however: any other choice of a family of diffeomorphisms $F_\xi'$ such that $F_\xi'(\mathcal Z_0)=F_\xi(\mathcal Z_0)$ is an equally valid choice. Thus the group $\text{Diff}(Y;\mathcal Z_0)$ of diffeomorphisms fixing $\mathcal Z_0$ acts as an infinite-dimensional gauge group on the deformation setup.

Section \ref{section6.3} introduces a more optimal choice of this infinite-dimensional gauge than the ones taken in  \cite{PartII, DonaldsonMultivalued, RyosukeThesis}, dubbed the {\bf Tangential Smoothing Gauge}. This gauge choice is formed from families of diffeomorphisms that temper the high Fourier modes of the extension of $\xi$ by a radially dependent family of smoothing operators in the direction tangential to $\mathcal Z_0$, hence the name. The motivation for this choice comes from elliptic edge theory, and is discussed in greater detail in that section. The estimates proved in Section \ref{section6.3} show that this tangential smoothing is sufficient to tame the unfavorable powers in the estimate (\refeq{badpower}) enough that the iteration converges. 
The tangential smoothing gauge can be regarded as the analogue of the Coulomb gauge in
classical gauge theory. Coulomb gauge allows stronger estimates than other gauges by making
the elliptic character of the equations manifest, and it is used to establish results that are
independent of gauge choice. In a similar way, using the tangential smoothing gauge leads to
stronger estimates in the proof of Theorem \ref{maina}, but the result is independent of this infinite-dimensional gauge choice.

 \subsection*{Appendix of Gluing Parameters}
 \label{appendixofnotation}
 This preliminary appendix provides some preliminary notation and conventions for the remainder of the article. The remainder of the notation is collected in the Glossary of Notation \ref{glossary}.

 (1) Gluing Decomposition Parameters

\begin{itemize}
\item $\e\in (0, \e_0)$ the $L^2$-norm parameter. 
\item $\tau \in (-\tau_0, \tau_0)$ the coordinate along the parameter path $p_\tau=(g_\tau, B_\tau)$. Assumed to satisfy (\ref{spectralcrossingdef}). 
\item $\delta = \e^{1/48}$ the convergence factor from a single cycle of the alternating iteration. 
\item $\gamma^+, \gamma^-, \underline \gamma<<1$ fixed small numbers, say $10^{-6}$. 
\item $\gamma <<1$ indeterminate small numbers on the same order $O(10^{-6})$, see Conventions below. 
 \item $\nu^+=\tfrac14 -10^{-6}$ the inside weight. 
\item $\nu^-=\tfrac12-10^{-6}$ the outside weight. 
\item $\lambda^+=\e^{1/2}$ the radius of the inside region
\item $\lambda^-=\e^{2/3-\gamma^-}$ the radius of (the complement of) the outside region. 
\end{itemize} 

(2) Regions and cutoffs
\begin{itemize}
\item $Y^+_{\e,\tau}=N_{\lambda^+}(\mathcal Z_\tau)$ the inside region: a tubular neighborhood of radius $\lambda^+$ of $\mathcal Z_\tau$. 
\item $Y^-_{\e,\tau}\!=\Yminus N_{\lambda^-}(\mathcal Z_\tau)$ the outside region: the complement of a neighborhood of radius $\lambda^-$ of $\mathcal Z_\tau$.
\item $\chi^+$ a  cutoff function equal to $1$ for $r\leq \lambda^+/4$ and vanishing for $r\geq \lambda^+/2$. 
\item $\chi^-$ a cutoff function equal to $1$ for $r\geq \lambda^-/2$ and vanishing for $r\leq \lambda^-/4$.
\item $\bold 1^+$ the indicator function of the set $r \leq \e^{2/3-\gamma^-}$. 
\item $\bold 1^-=1-\bold 1^+$ the indicator function of the complementary region. 
\end{itemize}

 \begin{impremark} \label{gammarem} We adopt the following conventions. 
 
 \begin{enumerate}
 \item[(i)] The small numbers $\e_0, \tau_0$ are allowed to decrease a finite number of times between successive appearances over the course of the proof of Theorem \ref{maina}. 
 \smallskip 
 
 \item[(ii)] Likewise, the constant $\gamma<<1$ is allowed to increase a finite number of times between successive appearances over the course of the proof. In each successive appearance, $\gamma$ is updated to a linear combination of the previous value of $\gamma$ and the fixed values $\gamma^\pm, \underline{\gamma}$ and related values. The coefficients of these linear combination are $O(1)$ universal constants, bounded independent of all other parameters. Over the course of the proof of Theorem \ref{maina}, a factor of $\e^{\pm \gamma}$ appears in many estimates; Section \ref{section11} collects the accumulated powers of $\e^\gamma$ in the final proof. Any choices of $\gamma$ for which the accumulated end value of $\gamma$ is less than $\frac{1}{48}$ are valid. With more careful bookkeeping, one could replace these factors with $(\log \e^{-1})^N$ for $N$ sufficiently large. 
 
 We emphasize that all values of $\gamma$ including deocirations, e.g. $\gamma^\pm, \underline \gamma$ remain fixed throughout. 
 
 \smallskip
 
 \item[(iii)] We adopt the convention on Sobolev spaces that the notation $H^{s}(Y;V)$ is used (with various decorations) to denote the space of sections of a vector bundle $V\to Y$ over the 3-dimensional manifold, while the notation $L^{s,2}(\mathcal Z; W)$ is used to denote the space of sections a vector bundle $W$ over the 1-dimensional singular sets. This convention is adopted to aid in visually distinguishing the many Sobolev spaces that appear.  
 \smallskip 
 \item[(iv)] Sub-subscripts and superscripts indicating the gluing parameters are often dropped to avoid excessively cluttered notation. Thus, for example, we often write $\mathcal L_{\Phi^{(1)}}$ rather than $\mathcal L_{\Phi^{(1)}_{\e,\tau}}$ to denote the linearization at the model solution $\Phi^{(1)}_{\e,\tau}$. It is understood in these cases that the dependence on $(\e,\tau)$ is retained. 
 \end{enumerate}
 \end{impremark}

 \section{$\Z_2$-Harmonic Spinors and Compactness}
 \label{section3}

 This section reviews the compactness properties of the two-spinor Seiberg--Witten equations from \cite{HWCompactness}, and begins the set-up of the gluing analysis.  More detailed expositions of the same material may be found in \cite{PartI, DWExistence, WalpuskiNotes, HWCompactness, TaubesU1SW}.

 \subsection{Compactness Theorem}
\label{section3.1}

Let $Y$ be a compact, oriented, 3-manifold. With $S,E\to Y$ as in Section \ref{section1.1} and $p=(g,B)$ as in (\refeq{mathcalPdef}), set $S_E:= S\otimes_\C E$. Clifford multiplication on $S$ induces a Clifford multiplication $\gamma:  T^*Y\to \text{End}(S_E)$ which acts by the identity on $E$. Define the {\bf moment map} $\mu: S_E\to \Omega^1(i\R)$ by 
 \be \frac{1}{2}\mu(\Psi, \Psi)= \sum_{j=1}^3 \frac{i}{2} \br i\gamma(e^j)\Psi, \Psi \kt e^j. \label{momentmap}\ee

\noindent where $\{e^j\}$ is a local orthonormal frame of $T^*Y$. Unlike for the single-spinor Seiberg--Witten equations, there are non-zero sections $\Psi\in \Gamma(S_E)$ with $\mu(\Psi,\Psi)=0$.

 In order for the two-spinor Seiberg--Witten equations (\refeq{prelimSW1}--\refeq{prelimSW2}) to be an elliptic system on $Y$, an auxiliary $0$-form is required.  To incorporate this 0-form, we extend Clifford multiplication to a a map $\gamma: (\Omega^0 \oplus \Omega^1)(i\R)\to \text{End}(S_E)$, revise our notation to upgrade $A$ to a pair $A= (a_0, A_1)\in \Omega^0(i\R)\oplus \mathcal A_{U(1)}$. For such a pair $A$, we denote
\be \slashed D_A = \slashed D_{A_1} -i \gamma(a_0)\hspace{2cm}\star F_A=\star F_{A_1} -d a_0,  \label{extended} \ee
\noindent where $\slashed D_{A_1}$ is the Dirac operator on $S_E$ formed using the spin connection of $g$, the connection $A_1$, and the background connection $B$ on $E$, and $F_{A_1}$ is the curvature of $A_1$.     
\begin{defn} \label{SWdef}The {\bf (extended) two-spinor Seiberg-Witten Equations} for the parameter $p=(g,B)$ are the following equations for configurations $(\Psi, A)\in \Gamma(S_E)\times (\Omega^0(i\R)\times \mathcal A_{U(1)})$ 
\begin{eqnarray}
\slashed D_{A}\Psi&=&0 \label{SW1}\\ 
\star F_A +\tfrac{1}{2}\mu(\Psi, \Psi)& =& 0,
\label{SW2}
\end{eqnarray}

\noindent where $\star F_A$ and $\slashed D_A$ are as in (\refeq{extended}). These equations are invariant under the action of the {\bf gauge group} $\mathcal G= \text{Maps}(Y; U(1))$.   
\end{defn}

\noindent Note that the dependence of the equations on $p=(g,B)$ is kept implicit in the notation. Note also for comparison, many authors include a minus sign in the definition (\refeq{momentmap}) of $\mu$, and reverse the sign in Eq. (\refeq{SW2}) \cite{MorganSW, KM, DWDeformations}. If $(\Psi, A)$ solves (\refeq{SW1}--\refeq{SW2}) and $\Psi\neq 0$, then integration by parts shows that $a_0=0$. For the purposes of Theorem \ref{maina}, it therefore suffices to solve the extended equations. {\it From here onward, we work exclusively with the extended equations.}

Standard elliptic theory shows that a sequence of solutions to (\refeq{SW1}--\refeq{SW2}) with a uniform bound on the spinors' $L^2$-norm admits a convergent subsequence \cite{MorganSW, KM, HWCompactness}. In the case of the standard (one-spinor) Seiberg-Witten equations, a simple argument using the Weitzenb\"ock formula, the maximum principle, and a pointwise identity for $\mu$ gives such a bound, namely $\|\Psi\|_{L^2}\le \frac12 \sup |s|$, where $s$ is the scalar curvature of $g$. In the case of two-spinors, the corresponding pointwise identity for $\mu$ fails, and there may be sequences of solutions $(\Psi_i, A_i)$ such that $\|\Psi_i\|_{L^2}\to \infty$ which therefore admit no convergent subsequence.  

The behavior of such sequences can be understood by renormalizing the spinor to have unit $L^2$-norm. Thus, with $\e = \frac{1}{\|\Psi\|_{L^2}}$ we define renormalized spinors \be \boxed{\Phi := \e \Psi}\label{epsilondef}\ee

\noindent so that $\|\Phi\|_{L^2(Y)}=1$. As in the introduction, the re-normalized or {\bf blown-up} Seiberg-Witten equations for a triple $(\Phi, A,\e)$ become (\refeq{buSW1}--\refeq{buSW3}), where $A=(a_0, A_1)$ is now as in (\refeq{extended}) . The following theorem of Haydys--Walpuski describes the convergence behavior of sequences of solutions to the blown-up equations.  The theorem states that sequences of solutions along which $\e\to 0$ converge to solution  of the $\e=0$-version of  (\refeq{buSW1}--\refeq{buSW3})
away from a singular set $\mathcal Z_0$. 
\medskip 

\begin{thm} ({\bf  Haydys--Walpuski} \cite{HWCompactness}, {\bf Zhang} \cite{ZhangRectifiability},  {\bf Taubes}\cite{TaubesZeroLoci}, {\bf Parker}\cite{ConcentratingDirac})
\label{HWcompactness} 
Suppose that $(\Phi_i, A_i, \e_i)$ is a sequence of solutions to (\refeq{buSW1}--\refeq{buSW3}) with respect to a sequence of parameters $p_i\to p_0=(g_0, B_0)$ converging in $C^\infty$ such that $\e_i \to 0$. Then there exists a triple $(\mathcal Z_0, A_0, \Phi_0)$  where

\begin{itemize}
\item $\mathcal Z_0\subset Y$ is a closed, rectifiable subset of Hausdorff codimension 2, 
\item $A_0$ is a flat $U(1)$-connection on $Y-\mathcal Z_0$ with holonomy in $\Z_2$,
\item $\Phi_0$ is a spinor on $\Yminus \mathcal Z_0$ satisfying \be \slashed D_{A_0}\Phi_0 = 0 \hspace{1.5cm} \mu(\Phi_0, \Phi_0)=0 \hspace{1.5cm} \|\Phi_0\|_{L^2}=1,\label{limitingconfigs}\ee
and $|\Phi_0|$ extends continuously to $Y$ with $\mathcal Z_0=|\Phi_0|^{-1}(0)$, 
\end{itemize}
and, after passing to a subsequence, $\Phi_i \to \Phi_0$, and  $A_i \to A_0$ in $C^\infty_{loc}(\Yminus \mathcal Z_0)$ modulo gauge transformations, and $|\Phi_i|\to |\Phi_0|$ in $C^{0,\alpha}(Y)$ for some $\alpha>0$.    
\end{thm}

\begin{rem}
The above statement combines the original result of Haydys--Walpuski with refinements proved by Taubes \cite{TaubesZeroLoci}, Zhang \cite{ZhangRectifiability} and the author \cite{ConcentratingDirac}. In \cite{TaubesZeroLoci}, Taubes showed that the singular set $\mathcal Z_0$ has finite 1-dimensional Hausdorff content; building on this Zhang showed  in \cite{ZhangRectifiability} that the singular set is rectifiable. In \cite{ConcentratingDirac}, the author improved the convergence to the limit from weak $L^{2,2}_{loc}$ for the spinor and weak $L^{1,2}_{loc}$ for the connection to $C^\infty_{loc}$ for both. Taubes also proved a four-dimensional version of  Theorem \ref{HWcompactness} in \cite{TaubesU1SW}, to which the same refinements apply.  

\end{rem}
\begin{rem}
 Theorem  \ref{HWcompactness} i is one of a family of similar compactness results for generalized Seiberg--Witten equations stemming from C. Taubes's generalization of Uhlenbeck Compactness to $\text{PSL}(2,\C)$  connections. Other such compactness theorems can be found in \cite{TaubesKWNahmPole,TaubesVW, TaubesU1SW, Taubes4dSL2C, WalpuskiZhangCompactness, WalpuskiNotes, MinhCompactness}.
\end{rem}

\subsection{The Haydys Corresondence}

\label{section3.2}
The limiting configurations in Theorem \ref{HWcompactness} are equivalent to $\Z_2$-harmonic spinors as defined in the Section \ref{section1.2}. This equivalence, which we now describe, is a particular instance of the Haydys Correspondence  \cite{HaydysCorrespondence, DWDeformations}.

A limiting configuration $(\mathcal Z_0, A_0, \Phi_0)$ as in Theorem \ref{HWcompactness} induces a decomposition of the restriction of the two-spinor bundle $S_E$ to $Y-\mathcal Z_0$ as follows. Since $A_0$ is flat with holonomy in $\Z_2$, $\det( S)|_{Y-\mathcal Z_0}\simeq \underline{\C}$ is trivial, and $S|_{Y-\mathcal Z_0}$ admits a reduction of structure to $SU(2)$. Thus, there is a ``charge conjugation'' map $J \in \text{End}(S |_{Y-\mathcal Z_0})$ such that $J^2=-\text{Id}$; since $E$ is an $SU(2)$-bundle it admits a similar map, denoted $j$. The product $\sigma=J\otimes_\C j$ satisfies $\sigma^2=\text{Id}$, i.e. it is a real structure on $S_E |_{Y-\mathcal Z_0}$.  Consequently, there is a decomposition 

\be  S_E |_{Y-\mathcal Z_0} = S^\text{Re} \oplus S^\text{Im}\label{realimdecomp}\ee

\noindent where $$ S^\text{Re}= \Big\{ \tfrac12 (\Psi + \sigma \Psi) \ | \ \Psi \in \Gamma(S_E|_{Y-\mathcal Z_0})\Big\} \hspace{2cm}S^\text{Im}= \Big\{ \tfrac12 (\Psi - \sigma \Psi) \ | \ \Psi \in \Gamma(S_E|_{Y-\mathcal Z_0})\Big\}$$
are the ``real'' and ``imaginary'' subbundles respectively. 

These subbundles have the following useful characterization, which is proved in \cite[Sec. 2]{PartI}, and  \cite[Sec. 3]{DWExistence}. 

\begin{lm} \label{subbundles}Let $(\mathcal Z_0, A_0, \Phi_0)$ be a limiting configuration as in Theorem \ref{HWcompactness}. The decomposition (\refeq{realimdecomp}) satisfies the following: 

\begin{enumerate}
\item[(A)] The decomposition is parallel with respect to the connection $\nabla_{A_0}$ induced by $A_0$. 
\item[(B)] Clifford multiplication by $\R$-valued forms preserves the decomposition, i.e. $$\gamma : (\Omega^0 \oplus \Omega^1)(\R) \times S^\text{Re}\to S^\text{Re} $$
and likewise for $S^\text{Im}$. Conversely, Clifford multiplication by $i\R$-valued forms reverses it. 
\item[(C)] $\Phi_0 \in \Gamma(S^\text{Re})$ is a section of the first summand, and there exists a spin structure on $Y$ with spinor bundle $S_0$ and a real Euclidean line bundle $\ell\to Y-\mathcal Z_0$ such that $$S^\text{Re}\simeq S_0\otimes_\R \ell $$
on $Y-\mathcal Z_0$. Moreover, under this isomorphism, $\nabla_{A_0}$ is taken to the connection formed from the spin connection on $S_0$ and the unique flat connection on $\ell$, with an $\R$-linear perturbation commuting with $\gamma$ arising from $B_0$. \qed
\end{enumerate}
\end{lm}

\medskip As a consequence of items (A) and (B) above, the Dirac operator on $S_E$  restricts to a Dirac operator

\be \slashed D_{A_0}^\text{Re} : \Gamma(S^\text{Re})\to \Gamma(S^\text{Re}), \label{DA0preserves}\ee
\noindent and likewise for the imaginary part. When the subbundle in question is evident, we will omit the superscript from the notation. The isomorphism in Item (C) intertwines (\refeq{DA0preserves}) and the Dirac operator on $S_0\otimes_\R \ell$ formed using the connection in Item (C). This leads to the following equivalence. 

\begin{cor}\label{givesZ2harmonic}
Suppose that $\mathcal Z_0\subset Y$ is a smooth, embedded link. Then the data of a  limiting configuration satisfying (\refeq{limitingconfigs}) as in Theorem \ref{HWcompactness} is equivalent to a $\Z_2$-harmonic spinor $(\mathcal Z_0, A_0, \Phi_0)$ on $S^\text{Re}\simeq S_0\otimes_\R \ell$ satisfying 

$$\slashed D_{A_0}\Phi_0 =0  \hspace{2cm}\nabla_{A_0}\Phi_0\in L^2$$

\noindent with respect to $(g_0, B_0)$.    
\end{cor} 
\begin{proof} Except for the integrability condition, the corollary follows directly from isomorphism of item (3) in Lemma \ref{subbundles}. The fact that $\nabla \Phi_0 \in L^2$ will follow from Lemma \ref{asymptoticexpansion} in Section \ref{section4}, which shows requirement that $\nabla \Phi_0\in L^2$ is equivalent (for regular $\Z_2$-harmonic spinors) to requiring that $|\Phi_0|$ extend continuously over $\mathcal Z_0$ with $\mathcal Z_0\subset |\Phi_0|^{-1}(0)$.  
\end{proof}

Corollary \ref{givesZ2harmonic} is a manifestation of the Haydys Correspondence in the setting of $\Z_2$-harmonic spinors. Moving through the Hayds correspondence allows one to take advantage of the gauge freedom to temper the singular nature of the limiting equations. To explain this further, the limiting configurations in Theorem \ref{HWcompactness} are considered up to $U(1)$ gauge transformations and solve the globally degenerate system of equations (\refeq{limitingconfigs}). Here, the degeneracy arises because the symbol of the curvature equation (\refeq{buSW2}) vanishes as $\e\to 0$, hence one loses ellipticity everywhere on $Y$. On the other side of the Haydys correspondence, $\Z_2$-harmonic spinors are considered only up to the action of $\Z_2$ (acting by $\pm 1$ on $S_0$), and solve the Dirac equation on $S_0\otimes_\R \ell$ which is a singular elliptic equation whose symbol degenerates only locally along $\mathcal Z_0$ (as will follow from the local description in Section \ref{section4.1}). While the first type of degeneracy appears to at first be rather intractable, the latter description places the problem in the well-studied class of elliptic edge problems \cite{MazzeoEdgeOperators, MazzeoEdgeOperatorsII}. 

\medskip

\begin{rem}
Items (A) and (B) of Lemma \ref{subbundles} show that $S^\text{Re}$ is a 4-dimensional real Clifford module on $Y-\mathcal Z_0$. The isomorphism in Item (C) of Lemma \ref{subbundles} endows it with a complex structure, but not canonically so. In particular, the induced Dirac operator (\refeq{DA0preserves}) is only $\R$-linear if the $SU(2)$-connection $B$ is non-trivial (as it must be for condition (2) of Definition \refeq{regulardef} to be met). 
\end{rem}

\subsection{Recovering $\text{Spin}^c$ Structures}
\label{section3.3}

By Corollary \ref{givesZ2harmonic}, a limiting configuration $(\mathcal Z_0, A_0, \Phi_0)$ gives rise to a $\Z_2$-harmonic spinor. It is not immediately clear how to reverse this process because, in contrast to the Seiberg--Witten equations, the definition of a $\Z_2$-harmonic spinor makes no references to a $\text{Spin}^c$ structure. The topological information of the $\text{Spin}^c$ structure is lost in the limiting process of Theorem \ref{HWcompactness} and must be reconstructed before the gluing analysis begins. 

Specifically, we seek a $\text{Spin}^c$ structure with spinor bundle $S$ such that $S^\text{Re}$ as defined by (\refeq{realimdecomp}) satisfies the isomorphism of Lemma  \ref{subbundles}(C) for the twisted spinor bundle $S_0\otimes_\R \ell$ that hosts $\Phi_0$. Given such an $S$,  Corollary \ref{givesZ2harmonic} implies that $(\mathcal Z_0, A_0, \Phi_0)$ may be viewed as a (non-smooth along $\mathcal Z_0$) configuration on the subbundle $S^\text{Re}\subset S_E$ of two-spinor bundle formed from $S$, and the gluing analysis begins from there.
  
 The following lemma reconstructs the correct $\text{Spin}^c$-structure for the gluing, given an orientation of $\mathcal Z_0$. The proof may be found in Section 3 of \cite{PartI}; see also \cite{HaydysBlowupsets} for more results in this direction. 
\begin{lm}\label{reconstructingspinc}
 Let $(\mathcal Z_0, A_0, \Phi_0)$ be a regular $\Z_2$-harmonic spinor on $S_0\otimes_\R \ell$. An orientation of $\mathcal Z_0$ determines a unique $\text{Spin}^c$-structure with spinor bundle $S\to Y$ satisfying the following criteria. 
 
 \begin{enumerate}
  \item[(1)] The first Chern class satisfies $$c_1(S)= - \text{PD}[\mathcal Z_0]$$
 \noindent with the specified orientation of $\mathcal Z_0$.  
 \item[(2)]  $S$ extends $S_0\otimes_\R \ell$ in the sense that $S|_{Y-\mathcal Z_0}\simeq S_0\otimes_\R \ell$. Moreover, there is an isomorphism $$S_0\otimes_\R \ell \simeq S^\text{Re} \subset S_E$$
 where $S_E=S\otimes_\C E$, which makes $\Phi_0$ a smooth section of $S^\text{Re}\to \Yminus \mathcal Z_0$.  
 \end{enumerate}
\end{lm}

\noindent Notice that we do {\it not} assume $\mathcal Z_0$ is connected, thus there are $2^k$ possible choices of orientation when $\mathcal Z_0$ has $k$ components.

Given Lemma \ref{reconstructingspinc}, the data of a regular $\Z_2$-harmonic spinor (with an orientation of $\mathcal Z_0$) is equivalent to that of a limiting configuration $(\mathcal Z_0, A_0, \Phi_0)$ using the induced $\text{Spin}^c$-structure on $Y$. From here on, we therefore cease to distinguish between a regular (oriented) $\Z_2$-harmonic spinor $(\mathcal Z_0, A_0,\Phi_0)$ and the corresponding limiting configuration denoted (purposefully) by the same triple.

\subsection{Adapted Coordinates}
\label{section3.4}
In order to describe Seiberg--Witten configurations in a neighborhood of $\mathcal Z_0$, we use adapted coordinate systems, constructed as follows. 

Fix a component $\mathcal Z_j$ of $\mathcal Z_0$ with length $|\mathcal Z_j|$ and an arclength parameterization $ p: S^1\to \mathcal Z_j$. Choose an orthonormal frame $\{n_1, n_2\}$ for the pullback $ p^* N\mathcal Z_0$ of the normal bundle, ordered so that $\{\dot { p}, n_1, n_2\}$ is an oriented frame of $p^*TY$ along $\mathcal Z_j$. Let $N_r(\mathcal Z_0)$ denote the tubular neighborhood of radius $r$ around $\mathcal Z_0$, measured in the geodesic distance of $g_0$. 

\begin{defn}
\label{Fermicoords}

A system of {\bf Fermi coordinates}  $(t,x,y)$ of radius $r_0 < r_{\text{inj}}$ where $r_{\text{inj}}$ is the injectivity radius of $(Y,g_0)$ is the diffeomorphism $\bigsqcup_j S^1\times D_{r_0}\simeq N_{r_0}(\mathcal Z_0)$ defined by $$(t,x,y)\mapsto \text{Exp}_{ p(t)}(xn_1 + yn_2).$$

\noindent on each component of $\mathcal Z_0$, where $t\in [0, |\mathcal Z_j|)$ is the normalized coordinate in the $S^1$ direction. In these coordinates the  metric $g_0$ has the form \be g_0=dt^2 + dx^2 + dy^2 \  +\  O(r).\label{metriccoords}\ee

\medskip 

\noindent We denote the corresponding cylindrical coordinates by $(t,r,\theta)$. Given a smooth family $(g_\tau, \mathcal Z_\tau)$ as in Theorem \ref{SpectralCrossing} it may be arranged that the Fermi coordinate systems depend smoothly on $\tau$.  \end{defn}

A system of Fermi coordinates induces a trivialization $\{dt,dx,dy\} $ of $T^*Y$, and thus one of $(\Omega^0\oplus \Omega^1)(i\R)$. The next lemma, which is proved in Section 3 of \cite{PartI}, shows that a choice of Fermi coordinates {also} induces trivialization of the two-spinor bundle $S_E$.

\begin{lm}\label{Fermitrivialization}
In the neighborhood $N=N_{r_0}(\mathcal Z_j)$ of each component of $\mathcal Z_j$, there is a local trivialization 
\be (S\otimes_\C E)|_{N} \simeq N\times (\C^2\otimes_\C \mathbb H)\label{trivializationeq}\ee

\noindent with the following properties. 

\begin{itemize}
\item[(1)] The connection $A_0$ has the form $$A_0:= \text{d}  \ + \ \frac{i}{2}d\theta \  + \  O(1),$$

\noindent where $O(1)$ denotes a smooth term whose derivatives are bounded in terms of those of the background data $p=(g,B)$. 

\item[(2)] There is an $\epsilon_j \in \{0,1\}$ such that the restriction $S^\text{Re}|_N$ is given in the trivialization (\refeq{trivializationeq}) by

$$S^\text{Re}\big |_{N_{r_0}(\mathcal Z_j)}= \left\{\begin{pmatrix} \alpha \\ \beta \end{pmatrix}\otimes 1 \  + \ e^{-i\theta}e^{-i\epsilon_j t} \begin{pmatrix} -\overline \beta \\ \overline \alpha \end{pmatrix}\otimes j \ \Big | \ \alpha,\beta: N\to \C\right\}. $$
\end{itemize}

\noindent  Again, for the $\tau$-parameterized family of Theorem \ref{SpectralCrossing}, and it may be assumed that this family of trivializations depends smoothly on $\tau$ (using $\mathcal Z_\tau$ and $A_\tau$ respectively). 
\end{lm}

Henceforth, we fix, once and for all, a smooth family of Fermi coordinates and accompanying trivializations (\refeq{trivializationeq}) for $\tau \in (-\tau_0, \tau_0)$. We allow all universal constants to depend on this choice.

 \section{The Singular Linearization}
 \label{section4}

 The linearized Seiberg--Witten equations play an essential role in carrying out the alternating iteration outlined in Section \ref{section2.4}. This section introduces the linearized equations, first at a smooth solution and then at a $\Z_2$-harmonic spinor.  The key observation is that the linearization at a  $\Z_2$-harmonic spinor is a singular elliptic system in which some components are elliptic edge operators  \cite{MazzeoEdgeOperators, MazzeoEdgeOperatorsII, grieser2001basics}. 
 
 \subsection{Singular Linearization}
\label{section4.1}

Differentiating, (\refeq{SW1}--\refeq{SW2}), the linearized (extended) Seiberg--Witten equations at a general smooth (renormalized) configuration $(\Phi, A)$ acting on a linearized deformation $(\ph,a)$ are  
\begin{eqnarray}
\slashed D_{A}\ph + \gamma(a)\tfrac{\Phi}{\e}&=& 0 \label{LSW1} \\
(\star d, -d) a + \tfrac{\mu(\ph, \Phi)}{\e} &=&0,  \label{LSW2}
\end{eqnarray}

\noindent where $\e$ is as in (\refeq{epsilondef}).

To make (\refeq{LSW1}--\refeq{LSW2}) into an elliptic system, we impose the $\Omega^0(i\R)$-valued gauge-fixing condition 
\be \label{gaugefixing}
 -d^\star a - \frac{i\br i \ph, \Phi \kt}{\e}=0,
\ee

\noindent where $d^\star $ denotes the adjoint of the exterior derivative. Then the polarization of $\mu$  is extended to a bilinear map $\mu: S_E\otimes S_E\to (\Omega^0\oplus \Omega^1)(i\R)$ by
\bea
\mu(\ph, \psi)&:=& (-i \br i \ph, \psi \kt , \mu_1(\ph, \psi)),
\eea

\noindent where $\mu_1$ is what were previously denoted by $\mu$. 

\medskip 

\begin{lm}\label{linearizedequations} Suppose that $(\Phi, A)$ is a smooth configuration on $Y$. Then the linearization of the (extended, gauge-fixed) Seiberg--Witten equations at $(\Phi,A)$ is the operator $\mathcal L_{(\Phi,A)}$ defined on linearized deformations $(\ph,a)\in \Gamma(S_E)\oplus (\Omega^0\oplus \Omega^1)(i\R)$ by
\be \mathcal L_{(\Phi, A)}(\ph,a)=\begin{pmatrix}\slashed D_A & \gamma(\_)\tfrac{\Phi}{\e} \\ \tfrac{\mu(\_,\Phi)}{\e} & \bold d \end{pmatrix}\begin{pmatrix}\ph \\ a \end{pmatrix},\label{linearization}\ee
\noindent where $\bold da= \begin{pmatrix} 0 & -d^\star \\ -d  & \star d \end{pmatrix}\begin{pmatrix} a_0 \\ a_1 \end{pmatrix}$. Moreover, $\mathcal L_{(\Phi, A)}$ is a self-adjoint elliptic operator. \qed
\end{lm}

\noindent Notice that the parameter $\e$ is kept implicit in the notation $\mathcal L_{(\Phi, A)}$ 

A $\Z_2$-harmonic spinor $(\Z_0,A_0,\Phi_0)$ (or eigenvector) is {\it not} smooth on $Y$, and the corresponding  linearization ${\mathcal L}_{(\Phi_0,A_0)}$  is not elliptic.  Nevertheless,  the linearization acts on sections of bundles defined on $Y-\Z_0$ and, because of the decomposition (3.8),  it admits the following block decomposition.

\begin{lm}\label{linearizedequationsdef} The (extended, gauge-fixed) linearized Seiberg--Witten equations at $(\mathcal Z_0, A_0, \Phi_0)$ take the following form on a linearized deformation $(\ph_1, \ph_2, a ) \in \Gamma(S^\text{Re}) \oplus \Gamma(S^\text{Im})\oplus (\Omega^0 \oplus \Omega^1)(i\R)$:

\be \mathcal L_{(\Phi_0, A_0)} (\ph_1, \ph_2, a)  = \begin{pmatrix} \slashed D_{A_0} & 0 & 0 \\ 0 & \slashed D_{A_0} & \gamma(\_)\frac{\Phi_0}{\e} \\ 0 & \frac{\mu(\_, \Phi_0)}{\e}. & \bold d \end{pmatrix} \begin{pmatrix}\ph_1 \\ \ph_2 \\ a \end{pmatrix}\label{blockdiagonaldecomp}\ee

\noindent The same applies at an eigenvector $(\mathcal Z_\tau, A_\tau, \Phi_\tau)$. 
\end{lm}

\begin{proof}The proof is an immediate consequence of Lemma \ref{linearizedequations} and Eq. (\refeq{DA0preserves}). 
\end{proof}

\subsection{The Singular Dirac Operator}

 \label{section4.2} 
 
 As explained in the introduction, the Dirac operator $\slashed D_{A_0}$ at the singular connection $A_0$ is a degenerate {\bf elliptic edge operator} \cite{MazzeoEdgeOperators}. 
In the local coordinates and trivializations of Lemma \ref{Fermitrivialization}, the degenerate nature becomes manifest. Near $\mathcal Z_0$, it has the form \be \slashed D_{A_0}= \begin{pmatrix} i\del_t & -2\del \\ 2\delbar & -i\del_t\end{pmatrix} \ + \ \frac{1}{4r}\begin{pmatrix} 0 & e^{-i\theta} \\ -e^{i\theta}  &0 \end{pmatrix} \ + \ \frak d_1 \ + \ \frak d_0 \label{localform}\ee

\noindent where $\frak d_1=O(r)\nabla$ is a first order operator vanishing along $\mathcal Z_0$, and $\frak d_0$ is a bounded zeroth order operator. In particular, the second term, which arises from the non-trivial holonomy of $A_0$, is unbounded on $L^2$. Equivalently, $r\slashed D_{A_0}$ is a elliptic operator with $L^2$-bounded terms whose symbol degenerates along $\mathcal Z_0$, i.e. an elliptic edge operator\footnote{In general, an elliptic edge operator is an elliptic combination of the derivatives $r\del_r, \del_\theta, r\del_t$ in Fermi coordinates; technically speaking, $r\slashed D_{A_0}$ is the edge operator in question, but the factor of $r$ only shifts the weight.  } with ``edge'' $\mathcal Z_0$. Standard elliptic theory fails for elliptic edge operators, and specialized Sobolev spaces, notions of elliptic regularity, and Fredholm theory are required in this setting.

The remainder of Sections \ref{section4}--\ref{section5} develop the analysis of (\refeq{localform}) from the perspective of elliptic edge theory. The lower $2\times 2$ block in (\refeq{blockdiagonaldecomp}) is analyzed in Section \ref{section7.3}, using a trick that reduces it to standard elliptic theory. More generally, Lemmas \ref{subbundles} and \ref{linearizedequationsdef}  and the expression (\refeq{localform}) apply for any $\tau\in (-\tau_0,\tau_0)$, and we consider the family of operators   \be\slashed D_{A_\tau}: \Gamma(\Yminus \mathcal Z_\tau; S^\text{Re})\lre \Gamma(\Yminus \mathcal Z_\tau; S^\text{Re})\label{singulardirac},\ee

\noindent  where the dependence of $S^\text{Re}$ on $\tau$ is suppressed in the notation. More detailed discussion and proofs of the statements in this subsection may be found in Sections 2--4 of \cite{PartII}, and a discussion from the perspective of the microlocal analysis of elliptic edge operators is contained in \cite{MazzeoHaydysTakahashi}.

 To begin, we define function spaces adapted for edge operators. Let $r_\tau$ be a smooth positive weight function such that \be r_\tau(y)=\begin{cases}\text{dist}(y,\mathcal Z_\tau)  \hspace{1cm} y\in N_{r_0/2}(\mathcal Z_\tau) \\ \text{const.} \hspace{1.7cm}  y\in \Yminus N_{ r_0}(\mathcal Z_\tau) \end{cases}\label{weightfunctionr}\ee where the distance is measured using the metric $g_\tau$ (though we omit this from the notation), and $r_0$ is as in Definition \ref{Fermicoords}.

 \begin{defn}
\label{rH1edef}
For a constant $\nu \in \R$,  the weighted {\bf edge Sobolev spaces} (of regularity $m=0,1$) are defined by 
\begin{eqnarray}
r^{1+\nu}H^1_e(\Yminus \mathcal Z_\tau; S^\text{Re})&: =& \Big \{ \ u \ \ \Big | \  \  \int_{\Yminus \mathcal Z_\tau}\left( |\nabla u |^2 + \frac{|u|^2}{r_\tau^2} \  \right)r_\tau^{-2\nu} d V  \ < \ \infty \Big \}\\
r^{\nu}L^2(\Yminus \mathcal Z_\tau; S^\text{Re})&: =& \Big \{ \ v \ \  \Big | \  \ \int_{\Yminus \mathcal Z_\tau} {|v |^2}\ r_\tau^{-2\nu}  d V\ < \ \infty \Big \}\label{L2nunorm}
\end{eqnarray}
\noindent where $\nabla$ denotes the covariant derivative on $S^\text{Re}$ formed from $A_\tau$ and the background pair $(g_\tau,B_\tau)$.  These spaces are equipped with the norms given by the (positive) square root of the integrals required to be finite, and the Hilbert space structures arising from their polarizations. In a slight abuse of notation, we drop the subscript and write simply $r$ for the weight function, which coincides with the radial distance in Fermi coordinates.

\end{defn}
 \bigskip

The expression (\refeq{localform}) shows that $\slashed D_{A_\tau}$ extends to a bounded linear operator  \be \slashed D_{A_\tau}: r^{1+\nu}H^1_e(\Yminus \mathcal Z_\tau; S^\text{Re})\lre r^{\nu}L^2(\Yminus \mathcal Z_\tau; S^\text{Re}).\label{semifredholm}\ee 

\noindent The results of \cite[Prop 2.4]{PartII}, and \cite[Thm 6.1]{MazzeoEdgeOperators} show the following. 
\begin{lm}\label{mappingpropertiesI}
For $-\tfrac12 < \nu < \tfrac12$, 
\begin{enumerate} 
\item[(A)] (\refeq{semifredholm}) is left semi-Fredholm (i.e. $\ker(\slashed D_{A_\tau})$ is finite-dimensional, and $\text{Range}(\slashed D_{A_\tau})$ is closed.)
\smallskip
\item[(B)] The``semi-elliptic'' estimate \be
\|u\|_{r^{1+\nu}H^1_e} \leq C_\nu \Big( \|\slashed D_{A_\tau}u \|_{r^\nu L^2} \ + \ \| \widetilde \pi_\tau (u)\|_{r^\nu L^2}\Big)\hspace{1.5cm} \text{ for all } \ \ u \in r^{1+\nu}H^1_e\label{semiellipticest}
\ee

holds, where $\widetilde \pi_\tau$ is the $L^2$-orthogonal projection onto the finite-dimensional kernel. 
\smallskip
\item[(C)]When the assumptions of Theorem \ref{maina} hold, \eqref{semiellipticest} holds uniformly for $\tau\in(\tau_0, \tau_0)$ when $\pi_\nu$ is replaced by the projection to the 1-dimensional eigenspace spanned by $\Phi_\tau$. 
\end{enumerate}
\end{lm}

\medskip

\noindent For $\nu=0$, the finite-dimensional kernel is, by definition, the space of $\Z_2$-harmonic spinors. The upcoming Lemma \ref{asymptoticexpansion} implies that this space is independent of $\nu$ in the range $-\tfrac{1}{2}<\nu<\tfrac{1}{2}$. 
\medskip

Notice that the estimate (\refeq{semiellipticest}) assumes {\it a priori} that $u\in r^{1+\nu}H^1_e$ and does not imply that an $r^\nu L^2$-solution can be bootstrapped to $u\in r^{1+\nu}H^1_e$. Thus elliptic bootstrapping in the standard sense fails for $\slashed D_{A_\tau}$. Consequently, even for $\nu=0$, the kernel and cokernel of (\refeq{semifredholm}) need not coincide, despite the fact that $\slashed D_{A_\tau}$ is formally self-adjoint. 

Fix a choice of Fermi coordinates near $\mathcal Z_\tau$ as in Definition \ref{Fermicoords}.  The results of \cite[Sec. 7]{MazzeoEdgeOperators}  imply that $\Z_2$-harmonic spinors have the following {\bf polyhomogeneous expansions} along the singular set $\mathcal Z_\tau$. These expansions serve as the appropriate substitute for standard elliptic regularity in the edge setting (see also \cite[App. A]{SiqiSLag}, \cite[Section 3.3]{PartII}, and \cite{grieser2001basics} for more general exposition). 

\begin{lm} \label{asymptoticexpansion}Suppose that $\Phi \in r^{1+\nu}H^1_e(Y-\mathcal Z_\tau; S^\text{Re})$ is a $\Z_2$-harmonic spinor or eigenvector. Then $\Phi$ admits a local polyhomogenous expansion of the following form: 

\begin{eqnarray} \Phi &\sim &\begin{pmatrix} c(t)  \ \ \  \ \  \\ d(t)e^{-i\theta} \end{pmatrix}r^{1/2}  \ + \  \sum_{ n\geq 1  }\sum_{k= -2n }^{2n+1} \sum_{p=0}^{n-1} \begin{pmatrix}  \ \ c_{n,k,p}(t)  e^{ik\theta}  \ \ \ \ \   \\  \ \ d_{n,k,p}(t) e^{ik\theta}e^{-i\theta} \ \ \end{pmatrix} r^{n+1/2}(\log r)^p \label{Phipoly} \end{eqnarray}   
 
\noindent where $c(t), d(t),c_{k,m,n}(t), d_{k,m,n}(t) \in C^\infty(S^1;\C)$. Here, $\sim$ denotes convergence in the following sense: for every $N\in \N$, the partial sums $\Phi_N$ given by the truncation of (\refeq{Phipoly}) at $n=N$ satisfy the pointwise bounds 
\begin{eqnarray} |\Phi-\Phi_N| &\leq& C_{N} r^{N+1+\tfrac14}\label{polyhom1}  \hspace{2cm} |\nabla_t^{\alpha}\nabla^{\beta}(\Phi-\Phi_N) |\leq C_{N,\alpha,\beta} r^{N+1+\tfrac14-|\beta|}\label{polyhom2}\end{eqnarray}

\noindent for constants $C_{N,\alpha,\beta}$ determined by the background data and choice of local coordinates and trivialization. Here, $\beta$ is a multi-index of derivatives in the directions normal to $\mathcal Z_\tau$. 

Moreover, if $\Phi_\tau$ is a family of $\Z_2$-harmonic spinors or eigenvectors,  then all the coefficients $c(t), d(t)$, $c_{n,p,t}(t), d_{n,k,p}(t)$ depend smoothly on $\tau$, and the bounds (\refeq{polyhom2}) are uniform for $\tau$ in a compact set. \qed

\end{lm} 

\medskip 

\noindent Notice that the non-degeneracy condition of Definition \ref{regulardef} is equivalent to the statement that the leading coefficients satisfy $|c(t)|^2 + |d(t)|^2>0$ for all $t\in \mathcal Z_\tau$. Note also that the existence of the expansion (\refeq{Phipoly}) implies that condition $\nabla \Phi\in L^2$ of \eqref{1.1} is equivalent to the requirement that $|\Phi_0|$ extends continuously to $Y$ with $\mathcal Z_0=|\Phi_0|^{-1}(0)$ as in Theorem \ref{HWcompactness}.

\section{The Obstruction Bundle}
\label{section5}

This section describes the infinite-dimensional cokernel of (\refeq{semifredholm}) for the weight $\nu=0$ more precisely. As explained in the introduction, this cokernel obstructs solving the Dirac equation during the gluing iteration, necessitating the introduction of deformations of the singular set as a gluing parameter.

\subsection{The Obstruction Basis}
\label{section5.1}
This first subsection constructs an isomorphism of the cokernel of the operator $\slashed D_{A_\tau}$ from (\refeq{semifredholm}) with sections of a complex line bundle over $\mathcal Z_\tau$ (see Proposition \ref{cokerpropertiesII} below).  The cokernel of (\refeq{semifredholm}) is canonically identified with the $L^2$-solutions of the formal adjoint operator, which for the weight $\nu=0$ is $\slashed D_{A_\tau}$ itself, now understood in a weak sense with domain $L^2$, see Section 2.2 of \cite{PartII}). These weak $L^2$ solution have a singularity of order $r^{-1/2}$ along $\mathcal Z_\tau$; we refer to them as {\bf singular harmonic spinors}. The isomorphism we construct associates the space of such singular harmonic spinors with an appropriate space of boundary data. 

Recall the elementary fact that the space of harmonic functions on a bounded domain with smooth boundary is isomorphic to a Hardy space of functions on the boundary via the Poisson extension operator. The isomorphism in Proposition \ref{cokerpropertiesII} is analogous, with $\mathcal Z_\tau$ playing the role of the boundary, and the leading terms of a (weak) expansion of $\Psi$ playing the role of the boundary values. The theory of such boundary value problems was developed extensively in \cite{MazzeoEdgeOperatorsII, BW25}, and we content ourselves here with only the minimal exposition necessary for our purposes (see also \cite[Sec. 4]{PartII}).

Consider a compact manifold $Y$ with parameter $p_\tau=(g_\tau,B_\tau)$ and $S^\text{Re}, \slashed D_{A_\tau}$ as before. If $\ph  \in  \mathcal D^\text{max}:=\{ u \in L^2(\Yminus \mathcal Z_\tau; S^\text{Re}) \ | \ \slashed D_{A_\tau}u \in L^2\}$ is in the $L^2$-maximal domain, \cite{BW25} shows that it admits a partial {\bf weak polyhomogeneous expansion} of the form 
\be \ph=\frac{1}{2}\left[\begin{pmatrix} \frac{a(t)}{\sqrt{z}} \\ \frac{b(t)}{\sqrt{\overline {z}}}\end{pmatrix}e^{-i\theta/2} \ + \ u\right] +  \frac{1}{2}\sigma\Bigg[ \ \ \ldots \ \ \ \Bigg] \hspace{2cm} u \in rH^1_e(\Yminus \mathcal Z_\tau; S_E)\label{maximaldom}\ee

\noindent in Fermi coordinates and the induced trivialization of Lemma \ref{Fermitrivialization}, where $a(t)\in \Gamma(\mathcal Z_\tau ; N\mathcal Z_\tau^{\otimes k_a})$ and  $b(t)\in \Gamma(\mathcal Z_\tau ; N\mathcal Z_\tau^{\otimes k_b})$ transform as sections of an appropriate power of the normal bundle under changes of Fermi coordinates, and $\sigma$ symmetrizes so that $\ph \in S^\text{Re}$ as in (\refeq{realimdecomp}). These expansions may be viewed as a weaker version of the expansions in Lemma \ref{asymptoticexpansion}, where $u$ collects all the higher-order terms. In fact, \cite[Cor. 3.9]{PartII} shows that in this case, $k_a=k_b=0$. The edge calculus analogue of the boundary trace is the map 
$$\text{tr}_e:\mathcal D^\text{max} \to L^{-1/2,2}\left(\mathcal Z_\tau ; \mathcal C_\tau \oplus \mathcal C_\tau^{-1}\right),$$

\noindent which extracts the leading coefficients $a(t),b(t)$ above, where $\mathcal C_\tau \to \mathcal Z_\tau$ is a trivial complex line bundle with a fixed trivialization induced by {\it any} choice of Fermi coordinates. We refer to $\mathcal C_\tau$ as the {\bf Calder\'on bundle}. The {\bf Calder\'on subspace} $\Lambda_\tau^\text{Cald}$ is the image of the space of singular harmonic spinors under  $\text{tr}_e$, namely 
\be \Lambda_\tau^\text{Cald}:= \text{tr}_e\left(\ker \slashed D_{A_\tau}^\text{Re}\big|_{L^2}\right) \subseteq L^{-1/2,2}(\Yminus \mathcal Z_\tau; \mathcal C_\tau \oplus \mathcal C_\tau^{-1}). \label{calderonbunddef}\ee

\noindent In the general theory, this is viewed as the inclusion of a closed Lagrangian subspace with respect to an appropriate symplectic form, \cite{BW25}.

\begin{eg} \label{Poissoneg}
To motivate the statement of Proposition \ref{cokerpropertiesII}, first consider the model operator $\slashed D_{A_\circ}$ on $Y^\circ=S^1\times \R^2$ equipped with the product metric where $\mathcal Z_\circ=S^1\times \{0\}$ has length $2\pi$, and $E=\underline \C^2$ is the trivial bundle. This model operator has the form  (\refeq{localform}) with $\frak d_1,\frak d_0=0$, and (\refeq{realimdecomp}) shows that $S^\text{Re}\subset \underline \C^2 \otimes E$ consists of elements of the form $\tfrac{1}{2}(\Psi + \sigma \Psi)$ where $\Psi\in \Gamma(S_E)$. Direct computation via separation of variables (see \cite[Sec. 3]{PartII}, \cite{RyosukeThesis}) shows that the $L^2$-kernel of $\slashed D_{A_\circ}^\text{Re}$ is the $L^2$ span of

\be \Psi_\ell^\circ:=\frac{1}{2} \left[\sqrt{|\ell|}e^{i\ell t}e^{-|\ell|r}\begin{pmatrix}\tfrac{1} {\sqrt{z}} \\ \tfrac{\text{sgn}(\ell)}{ \sqrt{\overline z}}\end{pmatrix} e^{-i\theta/2} \right]  \ \ + \ \  \frac{1}{2}\sigma\Bigg[ \ \ \ \ \ldots \ \ \ \Bigg] \label{Psilmodel}\ee 

\noindent for $\ell \in \Z$, where the second term symmetrizes so that $\Psi_\ell^\circ\in S^\text{Re}$. Notice that the leading coefficient can be written $(e^{i\ell t}, H e^{i\ell t})$ as functions of $t$, where $H$ is the zeroth-order pseudodifferential operator with symbol $\text{sgn}(\ell)$\footnote{Here, we are glossing over the fact that the $\ell=0$ modes are not in $L^2$ on the non-compact $Y^\circ$; compactness of $Y$ ameliorates this, see Section 4.3 of \cite{PartII}.} . Thus, in the model case, $\zeta(t)\mapsto (\zeta(t), H \zeta(t))$ defines an isomorphism 

\be L^{-1/2,2}(\mathcal Z_\circ; \mathcal C_\circ)]\overset{\cong}\lre \Lambda_\circ^\text{Cald} \label{Caldiso}\ee
to the model Calder\'on subspace, where $\mathcal C_\circ$ is the model Calder\'on bundle. There is also a Poisson extension operator $\frak P_\circ: \Lambda_\circ^\text{Cald}\to \ker\left(\slashed D_{A_\circ}^\text{Re}|_{L^2}\right)$ given by

\bea
\frak P_0 \begin{pmatrix} \zeta(t) \\  H(\zeta(t))\end{pmatrix}  &:=& \sum_{\ell} \zeta_\ell \Psi_\ell^\circ 
\eea 
where $\zeta_\ell$ are the Fourier coefficients of $\zeta(t)$ in a trivialization of $\mathcal C_0$ and $\Psi_\ell^\circ$ is as in (\refeq{Psilmodel}. This Poisson operator is an isomorphism with inverse $\text{tr}_e$. 
\end{eg}

Returning to the general case, given a link $\mathcal Z_\tau\subseteq Y$ with a Fermi coordinate neighborhood of each component, we may extend the model Poisson operator $\frak P_\circ$  in these coordinates and the trivializations of Lemma \ref{Fermitrivialization} to all of $Y$ using a radial cut-off function $\chi$. Composing the resulting operator $\chi \frak P_\circ$ with the $L^2$-projection to the kernel of $\slashed D_{A_\tau}$ gives linear maps 
    \begin{center}\be 
\tikzset{node distance=3.5cm, auto}
\begin{tikzpicture}
\node(A){$L^{-1/2,2}(\mathcal Z_\tau; \mathcal C_\tau)$};
\node(C)[right of=A]{$L^2(\Yminus \mathcal Z_\tau; S^\text{Re})$};
\node(C')[right of=C, yshift=0cm]{$\ker\left(\slashed D_{A_\tau}^\text{Re}\Big |_{L^2}\right). $};
\draw[->] (A) to node {$\chi \frak P_\circ$} (C);
\draw[->] (C) to node {$  \text{proj}_{L^2} $} (C');
\end{tikzpicture}
\label{compositionpoisson}
\ee \end{center}

\noindent Proposition \ref{cokerpropertiesII} below essentially says that (\refeq{compositionpoisson}) is an isomorphism (see \cite[Sec. 4]{PartII} for details). Thus (\refeq{compositionpoisson}) gives a way of identifying the space of singular harmonic spinors (the right hand side of (\refeq{compositionpoisson})) with the model Calder\'on subspace $\Lambda_\circ^\text{Cald}$, which in turn, is identified with $L^{-1/2,2}(\mathcal Z_\tau; \mathcal C_\tau)$ via (\refeq{Caldiso}). This approach has the advantage that it identifies the space of singular harmonic spinors with a space of sections of a bundle over $\mathcal Z_\tau$, rather than the more abstract subspace $\Lambda^\text{Cald}_\tau$. In particular, this allows the interpretation of (\refeq{mainamap}) in Theorem \ref{PartIImainb} as a pseudodifferential operator.

Before stating the proposition, there is one subtlety to address. The space of singular harmonic spinors need not form a vector bundle over the parameter space $\tau\in (-\tau_0, \tau_0)$. In the classical elliptic setting of operators with a finite-dimensional kernel, this fails because the dimension of the kernel may jump; here there may be similar discontinuities, although the dimension is always infinite. To ameliorate this, we employ the standard trick of ``thickening'' the space by low eigenvectors. Here, because $\Z_2$-harmonic spinor $\Phi_0$ is assumed to be isolated (Definition \ref{regulardef}), we need only consider a 1-dimensional thickening. Let \be r H^1_\perp(\Yminus\mathcal Z_\tau: S^\text{Re}):= \Big\{\ph \in rH^1_e \ | \ \br \ph, \Phi_\tau\kt_{L^2}=0\Big\}\label{rH1perp}\ee 

\noindent where $\Phi_\tau$ is the $\Z_2$-harmonic eigenvector from Definition \ref{spectralcrossingdef}.

\begin{defn}\label{Obdefn}
The {\bf obstruction space} associated to the data $(\mathcal Z_\tau, g_\tau, B_\tau)$ is defined as the $L^2$-orthogonal complement of the range of $\slashed D_{A_\tau}$ on the spaces (\refeq{rH1perp}). Thus for every  $\tau\in (-\tau_0, \tau_0)$, there is an orthonormal decomposition

$$L^2(\Yminus \mathcal Z_\tau; S^\text{Re})=: \text{\bf Ob}(\mathcal Z_\tau)\oplus  \text{Range}\left(\slashed D_{A_\tau}^\text{Re}\Big |_{rH^1_\perp}\right)^\perp.$$

\noindent We denote the projections to the first and second factors by $\Pi_\tau, \Pi_\tau^\perp$ respectively. We also define the {\bf obstruction bundle} $$\text{\bf Ob}:=\bigsqcup_{\tau}\text{\bf Ob}(\mathcal Z_\tau) \to (-\tau_0,\tau_0)$$
 as the parameterized family of obstruction spaces
\end{defn}

\medskip

\noindent At this point, the obstruction bundle is simply a family of Hilbert spaces -- we will show in Section \ref{section6} that it is indeed a Hilbert vector bundle. By construction, $\ker(\slashed D^\text{Re}|_{L^2})\subseteq \text{\bf Ob}(\mathcal Z_\tau)$. This is an equality at $\tau=0$, and an inclusion complemented by the span of $\Phi_\tau$ for $\tau \neq 0$ (note that $\Phi_\tau$ ).

The following proposition is proved in \cite{PartII}. It formalizes the construction (\refeq{compositionpoisson}), and provides regularity estimates needed in later sections. Compared to  (\refeq{compositionpoisson}), the proposition includes a shift by $1/2$ a degree of regularity, see Remark \ref{regrem} below.

\begin{prop}{\bf (\cite[Prop 4.4]{PartII}}\label{cokerpropertiesII} 
For $\tau \in (-\tau_0,\tau_0)$, there is a bounded linear isomorphism 

\be \text{\bf ob}_\tau : L^2(\mathcal Z_\tau; \mathcal C_\tau)\oplus \R \to \text{\bf Ob}(\mathcal Z_\tau)\ee

\noindent whose components on the summands are denoted $\text{\bf ob}_\tau=(\text{ob}_\tau, \iota_\tau)$ where $\iota_\tau(a)=a\Phi_\tau$ and such that the following properties hold uniformly for $\tau\in (-\tau_0, \tau_0)$. 

\begin{enumerate}
\item[(A)] When $\slashed D_{A_\tau}$ is complex linear and $\mathcal Z_\tau$ has a single component, the inverse $\text{\bf ob}_\tau^{-1}$ is given in terms of the spinors $\Psi_\ell:= \text{ ob}_\tau(e^{i\ell t})$ for $\ell \in 2\pi \mathbb Z/ |\mathcal Z_\tau|$ as follows. For a spinor $\psi\in L^2$, the inverses of $(\text{\bf ob}_\tau, \iota_\tau)$ are given respectively by 
\be
 \text{ob}_\tau^{-1} (\Pi_\tau \psi)= \sum_{\ell} \br\psi,  \Psi^{}_\ell  \kt_{{\C}} \cdot  e^{i\ell t}\label{innerprods} \hspace{3cm} \iota^{-1}(\Pi_\tau\psi)= \br  \psi, \Phi_\tau\kt_{L^2}\cdot  \Phi_\tau,\ee

\noindent where $\br -,-\kt_\C$ is the hermitian inner product. Here, $e^{i\ell t}$ are Fourier modes on $\mathcal Z_\tau$ in the trivialization of $\mathcal C_\tau$ induced by a choice of orientation.

\smallskip 

\item[(B)] The basis vectors $\Psi_\ell$ as in part (A) each admit a decomposition 
\be \Psi_\ell = \chi \Psi_\ell^\circ + \zeta_\ell + \xi_\ell \ee

\noindent wherein 
\begin{itemize}
\item[{\it (i)}] $\Psi_\ell^\circ$ are the model singular harmonic spinors  in (\refeq{Psilmodel}) in Fermi coordinates, and $\chi$ is a fixed radial cutoff function.  
\item[{\it (ii)}] There are constants $c_1<1$ and $C>1$ such that the linear extension of $Z(e^{i\ell t})=\zeta_\ell$ obeys the bounds

\be \hspace{3cm}\|Z( \nabla_t \varrho(t))\|_{L^2(\Yminus\mathcal Z_\tau)} \leq C\|\varrho(t)\|_{L^2(\mathcal Z_\tau)} \hspace{1cm} \| e^{c_1 r\sqrt{\Delta_t}} Z(\varrho(t))\|_{L^2(\Yminus \mathcal Z_\tau)}\leq C\|\varrho(t)\|_{L^2(\mathcal Z_\tau)}\ee

for $ \varrho(t)\in L^2(\mathcal Z_\tau, \mathcal C_\tau)$. Moreover, $Z$ is supported in the Fermi coordinate chart of $\mathcal Z_\tau$ and respects Fourier modes in the $t$-direction in the sense that $Z(e^{i\ell t})$ is supported in Fourier modes between $\ell-\frac{|\ell|}{2}$ and $\ell+ \frac{|\ell|}{2}$. 
\item[{\it (iii)}] For $N\in \N$, there are constants $C_N>1$ such that the linear extension of $X(e^{i\ell t})=\xi_\ell$ obeys the bounds 
\bea \hspace{3cm} \|X(\nabla^N_t \varrho(t))\|_{L^2(\Yminus\mathcal Z_\tau)}& \leq& C_N \|\varrho(t)\|_{L^2(\mathcal Z_\tau)} \\
\|  \nabla_b^\alpha\nabla_t^\beta X( \varrho(t))\|_{L^2(\Yminus\mathcal Z_\tau)}& \leq &  C_{\alpha\beta} \|X(\nabla^\beta_t \varrho(t))\|_{L^2(\Yminus\mathcal Z_\tau)}   \eea

for $ \varrho(t)\in L^2(\mathcal Z_\tau, \mathcal C_\tau)$, and $\nabla_b=r\nabla_r$ or $\nabla_\theta$ in Fermi coordinates. The constants $C_N$ depend on up to the $C^{N+3}$-norm of $(g_\tau, B_\tau)$. 
\noindent 
\end{itemize}

\bigskip

 \item[(C)] $\text{\bf ob}_\tau$ respects regularity in the sense that it restricts to a bounded linear isomorphism 
 
 \be \text{\bf ob}_\tau : L^{s,2}(\mathcal Z_\tau; \mathcal C_\tau)\oplus \R \to \text{\bf Ob}(\mathcal Z_\tau) \cap H^s_b(\Yminus\mathcal Z_\tau; S^{\text{Re}})\ee

\noindent \noindent for any $s\in \N$ (see below for the definition of $H^s_b$).
\bigskip

\item[(D)] In the case that $\slashed D_{A_\tau}$ is only $\R$-linear and $\mathcal Z_\tau$ has multiple components, then (A) --(B) continue to hold when we define  $$\text{ ob}_\tau(e^{i\ell t_j}):=\Psi_{\ell,j}^{\text{re}} \hspace{2cm}\text{ob}_\tau(ie^{i\ell t_j}):= \Psi_{\ell,j}^{\text{im}}$$ for $j$ indexing the components of $\mathcal Z_\tau$, and (\refeq{innerprods}) uses the inner product $\br \psi, \Psi_{\ell,j}^{}\kt_\C= \br \psi, \Psi_{\ell,j}^{\text{re},{}}\kt \ + \ i \br \psi, \Psi_{\ell,j}^{\text{im},{}}\kt.$ Each of these also admits a decomposition $\Psi_{\ell,j}^\text{re}=\chi_j \Psi_{\ell}^\circ + \zeta_{\ell,j}^\text{re}+\xi_{\ell,j}^\text{re}$ where the bounds of (B) hold uniformly in $j$, and likewise for $\Psi_{\ell,j}^\text{im}$.

\end{enumerate}\qed

\end{prop}

To clarify the notation used in Item (B.ii), recall that the pseudodifferential operator $e^{c_1 r \sqrt{\Delta t}}$ is defined via functional calculus to act on each Fourier mode by $e^{c_1 r \sqrt{\Delta t}}(e^{i\ell t})=e^{c_1|\ell| r}e^{i\ell t}$. In particular, Item (B.ii) shows that each $\zeta_\ell$ enjoys exponential decay similar to (\refeq{Psilmodel}) with $1/e$ length of size $O(1/|\ell|)$ in each mode. In the same item, note that $Z(e^{i\ell t})$ is a spinor on $\Yminus\mathcal Z_\tau$, so its Fourier mode decomposition in the $t$-direction depends parametrically on the normal variables $(x,y)$, and satisfies the assertion for every pair $(x,y)$ individually.  Finally, for part (C), $H^s_b(\Yminus \mathcal Z_\tau; S^\text{Re})$ is defined as the space of spinors $\Psi$ such that 
$$\nabla^{\alpha_1}\nabla^{\alpha_2}\ldots \nabla^{\alpha_k} \Psi \in L^2$$

\noindent for all multi-indices $\alpha=(\alpha_1,\ldots, \alpha_k)$ of weight $s$, where each index corresponds to one of the ``boundary'' derivatives $\nabla_t, r\nabla_r, \nabla_\theta$ in Fermi coordinates and standard covariant derivatives away from a neighborhood of $\mathcal Z_\tau$.

\begin{rem} \label{regrem}The choice of the regularity convention on the domain of $\text{\bf ob}_\tau$ is a convention, since this isomorphism can be precomposed with $(1+\Delta_t)^{s/2}$ for any power $s$. There are two natural choices: the ``boundary regularity'' convention, where the domain is $L^{-1/2,2}(\mathcal Z_\tau, \mathcal C_\tau)$, and the ``ambient regularity'' convention so that the map in Item (C) has order $0$. In Proposition \ref{cokerpropertiesII}  (and henceforth), we adopt the latter convention. 
\end{rem}

\subsection{Conormal Regularity}
\label{conormalsection}

A fundamental aspect of elliptic edge calculus is the intertwining of tangential regularity along the singular set $\mathcal Z_\tau$ with the rate of growth of spinors in the radial direction \cite{MazzeoEdgeOperators, MazzeoEdgeOperatorsII}. One manifestation of this phenomenon appears when considering the regularity properties of the projection $\Pi_\tau: L^2 \to \text{\bf Ob}(\mathcal Z_\tau)$. This regularity relationship is fundamental in \cite{PartII} and in the proof of Theorem \ref{maina}, as it gives rise to both the problematic loss of regularity and also the solution thereto. This section, specifically Lemma \ref{fourierregimescokernel} below, formalizes the discussion in Subsection \ref{smoothingprojections}.

For a simple example, consider again the model problem on $Y^\circ=S^1\times \R^2$, where the obstruction is spanned by square summable linear combinations of (\refeq{Psilmodel}). For a spinor $\psi\in L^2(\Yminus \mathcal Z_\tau)$ supported where $r\geq R_0$ for some $R_0$ the projection $\text{ob}^{-1}\Pi_\tau (\psi)$ can be calculated in Fourier modes by direct integration as

$$  \text{ob}^{-1}\Pi_\tau(\psi)=\sum_{\ell \in \Z} \br \psi, \Psi_\ell^\circ\kt_{L^2(r\geq R_0)}\cdot e^{i\ell t} = \sum_{\ell \in \Z} c_\ell e^{i\ell t} \hspace{1cm}  \text{ where }\hspace{1cm}|c_\ell|\leq \sqrt{|\ell|} e^{-|\ell| R_0}\cdot  \|\psi\|_{L^2}.$$

\noindent by Cauchy-Schwartz. In particular, $\text{ob}^{-1}\Pi_\tau(\psi)\in L^{s,2}$ for every $s\geq 0$, even if $\psi$ has no weak derivatives in $L^2$. Proposition \ref{cokerpropertiesII}(A,B) shows that on a closed $Y$, the regularity of this projection remains a question of how fast the sequence of inner products (\refeq{innerprods}) decays as $|\ell|\to \infty$.

More precisely (see \cite[Sec. 6.1]{PartII} for additional details), we say that a spinor $\psi$ has obstruction component of regularity $s$ if $\Pi_\tau(\Psi)\in \text{\bf Ob}(\mathcal Z_\tau) \cap H^s_b$ for some $s\geq 0$. Equivalently, by of Proposition \ref{cokerpropertiesII}(C), this means that $\text{ob}_\tau^{-1}\circ \Pi_\tau(\Psi)\in L^{s,2}(\mathcal Z_\tau; \mathcal C_\tau) \subseteq L^2(\mathcal Z_\tau; \mathcal C_\tau)$. Given $L_0\in \N$, we also define Fourier projections  
  $${ \pi}_{L_0}(e^{i\ell t})=\begin{cases} e^{i\ell t} \hspace{1cm}|\ell|\leq L_0 \\ 0 \  \  \hspace{.95cm} \  |\ell| > L_0\end{cases}  $$
   \noindent to modes less than $|L_0|$ in $L^2(\mathcal Z_\tau;\mathcal C_\tau)$. The following is a fundamental lemma for the gluing iteration, as explained in Subsection \ref{smoothingprojections}: it says that for spinors compactly supported away from $\mathcal Z_\tau$, the projection to the obstruction is smooth norms bounded by the distance from $\mathcal Z_\tau$.

   \begin{lm}\label{fourierregimescokernel}
Fix $0<\gamma<<1$. Then for any $N\in \N$,   there are $C_N,R_N>0$ such that the following holds. If $\Psi \in L^2(Y-\mathcal Z_\tau; S^\text{Re})$ is a spinor and $\text{dist}(\text{supp}(\Psi),\mathcal Z_\tau)
\geq R_N$. 

$$\| (1-\pi_{L_0})\circ \text{ob}^{-1}_\tau \circ \Pi_\tau(\Psi)\|_{L^2} \leq \frac{C_N}{(L_0)^N} \|\Psi\|_{L^2}$$

\noindent holds for any $L_0>R_N^{-\tfrac{1}{1-\gamma}}$. In particular, $\text{ob}_\tau^{-1}\circ \Pi_\tau(\Psi)\in C^\infty(\mathcal Z_\tau; \mathcal C_\tau)$.

  \end{lm}
  
  \begin{proof}
By Proposition \ref{cokerpropertiesII}, the projection $(1-\pi^\text{low})\circ \text{ob}_\tau^{-1} \circ \Pi_\tau (\Psi)$ is given by 

\be \sum_{|\ell|> L_0} \br  \Psi_\ell , \Psi \kt_{\C} \cdot e^{i\ell t}. \label{innerproductseq}\ee

\noindent The result then follows from using the decomposition $\Psi_\ell= \chi\Psi^\circ_\ell + \zeta^{}_{\ell}+ \xi^{}_{\ell}$. Using Items (B.i) and (B.ii) in Proposition \ref{cokerpropertiesII}, the first two terms --- $\Psi^\circ_\ell$ and $\xi_\ell$ --- have $L^2$-norm at most $O(\text{Exp}(-|\ell|R_N/c_1)$. In particular, for $|\ell|\geq L_0$, at most $O(\text{Exp}(-L_0R_N/c_1))$. The assumption on $L_0$ implies that $|\ell| R>|\ell|^\gamma$ once $|\ell|\geq L_0$, thus the $L^2$-norm of the first two term in the region $r\geq R$ is dominated by $O(\text{Exp}(-|\ell|^\gamma/c_1))$. For $R_N$ sufficiently small (thus $|\ell|$ sufficiently large), this is smaller than $|\ell|^{-N-2}$. Using Cauchy-Schwartz on $\xi_\ell$ and the third bullet point of Item (A) with $N'=N+2$ then gives the desired bound. Then, summing over $|\ell|$ one has 

$$\sum_{|\ell|> L_0} \br  \Psi_\ell , \Psi \kt_{\C} \cdot e^{i\ell t} \leq \sum_{|\ell|\geq L_0} \frac{C_N}{|\ell|^{N+2}} \|\Psi\|_{L^2}\leq \frac{C_N}{|L_0|^N} \|\Psi\|_{L^2}$$

\noindent as desired. Since $N$ was arbitrary, and the projection $(1-\pi_{L_0})$ excludes finitely many low Fourier modes, the fact that the projection is $C^\infty$ follows.

\end{proof}

\subsection{The Surjective Weights}
\label{section5.2}
The fact that the singular Dirac equation has an obstruction to solving in Lemma \ref{mappingpropertiesI} is a factor of the weights on the function space. In general, one cannot ensure {\it both} that there is a solution of $\slashed Du=f$ {\it and} that the solution $u$ vanishes along $\mathcal Z_0$; if one allows $u$ to become singular along $\mathcal Z_0$ then the problem has only a finite-dimensional obstruction. As explained in Section \ref{section2.5}, these singular spinors will be used to cancel the obstruction in high Fourier modes. Sections \ref{section9}--\ref{section10} show that for a particular choice of $L_0$, Lemma \ref{fourierregimescokernel} ensures the Fourier modes above a sufficient choice of  $L_0$ are sufficiently small that canceling them this way does not disrupt the convergence of the iteration scheme, despite the growth of these solutions across the neck region.

In the present section, we establish the semi-elliptic estimates that allow the singular Dirac equation to solved, provided solutions with singularities along $\mathcal Z_\tau$ are permitted. By the general theory of elliptic edge operators (\cite[Thm 6.1]{MazzeoEdgeOperators}), a first order elliptic edge operator $L$$$L:r^{1+\nu}H^1_e \to r^{\nu} L^2$$

\noindent  is semi-Fredholm provided that the weight $\nu$ lies outside the discrete set $I(L)$ of indicial roots. More specifically, there are critical weights $ \underline \nu<\overline \nu,$ such that (i) $L$ is left semi-Fredholm  (finite-dimensional kernel and closed range) for  $\overline \nu<\nu \notin I(L)$, and (ii) is right semi-Fredholm (finite-dimensional cokernel) for   $\underline \nu>\nu \notin I(L)$. For the singular Dirac operator $\slashed D_{A_\tau}$, the critical weights (by \cite[Prop 3.9]{MazzeoHaydysTakahashi}) are:
\be
I(r\slashed D_{A_\tau})= \Z + \tfrac{1}{2}\hspace{2.5cm}
\overline \nu=
\underline \nu =- \tfrac{1}{2}\label{indicialroots}.
\ee

\noindent Thus when the weight $\nu$ decreases past the critical weight $-\tfrac12$, $\slashed D_{A_\tau}$ flips from being left semi-Fredholm to being right semi-Fredholm. 

Lemma \ref{mappingpropertiesI} shows \eqref{semifredholm}  is left semi-Fredholm for $\nu\in (-\tfrac{1}{2},\tfrac12)$. The next lemma gives a precise statement in the right semi-Frehdolm, for the weight $\nu=-1$. We use the notation $H^1_e$ for the space $r^0H^1_e$ as in Definition \ref{rH1edef} for this weight. 

\begin{lm}\label{hittingob} For $\nu=-1$,  \be \slashed D_{A_\tau}: H^1_e(\Yminus \mathcal Z_\tau; S^\text{Re})\lre r^{-1}L^2(\Yminus \mathcal Z_\tau ; S^\text{Re})\label{weightedmapping}\ee
has a finite-dimensional cokernel. Moreover, if $f\perp_{L^2}\text{Span}(\Phi_\tau)$, then there is a unique solution of $\slashed D_{A_\tau}u=f$ such that $u$ is $L^2$-orthogonal to  $\text{\bf Ob}(\mathcal Z_\tau)\cap H^1_e$ and obeys the elliptic estimate
\be \|u\|_{H^1_e}\leq C \|f \|_{r^{-1}L^2} .\label{surjectiveest}\ee

\noindent uniformly in $\tau$. 
\end{lm}
\begin{proof}The (weakly defined) second order operator 

$$\slashed D_{A_\tau}^\star\slashed D_{A_\tau}: rH^1_e(\Yminus \mathcal Z_\tau; S^\text{Re})\lre r^{-1}H^{-1}_e(\Yminus\mathcal Z_\tau; S^\text{Re})$$

\noindent is surjective for $|\tau|>0$ and has 1-dimensional cokernel spanned by $\Phi_0$ at $\tau=0$ via integration by parts. The weakly-defined adjoint operator $\slashed D^\star_{A_\tau}:\{u \in L^2\ | \ u \perp \text{\bf Ob}(\mathcal Z_\tau)\}\to r^{-1}H^{-1}_e$ is therefore surjective onto the codomain modulo the span of $\Phi_\tau$ with uniform elliptic estimates (see \cite[Sec. 2]{PartII} for details). The bootstrapping results of \cite[Thm 6.1]{MazzeoEdgeOperators} and \cite[Thm 3.18]{MazzeoHaydysTakahashi} show that if $f \in r^{-1}L^2 \cap r^{-1}H^{-1}_e$, then the solution is in $H^1_e$, and (\refeq{surjectiveest}) holds. 
\end{proof}

The following corollary is a direct application of the previous lemma. In it, we denote 

$$\text{\bf Ob}^\perp(\mathcal Z_\tau):= \{ \psi \in \text{\bf Ob}(\mathcal Z_\tau) \ | \ \br \psi, \Phi_\tau\kt_{L^2}=0\}.$$

\noindent Note that the $L^2$ norm dominates the $r^{-1}L^2$ norm, so this is a closed subspace in $r^{-1}L^2$. 
\begin{cor}\label{solvingcokernel} There is a closed subspace $\mathcal X_\tau\subseteq H^1_e$ such that $$\slashed D_{A_\tau}: \mathcal X_\tau\to \text{\bf Ob}^\perp(\mathcal Z_\tau)$$ is an isomorphism. In particular, there is a $C>0$ such that if $\Psi \in \text{\bf Ob}^\perp(\mathcal Z_\tau)$, then there exists a unique $u_\Psi \in \mathcal X_\tau$  satisfying  $$\slashed D_{A_\tau} u_\Psi = \Psi, \hspace{2cm}\|u_\Psi\|_{H^1_e}\leq C \|\Psi\|_{L^2},$$
\noindent where $C$ is uniform in $\tau$. 
\end{cor}

\begin{proof} The conclusion then follows directly from Lemma \ref{hittingob}, taking $f=\Psi \in \text{\bf Ob}(\mathcal Z_\tau)$. \end{proof}

\begin{rem}In fact, \cite[Thm. 7.14]{MazzeoEdgeOperators} describes the form of $u_\Psi$ more precisely. Recall elements of the $L^2$-maximal domain of $\mathcal D_{A_\tau}$ may be written in the form (\refeq{maximaldom}). Proposition \ref{cokerpropertiesII} implies that $\text{\bf Ob}^\perp(\mathcal Z_\tau)$ consists of $L^2$-spinors whose leading coefficients (modulo a compact operator) lie in the subspace where $b(t)=H(a(t))\in L^{-1/2,2}$. The space $\mathcal X_\tau$, modulo another compact operator, has boundary values which fill out the complementary subspace where $b(t)=-H(a(t))$.

\end{rem}

 \section{Deformations of Singular Sets}
 \label{section6}
  
  This section develops the theory of the singular Dirac operator in the case that the singular set may vary. We begin by carefully constructing charts and trivializations for the relevant Banach manifolds and vector bundles. Then, we define the universal Dirac operator $\slashed{\mathbb D}$ originally introduced in (\refeq{universalDiracdef}) more precisely as a map on these bundles. The main result of this section is the calculation of the partial linearization  $\text{d}\slashed{\mathbb D}$ with respect to deformations of the singular set given in Theorem \ref{deformationsmaina},  which is a more precise version of Theorem \ref{PartIImainb}.

\subsection{Deformation Families}
\label{embeddings}
For each fixed $\tau\in (-\tau_0, \tau_0)$, the metric $g_\tau$ determines an exponential map $\text{exp}^\tau: N\mathcal Z_\tau \to Y$, where $N\mathcal Z_\tau$ is the normal bundle to $\mathcal Z_\tau$. For $r_0$ sufficiently small, the restriction of this map to the $r_0$ neighborhood of the zero-section is a diffeomorphism onto its image (in particular, we may assume $r_0$ is sufficiently small that these images are disjoint for the components of $\mathcal Z_\tau$). Via this identification, a choice of framing of $N\mathcal Z_\tau$ induces a choice of Fermi coordinates as in Definition \ref{Fermicoords}, in which $\mathcal Z_\tau$ is identified with the zero-section in $N\mathcal Z_\tau$. 

Now fix an $r_\mathcal E>0$ and set 
\begin{eqnarray}\label{6.DefDef}
\mathcal E_\tau =\ \Big\{ \eta\in L^{2,2}(\mathcal Z_\tau, N\mathcal Z_\tau)\, \Big|\, \|\eta\|_{2,2}<r_\mathcal E\, \Big\}.
\end{eqnarray}
  By the Sobolev embedding $L^{2,2}\hookrightarrow C^1$ in dimension~1, we can choose $r_\mathcal E$ small enough that the image of each $\eta$ in \eqref{6.DefDef}, regarded as a subset of $N\mathcal Z_\tau$,   is $C^1$ close to the zero section and hence defines a link in $Y$ that is a perturbation of $\mathcal Z_\tau$.  Conversely, each $C^1$ small perturbation of $\mathcal Z_\tau$, also regarded as a subset of $N\mathcal Z_\tau$,  is transverse to the fibers of $N\mathcal Z_\tau$, so is the image $\eta(\mathcal Z_\tau)$ of a section $\eta$.  Accordingly, for such a choice of $r_\mathcal E$,  we take \eqref{6.DefDef} as our space of deformations of $\mathcal Z_\tau$ in $Y$.  This is an open set in a Hilbert space, and hence a Hilbert manifold. Elements of $\mathcal E_\tau$ are identified with the corresponding set of embedded links $\text{Emb}(Y)$  by 
  \begin{eqnarray}
\text{Exp}_\tau: \mathcal E_\tau &\to& \text{Emb}(Y) \label{Expdef} \\
\eta &\mapsto &\text{exp}^\tau(\eta) \nonumber
\end{eqnarray}

\noindent where $\text{exp}^\tau: T\mathcal Z_\tau \to Y$ is the exponential map, and $\text{exp}^\tau(\eta)$ is regarded as a subset of $Y$.

In this context, the conclusion of Theorem \ref{SpectralCrossing} that each $\mathcal Z_\tau$ is regular (in particular smooth by Definition \ref{regulardef}(i)), and depend smoothly on $\tau$ means the family $\{\mathcal Z_\tau\}$ is given in $\mathcal E_0$ by the image of a smooth map \be\iota_{\mathcal Z}: (-\tau_0, \tau_0)\to \mathcal E_0 \cap C^\infty(\mathcal Z_0; N\mathcal Z_0).\label{iotzembedding}\ee Since the proof of Theorem \ref{maina} requires perturbing the link for each parameter pair $(\e,\tau)$ independently, we work  with a family of exponential charts centered at $\mathcal Z_\tau$ for each $\tau$. To this end, consider the bundle
\begin{eqnarray}\label{6.defE1}
\begin{tikzcd}
{\mathcal E}  \arrow[d, "p_2"] &  \hspace{-16mm}  =  \Big\{(\eta,\tau)\, |\, \eta\in {\mathcal E}_\tau \Big\}  \\
  (-\tau_0, \tau_0) &
\end{tikzcd}
\end{eqnarray}
The original family $\{\mathcal Z_\tau\}$ then corresponds to the zero section of ${\mathcal E}$, and the perturbations of this family that will be considered in Sections~\ref{section6.3}--\ref{section12} are defined by sections of ${\mathcal E}$. 

To endow $\mathcal E$ with the structure of a smooth fiber bundle, let $i :(-\tau_0, \tau_0) \times  \bigsqcup S^1 \to (-\tau_0, \tau_0)\times Y$ be the parameterized family of embeddings induced by (\refeq{iotzembedding}), so that $\iota(\tau, -)$ restricts to a parameterization of $\mathcal Z_\tau$ with constant velocity for every fixed $\tau$. Parallel transport in the $\tau$ direction using $d\tau^2 + g_\tau$ along with a trivialization of $N\mathcal Z_0$ at $\tau=0$ yields a trivialization of the pullback bundle $i^*(N\mathcal Z_\tau)$, and thus of the bundle $L^{2,2}(\mathcal Z_\tau; N\mathcal Z_\tau)\to (-\tau_0,\tau_0)$. $\mathcal E$ inherits a trivialization as open subsets of these fibers. We may assume, in this trivialization, that $r_\mathcal E$ is chosen uniformly in $\tau$, and that the family of framings used to define Fermi coordinates is parallel.

\subsection{The Gauge Choice of Diffeomorphisms}
\label{gaugefreedomsection}

In this subsection, we explain the gauge freedom coming from diffeomorphisms in more detail. For the purposes of exposition, the motivating geometric picture is described here in some detail, although the constructions in the subsequent section \ref{admissiblefamiliessection} have no explicit logical dependence on this picture.

Over this space $\mathcal E$ of embedded links there is a universal bundle of 3-manifolds  \begin{eqnarray}\label{6.family2}
\begin{tikzcd}
  {\mathbb Y} \arrow[d, "p"] &  \hspace{-10mm}  =\mathcal E \times  Y  \\
\mathcal E& &
\end{tikzcd}
\end{eqnarray}
 where $p$ is the projection onto the first factor, and we give the fiber over $(\eta,\tau)$ the metric $g_\tau$. There is a universal link $\mathcal Z^\text{univ}:=\{(\eta, \text{exp}^\tau(\eta(t))) \ | \ \eta \in \mathcal E_\tau \ , \ \tau \in (-\tau_0, \tau_0) \ , \ t\in \mathcal Z_\tau\} \subseteq \mathbb Y$, whose intersection with the fiber over $\eta$ is precisely the link $\mathcal Z_{\eta,\tau}=\text{Exp}_\tau(\eta)$.

There are ``universal'' vector bundles over $\mathbb Y, \mathbb \Yminus \mathcal Z^\text{univ}$, whose restriction to fibers $p^{-1}(\eta,\tau)=(Y,g_\tau)$  or are given as follows. 
 
 \begin{itemize}
\item[(1)] $S, E\to (Y,g_\tau)$ the restrictions of the product bundles $\mathcal E\times S$ and $\mathcal E\times E$, and $S_E$ the restriction of $\mathcal E\times S_E$. 
\smallskip
 
\item[(2)] $S^\text{Re}_{\eta,\tau}\to (\Yminus \mathcal Z_{\eta,\tau}, g_\tau)$, the corresponding real spinor bundle as in Eq. (\refeq{realimdecomp})\footnote{Since $S_E |_{\Yminus \mathcal Z_{\eta,\tau}}$ has vanishing first Chern class for every $\eta,\tau$, one may choose a continuous family of charge-conjugations $J$, thus of real structures $\sigma$ as in (\refeq{realimdecomp}).}. It is isomorphic, via Lemma \ref{subbundles}, to $S_0\otimes \ell_{\eta,\tau}$ where $\ell_{\eta,\tau}\to \Yminus \mathcal Z_{\eta,\tau}$ is a real line bundle twisted around $\mathcal Z_{\eta,\tau}$ equipped with its unique a flat connection ${A_{\eta,\tau}}$ with holonomy in $\Z_2$.  
\end{itemize}

\noindent Each line bundle $\ell_{\eta,\tau}$ is isomorphic to the original $\ell_0 \to \Yminus\mathcal Z_0$ after a homotopy of the link. Integration of sections over the fiber of $p$ with respect to the $rH^1_e$ or $L^2$ norm gives $\mathcal E$-parameterized families of Hilbert spaces whose restrictions to fibers $\mathcal E_\tau \subseteq \mathcal E$ are given by
\begin{eqnarray}
\mathbb H^1_e(\mathcal E_\tau) \to \mathcal E_\tau&:= &\{(\eta,u) \ | \ \mathcal Z \in \mathcal E_\tau \ , \ u \in rH^1_e(\Yminus \mathcal Z_{\eta,\tau}; S^\text{Re}_{\eta,\tau}) \} \label{Hilbertfamilies}\\
\mathbb L^2(\mathcal E_\tau) \to \mathcal E_\tau&:= &\{(\eta,v) \ | \ \mathcal Z \in \mathcal E_\tau \ , \ v \in L^2(\Yminus \mathcal Z_{\eta,\tau}; S^\text{Re}_{\eta,\tau}) \} \nonumber.
\end{eqnarray}

\noindent where $r=r_\eta$ now denotes a weight function given by the Riemannian distance to the marked curve $\mathcal Z_{\eta,\tau}$ in $g_\tau$. The gluing analysis takes place in trivializations of these (and other related) families of Hilbert vector bundles. Rather than precisely constructing the universal bundles (1) and (2) above, we will directly construct local trivializations of the families (\refeq{Hilbertfamilies}),  which is done in Lemma \ref{trivializationcompatible} in the next subsection.

There is an infinite-dimensional gauge freedom arising from the choice of these trivializations, which we describe in more detail before proceeding. While the universal family of 3-manifolds \eqref{6.DefDef} is a product, the universal link $\mathcal Z^{\text{univ}}\subseteq \mathbb Y$ is not. A trivialization of $\mathbb Y$ respecting the universal family of links is required to trivialize the Hilbert bundles $\mathbb H^1_e, \mathbb L^2$. Such a trivialization $\Upsilon_{\mathbb Y}$ is map

  \begin{eqnarray}\label{6.trivialization1}
  \begin{tikzcd}
\mathcal E  \times Y \arrow[rr,"\Upsilon_{\mathbb Y}"] \arrow[dr,"p'"] & &  {\mathbb Y} \arrow[dl,"p"] \\
& \mathcal E&
\end{tikzcd}
\end{eqnarray}
under which the product link $\mathcal E_\tau \times \{\mathcal Z_\tau\}$ is sent to $\mathcal Z^\text{univ} \cap p^{-1}(-,\tau)$ for each fixed $\tau$. Such a trivialization is determined by the choice of a family of diffeomorphisms \be \mathbb F: \mathcal E\to \text{Diff}(Y)\label{diffeospre}\ee

\noindent that associates to each $(\eta,\tau)\in \mathcal E$ a diffeomorphism ${\Bbb F}_\tau(\eta)=F_{\eta,\tau}:Y\to Y$ that preserves orientation and spin structure, and restricts to $\mathcal Z_\tau$ so that $F_{\eta,\tau}[\mathcal Z_\tau]=\text{Exp}_\tau(\eta)$ as subsets.  The induced trivialization \eqref{6.trivialization1} of the universal family of $3$-manifolds $\mathbb Y$ is then
  \begin{eqnarray}\label{6.trivialization3}
\Upsilon_{\mathbb Y}(\eta,\tau, y) = \big(\eta, F_{\eta,\tau}(y)\big).
\end{eqnarray}

\noindent The trivializations of the families of Hilbert spaces (\refeq{Hilbertfamilies}) are then constructed from a trivialization of $\mathbb Y$ via parallel transport maps (see Section  \ref{admissiblefamiliessection} for details).

For each fixed $\tau \in (-\tau_0, \tau_0)$, the choice of family (\refeq{diffeospre}) is far from unique. In fact, setting 

\smallskip 

\be \text{Diff}(Y; \mathcal Z_\tau):=\{F: Y\to Y  \ | \ F \text{ is a $C^\infty$ diffeomorphism and } F|_{\mathcal Z_\tau}=\text{Id}\}, \label{subdiffeos}\ee 

\smallskip

\noindent then pre-composing $\Upsilon_{\mathbb Y}$ with any family of diffeomorphisms $\mathbb G_\tau : \mathcal E_\tau \to \text{Diff}(Y; \mathcal Z_\tau)$ yields another trivialization $\Upsilon_{\mathbb Y'}=\mathbb F_\tau \circ \mathbb G_\tau$ for each fixed $\tau$. Thus (\refeq{subdiffeos}) acts as a group of gauge transformations on the Hilbert bundles $r\mathbb H^1_e, \mathbb L^2$, and the choice of the family (\refeq{diffeospre}) determines a choice of gauge. This is the gauge freedom alluded to in Section \ref{section2.5}, see also \cite[Sec. 4.1]{DonaldsonMultivalued} for a related discussion.

Although the relevant properties of the universal Dirac operator $\slashed {\mathbb D}$ (defined precisely in Definition \refeq{universalDirac} below), such as the Fredholmness of (\refeq{mainamap}) are independent of the choice of gauge, the concrete expressions for the derivative $\text{d}\slashed{\mathbb D}$ depend on this choice. Consequently, certain choices of gauge reveal different analytic properties of the operator. In Section \ref{section6.3}, we introduce a particular choice of gauge --- called the {\bf tangential smoothing gauge}  --- that behaves particularly nicely with respect to the elliptic edge theory of the Dirac operator. This gauge choice can be regarded as the analogue in the present setting of the Coulomb gauge in standard gauge theory, where the Coulomb gauge reveals the ellipticity of the gauge-fixed system of equations. Here, the tangential smoothing gauge reveals particularly nice properties relating regularity in the tangential directions along $\mathcal Z_\tau$ to the radial growth rate of spinors. Before introducing the tangential smoothing gauge in Section \ref{section6.3}, we review results from  \cite{PartII} using the standard choice of family $\mathbb F$ taken there.

\subsection{Admissible Families of Diffeomorphisms}
\label{admissiblefamiliessection}

This subsection gives a precise description of the allowable families of diffeomorphisms (\refeq{diffeospre}) and  describes the resulting trivializations of (\refeq{Hilbertfamilies}) over ${\mathcal E}$.  Since the gluing problem is local, we must only ever make a single choice of chart and accompanying trivialization.  For each $s\ge 2$, let
 $$
 \mbox{Diff}^s(Y) :=  \mbox{Diff}^{s,2}(Y) \cap  \mbox{Diff}^{C^1}(Y) 
 $$
be the set of diffeomorphisms that are both $C^1$ and are given by collections of $L^{s,2}$ functions in local coordinate charts. This space,  endowed with the topology generated by open sets in $C^1$ and open sets in $L^{s,2}$ of the coordinate functions, is a smooth Banach manifold (see e.g.\cite[Sec. 3]{ebin1970manifold} and \cite{Palais}). (We will never need to compose these diffeomorphisms, so do not need a Lie group structure.) 

We impose the following constraint on the family (\refeq{diffeospre}). Let $\mathcal E,{\mathcal E_\tau}$ and $\mbox{Exp}_\tau$ be as in Section \ref{embeddings}.

\begin{defn}{\rm
 An {\bf admissible family of diffeomorphisms} is a  smooth map
\begin{eqnarray}
{\Bbb F}: {\mathcal E} &\to&  \mbox{Diff}^s(Y) \nonumber \\
(\eta,\tau)&\mapsto& F_{\eta,\tau}: Y\to Y  \label{Fdiffeofamiliy}
\end{eqnarray}
that satisfies the following properties:
\begin{enumerate}
\item  $F_{0,\tau}=\mbox{Id}$ for all $\tau$, and each $F_{\eta,\tau}$ is the identify on $\Yminus N_{r_0}(\mathcal Z_\tau)$.
\item $F_{\eta,\tau}[\mathcal Z_\tau]= \mbox{Exp}_\tau(\eta)$ as subsets of $Y$ for all $(\eta,\tau) \in{\mathcal E}$.
\item ${\Bbb F}$  restricts to a smooth map ${\Bbb F}: L^{s,2}(\mathcal Z_\tau;N\mathcal Z_\tau) \to  \mbox{Diff}^s(Y)$ for every $s\ge 2$.
\item The bound $\|g_\tau- F_{\eta,\tau}^*g_\tau\|_{L^{1,2}(Y)}\le C \| \eta\|_{ L^{s,2}(\Z_\tau)}$ holds uniformly in $\eta,\tau$. 
\end{enumerate}
}
\noindent Here, smoothness of $\mathbb F$ is defined using the trivialization of $\mathcal E$ following (\refeq{6.defE1}).
\label{admissibledef}
\end{defn}

\begin{eg} \label{standarddeformations} \cite{PartII} uses the following natural choice of admissible family (\refeq{Fdiffeofamiliy}). Fermi coordinates induce a trivailization $N\mathcal Z_\tau\simeq \underline \C$, in which we can write $\eta(t)=\eta_x(t)+ i\eta_y(y)$. The diffeomorphism corresponding to $\eta\in L^{2,2}(\mathcal Z_\tau; N\mathcal Z_\tau)$ is then defined by 

\be F^\circ_\eta(t, z):=(t, z + \chi(|z|)\eta(t))\label{standardfamily}\ee

\noindent in Fermi coordinates, where $z=x+iy$, and $\chi(|z|)$ is a smooth radially dependent cutoff function equal to $1$ for $|z|\leq r_1/2$ and vanishing for $|z|\geq r_1$ for some $r_1<r_0$ smaller than the radius $r_0$ of the Fermi coordinates. $F_\eta$ is extended by the identity to the remainder of $Y$. \cite[Lem. 5.3]{PartII} shows that $F_\eta$ is indeed a diffeomorphism for $r_\mathcal E$ as in Eq. (\refeq{6.DefDef}) sufficiently small. 
\end{eg}

\begin{notation} \label{Notation6.1} Extending the notation in Section \ref{embeddings}, we have the following for objects associated to an admissible family $\mathbb F$. 

\begin{enumerate}
\item In a mild abuse of notation, we use $\mathcal E_\tau$ to denote both the open ball (\refeq{6.DefDef}) and its image under (\refeq{Expdef}). 
\item The family for fixed $\tau$ is denoted $\mathbb F_\tau$ and the subscript $\tau$ is dropped when it is clear from context, so that we write $F_\eta=\mathbb F_\tau(\eta)$. 
\item $\mathcal Z_{\eta}:=F_\eta[\mathcal Z_\tau]$ denotes the link corresponding to $\eta \in \mathcal E_\tau$ ($\mathcal Z_{\eta,\tau}$ when ambiguity may arise) 
\item $g_{\eta}:=F_\eta^*(g_\tau)$ denotes the pullback metric (or $g_{\eta,\tau}$ when ambiguity may arise). 
\end{enumerate}
\end{notation} 
\medskip

A choice of admissible family $\mathbb F_\tau$ induces a chart on the space of embeddings via $\eta \mapsto F_\eta[\mathcal Z_\tau]$. Item (2) in Definition \ref{admissibledef} ensures that this induced chart  coincides with (\refeq{Expdef}), and that the associated pullback metrics vary smoothly with parameters.  
\begin{lm}
\label{Fchart}
For an admissible family of diffeomorphisms $\mathbb F_\tau$, the  induced chart \begin{eqnarray}
\mathcal E_\tau &\to& \text{Emb}(\mathcal Z_\tau; Y)\label{Expchi1}\\
\eta&\mapsto& \mathcal Z_{\eta,\tau}:= F_{\eta}[\mathcal Z_\tau]. \label{Expchi2}
\end{eqnarray}
\noindent coincides with $\text{Exp}_\tau$ as defined in (\refeq{Expdef}) for each $\tau\in (-\tau_0,\tau_0)$. Moreover, the family of pullback metrics $g_{\eta,\tau} \in L^{1,2}(Y; \text{Sym}^2(T^*Y))$ depends smoothly on $(\tau,\eta)$. 
\end{lm}

\begin{proof} The first statement follows directly from Item (2) of Definition \ref{admissibledef} and the definition (\refeq{Expdef}). 

For smoothness, first observe that since $F_\eta \in L^{2,2}$ for each $\eta$, the pullback metrics lie in \newline $L^{1,2}(Y; \text{Sym}^2(T^*Y))$ which is a Banach space is the standard way using the smooth reference metric $g_0$. Item (1) in Definition \ref{admissibledef} ensures that the coordinate functions of $F_\eta$, and thus the entries $g_{\tau,\eta}$ vary smoothly in $(\tau,\eta)$ in the family of Fermi coordinate charts $(-\tau_0, \tau_0)\times N_{r_0}(\mathcal Z_0)$ used to define smoothness following Definition \ref{admissibledef}. These charts are formed using the smoothly parameterized parallel transport and exponential maps of the metrics $g_\tau$, thus the same smoothness properties hold on $Y$.  
\end{proof}

We now proceed to define the families (\refeq{Hilbertfamilies}) more precisely, and show that a choice of admissible family induces a trivialization. For each $\mathcal Z_\eta\in \mathcal E_\tau$, there is an associated spinor bundle $S^\text{Re}_{\eta}$ defined as in Eq. (\refeq{realimdecomp}) using the metric $g_\tau$, but with $\mathcal Z_\eta$ in place of $\mathcal Z_0$. More precisely, since any such $\mathcal Z_\eta$ is homotopic to $\mathcal Z_\tau$, the determinant $\text{det}(S)$ restricts trivially to $\Yminus \mathcal Z_\eta$. Thus, there is a $U(1)$-connection $A_\eta$ on $\text{det}(S|_{\Yminus \mathcal Z_\eta})$ unique up to the action of $U(1)$-gauge transformations, such that $A_\eta$ is flat with the same holonomy representation as $A_0$ (after a homotopy inducing an isomorphism $\pi_1(\Yminus\mathcal Z_\eta)\simeq \pi_1(\Yminus \mathcal Z_0)$). Such a connection induces an $SU(2)$ structure on $S|_{\Yminus \mathcal Z_\eta}$, and $S^\text{Re}_{\eta,\tau}$ is defined, mutatis mutandis, as in (\refeq{realimdecomp}), with $\mathcal Z_\eta$.

We next show how the choice of an admissible family of diffeomorphisms induces  trivializations of the families (\refeq{Hilbertfamilies}) for a fixed $\tau$; these are the maps $\Upsilon_{\mathbb F}$ in Lemma~ \ref{trivializationcompatible} below.  For this, we follow the construction of \cite[Sec. 5.1]{PartII}, which is based on a method for associating spinor bundles of different metrics, originally due to \cite{Bourguignon} (see also \cite{DiracVariationNonlinear}).

To start,  fix a spin structure  on $(Y,g_\tau$) with
associated spinor bundle $S_{g_\tau}$ and a complex line bundle $L$ so that the $\text{Spin}^c$ structure is given by $S=S_{g_\tau}\otimes L$. We may assume that the spin structure is that of  Lemma \ref{subbundles}(3). Then, for each fixed $\eta$, consider the cylinder with Riemannian metric
 \be
 X= \big([0,1]\times Y, \, ds^2 +   g_{s\eta}\big)\label{metriccylinder}
 \ee
\noindent where $s$ is the coordinate on $[0,1]$ and $g_{s\eta}=F_{s\eta}^*(g_\tau)$.  Let $W^\pm_\eta$ be the positive and negative spinor bundles of the pullback spin structure on $X$. For $s=0$, the positive spinor bundle $W_\eta^+\to X$ is isomorphic to the spinor bundle of $Y$ with the metric $g_\tau$, while for $s=1$ it is isomorphic to that with the metric $g_\eta$.  Let
\be
(\frak T_S)_{\tau}^{{\eta}}:  S_{g_\eta}  \longrightarrow S_{g_\tau}\label{spinorparalleltrans}
\ee
 denote the isomorphism between the two spinor bundles for $g_\tau$ and $g_\eta$ obtained by parallel transport along rays $\{y\}\times [0,1]$ using the spin connection on $W^+$.

In a similar fashion, for each fixed $\eta$, there are vector bundles $E_\eta, L_\eta \to X-([0,1]\times \mathcal Z_\tau)$ given by the pullbacks of $E, L\to X$ via the map {
$ X\to Y$ by $(s,y)\mapsto F_{s\eta}(y)$,} i.e. the restriction of $E_\eta$ to $\{s\}\times Y$ is $F_{s\eta}^*(E)$ and likewise for $L$. These bundles carry the pullback connections, denoted $B^X_\eta$ and  $A^X_\eta$, the latter defined over $X\mathrm{-}([0,1]\times \mathcal Z_\tau)$\footnote{We emphasize here that the 1-parameter family of connections $A_{s\eta}$ define a connection (in temporal gauge) over $X-\bigcup \mathcal Z_{s\eta}$. Since $F_{s\eta}[\mathcal Z_\tau]=\mathcal Z_{s\eta}$, the pullback connections $A_\eta^*=F_{s\eta}^*A_\eta$ then define a connection over $X$ with the {\it product} singular set $[0,1]\times \mathcal Z_\tau$ excised. Thus parallel transport along rays is indeed well-defined.  }. Denote by  

$$
\frak T_E:  F_\eta^*(L\otimes E)|_{\Yminus \mathcal Z_\tau} \to L\otimes E|_{\Yminus \mathcal Z_\tau}
$$
the map defined by parallel transport along rays $\{y\}\times [0,1]$. Then define 
 \be \mathfrak T_{g_\tau}^{g_\eta}=(\mathfrak T_S)_{g_\tau}^{g_\eta}\otimes \mathfrak T_E \label{paralleltransport}\ee on the tensor product $W^+_\eta \otimes L_\eta \otimes E_\eta$.

\begin{defn}\label{trivializationdef}
Define the {\bf trivialization induced by an admissible family $\mathbb F_\tau$} as follows. For each fixed $\tau\in (-\tau_0, \tau_0)$ and $\eta \in \mathcal E_\tau$, let $v_\eta$ be the following composition:  
\begin{center}
\tikzset{node distance=3.7cm, auto}
\begin{tikzpicture}
\node(C){$S_{g_\tau}\otimes  L \otimes E $};
\node(D)[left of=C]{$S_{g_\eta}\otimes   F_{\eta}^*(  L \otimes E) $};
\node(E)[left of=D]{$  F_{\eta}^*(S_{g_\tau}\otimes  L \otimes E)$};
\node(F)[left of=E]{$S_{g_\tau}\otimes L \otimes E$};
\draw[<-][swap] (C) to node {$\frak T_{g_\eta}^{g_\tau}$} (D);
\draw[<-][swap] (D) to node {$\mathfrak S_{\eta}$} (E);
\draw[<-][swap] (E) to node {$ F_{\eta}^*$} (F);
\end{tikzpicture}
\end{center}
\noindent where 
\begin{itemize}
\item[(i)] $ F_\eta^*$ is the pullback by the diffeomorphism $ F_\eta$. 
\item[(ii)] $\frak S_\eta$ is the canonical isomorphism $ F_\eta^*(S_{g_\tau})\simeq S_{g_\eta}$ on the first factor, and $\text{Id}$ on $F^*_\eta  (L\otimes E)$. 
\item[(iii)] $\frak T_{g_\eta}^{g_\tau}$ is the parallel transport map Eq. (\refeq{paralleltransport}). 
\end{itemize}
\noindent 
Note that $v_\eta$ is a fiberwise linear isomorphism that covers $F_\eta$. Then trivialization induced by $\mathbb F_\tau$, denoted $\Upsilon_{\mathbb F}$ is the map on sections 
 \bea \Upsilon_{\mathbb F}: rH^1_e(\Yminus \mathcal Z_\eta; S^\text{Re}_{\eta}) &\to& rH^1_e(\Yminus \mathcal Z_\tau; S^\text{Re}_\eta) \\ \psi &\mapsto& v_\eta^{-1}\circ \psi, \eea

 \noindent and equivalently for $L^2$. Lemma \ref{trivializationcompatible} below ensures that this map preserves $S^\text{Re}$ and regularity, thus a choice of admissible family $\mathbb F_\tau$ endows the families (\refeq{Hilbertfamilies}) with the structure of a locally trivial Hilbert vector bundle. 
 
\end{defn}

In the following lemma, we set $\mathcal E_\tau^s:= L^{s,2}(\mathcal Z_\tau ; N\mathcal Z_\tau) \cap \mathcal E_\tau$ for any $s\geq 2$.  
\begin{lm} \label{trivializationcompatible} Let $\mathbb F_\tau$ be an admissible family of diffeomorphisms. Then, for $s\geq 5$, the restriction of the induced trivializations $\Upsilon_{\mathbb F}$ to $\mathcal E_\tau^s$ 
\bea
\Upsilon_{\mathbb F_\tau}:\mathbb H^1(\mathcal E_\tau^s)&\simeq& \mathcal E_\tau^s \times rH^1_e(\Yminus \mathcal Z_\tau ;S^\text{Re}_\eta) \\
\Upsilon_{\mathbb F_\tau}: \mathbb L^2(\mathcal E_\tau^s)&\simeq& \mathcal E_\tau^s \times L^2(\Yminus \mathcal Z_\tau ;S^\text{Re}_\eta) 
\eea 
\noindent is a fiberwise bounded linear isomorphism. Moreover, for different choices of admissible family, these trivializations are compatible for different choices of $\mathbb F$ in the sense that $\mathbb F_\tau, \mathbb F_\tau'$ is a fiberwise bounded linear isomorphism depending continuously on $\eta \in \mathcal E^s_\tau$. Finally, these trivialization depend smoothly on $\tau\in (-\tau_0, \tau_0)$. 
\end{lm}

\begin{proof} The flat connection $A_\eta^X$ used to form (\refeq{paralleltransport}) induces an $SU(2)$ structure on $S_E|_{X\mathrm{-} ([0,1]\times \mathcal Z_\tau)}$, which defines a real structure as in Eq. (\refeq{realimdecomp}) that restricts to those defining $S^\text{Re}$ and $S^\text{Re}_\eta$ on the two ends. The connection on $X$ formed from the spin connection, $A_\eta^X, B_\eta^X$ is compatible with this real structure, thus the parallel transport map (\refeq{paralleltransport}) preserves the real subbundle. Since $F_\eta, \mathfrak G_\eta$ obviously preserve the real structure, it follows that $\psi \mapsto v_\eta^{-1}\circ \psi$ restricts to a map of the real subbundles. 

It remains to show that this map carries $rH^1_e$ sections defined by the connections and metric associated to $\eta$ to $rH^1_e$ sections defined by the data associated to $(g_\tau, B_\tau)$. Since $\eta\in L^{5,2}(\mathcal Z_\tau; N\mathcal Z_\tau)$, the pullback metric is $L^{4,2}(Y; \text{Sym}^2(T^*Y))$, and in particular is $C^2$ by the Sobolev embedding. The different volume form and Christoffel symbols of the spin connection are therefore induce a bounded equivalence of the $rH^1_e$ norms. Finally, by the same argument, for two different choice of admissible familes $\mathbb F_\tau, \mathbb F_\tau'$, the composition $\Upsilon_{\mathbb F'}\circ \Upsilon_{\mathbb F}^{-1}$ is fiberwise a bounded linear isomorphism. Moreover, since $\mathbb F_\tau, \mathbb F_\tau'$ are assumed to depend continuously on $\eta$ and the constructions of Definition \ref{trivializationdef} are natural (thus e.g. the pullback metrics also depend continuously on $\eta$), these fiberwise bounded linear isomorphisms depend continuously on $\eta$. 

\end{proof}

\begin{rem}  \label{regularityrem}Lemma \ref{trivializationcompatible} is true for $s=2$ as well. The proof relies on the mixed-regularity Sobolev inequalities in \cite[Sec. 5.1]{PartII}. Here we omit these sharper regularity statements, since the proof of Theorem \ref{maina} only requires dealing with smooth deformations $\eta$. 
\end{rem}

With these trivialization in place, we can now tie up a loose end from Definition \ref{Obdefn}. Until now, the obstruction bundle $\text{\bf Ob}\to (-\tau_0, \tau_0)$ is simply a parameterized family of Banach spaces. The next lemma shows that

\begin{prop} \label{cokerpropertiesI}
There is a family of bounded linear isomorphisms $$\Upsilon_{\text{\bf Ob}}: \text{\bf Ob}\to (-\tau_0,\tau_0)\times (L^2(\mathcal Z_0 ;\mathcal C_0 )\oplus \R)$$

\noindent that endow the obstruction bundle with the structure of a smooth Hilbert vector bundle over $(-\tau_0, \tau_0)$. This bundle structure makes the natural inclusion $\text{\bf Ob}\to \mathbb L^2|_{(-\tau_0,\tau_0)}$ a continuous inclusion of Hilbert vector bundles, where the restriction means to the 1-parameter family of embeddings $\mathcal Z_\tau$ for $\tau\in (-\tau_0, \tau_0)$. Finally, the map $\text{\bf ob}_\tau $ from Proposition \ref{cokerpropertiesII} is a continuous map of vector bundles. \end{prop}
\begin{proof} With $rH^1_\perp$ as defined in (\refeq{rH1perp}), the family of spaces $\{ rH^1_\perp(\Yminus \mathcal Z_{\tau}) \ | \ \tau \in (-\tau_0, \tau_0))\}\subseteq \mathbb H^1|_{(-\tau_0, \tau_0)}$ form a vector subbundle of codimension 1 by Lemma \ref{trivializationcompatible}, where the restriction is again to the 1-dimensional submanifold  $\{\mathcal Z_\tau  \ | \ \tau \in (-\tau_0 ,\tau_0) \subseteq \mathcal E^s_0$. Since $\slashed{\mathbb D}$ is a fiberwise injective bundle map on this subbundle, its image is a subbundle of $\mathbb L^2|_{(-\tau_0,\tau_0)}$, and $\text{\bf Ob}$ is its orthogonal complement by Definition \ref{Obdefn}, hence also a Hilbert vector bundle. See Propositions 4.2 and 8.5 in \cite{PartII} for further details.  
\end{proof}

\subsection{The Universal Dirac Operator}
\label{universalDirac}
In this subsection, we define the universal Dirac operator and calculate its derivative. The universal Dirac operator is the ({\it a priori} discontinuous) section defined as follows, which gives a more precise meaning to the operator defined in (\refeq{universalDiracdef}).

\begin{defn}\label{universalDirac}
The {\bf Universal Dirac Operator} is the section $\slashed{\mathbb D}$ defined by

\begin{center}
\tikzset{node distance=1.6cm, auto}
\begin{tikzpicture}
\node(C)[yshift=-1.7cm]{$\mathbb H^1_e(\mathcal E_\tau).$};
\node(D)[above of=C]{$p_{1}^*\mathbb L^2(\mathcal E_\tau)$};
\node(D')[xshift=2.2cm, yshift=-0.8cm]{$\slashed{\mathbb D}(\mathcal Z, u):=\slashed D_{\mathcal Z}u$};
\draw[<-] (C) to node {$$} (D);
\draw[->] (C) edge[bend right=45]node[below] {$$} (D);
\end{tikzpicture}
\end{center}
\end{defn}

\noindent Morally speaking, $\slashed{\mathbb D}$ is a smooth section, though this is not strictly true on these low-regularity Sobolev spaces. Working with higher regularity would require many tedious bootstrapping arguments later in the proof of Theorem \ref{maina}, so we instead retain this low regularity and specify the continuity of $\slashed{\mathbb D}$ and the boundedness of its derivative more precisely in Lemma \ref{boundedUniversalDirac} below. Ultimately, as in Remark  \ref{regularityrem} (cf. Section \ref{losssection}), only smooth deformations $\eta$ are needed, so imposing extra regularity on $\eta$ causes no issues.

The formula for $\slashed{\mathbb D}$ in the local trivializations of Definition \ref{trivializationdef} depends on the formula for the spin Dirac operator with respect to an arbitrary metric, originally due to Bourguignon-Gauducho \cite{Bourguignon}. To state Bourguignon-Gauduchon's formula, let $p_s=(g_s,B_s)$ be a 1-parameter family of metric and perturbation pairs on $Y$ for $s\in [0,1]$. Here, we may view $B$ as a perturbation to the spin connection on $S^\text{Re}$ under the isomorphism in Item (2) of Lemma \ref{reconstructingspinc}. Let $S_{g_s}$ be the spinor bundle of the metric $g_s$, and $\frak T_{g_0}^{g_s}: S^\text{Re}_{g_s}\to S_{g_0}^\text{Re}$ be the parallel transport map as in (\refeq{spinorparalleltrans}). Then the conjugated operators 

\be \left( \frak T_{g_\tau}^{g_{s}}\circ \slashed D_{p({s})} \circ (\frak T_{g_0}^{g_{s}})^{-1}\right) : \Gamma(Y; S^\text{Re}_{g_0})\lre \Gamma(Y;S^\text{Re}_{g_0}) \label{conjugatedDirac}\ee

\noindent form an $s$-dependent family of first order differential operators on sections of the spinor bundle with the fixed metric $g_0$. Define $a_{g_0}^g(s), \frak a(s)\in \text{End}(TY)$ be defined respectively by $$g_s(V,W)=g_0(a_{g_0}^{g(s)}(s) V, W) \hspace{3cm} \frak a(s)=(a_{g_0}^g(s))^{-1/2}$$  
where $V,W\in \Gamma(TY)$ and the latter is understood via the eigenvalues of $(a_{g_0}^g)^\star a_{g_0}^g$, which are non-zero for $g_s$ sufficiently close to $g_0$.

\begin{thm} {\bf (Bourguignon-Gauduchon, \cite{Bourguignon})}The following expressions hold for the family of conjugated Dirac operators (\refeq{conjugatedDirac}) acting on a spinor $\Psi \in \Gamma(Y;S_{g_0}^\text{Re})$

\begin{itemize}
\item[(A)] The Dirac operator $\slashed D_{p(s)}$ is given by
 \be \left(\frak T_{g_0}^{g_s}\circ \slashed D_{p(s)} \circ (\frak T_{g_0}^{g_s})^{-1} \right)\Psi=\left( \sum_{i} e^i. \nabla^{p(s)}_{\frak a(e_i)} \  + \  \frac{1}{4}\sum_{ij}e^i e^j.   \left(\frak a^{-1}(\nabla^{g_0}_{\frak a(e_i)} \frak a) e^j + \frak a^{-1}(\nabla^{g(s)}-\nabla^{g_0})_{\frak a(e_i)} \frak a(e^j) \right).\right) \Psi  \label{NLBG}\ee 
 where $e^i$ and $.$ are an orthonormal basis and Clifford multiplication for $g_0$, and $\nabla^{g(s)}$ denotes the unperturbed spin connection of the metric $g$ and likewise for $g_0$. Here we use the shorthand $\frak a=\frak a(s)$. 
 
 \medskip 
 
 \item[(B)] Denoting the $s$-derivative of $g_s$ by $\dot g_s$, the derivative of the family of Dirac operator with respect to $s$ at $s=0$ is given by 

\begin{eqnarray}\left(\d{}{s}\Big |_{s=0} \frak T_{g_0}^{g_s}\circ \slashed D_{p(s)} \circ (\frak T_{g_0}^{g_s})^{-1}\right)\Psi&=& {\Bigg(}-\frac{1}{2}\sum_{ij} \dot g_s(e_i,e_j) e^i . \nabla^{g_0}_j  + \frac{1}{2} d \text{Tr}_{g_0}(\dot g_s). \nonumber \\ & &  +\frac{1}{2} \text{div}_{g_0}(\dot g_s). +  \mathcal R(B_0, \dot g_s). \Bigg)\Psi\label{bourguignon}\end{eqnarray}
where $\mathcal R(B_0,\dot g(s))$ is a smooth term involving up to first derivatives of $B_0$, and $e_i, e^i, \ . \ , \text{div}_{g_0}, \nabla^{g_0}$ are respectively an orthonormal frame and co-frame, Clifford multiplication, the divergence of a symmetric tensor, and the spin connection of the metric $g_0$.  
 \end{itemize}
 \label{nonlinearBG}
 \end{thm}
 
 \begin{proof} \cite{Bourguignon} derives both formulas for the case of the spin Dirac operator (see also \cite{DiracVariationNonlinear}). The case of a perturbed spin connection appears in \cite[Cor. 5.19]{PartII}, and differs only in the appearance of the term $\mathcal R(B_0, \dot g(s))$. Using the isomorphism in Item (2) of Lemma \ref{reconstructingspinc}, the case for the Dirac operator on the real spinor bundle is identical. 
 \end{proof}
 
\noindent To give intuition for (\refeq{bourguignon}) briefly, the first term arises from differentiating the symbol/Clifford multiplication of the Dirac operator, while the next two terms arise from differentiating the spin connection, and the last from differentiating the perturbation $B$. This final term, for our purposes, is a lower order term in a meaning that will be made precise in the upcoming sections.

We now verify the smoothness of the universal Dirac operator for spaces of higher regularity.

\begin{lm} \label{boundedUniversalDirac}The universal Dirac operator in Definition \ref{universalDirac} satisfies the following. 
\begin{itemize}
\item[(A)] For $s\geq 5$, the universal Dirac operator 

$$\slashed{\mathbb D}: \mathbb H^1(\mathcal E_\tau^s)\to p_1^*\mathbb L^2(\mathcal E_\tau^s)$$

\noindent is a smooth section over the total space of the bundle $\mathbb H^1$ restricted to the higher regularity locus $\mathcal E^s_\tau=\mathcal E_\tau\cap L^{s,2}(\mathcal Z_\tau; N\mathcal Z_\tau)$, and depends smoothly on $\tau \in (-\tau,\tau)$. 

\smallskip 

\item[(B)] Provided $\Psi\in rH^1_e\cap C^0$, the linearization at $(0, \Psi)\in \mathbb H^1(\mathcal E_\tau^s)$ extends to a bounded linear map 

\be \text{d}\slashed {\mathbb D}_{(\mathcal Z_\tau,\Psi)}: L^{2,2}(\mathcal Z_\tau; N\mathcal Z_\tau)\oplus  rH^1_e(\Yminus \mathcal Z_\tau; S^\text{Re}) \ \lre \ L^2(\Yminus \mathcal Z_\tau; S^\text{Re})\label{universalderivmap}\ee
on the lower regularity tangent spaces, where the domain is decomposed in the splitting $T_{(\eta,\Phi)}\mathbb H^1\simeq L^{2,2}(\mathcal Z_\tau; N\mathcal Z_\tau)\oplus rH^1_e(\Yminus \mathcal Z_\tau; S^\text{Re})$ is that induced by the trivialization $\Upsilon_{\mathbb F}$. This derivative also and depends smoothly on $\tau \in (-\tau,\tau)$. 

\noindent 
\end{itemize}
\end{lm}

\begin{proof}
(A) By Theorem \ref{nonlinearBG}, $\slashed{\mathbb D}(\eta,\Phi)$ is given in coordinates by (\refeq{NLBG}), where $\frak a$ is formed using the pullback metric $g_\eta$. By the admissibility of $\mathbb F_\eta$, the diffeomorphisms $F_\eta \in \text{Diff}^{s,2}(Y)$ for all $\eta \in \mathcal E^s_\tau$, thus the pullback metrics $$g_\eta:=F_\eta^*g_\tau \in L^{s-1,2}(Y;\text{Sym}^2(T^*Y))$$

\noindent lie in the multiplicative range of Sobolev regularity in dimension $3$ for $n\geq 4$. The algebraic operators $a_{g_0}^{g_\eta}, \frak a$ have entries consisting of smooth combinations of sums, products, and compositions of the components of $g_\eta$ (see \cite[Sec. 5.3 and 8.3]{PartII} for precise expressions). Differentiating these smooth combinations with respect to $\eta$ shows the smoothness of $\slashed{\mathbb D}$ with respect to this variable, and the operator is linear as a function of the spinor (so {\it a fortiori} smooth). By Lemma \ref{Fchart} the pullback metrics depends smoothly on $\tau$, and smoothness as a function of $\tau\in (-\tau_0, \tau_0)$ follows. 

(B) By Theorem \ref{nonlinearBG}, the linearization of the universal Dirac operator at $(0, \Psi)$ is given by (\refeq{bourguignon}) where $g_s=g_{s\eta}$. This is a bounded map into $L^2$ for $\eta \in L^{s,2}(\mathcal Z_\tau; N\mathcal Z_\tau)$ by part (A) above. For $\eta \in L^{2,2}$ only, boundedness is concluded from the following. Inspection of (\refeq{bourguignon}) shows that it consists of terms schematically having the form $( \dot g_\eta)\nabla \Psi$ and $(\nabla \dot g_\eta)\Psi$. Since $F_\eta \in L^{2,2}$ for all $\eta \in L^{2,2}$, then the pullback metrics $g_{s\eta}$ are $L^{1,2}$ and are bounded by $s\|\eta\|_{L^{2,2}}$ by Item (4) of Definition \ref{admissibledef}, thus $\|\dot g_\eta\|_{L^{1,2}}\leq C\|\eta\|_{L^{2,2}}$. Both types of terms are therefore $L^2$, with norms bounded in terms of $\|\eta\|_{2,2}$. Smooth dependence on $\tau$ follows from Lemma \ref{Fchart} as before.

\end{proof}

\subsection{The Deformation Operator}
\label{section6.2}  This section calculates the projection of the derivative (\refeq{bourguignon}) of $\text{d}\slashed{\mathbb D}$ at $\Z_2$-harmonic eigenvector to the obstruction bundle (see  Definition \ref{Obdefn}).

Let  $\pi_\tau=\br \Phi_\tau, - \kt\Phi_\tau$ be the $L^2$-orthogonal projection onto the span of the eigenvector $\Phi_\tau$.  We denote the derivative (\refeq{bourguignon}) along the family of pullback metrics $g_{s\eta}$ for $\eta \in L^{2,2}(\mathcal Z_\tau; N\mathcal Z_\tau)$ by  \begin{eqnarray} \mathcal B_{\Phi_\tau}(\eta)&:=& \left(\d{}{s}\Big |_{s=0} \frak T_{g_0}^{g_{s\eta}}\circ \slashed D_{p(s)} \circ (\frak T_{g_0}^{g_{s\eta}})^{-1}\right)\Phi_\tau.\label{calBdef} \\  &=&   {\Bigg(}-\frac{1}{2}\sum_{ij} \dot g_{s\eta}(e_i,e_j) e^i . \nabla^{g_0}_j  + \frac{1}{2} d \text{Tr}_{g_0}(\dot g_{s\eta}).  +\frac{1}{2} \text{div}_{g_0}(\dot g_{s\eta}). +  \mathcal R(B_0, \dot g_{s\eta}). \Bigg)\Phi_\tau. \nonumber \end{eqnarray} 
Then, using the orthogonal splitting $$L^2(\Yminus\mathcal Z_\tau)\simeq \text{\bf Ob}(\mathcal Z_\tau) \oplus \text{Range}^\perp_\tau$$

\noindent from Definition \ref{Obdefn}, the derivative (\refeq{universalderivmap}) can be written as a block matrix: 

\be \text{d}\slashed {\mathbb D}_{(\mathcal Z_\tau,\Phi_\tau)} = \begin{pmatrix}  \Pi_\tau \mathcal B_{\Phi_\tau} & \Lambda(\tau) \pi_\tau \\  \\  (1-\Pi_\tau)\mathcal B_{ \Phi_\tau}& \slashed D_{A_\tau}\end{pmatrix} \ : \  \begin{matrix} L^{2,2}(\mathcal Z_\tau; N\mathcal Z_\tau)\\ \oplus \\ r H^1_\perp \end{matrix} \ \  \lre \ \   \begin{matrix}  \text{\bf Ob}(\mathcal Z_0)\\ \oplus \\ \text{Range}_\tau^\perp.\end{matrix}\label{blockdecomp}\ee

\noindent where the top right entry has rank $1$. Recall that $rH^1_\perp$ was defined in (\refeq{rH1perp}). 

Composing with the isomorphism $\text{ob}^{-1}_\tau\oplus \iota: \text{\bf Ob}(\mathcal Z_\tau)\to  L^2(\mathcal Z_\tau;\mathcal C_{\tau})\oplus \R$ from Proposition \ref{cokerpropertiesII} where $\mathcal C_\tau$ is the Calderon bundle (see Definition \ref{calderonbunddef}), the top left block of (\refeq{blockdecomp}) can be written as $(T_{\mathbb F,\Phi_\tau}, \pi_\tau)$ where $T_{\mathbb F,\Phi_\tau}$ is the composition:

    \begin{center}
\tikzset{node distance=3.2cm, auto}
\begin{tikzpicture}
\node(A){$L^{2,2}(\mathcal Z_\tau;N\mathcal Z_\tau)$};
\node(C)[right of=A]{$\text{\bf Ob}(\mathcal Z_\tau)$};
\node(C')[right of=C, yshift=0cm]{$L^{2}(\mathcal Z_\tau;\mathcal C_{\tau} ). $};
\draw[->] (A) to node {$ \Pi_\tau\mathcal B_{\Phi_\tau}$} (C);
\draw[->] (C) to node {$ \text{ob}_\tau^{-1}$} (C');
\draw[->,>=stealth] (A) edge[bend right=20]node[below]{$T_{\mathbb F, \Phi_{\tau}}$} (C');
\end{tikzpicture}\end{center}

\noindent The operator $T_{\mathbb F,\Phi_\tau}$ depends on the choice of local trivialization in Definition (\ref{trivializationdef}), and thus in particular on the choice of an admissible family $\mathbb F_\tau$. Different choice of admissible family lead to different expressions for  $\mathcal B_{\Phi_\tau}$ and thus for $T_{\mathbb F,\Phi_\tau}$.  Note that $T_{\mathbb F,\Phi_\tau}$ is an operator on sections of vector bundles over the fixed curve $\mathcal Z_\tau$. We refer to it as the {\bf deformation operator}.

In \cite[Sec. 6]{PartII}, an explicit expression for $T_{\mathbb F,\Phi_\tau}$ is calculated, using the choice of admissible family $\mathbb F_\tau$ described in Example (\ref{standarddeformations}). For this particular family, we denote the operator simply by $T_{\Phi_\tau}$ unadorned by $\mathbb F_\tau$. The expression results from explicitly computing the sequence of inner products 
\be T_{ \Phi_\tau}(\eta(t)):=\sum_{\ell \in \Z}\br \mathcal B_{ \Phi_\tau}(\eta(t)) , \Psi_\ell \kt_\C \cdot e^{i\ell t}\label{deformation}\ee

\noindent which gives an expression for the deformation operator in terms of Fourier modes, via Item (A) of Proposition \ref{cokerpropertiesII}. The expression also involves the zeroth-order pseudodifferential operator defined as follows
\begin{eqnarray}
\mathcal T_{\Phi_\tau}:\Gamma(\mathcal Z_\tau; N\mathcal Z_\tau)&\to& \Gamma(\mathcal Z_\tau; \mathcal C_\tau) \label{calTdef}\\
 s(t)&\mapsto& H(c_\tau(t)s(t))-\overline s(t)d_\tau(t) \nonumber, 
\end{eqnarray}

\noindent where $-iH$ is the Hilbert transform in the trivialization $\mathcal C_\tau\simeq \underline \C$ below (\refeq{Psilmodel}), and $c(t), d(t)$ are the leading coefficients of $\Phi_\tau$ from the expansion of Lemma \ref{asymptoticexpansion}. By \cite{PartII}[Cor. 3.9], these coefficients transform as sections $c_\tau(t)\in \Gamma(N\mathcal Z_\tau^{-1}), d(t)\in \Gamma(N\mathcal Z_\tau)$ under changes of Fermi coordinates, so that the pointwise multiplication of both terms is well-defined as a map into the trivial $\underline \C$-bundle. By Theorem  \ref{SpectralCrossing} and Lemma \ref{asymptoticexpansion}, $\mathcal T_{\Phi_\tau}$ depends smoothly on $\tau$.

\cite[Sec. 6.2]{PartII} proves: 

\begin{thm}{\label{deformationsmaina}}  {\bf (\cite{PartII})}. 
For the family of admissible diffeomorphisms $\mathbb F_\tau$ in Example \ref{standarddeformations}, the operator $T_{\Phi_\tau}$ is given by
 \be
 T_{\Phi_\tau}(\eta(t))=\left( -\tfrac{3|\mathcal Z_\tau|}{2} (\Delta+1)^{-\tfrac34}  \ \circ \ \mathcal T_{\Phi_\tau}\circ \d{}{t^2}\right)\eta(t) + K_\tau(\eta(t))
\label{formofT}\ee

\noindent where $|\mathcal Z_\tau|$ is the length of $\mathcal Z_\tau$, $\Delta$ is the Laplacian on $\mathcal C_\tau$, $\d{}{t^2}$ is the second covariant derivative on $\Gamma(N\mathcal Z_\tau)$ induced by the Levi-Civita connection of $g_\tau$, and $K_\tau$ is a pseudo-differential operator of order at most $\tfrac14$ depending smoothly on $\tau$. 

In particular, $T_{\Phi_\tau}$ is an elliptic pseudo-differential operator of order $\frac{1}{2}$, and its Fredholm extension
\smallskip 
\be T_{\Phi_\tau}: L^{1/2,2}(\mathcal Z_\tau;N\mathcal Z_\tau)\lre L^{2}(\mathcal Z_\tau;\mathcal C_{\tau})\label{TPhifredholmness}\ee has index 0.  \qed
 
\end{thm}

  \bigskip 
  
  The unobstructed condition in Definition  \ref{unobstructeddef} can be restated in terms of the operator $T_{\Phi_\tau}$. 
\begin{cor}
If $\Z_2$-harmonic spinor $(\mathcal Z_0, A_0, \Phi_0)$ has unobstructed deformations, then $T_{\Phi_\tau}$ is invertible for $\tau$ sufficiently small. 
\end{cor}
\begin{proof}
Definition \ref{unobstructeddef} means that $(\mathcal Z_0, A_0, \Phi_0)$ is unobstructed if and only $T_{\Phi_0}$ is injective. Since it is index 0, injectivity implies (\refeq{TPhifredholmness}) is invertible. It follows from smoothness that $T_{\Phi_\tau}$ is invertible for $\tau$ sufficiently small (see also \cite[Lem. 8.17]{PartII}).  
\end{proof}

\noindent

As a consequence of Theorem \ref{deformationsmaina}, the following version of standard elliptic estimates hold.  They are proved by repeated differentiation (or integration by parts for $m<2$). 
\begin{cor}\label{Cor4.6}
For any $m\geq 0$, the extension $$T_{\Phi_\tau}: L^{m+1/2,2}(\mathcal Z_\tau;\mathcal S_\tau)\to L^{m,2}(\mathcal Z_\tau;\mathcal S_\tau)$$

\noindent is Fredholm of index 0 and there are constants $C_m$ such that it satisfies
\be
\|\eta\|_{m+1/2,2} \leq C_m \ (\|T_{\Phi_\tau}(\eta)\|_{{m,2}} +  \|\eta\|_{m+1/4,2} ). \label{ellipticestimatesTPhi}
\ee
Moreover, in the case that $(\mathcal Z_0, A_0, \Phi_0)$ has unobstructed deformations, the $\|\eta\|_{m+1/4,2}$ term is not needed for $\tau$ sufficiently small. 
\end{cor}
\begin{proof} Composing the parametrix of $\mathcal T_{\Phi_\tau}$ from \cite[Lem. 6.11]{PartII} with the appropriate multiple of $(\Delta+1)^{3/4}$ yields a parametrix for the first term of (\refeq{formofT}). The estimates then follow from the boundedness of this parametrix in the standard way, with $(m+1/4,2)$-norm used to bound the compact error term $K_{\tau}$ of order $1/4$. Eliminating the lower order term in the invertible case follows in the standard way. 
\end{proof}

A more quantitative version of these elliptic estimates will also be needed, which is given in the next proposition. To motivate these estimates, we offer the spoiler that the gluing problem only requires solving the equation \be T_{\Phi_\tau}(\eta)=\psi\label{solvingtruncated}\ee

\noindent where $\psi$ is supported in the lowest $\e^{-1/2}$ Fourier modes in $L^2(\mathcal Z_\tau; \mathcal C_\tau)$ (cf. Section \ref{losssection}). This is ultimately a consequence of Lemma \ref{fourierregimescokernel}, since error terms are mostly supported in the region where $r=O(\e^{-1/2})$, and this matter is dealt with precisely in Sections \ref{section9}--\ref{section11}. For now, notice that if the metric were a product, and $\Phi_\tau$ had only constant Fourier modes, then solving (\refeq{solvingtruncated}) would preserve Fourier modes and $\eta$ would likewise have support in the lowest $\e^{-1/2}$ modes. The below proposition extends the elliptic estimates to ensure that, for sufficiently smooth metrics and eigenspinors $\Phi_\tau$, the norms of the solution of (\refeq{solvingtruncated}) grow {\it as if} they were supported in these same Fourier modes as the right-hand side.

In the statement of the proposition,  $g_\circ$ is used to denote the product metric in Fermi coordinates on $N_{ r_0}(\mathcal Z_\tau)$ defined using $g_\tau$. As in Definition \ref{Fermicoords}, $g_\circ$ differs from $g_\tau$ by a symmetric tensor of size $O(r)$.  $B_\tau$ continues to denote the perturbation coming from the background connection on $SU(2)$ (here the difference from the product connection). 
\begin{prop} \label{quantitativehigherorder}Suppose that there is an $M>1$ such that for each $m\in \N$ the bounds 
\bea  |\del^m_t (g_\tau - g_\circ)|  &\leq & M^{m} \|g_\tau - g_\circ \|_{C^3(Y)}   \hspace{1.5cm} |\del^m_t\Phi_\tau |\leq  M^{m} \|\Phi_\tau \|_{C^1(Y)} \\  |\del_t^m B_\tau|& \leq& M^m \|B_\tau\|_{C^2} \eea 
\noindent hold on $N_{r_0}(\mathcal Z_\tau)$, and that $(\mathcal Z_0, A_0, \Phi_0)$ has unobstructed deformation. 

 Then there is a constant $C_m$ independent of $M$ such that if $T_{\Phi_\tau}(\eta)=\psi$, then the following estimate holds for every  $m\geq 0$:
\be
\|\eta\|_{m+1/2,2} \leq C_m   \| \psi\|_{m,2}   \ + \  C_mM^m  \| \psi \|_{2}\label{strongellipticest}. 
\ee
\end{prop}

\begin{proof}
Differentiating the elliptic estimate of Corollary \ref{Cor4.6} for $m=0$ and using commutators and interpolation inequalities leads to the following tame estimate (cf. \cite[Lem. 8.17]{PartII}): 

\be
\|\eta\|_{m+1/2,2} \leq C_m \left(\|\psi\|_{m,2} \ + \ \|(g_\tau,  B_\tau, \Phi_\tau)\|_{C^{m+r}_{\text{tan}}}\|\psi\|_{L^2}\right) \label{tameelliptic}
\ee

\noindent for some natural number $r\geq 1$.  Here, $C^{m+l}_\text{tan}$ denotes the mixed regularity space with with $l$ continuous derivatives on $Y$ and up to $m$ additional derivatives {\it only in the directions tangential to $\mathcal Z_\tau$}, i.e. multi-indices in $x,y,t$ with at most $l$ instances of $x,y$. Since $\eta$ depends only on $t$, only these tangential derivatives appear when differentiating the elliptic estimates (see, e.g. the equation below (8.41) in \cite{PartII}). Since $g_\circ$ is constant in Fermi coordinates, the same estimate holds replacing $g_\tau$ by $g_\tau-g_\circ$. The proof of \cite[Lem 8.17]{PartII} shows that, in fact, $l=3,2,1$ suffices for the three components respectively. The result then follows from substituting the assumptions directly into (\refeq{tameelliptic}). 
\noindent 
\end{proof}

\section{The Tangential Smoothing Gauge}
\label{section6.3}

This section introduces a particular choice of gauge in the sense of Section \ref{gaugefreedomsection} by specifying a judicious choice of an admissible family $\mathbb F_\tau$ of diffeomorphisms. This particular choice of gauge is the Tangential Smoothing Gauge described in Section \ref{losssection}, so named because the definition of the admissible family involves smoothing operators in the tangential directions. This gauge choice (which depends on a fixed choice of Fermi coordinates in Definition \ref{Fermicoords}) provides stronger estimates for many terms in the expressions for $\slashed{\mathbb D}, \text{d}\slashed{\mathbb D}$ than the choice in Example \ref{standarddeformations}. The presence of these {\it tangential} smoothing operators should be viewed as the suitable replacement of the full Nash-Moser machinery in this setting.

\subsection{Radially Dependent Smoothing Operators}
To motivate the construction, recall (cf. Section \ref{conormalsection}) that the intertwining of radial growth rate and tangential regularity is a fundamental property of the edge calculus. This relationship appears very concretely in the expressions for the singular harmonic spinors (\refeq{Psilmodel}): these decay exponentially, with $1/e$ length $1/|\ell|$ where $\ell$ is the tangential Fourier mode. The key idea of the tangential smoothing gauge is that better estimates can be obtained in a gauge for which {\it we make a choice of admissible family that imposes, by hand, a similar relationship between the radial distance and tangential Fourier modes}.

The construction of the admissible diffeomorphisms relies on families of pseudo-differential operators in the tangential directions, parameterized by the radial distance. To begin, we introduce the following notation. Recall that $r_0$ denotes the radius of the Fermi coordinate chart around $\mathcal Z_\tau$, chosen uniformly in $\tau$. Given a family of smooth function $f_\ell: [0, r_0)\to \R$ indexed by $\ell\in \Z$ such that $|f_\ell(r)|\leq C$ are bounded uniformly in $r,\ell$, we let $\underline f$ denote the operator
\begin{eqnarray}
\underline f: L^{2}(\mathcal Z_\tau; \C) &\lre&  L^2(N_{ r_0}(\mathcal Z_\tau); \C)\label{underlinedef} \\ 
\underline f[\eta]&:= & \sum_{p \in \Z} f_p(r) \eta_p e^{ipt}. 
\end{eqnarray} 

\noindent where the bundle $N\mathcal Z_\tau \simeq \underline\C$ is a trivialization induced by a fixed choice of Fermi coordinates, and $\eta_p$ are the Fourier coefficients of $\eta(t)$. $\underline f$ is a $[0,r_0)$-parameterized family of pseudo-differential operators on $L^2(\mathcal Z_\tau; \C)$ whose Fourier multiplier is given by $\{f_\ell(r)\}_{\ell \in \Z}$ for each fixed $r$.

We now make a particular choice of such a family $f_\ell$. Let $\chi_\circ: [0, \infty)\to \R$ be a smooth cutoff function equal to $1$ for $r\leq 1$ and supported where $r\leq 2$. Next, let $R_0>0$ be a large positive number to be specified shortly, and denote by $\chi(r):=\chi_\circ(r/R_0)$ the dilated cutoff function supported where $r\leq 2R_0$. There is a constant $C$ independent of $R_0$ such that

 \be  |d\chi|\leq \frac{C}{R_0}\label{R0def}\ee
 
 \noindent holds. Additionally, let $\chi_{r_0}$ denote a second smooth cutoff function equal to 1 for $r\leq r_0/2$ and supported in $N_{r_0}(\mathcal Z_\tau)$.  Here $r_0$ is the radius of the Fermi coordinate charts (chosen uniformly in $\tau$). Then, for each $\ell \in \Z$, set \be \chi_\ell(r):= \chi_\circ(|\ell| r) \chi_{r_0}(r).\label{chiderivfamily}\ee

\noindent The family $\chi_\ell$ gives rise an operator $\underline \chi$ as in (\refeq{underlinedef}). Note that $|\nabla ^k\chi_\ell| \leq \tfrac{C}{R_0 |\ell|^k}$ by the Chain rule.

 \begin{defn}  \label{underlineFdef}The {\bf tangentially smoothing admissible family} $\underline{\mathbb F}_\tau$ is the family of diffeomorphisms 
 \begin{eqnarray}\underline{ \mathbb F}_\tau: \mathcal E_\tau \subseteq L^{2,2}(\mathcal Z_\tau; N\mathcal Z_\tau)&\to& \text{Diff}^{2,2}(Y) \label{Fdiffeofamiliyunderline} \\ \eta &\mapsto& \underline{ F}_{\eta}: Y\to Y\end{eqnarray}
 given by 
 \be \underline F_\eta (t, z) := \left(t, z+ \underline \chi[\eta]\right)\label{diffeosmd}\ee
 
 \noindent in Fermi coordinates $(t,x,y)$ with $z=x+iy$, and each is extended by the identity outside $N_{r_0}(\mathcal Z_\tau)$. 
 \end{defn}
 
 In order to justify this definition, we must show: 
 
 \begin{claim} For $r_\mathcal E$ as in (\refeq{6.DefDef}) sufficiently small, $\underline F_\eta$ is a diffeomorphism for every $\eta \in \mathcal E_\tau$, and $\mathbb F_\tau$ collectively form an admissible family. 
 \end{claim}
 
 \begin{proof}
To verify that the map (\refeq{diffeosmd}) is a diffeomorphism, a quick calculation in Fermi coordinates (see Eq. (\refeq{pullbackmetric}) in the upcoming proof of Lemma \ref{universalDnonlinear}) $\text{d}\underline F_\eta=\text{Id} + O(\|\eta\|_{C^1})$. Since $L^{2,2}\hookrightarrow C^1$ in dimension 1, we may choose $r_\mathcal E$ sufficiently small that $\text{d}\underline F_\eta$ is invertible everywhere. The Inverse Function Theorem then implies that $\underline F_\eta$ is a local $C^1$-diffeomorphism, thus a covering map (since its image is both open and closed). Since $\underline F_\eta=\text{Id}$ outside of $N_{r_0}(\mathcal Z_\tau)$, we conclude $\underline F_\eta$ is a degree 1 covering map, and so a bijection, and thus a global $C^1$-diffeomorphism.

To show $\mathbb F_\tau$ form an admissible family, note that, as above, $\underline F_0=\text{Id}$, and that ${\underline F_\eta}=\text{Id}$ outside $N_{r_0}(\mathcal Z_\tau)$. Thus $\underline {\mathbb F}_\tau$ obeys the two preliminary requirements of the family following (\ref{Fdiffeofamiliy}). We now show (1)--(4) in Definition \ref{admissibledef}. (2) is immediate, because $\chi_\ell(0)=1$ for every $\ell$, thus $F_\eta |_{\mathcal Z_\tau}=F_\eta|_{\mathcal Z_\tau}$ has the same restriction as the family in \ref{standarddeformations}. (3) Restricted to the annulus $|z|=r$, the diffeomorphisms are $\underline F_\eta(t,z)=(t, z+ \chi_{r_0}\widetilde \eta(t))$ where $\widetilde \eta(t)$ is the smooth truncation of $\widetilde \eta$ to the lowest $|\ell| \leq 2R_0/r$ Fourier modes. In particular, $\underline F_\eta$ is at least as smooth as $\eta$ all local coordinates, and is globally $C^1$ by the Sobolev embedding $L^{2,2}(S^1)\hookrightarrow C^1(S^1)$. (1) By the same observation, coordinate functions of $\underline F_\eta$ vary smoothly as a function $\eta$ in local coordinates, and are constant are functions of $\tau$ in the Fermi coordinate chart following Definition \ref{admissibledef}. Finally, (4) follows easily from calculating the pullback metric explicitly, which is done in the proof of the upcoming Lemma \ref{universalDnonlinear}. 
\end{proof}

\begin{notation} \label{underlinenotation}The tangentially smoothing admissible family $\underline {\mathbb F}_\tau$ induces its own versions of the trivializations and operators from Sections \ref{admissiblefamiliessection} -- \ref{section6.2}. We denote the corresponding version of each construction with an underline; in particular, $\underline{\mathbb F}_\tau$ induces, 

\begin{enumerate}
\item[(1)] $\underline \Upsilon_{\mathbb F_\tau}$  the associated trivializations as in Lemma \ref{trivializationcompatible}, 
\item[(2)] $\underline g_{\eta,\tau}:=\underline F_\eta^*(g_\tau)$ the associated family of pullback metrics, 
\item[(3)] $\underline {\mathcal B}_{\Phi_\tau}$ the partial derivative in the deformation direction as in  (\refeq{blockdecomp}) formed using $\underline g_{\eta,\tau}, \underline \Upsilon_{\mathbb F_\tau}$. 
\item[(4)] $\underline T_{\Phi_\tau}$ the deformation operator in the trivialization $\underline \Upsilon_{\mathbb F_\tau}$ as in Eq. (\refeq{deformation}).
\end{enumerate}
\noindent We emphasize that the chart $\text{Exp}$ in (\refeq{Expdef}) is independent of the choice of admissible family. 
\end{notation}
 
  The main results of Section \ref{section6.2} carry over to the version of the deformation operator $\underline T_{\Phi_\tau}$, provided $R_0$ in \refeq{R0def} is chosen sufficiently large.

  \begin{prop}
\label{deformationsmainamd} For $(\e,\tau)\in (0, \e_0)\times (-\tau_0, \tau_0)$, the operator $\underline T_{\Phi_\tau}$ is given by 
$$\underline T_{\Phi_\tau}=T_{\Phi_\tau} + T_{R_0}$$
\noindent where $T_{R_0}$ is a pseudo-differential operator of order $1/2$, and for some $K\in \N$ and $M>10$, there is a constant $C=C_{M}$ so that it satisfies $$ \|T_{R_0}\|\leq C R_0^K \text{Exp}(-R_0/c) \ + \ C R_0^{-M}$$ uniformly for $(\e,\tau)\in (0, \e_0)\times (-\tau_0, \tau_0)$. In particular, if $(\mathcal Z_0,A_0, \Phi_0)$ has unobstructed deformations then for $R_0$ sufficiently large, $\underline T_{\Phi_\tau}$ is invertible and the results of Corollaries \ref{Cor4.6} and \ref{quantitativehigherorder} continue to hold uniformly in $\e,\tau$.  

\end{prop}

The proof is an extension of the proof of Theorem \ref{deformationsmaina} in \cite[Sec. 6]{PartII}. Appendix \ref{appendixA} provides details.

\subsection{Tangential Smoothing Estimates}
In this section, we establish key estimates for terms of $\slashed{\mathbb D}$ and its linearization in the tangential smoothing gauge. A straightforward calculation of the pullback metric $g_{\eta,\tau}=F^*_\eta(g_\tau)$ (see \cite[Sec. 5.3]{PartII} and Lemma \ref{universalDnonlinear} below) shows that with the standard choice of gauge without tangential smoothing, i.e. the choice of admissible family from Example \ref{standardfamily}, $\slashed {\mathbb D}(\mathcal Z_\eta, \Phi)$ is a sum of terms of the form 

  \be  M_\Phi(\eta):=\left(\underline{\del^{m} \chi} \left[\eta^{(n)} \right] \right)\cdot \sigma_j \nabla^{k}\Phi. \label{typeA1}\ee

\noindent for various integers $m,n,k\geq 0$, where $\eta^{(n)}=\left(\d{}{t}\right)^n\eta$ and $\sigma_j=\gamma(e^j)$ is Clifford multiplication by a basis vector in an orthonormal frame.  When $\Phi$ is polyhomogeneous all its derivatives are bounded, the derivatives of $\eta$ are the harder terms to bound. In general, one cannot obtain any bound better than \be \|M_{\Phi}\|_{L^2}\leq C \|\eta\|_{n},\label{badbound}\ee where $\|\eta\|_{n}$ denotes the $L^{n,2}(\mathcal Z_\tau; N\mathcal Z_\tau)$-norm of $\eta$ for the highest value of $n$ that appears. Theorem \ref{nonlinearBG} shows that this value is $n=2$ for the terms appearing in $\slashed {\mathbb D}$.

When these terms are considered instead in the tangential smoothing gauge -- which has the effect of replacing $\chi$ by $\underline \chi$ in the expression (\refeq{typeA1}) -- the intertwining of the growth rate and the tangential regularity intrinsic to this gauge reveal that {\it $M_\Phi$ is secretly a smoothing operator}; thus in this gauge it is bounded $M_{\Phi}: L^{n,2}(\mathcal Z_\tau; N\mathcal Z_\tau)\to H^{s}(\Yminus \mathcal Z; S^\text{Re})$ for some $s>0$! With these better bounds, the need to explicitly smooth the tangential configuration $\eta$ during the gluing iteration, i.e. the use of the standard Nash-Moser framework, is eliminated. 

This smoothing property is ultimately a consequence of the interaction between the family of smoothing operators (\refeq{chiderivfamily}) and real-valued functions of fix growth rate. The following lemma, which is completely independent of anything relating to the Dirac operator, captures this property.

\begin{lm}\label{Cvaluedsmoothing} Let $\eta \in C^\infty(\mathcal Z_\tau; N\mathcal Z_\tau)$ be a deformation. Then the following bounds hold for weights $\beta \geq -\tfrac12$. 

\begin{itemize}
\item[(A)] If $G: N_{r_0}(\mathcal Z_\tau)\to \C$ is a complex-valued function with Fourier coefficients $\{\eta_p\}_{p\in \Z}$ such that $|G| \leq Cr^{\alpha}$ holds pointwise, then,
\smallskip  
$$\| r^\beta \cdot  \left(\underline{\del^{m} \chi} \left[\eta^{(n)} \right] \right) \cdot G \|_{L^2(Y)} \leq C \|\eta\|_{s_A},$$
\smallskip
\noindent where $s_A= m+n - (1+\alpha +\beta)$ under the assumption that $\alpha+\beta>-1$.  
\item[(B)] If $u: N_{r_0}(\mathcal Z_\tau)\to \C$ is a complex-valued function such that $u \in L^2(N_{r_0}(\mathcal Z_\tau))$, then 
\smallskip
$$\| r^\beta \cdot \left(\underline{\del^{m} \chi} \left[\eta^{(n)} \right] \right) \cdot u \|_{L^2(Y)} \leq C \|\eta\|_{s_B},$$

where $s_B= m+n + \tfrac12 + \underline \gamma - \beta$, with $\underline \gamma=10^{-6}$. 
\end{itemize}
 \noindent In these expressions, $\underline {\del^m \chi}$ denotes the operator (\ref{underlinedef}) formed analogously to (\ref{chiderivfamily}) but using $\del^m\chi$ for a a multi-index of order $m$ in $x,y$. 

\end{lm}

\begin{proof} In Fermi coordinates, $\eta: \mathcal Z_\tau \to \C$ becomes a complex-valued function. 

(A) We begin with the case that $m=0$. On each circle $S^1\times \{(x,y)\} \subseteq N_{ r_0}(\mathcal Z_\tau)$ the multiplication map $C^{0}(S^1)\times L^2(S^1)\to L^2(S^1)$
is bounded, and the bound can be taken uniform for $|(x,y)|< r_0,$ the radius of the Fermi coordinates. Applying this to the product $G(t,x,y) \underline \chi[\eta'']$ for each pair $(x,y)$, 
\bea
 \Big \|r^\beta  \left(\underline{\chi} \left[\eta^{(n)} \right] \right)G(t,x,y)\Big \|^2_{L^2(N_{r_0})}&=& \int_{D_{\bold r_0}} \int_{S^1} \Big |\underline \chi \left[\eta^{(n)} \right]  G\Big |^2 dt \ r^{2\beta +1 }d\theta dr \\
 &\leq &C \int_{D_{ r_0}}  r^{2\alpha} \cdot  \int_{S^1}  \Big | \ \underline \chi \left[\eta^{(n)} \right]\Big |^2 dt \ r^{2\beta+1}d\theta dr \\
 &\leq &  C\int_{D_{ r_0}}  \Big|\sum_{p\in \Z}\eta_p (ip)^{n} e^{ipt} \chi_p(r)\Big |^2 r^{2\alpha+2\beta+1} d\theta dr dt, \eea
\noindent where we have replaced $L^2(S^1)$-norm over with the square sum over Fourier modes via Parseval's Theorem, for each fixed $(x,y)\in D_{r_0}$. Since the integrand is dominated by a large $C^k$-norm of $\eta$ times, $r^{2\alpha+2\beta+1}$ which lies in the integrable range by assumption,  the Dominated Convergence Theorem shows the sum may be pulled through the integral, which yields:
 \begin{eqnarray}
 \phantom{\Big \| \underline{\chi}[\eta'']  \sigma_j \psi \Big \|^2_{L^2(N_{ r_0})}}&\leq& C \sum_{p\in \Z} |\eta_p|^2 |p|^{2n} \int_{D_{r_0}}  |\chi_p(r)|^2 r^{2\alpha+2\beta +1} d\theta dr dt \label{evaluatingintegral}\\
  &\leq &C \sum_{p\in \Z} |\eta_p|^{} |p|^{2n} \Big[r^{2\alpha+2\beta+2}\Big ]_{r=0}^{2R_0/|p|} \label{711}\\
  &\leq &C\sum_{p\in \Z} |\eta_p|^2 |p|^{2n-2\alpha-2\beta-2}=C \|\eta\|_{n-(1+\alpha+\beta)} \label{712},
\end{eqnarray}

\noindent because $\chi_p(r)$ is supported where $r\leq 2R_0/|p|$. When $m\neq 0$, we have that $\underline {\del^m \chi}$ is supported in the same region, but ${\del^m \chi}_p \leq C|p|^m$. In this case the same calculation shows the result where (\refeq{711}--\refeq{712}) have an additional factor of $|p|^{2m}$. 

(B) The proof in this case follows the same outline, with minor modifications due to the fact that $u$ no longer necessarily obeys pointwise bounds. First, divide $D_{ r_0}$ into the sequence of annuli $$A_j:= \{\tfrac{2R_0}{n+1} \leq r\leq \tfrac{2R_0}{j}\}$$

\noindent for $n\geq 1$. Recall here that $R_0$ is the constant from (\refeq{R0def}), chosen sufficiently large so that the conclusions of Proposition \ref{deformationsmainamd} hold.  Then, beginning again with the case that $m=0$,  
 \begin{eqnarray} \Big \|  r^\beta \underline{\chi}[\eta^{(n)}] \cdot u \Big \|^2_{L^2(N_{ r_0})}&=& \sum_{n\geq 1}\int_{A_j}  |\underline \chi[\eta^{(n)}] \cdot u|^2 r^{2\beta}  \ r dr dt d\theta \nonumber \\
 &\leq & \sum_{j\geq 1} \sup_{A_j} \|\underline\chi(\eta^{(n)})\|_{C^0(S^1)}\int_{A_j} |u|^2 r^{2\beta} \ rdr dt d\theta \nonumber  \\
 &\leq&   \sum_{n\geq 1}\sup_{A_j}  \left[\|\underline\chi(\eta^{(n)})\|^2_{C^0(S^1)} \cdot r^{2\beta} \right] \|u \|^2_{L^2(A_j)}.  \label{713} \end{eqnarray}

\noindent Now, since $\chi_p=0$ on $A_j$ for $p\geq 2(j+1)$, we see that $$\sup_{A_j}\|\underline \chi(\eta^{(n)})\|^2_{C^0} \leq C\|\pi^{2(j+1)} \eta^{(n)}\|^2_{C^0}\leq C\|\pi^{2(j+1)} \eta\|^2_{{n+1/2+\underline \gamma}}$$ 
where $\pi^{2(n+1)}$ denotes the projection to Fourier modes $|p|\leq 2(n+1)$, and we have used the continuous Sobolev embedding $C^0\hookrightarrow L^{1/2+\underline{\gamma}, 2}(\mathcal Z_\tau; \C)$ in dimension 1, since $\underline \gamma=10^{-6}>0$. Next, by the definition of $A_j$,  $\sup_{A_j} r^{2\beta}\leq \tfrac{C}{j^{2\beta}}$, hence the restriction on Fourier modes implies 

$$\sup_{A_j}  \left[\|\underline\chi(\eta^{(n)})\|^2_{C^0(S^1)} \right]\leq  C |j|^{-2\beta}\|\pi^{2(j+1)} \eta\|^2_{{n+1/2+\underline \gamma}} \leq C\|\eta\|^2_{n+1/2+ \underline \gamma -\beta}.$$

\noindent Thus in total, (\refeq{713}) is bounded by

\bea
&\leq & C \sum_{j=1}^\infty \|\eta\|_{{n+1/2+\underline \gamma-\beta}} \|u\|_{L^2(A_j)}^2 \leq C \|\eta\|^2_{{n+1/2+\underline \gamma-\beta}} \cdot \|u\|_{L^2}^2
\eea 
as desired. The modifications for $m\neq 0$ are identical to the equivalent modifications for part (A), i.e. the weight in the above sum is adjusted to $|j|^{-2\beta +2m}$. 
\end{proof}

We now extend these bounds from complex-valued functions to spinors. There are two specific classes of terms that appear in the expression for $\slashed{\mathbb D}$, which are controlled by the two parts of Lemma \ref{Cvaluedsmoothing} respectively. 

\begin{defn} \label{typesAB} We say that a linear operator $M: \Gamma(\mathcal Z_\tau; N\mathcal Z_\tau) \to \Gamma(\Yminus \mathcal Z_\tau; S_E)$ has {\bf deformation types (A) and (B)} respectively if has the following forms. 
\begin{itemize}
\item[(A)] Let $\Phi$ a spinor with a polyhomogeneous expansion as in Lemma \ref{asymptoticexpansion}. Terms of Type (A) have the form 

  \be  M_\Phi(\eta):=\left(\underline{\del^{m} \chi} \left[\eta^{(n)} \right] \right)\cdot \sigma_j \nabla^{k}\Phi. \label{typeA}\ee

\noindent using the notation of (\refeq{typeA1}), where $m,n\geq 0$ are integers, and $k=0,1$. 

\item[(B)] Let $\psi \in rH^1_e(\Yminus\mathcal Z; S_E)$ be a spinor and $k=0,1$. Terms of type (B) have the form
 
 \be M_\psi(\eta):=\left(\underline{\del^{m} \chi} \left[\eta^{(n)} \right] \right)\cdot \sigma_j \nabla^{k}\psi \label{typeB}\ee
 
 \noindent where $m,n\geq 0$ are integers, and $k=0,1$. 
\end{itemize} 

\noindent In the two cases respectively, we define the {\bf weight} of the term by 
$$w_A \ , \ w_B=m+n+k.$$
\end{defn}

\medskip
\begin{lm}\label{deformationbounds}
Let $\eta \in C^\infty(\mathcal Z_\tau;N\mathcal Z_\tau)$ be a deformation, and $\beta \geq -\tfrac{1}{2}$.   

\begin{enumerate}
\item[(A)] Suppose that $M_\Phi$ is a term of Type A as in (\ref{typeA}) having weight $w_A$, and that $\beta \in \R$. Then 
$$\|r^\beta M_\Phi(\eta)\|_{L^2(N_{r_0})} \leq C \|\eta\|_{w_A-(3/2+\beta)}$$

\noindent where $\|\eta\|_s$ denotes the $L^{s,2}$-norm on $\mathcal Z_\tau$. 
\smallskip 
\item[(B)] Suppose that $M_\psi(\eta)$ is a term of Type B as in (\ref{typeB}) having weight $w_B$. Then, for some $\underline \gamma$ as in Lemma \ref{Cvaluedsmoothing},   
$$\| r^\beta  M_\psi(\eta)\|_{L^2(N_{r_0})} \leq C \|\eta\|_{w_B+ \underline \gamma -(1/2+\beta)} \|\psi\|_{rH^1_e}.$$
\end{enumerate}

\end{lm}
  \begin{proof}
 The proof follows directly from various cases of Lemma \ref{Cvaluedsmoothing}. (A) For $k=0$, one has
 
$$  \Big \| r^\beta \cdot  \underline{\del^m\chi}[\eta^{(n)}]  \cdot (\sigma_j \Phi)\Big \|^2_{L^2(N_{r_0})}= \int_{D_{\bold r_0}} \int_{S^1} |\underline {\del^m\chi}[\eta^{(n)}] \cdot  \sigma_j \Phi|^2  r^{2\beta+1} \ dt d\theta dr $$

\noindent and we may apply Lemma \ref{Cvaluedsmoothing}(A) with  $G=|\sigma_j \Phi|$, which obeys $G\leq C r^{\alpha}$ with $\alpha=1/2$ by Lemma \ref{asymptoticexpansion}. For $k=0$, and $\alpha=1/2$, then $w_A=m+n -(1+\alpha+\beta)=s_A$. For $k=1$, the same applies with $G=|\sigma_j \nabla \Phi|$, and with $\alpha=-1/2$. 

(B) For $k=0$, one has that $\psi\in rH^1_e\Rightarrow \frac{\psi}{r}\in L^2$. Applying Lemma \ref{Cvaluedsmoothing}(B) to $g=|\frac{\psi}{r}|$ with $\beta'=\beta+1$ implies the desired bound. The same holds for $k=1$ applying Lemma \ref{Cvaluedsmoothing}(B) to $g=|\sigma_j \nabla \psi|$ (with $\beta'=\beta$). 
 \end{proof}

\smallskip 

 \begin{eg}
 \label{smoothingestexample}
 To digest the lemma briefly, observe that the term in the formula (\refeq{bourguignon}) for the linearization $\text{d}\slashed{\mathbb D}$ containing $\text{div}(\dot g_\eta)$ is of Type (A) with $m=0, n=2, k=0$, thus has weight $w_A=2$. In this case, the conclusion of Lemma \ref{deformationbounds} shows that 

$$\|M_{\Phi_\tau}(\eta)\|_{L^2}\leq C \|\eta\|_{1/2}.$$

\noindent This is a stronger estimate than that in Eq. (\refeq{badbound}), by a factor of $3/2$ regularity. In other words, although the term contains two derivatives of $\eta$, multiplication by the eigenspinor $\Phi_\tau$ with $r^{1/2}$ asymptotics behaves as an order $3/2$ smoothing operator in the tangential smoothing gauge. 

More generally, Theorem \ref{nonlinearBG}(B) shows that $(\text{d}\slashed{\mathbb D})_{(\mathcal Z_\tau, \Phi_\tau)}$ consists entirely of terms of weight $w_A=2$. Thus in particular, for the polyhomogeneous spinor $\Phi_\tau$ and an arbitrary $\psi\in rH^1_e$, one has 
\bea
\|r^\beta{\underline{\mathcal B}}_{\Phi_\tau}(\eta)\|_{L^2(N_{r_0}(\mathcal Z_\tau))} &\leq &C\|\eta\|_{1/2-\beta} \\
\|r^\beta{\underline{\mathcal B}}_{\psi}(\eta)\|_{L^2(N_{r_0}(\mathcal Z_\tau))} &\leq &C\|\eta\|_{3/2+\underline \gamma-\beta} \cdot \|\psi\|_{rH^1_e}
\eea
 \noindent by applying Lemma \ref{deformationbounds}(A) and (B) respectively. Here, $\underline{\mathcal B}_{\Phi_\tau}$ is as in Notation \ref{underlinenotation}, and $\underline{\mathcal B}_{\psi}$ is the equivalent, substituting $\psi$ for $\Phi_\tau$ in (the underlined version of) Eq. (\refeq{calBdef}). 
 \end{eg}\

 We return now to the context of the universal Dirac operator. Theorem \ref{nonlinearBG} shows that $\slashed{\mathbb D}$  
 is a quasi-linear function of the deformation $\eta$ when written in a local chart and the trivialization of Definition \ref{trivializationdef}. The non-linear terms in $\eta$ have a similar form to the linear terms of types (A) and (B) from Definition \ref{typesAB}. We conclude this section by extending Lemma \ref{deformationbounds} to an appropriate bilinear analogue used to bound these non-linear terms in the gluing iteration.

Given a configuration $h_0=(\mathcal Z, \Phi) \in \mathbb H^1_e$, 

$$\slashed {\mathbb D}(h_0 + (\eta,\psi))= \slashed {\mathbb D}h_0+ \text{d}\slashed{\mathbb D}_{h_0}(\eta,\psi) + Q_{h_0}(\eta,\psi)$$ 

\smallskip 

\noindent where $\text{d}_{h_0}\slashed{\mathbb D}$ is the linearization (\refeq{blockdecomp}) and $Q_{h_0}$ is the non-linear term. Here, $\mathcal Z + \eta$ means $\mathcal Z_{\eta}$ as in (\refeq{Expchi2}) using the chart centered at $\mathcal Z$. The following lemma characterizes the non-linear term $Q_{h_0}$. 

\begin{lm}

\label{universalDnonlinear}

The non-linear term $Q$  has  the following form: 

\bea Q_{h_0}(\eta,\psi)&=&\underline{ \mathcal B}_\psi(\eta) \ \ + \ \ \mathfrak m_\Phi(\eta, \eta)   \ \ + \ \  \frak m_\psi (\eta,\eta)  \ \ + \ \ F_{\Phi + \psi}(\eta) 
\eea
where \begin{itemize} 
\item[(Q')] $\underline {\mathcal B}_{\psi}(\eta)$ is as in (\refeq{bourguignon}) with $\psi$ in place of $\Phi_\tau$. 

\item[(A')] $\frak m_\Phi(\eta, \eta)$ is a term quadratic in $\eta$ and linear in $\Phi$ which is a finite sum of terms of type (A'), defined to have the form   
\smallskip 
$$m(y)\cdot a_1(\underline \chi[\eta])\cdot M_\Phi(\eta)$$
\smallskip

\noindent where $m(y) \in  C^\infty(Y; \text{End}(S^\text{Re}))$,  $M_{\Phi}(\eta)$ is a linear term of type (A) in the sense of Definition \ref{typesAB}  with weight $w_A=2$, and $a_1(\underline \chi[\eta])$ is a linear combination of $\underline{\chi} [\eta'], \underline{\del _a\chi }[\eta]$, and $\underline{\chi}[ \eta]$. 
\item[(B')] $\frak m_\psi(\eta,\eta)$ is a sum of terms of terms of type (B'), defined identically to (A') but with the term $M_\psi(\eta)$ of type (A) replaced by one of type (B) with weight $w_B=2$, in the sense of Definition \ref{typesAB}.
\item[(C')] $F_{\Phi+\psi}(\eta)$ is comprised of sums of terms of type (C') which have higher-order dependence on $\eta$, $\eta'$, so that it satisfies a bound 
$$|F_{\Phi+\psi}(\eta)| \leq C \|\eta\|_{C^1} \Big(\frak m_\Phi'(\eta, \eta) + \frak m_\psi'(\eta, \eta)\Big) $$
\noindent where $\frak m_\Phi', \frak m_\psi'$ are finite sums of terms of types (A') and (B') respectively. 
\end{itemize}
\end{lm}

\begin{proof}
This formula is derived by substituting the formula for the pullback metric $g_\eta = F_{\underline \chi}(\eta)^* g$ into the non-linear version of Bourguignon-Gauduchon's formula in Theorem \ref{nonlinearBG}. See Section 8.3 of  \cite{PartII}. 

To explain succinctly, the full pullback metric may be written in Fermi coordinates around $\mathcal Z_\tau$ as 

\begin{eqnarray} \underline g_{\xi}&=& (\text{d} \underline {F}_{\xi})^T  g_\circ (  \text{d} \underline F_{\xi}) \ + \ (\text{d} \underline {F}_{\xi})^T  \cdot h_\tau(t,z+ \underline \chi[\eta])  \cdot (  \text{d} \underline F_{\xi})\\  \nonumber  \\ 
\text{d}\underline F_{\xi}&= &  \text{Id} \ + \  \begin{pmatrix}  0 & \underline \chi[\eta'_x] & \underline \chi[\eta'_y] \\  \underline \chi[\eta'_x] &   2\underline{\del_x \chi}[\eta_x] & \underline{\del_x \chi}[\eta_y] + \underline{\del_y\chi}[\eta_x] \\  \underline \chi[\eta'_y] &   \underline{\del_x \chi}[\eta_y] + \underline{\del_y\chi}[\eta_x]& 2\underline{\del_y \chi}[\eta_y]  \label{pullbackmetric} \end{pmatrix} 
\end{eqnarray}

\noindent where $g_\circ$ is the product metric and $g_\tau=g_\circ + h_\tau$ with $h_\tau=O(r)$ in Fermi coordinates, and $\eta=(\eta_x, \eta_y)$. Each entry is of the matrix is a family of operators as in (\refeq{underlinedef}), formed using $\chi$ or its derivatives as indicated, and applied to $\eta$ or $\eta'=\d{}{t}\eta$.

The result now follows from plugging the formula for the metric into Theorem \ref{nonlinearBG} and collecting terms of various types. For terms arising form the matrix product with $g_\circ$, the metric terms consist of quadratic combinations of $\underline \chi[\eta'], \underline{\del \chi}[\eta]$, and the operations in the Bourguignon-Gauduchon formula (\refeq{NLBG}) involve at most a single derivative.  Subtracting off the terms linear in $\eta$, which yield the operator $\underline{\mathcal B}_{\psi}(\eta)$, this contributes the terms of Types (A') and (B') (denoted by $\frak m_{\Phi}, \frak m_{\psi}$).  

The remaining terms, denoted collectively by $F_{\Phi_\ph}$ arise from at most quadratic combinations of $\underline \chi[\eta'], \underline{\del \chi}[\eta]$ and factors of $h(t,z+\underline \chi[\eta])$, where $h$ is the smooth error from the product metric in Fermi coordinates as above.  \cite[Ch 13. Prop 3.9]{TaylorPDEsIII} shows that for such compositions, 

\be \|h(t,z+\underline \chi[\eta])\|_{H^s(N(\mathcal Z_\tau))}\leq C_0 + C_1 \|\eta\|_{C^0(\mathcal Z_\tau)}\Big(1 + \|\eta\|_{L^{s,2}(\mathcal Z_\tau)}\Big),\label{TaylorPDEs}\ee

\smallskip 

\noindent which can be derived by applying the result therein to the difference $h(t,z+\underline \chi[\eta])- h(t,z)$. The constant term $h(t,z)$ is combined with the quadratic combinations into terms of $\frak m_{\Phi}, \frak m_{\ph}$, while the terms arising from the remainder $h(t,z+\underline \chi[\eta])- h(t,z)$ are collected into $F_{\Phi+\ph}$, and the bound in (C') follows from  (\refeq{TaylorPDEs}) above with $s=2$, since we may assume that $\|\eta\|_{L^{2,2}}\leq 1$. 
\end{proof}

To bound the non-linear terms later, we have the following bilinear analogue of Lemma \ref{deformationbounds}. In the statement of the lemma, we tacitly use $\underline {\chi'}[\eta]$ to denote a term having one derivative, i.e. a linear combination of $\underline \chi[\eta']$ and $\underline{\del_a \chi}[\eta]$. $\underline \chi''[\eta]$ denotes the same but with up to second derivatives. 

\begin{lm}\label{nonlineardeformationbounds}
Let $\Phi, \psi$ be spinors as in Definition \ref{typesAB}. Retaining the notation from Lemma \ref{deformationbounds}, we have the following bounds on bilinear terms in a pair of deformations $\eta,\xi \in C^\infty(\mathcal Z_\tau; N\mathcal Z_\tau)$, where $\underline \gamma<<1$ is a fixed positive constant. 
\begin{enumerate}
\item[(A')] There are bounds
\bea \Big \| \underline {\chi' }[\eta]\underline{\chi'}[\xi]\nabla \Phi \Big \|_{L^2(N_{ r_0})} &\leq& C\|\eta\|_{3/2+\underline\gamma}\cdot \|\xi\|_{1-\alpha} \\ 
\Big \| \underline {\chi'' }[\eta]\underline{\chi'}[\xi] \Phi \Big \|_{L^2(N_{ r_0})}+\Big \| \underline {\chi' }[\eta]\underline{\chi''}[\xi] \Phi \Big \|_{L^2(N_{ r_0})} &\leq& C( \|\eta\|_{3/2+\underline\gamma}\cdot \|\xi\|_{1-\alpha}+ \|\eta\|_{1-\alpha}\cdot \|\xi\|_{3/2+\underline\gamma } )  \\ \eea
\item[(B')] Likewise, 
\bea \Big \| \underline {\chi' }[\eta]\underline{\chi'}[\xi]r^\beta\nabla \psi \Big \|_{L^2(N_{ r_0})} &\leq& C \|\eta\|_{3/2+\underline\gamma - \beta} \|\xi\|_{3/2+\underline\gamma} \|\tfrac{\psi}{r}\|_{L^2} \\ 
\Big \| \underline {\chi'' }[\eta]\underline{\chi'}[\xi]r^\beta \psi \Big \|_{L^2(N_{ r_0})}+\Big \| \underline {\chi' }[\eta]\underline{\chi''}[\xi]r^\beta \psi \Big \|_{L^2(N_{ r_0})} &\leq& C( \|\xi\|_{3/2+\underline\gamma}\|\eta\|_{3/2+\underline\gamma - \beta} + \|\eta\|_{3/2+\underline\gamma}\|\xi\|_{3/2+\underline\gamma - \beta} ) \|\tfrac{\psi}{r}\|_{L^2}  \\   \eea
\end{enumerate}
\end{lm}

\begin{proof}
Note that each term contains at most one instance of a second derivative (this is a reflection of the quasi-linearity of $\slashed {\mathbb D})$. The bound $d\chi_p \leq C|p|$ for each $p\in \Z$ implies that $\|\underline{\del_a\chi}[\eta]\|_{L^2}\lesssim \|\eta\|_{1,2}$. In each case above, the factor of $\eta, \xi$ with only a single derivative can therefore be pulled out using the Sobolev embedding $\|\chi'[\eta]\|_{C^0}\leq C \|\eta\|_{3/2+\underline\gamma}$ or equivalently for $\xi$,  after which the proofs proceed as in Lemma \refeq{deformationbounds}. 
\end{proof}

 \section{Concentrating Local Solutions}
 \label{section7}
 
This section introduces the model solutions that give the initial approximation in the gluing problem. These model solutions are defined on tubular neighborhoods of each singular set $\mathcal Z_\tau$. These $(\e,\tau)$-parameterized tubular neighborhoods which host the model solutions shrink in diameter as $\e\to0$ for each fixed $\tau$. More specifically, we recall the following from Appendix  \ref{appendixofnotation}.

\begin{defn}
\label{insideoutsidedef}
Let the {\bf inside} and {\bf outside} regions respectively be defined by 
\begin{eqnarray} Y^+_{\e,\tau}:=N_{\lambda^+(\e)}(\mathcal Z_\tau)  \hspace{1.0cm} & \text{ where }& \hspace{1.0cm}\lambda^+(\e)=\e^{1/2}\label{lambdadef} \\ \
{Y^-_{\e,\tau}=\Yminus N_{\lambda^-}(\mathcal Z_\tau).}  \hspace{1.0cm} & \text{ where } &  \hspace{1.0cm} \lambda^{-}(\e)=\e^{2/3-\gamma^-} \end{eqnarray} 

\noindent for $\e\in (0,\e_0)$, where $\gamma^-=10^{-6}$ and $N_\lambda(\mathcal Z_\tau)$ is the tubular neighborhood of the singular set for $\tau\in (-\tau_0,\tau_0)$ of radius $\lambda$. The overlap $Y^+_{\e,\tau} \cap Y^-_{\e,\tau}$ is called the {\bf neck region}. Both $Y^\pm_{\e,\tau}$ are equipped with Fermi coordinates as in Definition \ref{Fermicoords} using the metric $g_\tau$. Unless confusion can easily arise, we omit the specific dependence of the region on $\e,\tau$ and simply write $Y^\pm$.    
\end{defn}
 
Families of model solutions on $Y^+$ parameterized by $(\e,\tau)$ were constructed in \cite{PartI}. This section introduces these solutions and their relevant properties, as well discusses the linear analysis of the linearized equations at these model solutions. References to specific sections of \cite{PartI} will be indicated, where the details of the construction may be found. These model solutions are the concrete version of the model solutions (\refeq{initialmodel}) from the outline in Section \ref{section2};  the main result of this section, Theorem \ref{PartImain} below, gives a precise version of Hypothesis \ref{hyp1prime}(A).

\subsection{The Desingularized Configurations}
The construction of the model solutions has two steps, the first of which is to construct preliminary ``de-singularized'' configurations. These de-singularized solutions are formed from a radially-symmetric ODE solution on the normal planes, which de-singularize the $\e=0$ connection $A_\tau$ into a smooth connection with highly concentrated curvature. These de-singularized configurations have an error term, coming from the derivative in the tangential directions, which is $L^2$ bounded uniformly in $\e$. The second step of constructing model solutions is to show the linearization at the de-singularized configurations is sufficiently invertible, and use it to correct this tangential error.

  The de-singularized configurations arise from solving the ODE for the dimensionally-reduced equations in the normal directions to $\mathcal Z_\tau$, with a radially symmetric ansatz. The dimensional reduction of the Seiberg--Witten equations to the normal disks $D_\lambda(t_0):=\{(t_0, x, y) \in Y^+ \ | \ x^2 + y^2 \leq \lambda\}$ for fixed $t_0$ read 
  
  \be
  \begin{pmatrix} 0 &-2 \del_A \\  2\delbar_A & 0 
  \end{pmatrix}  \hspace{2cm} \mu_\C(\Phi)=0 \hspace{2cm} F_A + \frac{\mu_R(\Phi)}{\e^2}=0
  \label{dimred} 
  \ee

  \noindent where $\Phi=(\alpha,\beta)$ is a pair of $E$-valued sections, and $A=ia_x dx + i a_ydy$ has no $dt$ component (consequently its curvature is has only $dx\wedge dy$ component, and the moment map $\mu$ is split accordingly).  These equations are a vortex-type system on the Riemann surface with boundary $D_\lambda(t_0)$, now depending parametrically on $(\e,\tau)\in (0,\e_0)\times (-\tau_0, \tau_0)$ and the tangential coordinate $t_0 \in \mathcal Z_\tau$. The first two of these vortex-type equations are invariant under the action of complex gauge transformations.

 The vortex system admits a model solution on $\C$ with its flat metric. To simplify further, we work with the leader-order terms of $(\Phi_\tau, A_\tau)$ in $r$, defined by 
   \be \Phi^\bullet_\tau=\frac{1}{2}\begin{pmatrix} c(t) \\ d(t)e^{-i\theta}\end{pmatrix}  r^{1/2}  + \frac{1}{2}\sigma\begin{pmatrix} \ \\ \ldots \end{pmatrix}\hspace{2cm} A^\bullet_\tau=\frac{1}{4}\left(\frac{dz}{z}-\frac{d\overline z}{\overline z}\right) \label{bulletedversion}\ee
  
\noindent in Fermi coordinates, where $\sigma$ symmetrizes so that $\Phi_\tau\in \Gamma(S^\text{Re})$. To solve Eqns. (\refeq{dimred}), we take the ansatz that there is a radially symmetric, complex-valued gauge transformation $h_\e(r): D_\lambda(t_0)\to \C$ such that 
  $$(\Phi^{h_\e}, A^{h_\e}):= e^{h_\e}\cdot (\Phi^\bullet_\tau,A^\bullet_\tau)$$
  
  \noindent where $\cdot$ denotes the complex gauge action $e^h\cdot (\alpha, \beta)=(e^h\alpha, e^{-\overline h}\beta)$ and $e^h\cdot A=A + \del \overline h - \delbar h$. Since the first two equations are invariant under such gauge transformations and are solved by $(\Phi^\bullet_\tau,A^\bullet_\tau)$, the system (\refeq{dimred}) reduces to a single degenerate second-order ODE of Painlev\'e III type for $h_\e(r)$ arising from the curvature equation. We allow solutions that become singular at $r=0$. This ODE was solved in \cite{MWWW} in the context of a related gluing problem for Hitchin's equations; adapting their work, \cite[Sec. 4.1--4.2]{PartI} proves the following lemma. 
    
  \begin{lm} \label{fiducialbounds} For each $(\e,\tau)\in (0,\e_0)\times (-\tau_0,\tau_0)$ and $t_0 \in \mathcal Z_\tau$, there is a unique radially symmetric, complex valued gauge transformation $h_\e(r; \tau, t_0): D_\lambda(t_0)\to \C$ such that 
  
    $$(\ph^{h_\e}_{\tau, t_0}, a^{h_\e}_{\tau, t_0}):= e^{h_\e(r; \tau, t_0)}\cdot (\Phi^\bullet_\tau,A^\bullet_\tau)$$

\noindent are smooth configurations obeying (\refeq{dimred}) on $D_\lambda(t_0)$ and the following hold. 

\begin{enumerate}
\item[(A)] $h_\e(r; \tau, t_0)$ is a smooth function of $r$ for $r>0$ with $h \sim O(-\log(r \e^{-1/3}))$ as $r\to 0$, and depends smoothly on the parameters $\e, \tau, t_0$. 
\item[(B)] There is a constants $c_1, C$ uniform in $\e,\tau, t_0$ such that 

$$\Big \| \left(\ph^{h_\e}_{\tau, t_0}, a^{h_\e}_{\tau, t_0}\right)-\left(\Phi^\bullet_\tau, A^\bullet_\tau\right) \Big \|_{C^0(r\geq c_1\e^{2/3})} \leq C \text{Exp}\left(-\frac{r^{3/2}}{\e}\right).$$

\noindent holds in the region where $r\geq c_1 \e^{2/3}$. The function $\|e^{h_\e(r;\tau,t_0)}-1\|_{C^0(r\geq c_1\e^{2/3})}$ obeys the same bound. 
\item[(C)] For $r\leq \lambda^+=\e^{1/2}$, $|\ph_{\tau,t_0}^{h_\e}|$ is a monotonically decreasing function of $r$, and there is a constant $c_2$ uniform in $\e,\tau, t_0$ such that 

$$|\ph_{\tau,t_0}^{h_\e}| \geq c_2 \e^{1/3}$$

\noindent holds in $D_\lambda(t_0)$.  
\end{enumerate}
  \end{lm}
  
  \begin{proof}
  See \cite[Prop. 4.4]{PartI} and the references to the analysis of \cite{MWWW} therein. 
  \end{proof}

  Given the parameterized family on the normal planes, we define 
  
  \begin{defn}The {\bf de-singularized configurations} corresponding to the eigenvector $(\mathcal Z_\tau, A_\tau, \Phi_\tau)$ are defined as 
  \bea \Phi^{h_\e}_{\e,\tau}(t, r, \theta)&:=&e^{h_\e(r; \tau, t_0)}\cdot \Phi_\tau\\
  A^{h_\e}_{\e,\tau}(t, r, \theta)&:=& e^{h_\e(r; \tau, t_0)}\cdot A_\tau
  \eea
 and are defined over $Y^+=N_\lambda(\mathcal Z_\tau)$. Note that the right hand side uses the {\it unbulleted} versions of the eigenspinor and connection, thus includes the higher order terms, in contrast to  \eqref{bulletedversion}. 
  \label{desingularized}
  \end{defn}

 By the previous lemma, these are smooth configurations on the tubular neighborhood, and continue to obey the bounds in Items (B)--(C). The next subsection begins the analysis of their linearization. A cartoon (copied from \cite[Fig. 1]{PartI}) depicting the radial profiles of the desingularized solutions and of the corresponding curvature $F_{A^{h_\e}}$ for $\tau=0$ is depicted below. 
\begin{figure}[h!]
\begin{center}
\begin{picture}(200,150)
\put(-120,-1.3){\includegraphics[scale=1.6]{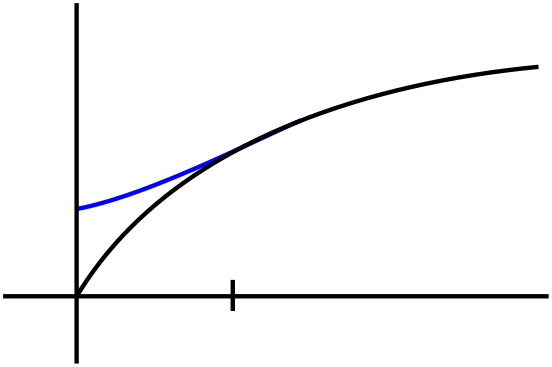}}
\put(115,0){\includegraphics[scale=1.6]{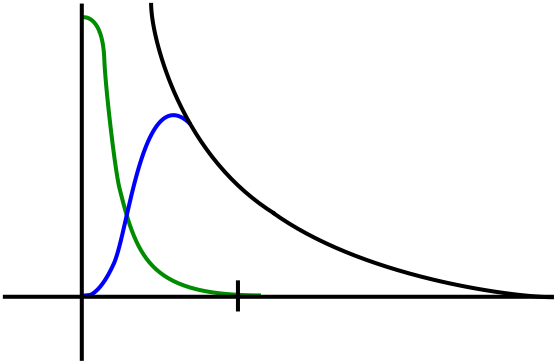}}
\put(20,85){\Large $|\Phi_0|$}
\put(215,69){\Large $|A_0|$}
\put(143,148){\color{OliveGreen}\large $|F_{A^{h_\e}}|$}
\put(170,60){\color{blue}\Large $|{A^{h_\e}}|$}
\put(-85,75){\color{blue}\Large $|\Phi^{h_\e}|$}
\put(-45,8){$\large O(\e^{2/3})$}
\put(190,8){$\large O(\e^{2/3})$}
\put(-125,58){$\large O(\e^{\tfrac{1}{3}})$}
\put(112,90){$\large O(\e^{-\tfrac{2}{3}})$}
\put(112,132){$\large O(\e^{-\tfrac{4}{3}})$}
\put(70,12){$r\to \infty$}
\put(300,12){$r\to \infty$}
\label{Fig1}
\end{picture}
\caption{The radial profiles of the de-singularized configurations compared to those of the limiting $\Z_2$-harmonic spinor. }
\end{center}
\end{figure}
  \noindent

  \begin{rem}
The de-singularization process used to obtain $(\Phi^{h_\e}_{\e,\tau}, A^{h_\e}_{\e,\tau})$ smoothes the $\Z_2$-harmonic spinor and the accompanying singular connection $A_\tau$. This smoothing process also re-introduces a highly concentrated ``bubble'' of curvature near $\mathcal Z_\tau$ so that the $-1$ holonomy around meridians is preserved for sufficiently small $\e$ up to small error. The curvature $F_{A^{h_\e}}$ is smooth with $C^0$-norm of size $O(\e^{-4/3})$, and $L^2$-norm of size $O(\e^{-2/3})$. This curvature reintroduces the curvature that bubbles away in the limit $\e\to 0$, as discussed in Section \ref{section3.3}. 
\end{rem}

\subsection{Hilbert Spaces and Boundary Conditions}
\label{section7.1}

This subsection begins the linear analysis of the linearization at the desingularized solutions from Definition \ref{desingularized} by defining Sobolev spaces with weights and boundary conditions on the tubular neighborhoods $Y^+=N_\lambda(\mathcal Z_\tau)$. These norms are specifically adapted so that the elliptic theory of the linearization at the desingularized solutions is (as close as possible) to being uniform in $(\e,\tau)$. The norms are defined in terms of the following weight function.  

Denote by $r_0$ the radius of the Fermi coordinate chart on $Y^+$. With $r=\text{dist}(-,\mathcal Z_\tau)$, and \be \kappa_\tau(t):=\sqrt{|c_\tau(t)|^2 +|d_\tau(t)|^2}\label{kappatau}\ee

\noindent where $c_\tau(t), d_\tau(t)$ are the (smooth) leading coefficients in the polyhomogeneous expansion of $\Phi_\tau$ in Lemma \ref{asymptoticexpansion}. Since $\Phi_0$ is regular -- in particular non-degenerate -- and these coefficients depend smoothly on $\tau$, $|\kappa_\tau(t)|>\kappa_0$ is bounded below uniformly for $\tau\in (-\tau_0, \tau_0)$ and $t\in \mathcal Z_\tau$. Let $R_\e$ be a smooth function such that $$R_{\e,\tau}(r) =\begin{cases} \sqrt{\kappa_0 \e^{4/3} + r^2} \hspace{1.2cm} r\leq r_0/2 \\ \text{const} \hspace{2.5cm} r\geq r_0. \end{cases}$$

\noindent Note that the weight function $R_{\e_\tau}$ is approximately equal to $r$ on the tubular neighborhood $\{r\leq r_0/2\}$, but levels off to be constant for $r\leq O(\e^{2/3})$. In particular, there is a uniform  (in $\tau$ and $t$) lower bound $R_\e \geq c_1 \e^{2/3}$.  The norms also use the norm  $|\Phi_{\e,\tau}^{h_\e}|$ as a weight. By Lemma \ref{fiducialbounds}, Item (B), this weight is exponentially close to $|\Phi_\tau|\sim r^{1/2}$ for $r\geq c_1\e^{2/3}$, thus this latter weight function is commensurate with $|\Phi_\tau^{h_\e}|\sim \sqrt{R_{\e,\tau}}$. It is used in place of $\sqrt{R_{\e,\tau}}$ simply because it naturally appears in the Weitzenb\"ock formula for the linearized operator.  
\begin{defn}
\label{insidenorms}
Let $\nu \in \R$ be a weight, and $\e\in (0, \e_0)$. The ``inside'' Sobolev norms are defined by 
 \begin{eqnarray}
\|(\ph,a)\|_{H^{1,+}_{\e,\nu}}&:=& \left(\int_{Y^+} \left(| \nabla \ph|^2 + |\nabla a|^2 + \frac{|\ph|^2}{R_\e^2}+ \frac{|\mu(\ph, \Phi^{h_\e}_\e)|^2}{\e^2} + \frac{|a|^2 |\Phi^{h_\e}_\e|^2}{\e^2}\right) R_\e^{2\nu} \ dV \right)^{1/2}\label{H1innorm}\\
\|(\ph,a)\|_{L^{2,+}_{\e,\nu}}&:=& \left(\int_{Y^+}\left( |\ph|^2 + |a|^2\right) R_\e^{2\nu} \ dV \right)^{1/2}\label{L2innorm}
\end{eqnarray}

\noindent where the dependence of $\Phi^{h_\e},dV$, and $R_\e$ on $\tau$ is suppressed in the notation, and $\nabla$ is formed using $A^{h_\e}_{\e,\tau}$, the Levi-Civita connection of $g_\tau$ and $B_\tau$. Both norms give rise to inner products via their polarizations. Because $N_\lambda(\mathcal Z_\tau)$ is compact, these norms are equivalent to the standard $L^{1,2}$ and $L^2$ norms respectively  (though not uniformly in $\e,\nu$).  \end{defn}

\medskip 

Note that we do not yet define the spaces $H^{1,+}_{\e,\nu}$ as the closures of smooth sections with respect to the above norm and likewise for $L^{2,+}_{\e,\nu}$. We instead define these spaces below as the subspaces with finite norm and subject to certain Atiyah-Patodi-Singer (APS) type boundary conditions. Since the above norms are equivalent to the standard Sobolev norms, there is a well-defined boundary trace 

$$\text{tr}: \left\{(\ph,a)  \ \big | \  \|(\ph,a)\|_{H^{1,+}_{\e,\nu}}<\infty \right\} \to L^{1/2,2}(\del Y^+; S_E \oplus (\Omega^0\oplus \Omega^1)(i\R)),$$

\noindent which we also denote by $\text{tr}(\ph)=\ph|_{\del Y^+}$. By standard theory (e.g. \cite{APS1, KM}), an index 0 APS boundary condition for the first-order elliptic (gauge-fixed) linearized Seiberg--Witten equations at a smooth configuration is a closed subspace 

$$\Lambda_0 \subseteq L^{1/2,2}\big(\del Y^+ ; S_E \oplus (\Omega^0\oplus \Omega^1)(i\R)\big),$$

\noindent that is Lagrangian with respect to the boundary pairing \be \Omega_{\del Y^+}(\phi, \psi):= \br \slashed D_{A}\ph, \psi \kt_{L^2} -  \br \ph, \slashed D_{A}\psi \kt_{L^2},\label{symplecticdY}\ee

\noindent  for any smooth connection $A$. Notice that this definition is invariant under changing $A\mapsto A+a$.  The pairing (\refeq{symplecticdY}) gives a symplectic form (see \cite[Sec. 2]{BW25} and Appendix \ref{appendixB}). The associated boundary-value problem 

\be
\begin{cases}
\mathcal L_{(\Phi,A)}(\ph,a)=(\psi,b)\\
\Pi_{\Lambda_0}(\ph|_{\del Y^+}, a|_{\del Y^+})=0,  
\end{cases}
\ee

\noindent is then Fredholm of index zero as a map $\mathcal L_{(\Phi, A)}:\{(\ph,a)\in L^{1,2} \ | \ \Pi_{\Lambda_0}(\ph|_{\del Y^+}, a|_{\del Y^+})=0\}\to L^2$. This follows from a simple integration by parts argument and the Weitzenb\"ock formula (see \cite[Sec. 7.1]{PartI}). Note that different choices of a smooth configuration $(\Phi, A)$ alter the linearization by a compact operator, thus the above discussion is insensitive to the choice of smooth configuration $(\Phi, A)$.

More generally, given such a Lagrangian, then any pair of $\Lambda_{-1}, \Lambda_1$ of closed isotropic and coisotropic subspaces respectively such that

$$\Lambda_{-1}  \subseteq \Lambda_0\subseteq  \Lambda_1$$

\noindent where all the inclusions have finite codimension define an APS boundary conditions such that the linearization 

$$\mathcal L_{(\Phi, A)}: \left\{(\ph,a)  \ \Big   |  \  \|(\ph,a)\|_{H^{1,+}_{\e,\nu}}<\infty  \ , \  \Pi_{\Lambda_i}(\ph,a)|_{\del Y^+}=0  \right\}  \to L^2(Y^+; S_E \oplus (\Omega^0 \oplus \Omega^1)(i\R))  $$

\noindent is Fredholm. Here $\Pi_{\Lambda_i}$ are the $L^2$-orthogonal projections to the subspaces. In the isotropic case, the operator has index $\text{ind}(\mathcal L_{(\Phi,A)})=\text{dim}(\Lambda_0 / \Lambda_{-1})\geq 0$, and in the coisotropic case it has negative index   $\text{ind}(\mathcal L_{(\Phi,A)})=-\text{dim}(\Lambda_1 / \Lambda_{0})$.

\begin{defn}  \label{mixedbddef}A {\bf mixed APS boundary and orthogonality condition} defined by a choice of closed Lagrangian subspace $\Lambda_0$, a coisotropic subspace $\Lambda_{-1}\subseteq \Lambda_0$ and a finite-dimensional subspace of sections $V\subseteq L^{1,2}(Y^+; S_E \oplus (\Omega^0\oplus \Omega^1)(i\R))$, is the condition that 
\bea
\Pi_{\Lambda_{-1}}(\ph |_{\del Y^+},a |_{\del Y^+})&=&0 \\ 
 \Pi_{V}(\ph,a)&=&0,
\eea 

\noindent where $\Pi_{\Lambda_{-1}}, \Pi_V$ are again the $L^2$-orthogonal projections. Associated to any choice of data $(\Lambda_0, \Lambda_{-1}, V)$ there is a closed subspace 

\be H_{\Lambda_{-1}, \Lambda_0, V}:= \left\{(\ph, a)  \ \ \Big | \ \  \Pi_{\Lambda_{-1}}(\ph |_{\del Y^+},a |_{\del Y^+})=0  \ \ , \ \   \\ 
 \Pi_{V}(\ph,a)=0 \ \right\} \subseteq L^{1,2}(Y^+), \label{H+bdortho}\ee
 
 \noindent and the gauge-fixed linearization $\mathcal L_{(\Phi, A)}: H_{\Lambda_{-1}, \Lambda_0, V}\to L^2$ is Fredholm for any smooth configuration $(\Phi, A)$ with $\text{ind}(\mathcal L_{(\Phi, A)})= \text{dim}(\Lambda_0/\Lambda_{-1}) - \text{dim}(V)$.

\end{defn}

\bigskip

 Although the linearized equations $\mathcal L_{(\Phi,A)}$ on $Y^+$ is Fredholm (thus has elliptic estimates) for any choice of mixed APS and orthogonality condition, this alone is not sufficient for the purpose of the gluing. The gluing requires that (1) the operator is invertible, and that (2) the elliptic estimates are nearly uniform. The main result of \cite{PartI} is the construction of a particular choice of mixed APS boundary and orthogonality conditions such that this is the case. In the below theorem, we use $\mathcal L^{h_\e}_{\e,\tau}$ to denote the gauge-fixed linearized Seiberg--Witten equations at the de-singularized configurations from Definition  \ref{desingularized}.
 
\begin{thm}\label{Insideinvertibility}(\cite{PartI}, Theorems 1.4 \& 7.1) For each $\e \in (0, \e_0)$ and $\tau \in (-\tau_0, \tau_0)$, there exists a choice of mixed APS boundary and orthogonality conditions defined by subspaces $(\Lambda_{-1}^+,\Lambda_0^+, V^+)$ such that 
\be \mathcal L^{h_\e}_{\e,\tau}: H_{\Lambda_{-1}^+, \Lambda_0^+, V^+}\lre L^{2,+}_{\e,\nu}(Y^+)\label{7.7}\ee

\noindent is Fredholm of index 0, where the domain is as in (\refeq{H+bdortho}), and the following hold, provided $\e_0, \tau_0$ are sufficiently small. 

\begin{enumerate}
\item[(A)] For any $\nu \in \left[0, \tfrac{1}{4}\right)$, the operator \refeq{7.7} is invertible, and the elliptic estimate

\be  \|(\ph,a)\|_{H^{1,+}_{\e,\nu}}\leq \frac{C_\nu}{\e^{1/12 + \gamma_\mathcal L}} \| \mathcal L^{h_\e}_{\e,\tau}(\ph,a)\|_{L^{2,+}_{\e,\nu}}\label{inversebound}\ee

\noindent holds uniformly for $\tau \in (-\tau_0, \tau_0)$ where $\gamma_\mathcal L=\tfrac{2}{3}\left(\tfrac{1}{4}-\nu\right)+ \gamma^+ \nu$. 
\item[(B)] The derivative $\del_\tau \mathcal L^{h_\e}_{\e,\tau}$ is uniformly bounded on the spaces (\refeq{7.7}).
\end{enumerate}
 \qed
\end{thm}
Recall that $\gamma^+$ is a fixed suitable choice of small positive number, say $\gamma^+=10^{-6}$. For the weight $\nu^+=\tfrac14-10^{-6}$ is as in Appendix \ref{appendixofnotation}, then $\gamma_\mathcal L<<1$.

\bigskip 

\begin{defn}\label{insidespaces}
For a weight $\nu \in \R$, define the ``inside'' Hilbert spaces by 
\bea
H^{1,+}_{\e,\nu}(Y^+; S_E\oplus (\Omega^0\oplus \Omega^1)(i\R))&:=& \Big\{ (\ph, a) \ \big | \ \|(\ph, a)\|_{H^{1,+}_{\e,\nu}}<\infty \  \ , \ \  \Pi_{\Lambda_{-1}^+}(\ph |_{\del Y^+},a |_{\del Y^+})=0  \ \ , \ \   \ 
 \Pi_{V^+}(\ph,a)=0 \Big\} \\ 
L^{2,+}_{\e,\nu}(Y^+; S_E\oplus (\Omega^0\oplus \Omega^1)(i\R))&:=& \Big\{ (\ph, a) \ \big | \ \|(\ph, a)\|_{L^{2,+}_{\e,\nu}}<\infty\Big\}, 
\eea

\noindent where the projections are those for the subspaces $\Lambda_{-1}^+,V^+$ for which the conclusions of Theorem \ref{Insideinvertibility} apply. These spaces are equipped with the norms from Definition \ref{insidenorms}, and the inner products arising from their polarizations. They are defined over the domain $Y^+=N_\lambda(\mathcal Z_\tau)$, and depend implicitly on $\tau$. When no confusion will arise, the domain $Y^+$ and the vector bundles are omitted from the notation. 
\end{defn}

\begin{rem} Because the desingularized configurations converge to $(\Phi_\tau, A_\tau)$ in $C^\infty_{loc}(\Yminus\mathcal Z_\tau)$, one expects that the linearizations 
$$ \mathcal L^{h_\e}_{\e,\tau} \to \mathcal L_{(\Phi_\tau, A_\tau)}$$

\noindent converge (in a sense we do not attempt to make precise) to the singular linearization from Eq. (\refeq{blockdiagonaldecomp}). Because the latter is semi-Fredholm with infinite-dimensional cokernel, one expects that there is a growing subspace of the codomain on which the elliptic estimates of the family $\mathcal L^{h_\e}$ blow up. Identifying this subspace and projecting away from it -- via the correct choice of the subspace $V$ in Definition \ref{mixedbddef} is crucial for obtaining elliptic estimates as in Theorem \ref{Insideinvertibility}(B) with sufficiently mild powers of $\e^{-1}$, and is the main challenge of \cite{PartI}. The precise definitions of $(\Lambda_{-1}^+, \Lambda_0^+, V^+)$ (which depend on $\e,\tau$) are described in \cite[Sec. 7]{PartI} and are not essential for our purposes here. Additional detail is provided in Appendix \ref{appendixB}. 
\end{rem}

\subsection{Model Solutions} The de-singularized configurations are defined over the ``inside'' neighborhood $Y^+=N_\lambda(\mathcal Z_\tau)$. In this subsection, we extend them to all of $Y$ using the cutoff functions $\chi^\pm$ and perform the first stage of the alternating iteration. 

Given the de-singularized solutions on $Y^+$, we define a spliced configuration on the closed manifold $Y$ as follows. Recall that $\chi^\pm$ are the cutoff functions defined in Appendix \ref{appendixofnotation}. Let 

\be (\Phi^{(0)}_{\e,\tau}, A^{(0)}_{\e,\tau}):= \chi^+\cdot  (\Phi^{h_\e}_{\e,\tau}, A^{h_\e}_{\e,\tau}) + (1-\chi^+) \cdot (\Phi_\tau, A_\tau).\label{approximate1}\ee
 
 \noindent These configurations have an error term that decays exponentially where $r\geq \e^{2/3-\gamma}$, by Lemma \ref{fiducialbounds}, Item (B). Theorem \ref{PartImain} below shows that one may correct for this error term, and define a corrected configuration

\be (\Phi^{(1)}_{\e,\tau}, A^{(1)}_{\e,\tau}):= (\Phi^{(0)}_{\e,\tau}, A^{(0)}_{\e,\tau})+ \chi^+\cdot  (\ph^{(1)}_{\e,\tau}, a^{(1)}_{\e,\tau}). \label{approximate1}\ee

\noindent This is the first stage of the alternating iteration -- the first ``inside'' correction.

\begin{defn} \label{defpreglued} \label{notationhn}
We define the {\bf pre-glued configurations} and the {\bf model solutions} by 
\bea
h_0&:=&   (\Phi^{(0)}_{\e,\tau}, A^{(0)}_{\e,\tau}) \\
h_1&:=&  (\Phi^{(1)}_{\e,\tau}, A^{(1)}_{\e,\tau})
\eea

\noindent respectively, where  $(\Phi^{h_\e}_{\e,\tau}, A^{h_\e}_{\e,\tau}) $ are the desingularized configurations from Definition \ref{desingularized}.
\end{defn}  

\medskip 
The following theorem characterizes the model solutions precisely. 

\begin{thm}\label{PartImain} (\cite{PartI}, Theorem 1.2 ) 
Suppose that $(\mathcal Z_\tau, A_\tau, \Phi_\tau)$ for $\tau \in (-\tau_0, \tau_0)$ are a family of $\Z_2$-harmonic eigenvectors satisfying the hypotheses of Theorem \ref{maina} with eigenvalues $\Lambda_\tau$. Then, for $\e<\e_0$ with the latter sufficiently small, there exist approximate solutions $ (\Phi^{(1)}_{\e,\tau}, A^{(1)}_{\e,\tau})$ smoothly parameterized by $(\e,\tau)$ and constructed as in (\refeq{approximate1}) with the following properties. 

\begin{enumerate}
\item [(A)] 
They satisfy $$\text{SW}\left( \frac{\Phi^{(1)}_{\e,\tau}}{\e}, A^{(1)}_{\e,\tau}\right)=\frac{\chi^- \Lambda_\tau \Phi_\tau}{\e} + e_1+f_1 $$
\noindent where $\text{SW}$ denotes the extended, gauge-fixed Seiberg--Witten equations with respect to $(g_\tau,B_\tau)$, and $e_1,f_1$ are error terms obeying

\begin{itemize}
\item[(1)] $e_1 \in \Gamma(S^\text{Re})$, while $f_1\in \Gamma(S^\text{Im}\oplus (\Omega^0\oplus \Omega^1)(i\R))$. Both have $$\text{supp}(e_1)\ , \ \text{supp}(f_1)\subseteq \text{supp}(d\chi^+).$$ 
\item[(2)] $\|e_1\|_{L^2(Y)}\leq C \e^{-1/24 + \gamma}. $
\item[(3)] $ \|f_1\|_{L^2(Y)} \leq C \e^{M}$ for $M>10$.
\end{itemize} 
\noindent for $\gamma<<1$. Moreover, the derivatives $\del_\tau e_1,\del_\tau f_1$ also satisfy (1)--(3). 
\item[(B)] There is a constant $C$ independent of $\e,\tau$ such that correction terms $ (\ph^{(1)}_{\e,\tau}, a^{(1)}_{\e,\tau}) \in H^{1,+}_{\e,0}(Y^+; S_E \oplus (\Omega^0\oplus \Omega^1)(i\R))$ in Eq. (\refeq{approximate1}) obey
 \be \| (\ph^{(1)}_{\e,\tau}, a^{(1)}_{\e,\tau})\|_{H^{1,+}_{\e,0}} \ + \  \|\del_\tau(\ph^{(1)}_{\e,\tau}, a^{(1)}_{\e,\tau})\|_{H^{1,+}_{\e,0}}\leq C\e^{-1/12-\gamma_\mathcal L}.\label{fordecayofinitial}\ee
 
 \noindent where $H^{1,+}_{\e,0}$ is the $\nu=0$ weight Sobolev norm from Definition \ref{insidenorms}, and $\gamma_\mathcal L$ is as in Theorem \ref{Insideinvertibility}.  
 \noindent 
\item[(C)] The $L^2$-norm satisfies $$\Big\|\frac{\Phi^{(1)}_{\e,\tau}}{\e}\Big\|_{L^2(Y)}= \frac{1}{\e} \ + \ o(1).$$ 

\noindent uniformly in $\e,\tau$. 

\item[(D)] The restriction of the linearization $$\mathcal L^+_{h_1}:=\mathcal L_{(\Phi^{(1)}_{\e,\tau},(A^{(1)}_{\e,\tau})}$$

\noindent to $Y^+$ subject to the mixed boundary and orthogonality conditions on $H^{1,+}_{\e,\nu}(Y^+)$ from Theorem \ref{Insideinvertibility} continues to obey the conclusions of Theorem \ref{Insideinvertibility}. 

\end{enumerate}

\qed 

\end{thm}

\begin{proof}
This theorem is a combination of Theorems 1.2, 7.1, and the error term estimates in Section 4.3 of \cite{PartI}. When the quantity $\frac{\chi^- \Lambda(\tau)\Phi_\tau}{\e}$ is subtracted from the initial error $\text{SW}\left( \e^{-1}{\Phi^{(0)}_{\e,\tau}}, A^{(0)}_{\e,\tau}\right)$, the error for $\tau \neq 0$ obeys the same bounds as that for $\tau=0$ covered in Section 4.3 of \cite{PartI}. Smoothness in $\tau$ and the derivative bound follow from the fact that the operators and mixed APS boundary and orthogonality conditions in Theorem \ref{Insideinvertibility} depend smoothly on $p_\tau=(g_\tau, B_\tau)$ and on $\Phi_\tau$. Item (D) follows from Theorem \ref{Insideinvertibility} and Item (B) via Neumann series. 
\end{proof}

\begin{rem} Note that in Item (A) of Theorem \ref{PartImain}, the components of the configuration are split into the $S^\text{Re}$ component $e_1$ and the $S^\text{Im}\oplus \Omega^0(i\R)\oplus \Omega^1(i\R)$ components $f_1$, which obey different bounds. Throughout the gluing iteration, the $S^\text{Re}$ components of the error term dominate, while the remaining components enjoy exponential or arbitrarily large polynomial decay in $\e$. Keeping track of the bounds on these individually provides stronger bounds on certain quadratic non-linearities later on, which are bilinear pairings between the two subbundles. More general results on the exponential decay properties for the $S^\text{Im}\oplus \Omega^0(i\R)\oplus \Omega^1(i\R)$  components are the subject of \cite{ConcentratingDirac}. 
\end{rem}

\subsection{The Outside Linearization}
\label{section7.3}

Theorem \ref{PartImain} Part (D) shows that the linearized Seiberg--Witten equations at the pre-glued configurations are invertible. This allows us to solve the linearized problem on the ``inside'' region $Y^+$ to make iterative corrections during the gluing procedure, giving a suitable version of Hypothesis \ref{hyp1prime}(A). This subsection shows an equivalent statement on the solvability of the linearized equations in the {\it outside} region $Y^-$, offering a suitable version of Hypothesis \ref{hyp1prime}(B).  

 Recall from Lemma \ref{linearizedequationsdef} that the linearized Seiberg--Witten equations at the limiting $\Z_2$-harmonic spinor take the form 

\be \mathcal L_{(\Phi_\tau, A_\tau)} (\ph_1, \ph_2,  a)  = \begin{pmatrix} \slashed D_{A_\tau} & 0 & 0 \\ 0 & \slashed D_{A_\tau} & \gamma(\_)\frac{\Phi_\tau}{\e} \\ 0 & \frac{\mu(\_, \Phi_\tau)}{\e}. & \bold d \end{pmatrix} \begin{pmatrix}\ph^\text{re} \\ \ph^\text{im}\\ a \end{pmatrix}\label{blockdiagonaldecomp2}.\ee

\noindent where $\ph=(\ph_1, \ph_2)\in S^\text{Re}\oplus S^\text{Im}$, and $\bold d$ is as defined in Lemma \ref{linearizedequations}.  In particular, the real component decouples from the imaginary and form components. The top left block is the operator that was studied in Section \ref{section4}; the bottom block is a copy of the standard (i.e. single spinor) Seiberg--Witten equations, can be reduced to standard elliptic theory by viewing it as a boundary-value problem using an adaptation of the boundary conditions defined in Section \ref{section7.1}.  

We consider the following Sobolev norms in the outside region. 
\begin{defn}
\label{outsidenorms}
 Let $\nu \in \R$ be a weight. The ``outside'' Sobolev norms are defined by 
 \begin{eqnarray*}
\|(\ph^\text{re},  \ph^\text{im},a)\|_{H^{1,-}_{\e,\nu}}&:=& \left( \|\ph^\text{re}\|_{r^{1-\nu}H^1_e(\Yminus \mathcal Z_\tau)}^2 \ + \ \int_{Y^-} \left(| \nabla \ph^\text{im}|^2 + |\nabla a|^2 + \frac{|\ph^\text{im}|^2|\Phi_\tau|^2}{\e^2} + \frac{|a|^2 |\Phi_\tau|^2}{\e^2}\right) R_\e^{2\nu} \ dV \right)^{1/2}\label{H1outnorm}\\
\|(\ph^\text{re},\ph^\text{im},a)\|_{L^{2,-}_{\e,\nu}}&:=& \left( \|\ph^\text{re}\|^2_{r^{-\nu}L^2(\Yminus \mathcal Z_\tau)} \ + \ \int_{Y^-}\left( |\ph^\text{im}|^2 + |a|^2\right) R_\e^{2\nu} \ dV \right)^{1/2}\label{L2outnorm}
\end{eqnarray*}

\noindent where the dependence of $dV$, and $R_\e$ on $\tau$ is suppressed in the notation. Here, $\nabla$ is formed using $A_\tau$, the Levi-Civita connection of $g_\tau$, and $B_\tau$.   \end{defn}

\medskip 

Notice that the norm for the section $\ph^\text{re}\in \Gamma(S^\text{Re})$ is defined by integration over $\Yminus \mathcal Z_\tau$, whereas those for $(\ph^\text{im},a)\in \Gamma(S^\text{Im}\oplus \Omega^0\oplus \Omega^1)$ only integrate over $Y^-$. We note also the relationship of the sign convention for the weight to the convention in Section \ref{section4} (which adopts the conventions standard for edge operators): in our notation $r^{\nu}L^2= L^2_{\e,-\nu}$. Finally, observe that $|\ph^\text{im}||\Phi_\tau|^2 = |\mu(\ph, \Phi_\tau)|^2$ since $\Phi\in S^\text{Re}$ and that $R_\e \sim r$ on $Y^-$; thus this norm is equivalent to the $H^{1,+}_{\e,\nu}$-norm on $Y^+\cap Y^-$ (since $\Phi_\tau- \Phi_{\e,\tau}^{h_\e}$ is exponentially small there).

\medskip 

The following lemma asserts the existence of a mixed APS boundary and orthogonality condition on $Y^-$, in the sense of Definition \ref{mixedbddef} (note this definition applies equally well with $Y^-$ in place of $Y^+$). In this case, the orthogonality condition is trivial. This boundary condition is applied to the operator 

\be \mathcal L^\text{Im}_{(\Phi_\tau, A_\tau)}(\ph^\text{im}, a):= \begin{pmatrix}   \slashed D_{A_\tau} & \gamma(\_)\frac{\Phi_\tau}{\e} \\  \frac{\mu(\_, \Phi_\tau)}{\e}. & \bold d \end{pmatrix} \begin{pmatrix}\ph^\text{im}\\ a \end{pmatrix}, \label{Limdef} \ee

\noindent which constitutes the bottom $2\times 2$ block of (\refeq{blockdiagonaldecomp2}), where $\ph^\text{im}\in \Gamma(S^\text{Im}), a\in \Gamma(\Omega^0\oplus \Omega^1)$. Here, the Lagrangian and co-isotropic subspaces (which coincide in the lemma) are closed subspaces 

$$ \Lambda_0\subseteq L^{1/2,2}(Y^-: S^\text{Im}\oplus (\Omega^0\oplus \Omega^1)(i\R)),$$

\noindent Lagrangian with respect to the (equivalent for $S^\text{Im}$ and $Y^-$) of the symplectic form (\refeq{symplecticdY}).

\begin{prop}\label{outboundaryconditionsdef}
For each $\e \in (0, \e_0)$ and $\tau \in (-\tau_0, \tau_0)$, there exists a choice of mixed APS boundary and orthogonality conditions defined by subspaces $(\Lambda_{0},\Lambda_0, \{0\})$ such that 

\be \mathcal L^\text{Im}_{(\Phi_\tau, A_\tau)}: H_{\Lambda_{0}^-}\lre L^{2,+}_{\e,\nu}(Y^-; S^\text{Im}\oplus (\Omega^0\oplus \Omega^1)(i\R))\label{outsidefredholm}\ee

\smallskip 

\noindent obeys the following, where $H_{\Lambda_0}= \{(\ph, a)  \ | \  \Pi_{\Lambda_{0}}(\ph |_{\del Y^-},a |_{\del Y^-})=0 \ \} \subseteq L^{1,2}(Y^+)$ is the subspace obeying the boundary conditions. 

\begin{enumerate}
\item[(A)]  The operator (\refeq{outsidefredholm}) is $L^{2,+}_{\e,\nu}$ self-adjoint in the sense that, 

\be \br \mathcal L^\text{Im}_{(\Phi_\tau, A_\tau)}(\ph, a) \ , \ (\psi,b) \kt_{L^2}= \br (\ph, a) \ , \   \mathcal L^\text{Im}_{(\Phi_\tau, A_\tau)}(\psi,b) \kt_{L^2} \ +\br K_\del(\ph,a),(\psi,b) \kt_{L^2} 
\nonumber \ee

\noindent for $(\ph,a), (\psi,b)\in C^\infty(Y^-; S^\text{Im}\oplus (\Omega^0\oplus \Omega^1)(i\R))$ obeying the boundary conditions. That is, there are no boundary terms when integrating by parts, up to an error term $K_\del$ supporting on the boundary obeying $\|K_\del(\psi,b)\|_{L^{2}}\leq C\e\|(\psi,b)\|_{H^{1,-}_{\e,\nu}}$ uniformly in $\e,\tau$.

\medskip 

\item[(B)]The operator (\refeq{outsidefredholm}) is Fredholm of Index 0. 
\end{enumerate}
\noindent

 \end{prop}
\begin{proof}See Apprendix \ref{appendixB}. 
\end{proof}

\bigskip 

Using the boundary condition of Proposition \ref{outboundaryconditionsdef}, 
\begin{defn}\label{outsidespaces}For a weight $\nu \in \R$, define the ``outside'' Hilbert spaces by 
\bea
H^{1,-}_{\e,\nu}(Y^-; S_E\oplus (\Omega^0\oplus \Omega^1)(i\R))&:=& \Big\{ (\ph^\text{re},\ph^\text{im}, a) \ \big | \ \|(\ph^\text{re},\ph^\text{im}, a)\|_{H^{1,-}_{\e,\nu}}<\infty \  \ , \ \  \Pi_{\Lambda_{0}^-}(\ph^\text{im} |_{\del Y^-},a |_{\del Y^-})=0  ,  \\  & & \ \ \text{ and } \ \   \br \ph^\text{Re}, \Phi_\tau\kt_{L^2}=0 \Big\} \\
L^{2,-}_{\e,\nu}(Y^-; S_E\oplus (\Omega^0\oplus \Omega^1)(i\R))&:=& \Big\{ (\ph^\text{re},\ph^\text{im}, a) \ \big | \ \|(\ph^\text{re},\ph^\text{im}, a)\|_{L^{2,-}_{\e,\nu}}<\infty\Big\}, 
\eea

\noindent where the projections are those for the subspace $\Lambda_{0}^-$ for which the conclusions of Proposition \ref{outboundaryconditionsdef}. These spaces are equipped with the norms from Definition \ref{insidenorms}, and the inner products arising from their polarizations. When no confusion will arise, the domain $Y^+$ and the vector bundles are omitted from the notation. 

{\it Notice:} the boundary conditions are only applied on the latter two components $(\ph^\text{im},a)$. The 1-dimensional projection is equivalent to requiring $\ph^\text{Re}\in rH^1_\perp$ as in Eq.  (\refeq{rH1perp}).
\end{defn}

\medskip 

The following lemma extends Proposition \ref{outboundaryconditionsdef} to show that the operator (\refeq{outsidefredholm}) is actually invertible, and obeys uniform elliptic estimates. Together with Lemma \ref{mappingpropertiesI}, it completes the linear elliptic analysis of (\refeq{blockdiagonaldecomp2}) required for the gluing iteration.

\begin{lm}
\label{solvingoutsideim}
For $-\tfrac12 < \nu < \tfrac12$, the boundary-value problem \be \mathcal L^\text{Im}_{(\Phi_\tau, A_\tau)}: H^{1, -}_{\e,\nu}(Y^-; S^\text{Im}\oplus (\Omega^0 \oplus \Omega^1)(i\R)) \lre L^{2,-}_{\e,\nu}(Y^-; S^\text{Im}\oplus (\Omega^0 \oplus \Omega^1)(i\R))\label{mathcalLIm}\ee

\noindent is invertible for $\e_0$ sufficiently small. Moreover, the estimate 

$$\|(\ph^\text{im}, a)\|_{H^{1,-}_{\e,\nu}} \leq C \|\mathcal L_{(\Phi_\tau, A_\tau)}^\text{Im} (\ph\text{im},a)\|_{L^{2,-}_{\e,\nu}} $$

\noindent holds uniformly in $\e, \tau$, and $\del_\tau  \mathcal L^\text{Im}_{(\Phi_\tau, A_\tau)}$ is uniformly bounded on the spaces (\refeq{mathcalLIm}). \qed

\end{lm} 

\begin{proof} The Weitzenb\"ock formula (\cite[Prop 2.13]{PartI}) shows that 
\be \mathcal L^\text{Im}\mathcal L^\text{Im}\begin{pmatrix}\ph^\text{im} \\ a\end{pmatrix}
=\begin{pmatrix} \slashed D_{A_\tau}\slashed D_{A_\tau}\ph^\text{im} \\  \bold d^\star \bold d a \end{pmatrix} + \frac{1}{\e^2}\begin{pmatrix}\gamma(\mu(\ph^\text{im},\Phi_\tau))\Phi_\tau \\ \mu(\gamma(a)\Phi_\tau, \Phi_\tau))\end{pmatrix} + \frac{1}{\e}\frak B(\ph^\text{im},a)\label{weitzenbock}\ee
\noindent where $\bold d$ is as in  Lemma \ref{linearizedequations} and $\bold d^\star =\bold c$ is its adjoint, and $\frak B(\ph^\text{im} ,a)$ schematically has terms of the form $a\cdot \nabla_{A_\tau}\Phi_\tau$ and $\ph\cdot \nabla_{A_\tau}\Phi_\tau$. Taking the inner product of (\refeq{weitzenbock}) with $(\ph_2,a)$ shows that 
$$\|\mathcal L^\text{Im}(\ph^\text{im} ,a)\|^2_{L^2}=\|(\ph^\text{im} ,a)\|^2_{H^{1,-}_\e}  \ + \  \frac{1}{\e}\br (\ph^\text{im} ,a), \frak B(\ph^\text{im} , a)  \kt \ + \ \text{b.d. terms}. $$

\noindent Proposition  \ref{outboundaryconditionsdef}(A) shows, in fact, that the boundary terms vanish. Using that $\Phi_\tau\sim r^{1/2}$ while $\nabla\Phi_\tau, A_\tau \sim r^{-1/2}$ on $Y^-$, the terms involving $\frak B$ and $A_\tau$ are dominated by the $\e^{-2}|\Phi_\tau|^2$ weight in  (\ref{outsidenorms}) on $Y^-$ where $r\geq \e^{2/3-\gamma}$, and may be absorbed. The additional term $K_\del$ in Part (A) of Lemma \ref{outboundaryconditionsdef} can likewise be absorbed. 
\end{proof}

\subsection{Decay across the Neck Region}
\label{section7.4}

As explained in Section \ref{section2}, the alternating gluing iteration requires the corrections on each region to decay over the neck region, else the iteration will not converge. Here, the neck region is 

$$Y^-\cap Y^+=\Big\{ \e^{2/3-\gamma^-} \leq r \leq \e^{1/2}\Big\}$$

\noindent where $r$ is the radial variable in Fermi coordinates of $g_\tau$ around $\mathcal Z_\tau$. With the analysis of the linearizations in hand (Lemma \ref{mappingpropertiesI}, Theorem \ref{PartImain}, and Lemma \ref{solvingoutsideim}), this subsection establishes the necessary decay results. These results constitute a precise version of Hypothesis \ref{hyp2}. 

 There are two regimes of decay across the neck, depending on the different bundles. For the $S^\text{Im}\oplus (\Omega^0\oplus \Omega^1)$ components, the very positive off-diagonal terms in the bottom block in (\refeq{blockdiagonaldecomp2}) give strong exponential decay results where $r\geq c\e^{2/3}$, generalizing (Item (B) of Lemma \ref{fiducialbounds}). These exponential decay bounds are proved using the concentration techniques established in \cite{ConcentratingDirac}.  For the $S^\text{Re}$-components, there is weaker polynomial decay like the leading asymptotics $r^{1/2}$ of $\Z_2$-harmonic spinors. There are two approaches to proving such estimates, as explained in Section \ref{section2} (cf. (II) in  \refeq{deltadef1}): analysis of the Green's function, and exploitation of weighted spaces. The latter is much simpler and suffices for our purposes.

\subsubsection{Decay on the ``Inside'' Region}. 
Let $g^+ \in L^{2,+}_{\e,\nu}(Y^+)$, and let $(\ph,a)$ be the unique solution of  

\begin{eqnarray} \mathcal L_{h_j}^+(\ph,a) \ &=&g^+ \label{decay1}\end{eqnarray}

\noindent guaranteed by Theorem \ref{PartImain} Item(D), where $h_j=h_0, h_1$ as in Definition \ref{notationhn}. 

This next lemma address the polynomial decay. 

\begin{lm}\label{insidepolynomialdecay}Suppose $(\ph,a)$ is the unique solution of (\refeq{decay1}).  If $\text{supp}(g^+)\subseteq \{r \leq c\e^{2/3-\gamma}\}$, then 
 $$\|d\chi^+.(\ph,a)\|_{L^2} \leq C  \e^{-1/24-\gamma} \|g^+\|_{L^2}, $$
 
 \noindent where $\chi^+$ is the cutoff function equal to $1$ where where $r\leq \e^{1/2}/4$ and vanishing for $r\geq \e^{1/2}/2$, where $\gamma <<1$ is a small number given by a linear combination of $\gamma_\mathcal L, \gamma^\pm$.  
\end{lm}

\begin{proof}
For the duration of the proof, we set $\nu=\nu^+$. By Item (D) of Theorem \ref{PartImain},  the estimate 

\be  \|(\ph,a)\|_{H^{1,+}_{\e,\nu}}\leq \frac{C_\nu}{\e^{1/12 + \gamma_\mathcal L}} \| \mathcal L_{h_j}(\ph,a)\|_{L^{2,+}_{\e,\nu}}\label{inversebound2}\ee

\noindent holds, for $\nu=\nu^+=\tfrac14 -10^{-6}$, where $\gamma_\mathcal L$ is as in Theorem \ref{Insideinvertibility}. Since $g^+$ is supported where $R_\e\leq \e^{2/3-\gamma^+}$, one has 
$$ \|g^+\|_{L^2_{\e,\nu}}=\left(\int_{Y+} |g^+|^2 R^{2\nu}_\e dV_{\tau}\right)^{1/2} \leq C\e^{\nu^+\left(\frac{2}{3}-\gamma^+\right)}\|g\|_{L^2}.$$

\noindent And $\nu^+\left(\frac{2}{3}-\gamma^+\right)=\frac{2}{12}-\gamma'$ where $\gamma'= \frac{2}{3}\gamma^+-\nu^+\gamma^+$. Consequently, (\refeq{inversebound2}) shows 

\bea  
 \|(\ph,a)\|_{H^{1,+}_{\e,\nu}}\leq C\e^{\left(\tfrac{2}{12}-\gamma'\right) -\left(\tfrac{1}{12}+\gamma_\mathcal L\right)}\|g\|_{L^2} =C\e^{\left(\frac{1}{12}-\gamma''\right)}\|g\|_{L^2}
\eea
\noindent where $C=C_{\nu^+}$, and $\gamma''=\gamma' + \gamma_{\mathcal L}$. Then $d\chi\sim R_\e^{-1}$, so $\|d\chi^+(\ph,a)\|_{L^{2,+}_{\e,\nu}}\leq  \|(\ph,a)\|_{H^{1,+}_{\e,\nu}}$. And since $d\chi^+$ is supported where $R_\e\sim \e^{1/2}$, one has

$$\|d\chi^+(\ph,a)\|_{L^2}\leq C \e^{-\frac12\nu^+)}\e^{\left(\frac{1}{12}-\gamma''\right)}\|g\|_{L^2}.$$

\noindent Since $\nu^+=\tfrac14-\gamma^+$, then the power of $\e$ is $-\tfrac{1}{8}-\tfrac{\gamma^+}{2} + \tfrac{1}{12}-\gamma''=-\frac{1}{24}-\gamma'''$, thus the conclusion holds with $\gamma=\gamma'''$. 

\end{proof}

\medskip

This proof demonstrates the line-by-line increase of $\gamma$ by linear combination pertaining to Remark \ref{gammarem}. After this subsection, we will often omit the exact arithmetic of the values of $\gamma$ and simply verify that the value is increase by a (uniform) linear combination of the previous value and the fixed $\gamma^\pm, \gamma_\mathcal L$ at each line.

\begin{rem}
 One could (perhaps justifiably) gripe that Lemma \ref{insidepolynomialdecay} does not, at first glance, constitute a  ``decay'' result because the decay factor diverges as $\e\to 0$. The point, however, is that this result still tempers that growth compared to the factor expected from (\refeq{inversebound}), and does so sufficiently much that when combined with Lemma \ref{polynomialdecayreal} below, the power of $\e$ for a full cycle of the alternating iteration is positive, (cf. Eq. (\refeq{deltapm})). 

\end{rem}

The exponential decay of the $S^\text{Im}\oplus (\Omega^0\oplus \Omega^1)$ components follows from the interpretation of generalized Seiberg--Witten equations a non-linear concentrating Dirac operators with degeneracy, developed extensively in \cite{ConcentratingDirac}. The following lemma specializes these results to the present case; it is proved in Appendix A of \cite{ConcentratingDirac}.

 For the statement of the lemma, let $K_\e \Subset Y^+-\mathcal Z_\tau$ denote a family of compact subsets of the complement of the singular set. Set $r_K:=\text{dist}(K_\e, \mathcal Z_\tau)$. Let $K'_\e$ be a slightly larger family of compact sets so $Y^+-N_{r_K/2}(\mathcal Z_\tau) \subseteq K_\e'$.

\begin{lm}\label{exponentialdecay} Let $(\ph,a)$ be the unique solution to (\refeq{decay1}), and write $(\ph,a)=(\ph_1, \ph_2, a)$ on $Y-\mathcal Z_\tau$. There exist constants $C,c$ such that if $\text{supp}(g)\subseteq (K_\e')^c$, then $S^\text{Im}\oplus \Omega$-components satisfy 

\be \|(\ph^\text{im} , a)\|_{C^m(K_\e)} \leq \frac{C}{\e^{2m+1} r_K^{3/2} } \text{Exp} \left(-\frac{c r_K^{3/2}}{\e}\right) \|(\ph,a)\|_{H^{1,+}_{\e,0}} \ee

\noindent uniformly for $\tau \in (-\tau_0, \tau_0)$. 

\end{lm}
\begin{proof} Both statements follow directly from Corollary A.2 of \cite{ConcentratingDirac}, which includes a similar conclusion for solutions of non-linear equations of the form in Theorem \ref{PartImain}.  
\end{proof}

In particular, we have the following corollary.
\begin{cor}  \label{corollaryexpdecay}Retaining the assumptions of Lemma \ref{exponentialdecay}, in the case that 
$$K_\e=\{r\leq 2\e^{2/3-\gamma^-}\} \hspace{2cm} K_\e'=\{r\leq 2\e^{2/3-\gamma^-}\},$$ 

\noindent Then there are constants $c,C$ independent of $\e,\tau$ such that

\be \|(\ph^\text{im} , a)\|_{C^m(K_\e)} \leq \frac{C}{\e^{2m+2+\gamma^-} } \text{Exp} \left(- c \e^{-\gamma^-}\right) \|(\ph,a)\|_{H^{1,+}_{\e,0}}. \ee

\noindent This applies, in particular, to the configurations $(\ph^{(1)}_{\e,\tau},a^{(1)}_{\e,\tau})$ from Item(B) of Theorem \ref{PartImain} which there obey the above bound. 
\end{cor}

\noindent Observe that for sufficiently small $\e$, exponential term overcomes any order of polynomial growth in $\e$.

\subsubsection{Decay on the ``Outside'' Region} We now turn to decay estimates on $Y^-$. It turns out that the exponential decay of Lemma \ref{exponentialdecay} is sufficiently strong to eliminate the need for decay results on the $(\ph^\text{im}, a)$ components, thus we focus on the $S^\text{Re}$-components.

 There is a subtlety here, however. While it is straightforward to show the analogue of Lemma \ref{insidepolynomialdecay} in this region, for error terms supported where $r\sim O(\e^{1/2})$, this is not sufficient. Indeed, solving of the $S^\text{Re}$ components can only occur after using deformations of the singular set to cancel the obstruction. The projection to the obstruction bundle, however, is a highly non-local operator on $Y$, and cancelling the obstruction disrupts the property that the error is supported where $d\chi^+\neq 0$. The following generalization of support is needed to address this.  
\medskip

\begin{defn}\label{effectivesupport}A spinor $g \in L^2(Y-\mathcal Z_\tau; S^\text{Re})$ is said to have { $\e^\beta$-{\bf effective support}} if there is a constant $C$ such that for all $-\nu \in (-\tfrac{1}{2}, 0],$ 

\be \|g\|_{r^{\nu }L^2}\leq C\e^{-\gamma} \e^{-{\nu\beta}}\|g \|_{L^2}. \label{effectivesupportest}\ee
 
 \noindent holds, for $\gamma=10^{-6}$. 
\end{defn}

\noindent  This definition may be interpreted as follows. A configuration has effective support in a region if its norms in different weighted spaces scale {\it as if they were supported in that region}. For example, if a configuration is supported where $r\sim O(\e^{1/2})$, one would expect changing the weight by $R^{\nu}$ would result in a change in norm of $\e^{\nu/2}$.  In particular, effective support generalizes the standard notion of support, but as we will see, the projections away from the obstruction bundle obey effective support bounds despite the fact that the projection delocalizes the true support.

The following lemma provides the relevant decay property for solutions of the Dirac equation in the outside region. In the statement, $\chi^-, \nu^-$ are in Section \ref{appendixofnotation}. Recall that restricted to $rH^1_\perp$ as defined in (\refeq{rH1perp}), the singular Dirac operator obeys uniform elliptic estimates by Lemma \ref{mappingpropertiesI},

\begin{lm} \label{polynomialdecayreal}Suppose that $g\in L^2\cap \text{Range}(\slashed D_{A_\tau}|_{rH^1_\perp})$ has $\e^{1/2}$-effective support, and let $\psi \in rH^1_e(Y-\mathcal Z_\tau; S^\text{Re})$ be the unique $\Phi_\tau$-perpendicular solution to $$\slashed D_{A_\tau}\psi=g.$$ 

\noindent Then \be \|d\chi^-.\psi\|_{L^2} \ + \ \|(1-\chi^-) g\|_{L^2}\leq C  \e^{1/12-\gamma} \|g\|_{L^2}, \label{insidedecayest}\ee

\noindent where $\gamma<<1$ is small (cf. Conventions in Appendix \ref{gammarem}). 
\end{lm}

The proof is almost identical to that of Lemma \ref{insidepolynomialdecay}. 
\begin{proof} For the duration of the proof set,  $\nu = \nu^-=\tfrac{1}{2}-10^{-6}$, and $\chi=\chi^-$, the cutoff function that localizes support to $Y^-$ (recall Appendix \ref{appendixofnotation}). The elliptic estimate (\refeq{semiellipticest}) from Lemma \ref{mappingpropertiesI} (restricted to $rH^1_\perp$), applies to show that 
$$
\|\psi\|_{r^{1+{\nu}}H^1_e} \ \leq  \ C \|g \|_{r^{\nu} L^2} \ \leq  \ C\e^{-\gamma} \e^{-\tfrac{\nu}{2}}\|g \|_{L^2} = C \e^{-\tfrac{1}{4}-\gamma'}\|g \|_{L^2},$$

\noindent where $\gamma'$ is a linear combination of $\nu-\frac12$ and the $\gamma$ from Definition \ref{effectivesupport} on the effective support bounds. 

Next, because $\chi$ is a cutoff function obeying $|d\chi. \psi| \leq \frac{C}{r}|\psi|$, and is supported where $r\sim \e^{2/3
-\gamma^-}$ (recall $\gamma^-=10^{-6}$ is fixed), one has 

\bea \|d\chi.\psi\|_{L^2} \leq  C \left(\int_{d\chi\neq 0}  \frac{|\psi|^2}{r^2} \frac{r^{2\nu}}{r^{2\nu}}\ dV\right)^{\tfrac12}& \leq & C \e^{\nu(\tfrac{2}{3}-\gamma^-)} \|\psi\|_{r^{1+\nu}H^1_e} \\& \leq&  C \e^{\nu(\tfrac{2}{3}-\gamma^-)} \e^{-\tfrac{1}{4}-\gamma'}\|g \|_{L^2} \eea 

\noindent In the first line, the factor of $r^{2\nu}\sim (\e^{2/3}-\gamma^-)^{2\nu}$ in the numerator has been pulled out of the integral, while the factor in the denominator has been used to form the weighted norm. The second line substitutes the assumption of effective support on $g$. Relabeling $\gamma \mapsto \gamma' + \nu\gamma$ and combining power of $\e$ gives the bound on the first term of (\refeq{insidedecayest}). 

For the second term, observe that because $\chi\neq 1$ only where $r\sim \e^{2/3-\gamma^-}$, the same manipulation as above writing $1=r^{2\nu}r^{-2\nu}$ shows that

\bea
\|(1-\chi)g\|_{L^2} &\leq & C \e^{\nu(\tfrac{2}{3}-\gamma^-)}  \|(1-\chi)g\|_{r^{1+\nu}L^2} \leq  C \e^{\nu(\tfrac{2}{3}-\gamma^-)} \|\psi\|_{r^{1+\nu}H^1_e}, 
\eea

\noindent from which point the proof proceeds exactly as in the first term. 
\end{proof}

 \section{Universal Seiberg--Witten Equations}
 
 \label{section8}
 
  This section uses the concentrating local solutions defined in Section \ref{section7} to construct an infinite-dimensional family of model solutions parameterized by deformations of the singular sets $\mathcal Z_\tau$. This family is used to define a universal version of the Seiberg--Witten equations akin to the universal Dirac operator (\refeq{universalDiracdef}). 
\subsection{Hilbert Bundles}
\label{section8.1}

This subsection extends the construction of Hilbert vector bundles in Section \ref{admissiblefamiliessection} to Hilbert bundles for full Seiberg--Witten configurations. Recalling the notation briefly from Notation \ref{underlinenotation}, $\mathcal E_\tau \subseteq L^{2,2}(\mathcal Z_\tau; N\mathcal Z_\tau)$ denotes an open neighborhood of the origin on which the chart $\text{Exp}_\tau$ is defined. The admissible family $\underline {\mathbb F}_\tau$ for the tangential smoothing gauge in Definition \ref{underlineFdef} give rise to to diffeomorphisms $\underline F_\eta$, pullback metrics $\underline g_{\eta,\tau}$ and trivializations $\underline \Upsilon_{\mathbb F_\tau}$ of the bundles of spinors, defined in Notation \ref{underlinenotation}. 

We first extend the trivialization $\underline \Upsilon_{\mathbb F_\tau}$ defined on the spinor bundle $S_E$ to one defined on the bundle $S_E\oplus (\Omega^0\oplus \Omega^1)(i\R)$ hosting the Seiberg--Witten configurations.

\begin{defn}\label{SWtrivializationdef}
Define the {\bf trivializations of Seiberg--Witten configurations} (induced by the admissible family $\underline {\mathbb F}_\tau$ defining the tangential smooth gauge) as the map

 \bea \underline \Upsilon_\eta=\left(\underline \Upsilon_{\mathbb F}, \underline \Upsilon_\Omega \right):  \Gamma(Y;  S_E \oplus (\Omega^0\oplus \Omega^1)) &\to& \Gamma(Y;  S_E \oplus (\Omega^0\oplus \Omega^1)) \\ (\psi,b) &\mapsto& \left(v_\eta^{-1}\circ \psi \ , \  w_\eta^{-1}\circ b\right)\eea
 
 \noindent for fixed $\eta \in \mathcal E_\tau$. That is, it is the map induced on sections by $v_\eta$ as in Definition \ref{trivializationdef}, and $w_\eta$ given by the composition

\begin{center}
\tikzset{node distance=3.7cm, auto}
\begin{tikzpicture}
\node(C){$(\Omega^0\oplus \Omega^1) $};
\node(D)[left of=C]{$(\Omega^0\oplus \Omega^1)$};
\node(E)[left of=D]{$(\Omega^0\oplus \Omega^1)$};
\draw[<-][swap] (C) to node {$\frak T_{\Omega}$} (D);
\draw[<-][swap] (D) to node {$\underline F_\eta$} (E);
\end{tikzpicture}
\end{center}
\noindent where $\underline F_\eta$ is the pullback, and $\frak T_\Omega$ is parallel transport by the Levi-Civita connection on the metric cylinder defined in (\refeq{metriccylinder}). Note that this composition $w_\eta$ is simpler than $v_\eta$ because the bundle of $i\R$-valued forms is canonically associated for different metrics, unlike the spinor bundle. Lemma \ref{SWtrivializations} below extends Lemma \ref{trivializationcompatible} to show that these maps indeed define trivialization of Hilbert vector bundles, justifying the name.
 
\end{defn}

\medskip

We now define families of Hilbert spaces of configuration. These families give versions of the spaces of Definition \ref{insidespaces} and \ref{outsidespaces} shifted so that the weights are centered along the curve defined by a deformation $\xi$. For each $\xi \in \mathcal E_\tau$ and  $\mathcal Z_{\xi,\tau}=\text{Exp}_\tau(\xi)=\underline F_\xi(\mathcal Z_\tau)$,  define  
$$Y^+_{\e, \tau ,\xi}:= \underline F_\xi[N_\lambda(\mathcal Z_\tau)] \hspace{2cm} Y^-_{\e,\tau,\xi}:= Y- \underline F_\xi [N_{\lambda^-}(\mathcal Z_\tau)]$$

\noindent where $\lambda, \lambda^-$ are as in (\refeq{lambdadef}) and Definition \ref{outsidenorms} respectively. For each triple $(\e, \tau, \xi)$ there are weight functions defined analogously to those in Section \ref{section7.1}: 
\begin{eqnarray}
R_{\e,\tau, \xi}&:=& R_{\e} \circ \underline F^{-1}_\xi\label{Rxiweight} \\
\Phi_{\e,\tau,\xi}^{h_\e}&:=& \underline{\Upsilon}^{-1}_{\mathbb F_\tau} (\xi,\Phi^{h_\e}_{\e,\tau}) \label{Phixiweight}
\end{eqnarray}

\noindent where $R_{\e,\tau}$ is the weight in Definition \ref{insidenorms}. Likewise, there are $\xi$-parameterized families of mixed APS boundary and orthogonality conditions defined via the pushforward of those in Definitions \ref{insidespaces} and \ref{outsidespaces}.  More precisely, there are subspaces 
\bea H_{(\e,\tau,\xi)}^+&:=& \left\{(\ph, a) \in L^{1,2}(Y^+_{
\xi,\e,\tau}) \ \ \Big | \ \  \Pi_{\Lambda_{-1}^-}( \underline{\Upsilon}(\ph, a) |_{\del Y^+})=0  \ \ , \ \   \Pi_{V^+}(\underline\Upsilon(\ph,a))=0 \ \right\} \\ 
H_{(\e,\tau,\xi)}^-&:=& \Big\{(\ph^\text{re}, \ph^\text{im}, a)  \ \ \Big | \ \  \ph^\text{re}\in rH^1_e(\Yminus \mathcal Z_{\xi,\tau}) \ , \ (\ph^\text{im}, a) \in L^{1,2}(Y^-_{\e,\tau,\xi})  \ , \ \\ 
& & \hspace{2.5cm} \text{ and }\hspace{.5cm}  \Pi_{\Lambda_{0}^-}( \underline{\Upsilon}(\ph, a) |_{\del Y^-})=0 \Big\}\eea

\noindent where the subspaces $\Lambda_{-1}^+,V^+, \Lambda_0^-$ are as in Definition \ref{insidespaces} and Definition \ref{outsidespaces} respectively. 

\begin{defn}\label{hilbertspaces}
Let $\nu \in \R$ be a weight. For each $(\e,\tau) \in (0, \e_0) \times (-\tau_0, \tau_0)$ and each $\xi\in \mathcal E_\tau$, define Hilbert spaces
\bea
H^{1,\pm}_{\e,\nu}(Y^\pm_{\e,\tau,\xi})&:=& \Big\{  \ \  \ \ \ (\ph, a) \in H^\pm_{\e,\tau,\xi}(Y^\pm)  \ \ \ \big |  \ \ \ \ \ \ \  \|(\ph,a)\|_{H^{1,\pm}_{\e,\nu,\xi}}<\infty \ \   \ \Big\}  \\ 
L^{2,\pm}_{\e,\nu}(Y^\pm_{\e,\tau,\xi})&:=& \Big\{  \ \  \ \ \ (\ph, a) \in L^2(Y^\pm) \ \ \  \ \ \ \ \big |  \ \ \ \ \ \ \  \|(\ph,a)\|_{L^{2,\pm}_{\e,\nu,\xi}}<\infty \ \ \ \ \Big\}
\eea

\noindent where the norms for parameters $(\e,\tau,\xi)$ are defined analogously to Definitions \ref{insidenorms} and \ref{outsidespaces} using the domains indicated and the weights (\refeq{Rxiweight}--\refeq{Phixiweight}), the metric $g_\tau$, and the pullback of the connections $A^{h_\e}_{\e,\tau},B_\tau$ and $A_\tau,B_\tau$  in the two cases $\pm$ respectively (cf. Remark \ref{pullbackconnection} below).

For each $(\e,\tau)$, denote the $\mathcal E_\tau$-parameterized families by 
\begin{eqnarray}
\mathbb H^{1,\pm}_{\e,\nu}(\mathcal E_\tau)&:= &\Big\{(\mathcal Z_{\xi,\tau}, \ph,a)  \ \  \ \hspace{.4cm}  |  \ \ \xi \in \mathcal E_\tau \ , \ (\ph,a) \in H^{1,\pm}_{\e,\nu}(Y^\pm_{\e,\tau,\xi})  \ \   \Big\}\\
\mathbb L^{2,\pm}_{\e,\nu}(\mathcal E_\tau)&:= &\Big\{(\mathcal Z_{\xi,\tau},  \ph,a   )\ \ \ \hspace{.4cm}  |  \ \ \xi \in \mathcal E_\tau \ , \ (\ph,a) \in L^{2,\pm}_{\e,\nu}(Y^\pm_{\e,\tau,\xi})  \ \ \  \Big\}.
\end{eqnarray} 

\noindent Finally, define global spaces by 
\be  \mathbb H^1_{\e,\nu}:= \mathbb H^{1,+}_{\e,\nu}(\mathcal E_\tau)\oplus \mathbb H^{1,-}_{\e,\nu}(\mathcal E_\tau) \hspace{2cm} \mathbb L^2_{\e,\nu}:=\Big \{ (\mathcal Z_{\xi,\tau},\ph, a) \ | \  \xi \in \mathcal E_\tau  \ , \ \|(\ph,a)\|_{L^2_{\e,\nu}(Y)} < \infty \Big \} \label{combinedH1def}\ee

\noindent where the $L^2_{\e,\nu}(Y)$-norm is defined identically to the $L^{2,+}_{\e,\nu}(Y^+)$-norm in Definition \ref{insidenorms}, but integrated over all $Y$ using the global geodesic distance in place of the local radial coordinate $r$. 
\end{defn}

We have the following generalization of Lemma \ref{trivializationcompatible}. Recall from that lemma that $\mathcal E^{s}_\tau:=\mathcal E_\tau \cap L^{s,2}(\mathcal Z_\tau; N\mathcal Z_\tau)$ denotes the subspace of higher-regularity embeddings. 
\begin{lm} \label{SWtrivializations}For $s\geq 5$, the restriction of the induced trivializations $\Upsilon_{\mathbb F}$ to $\mathcal E_\tau^s$ \begin{eqnarray}
\underline \Upsilon:  \mathbb H^{1,\pm}_{\e,\nu}(\mathcal E^s_\tau)&\lre&\mathcal E_\tau \ \times \  H^{1,\pm}_{\e,\nu}(Y^\pm_{\e,\tau,0})\\ 
\underline \Upsilon:  \mathbb L^{2,\pm}_{\e,\nu}(\mathcal E^s_\tau)&\lre&\mathcal E_\tau \ \times \  L^{2,\pm}_{\e,\nu}(Y^\pm_{\e,\tau,0})
\end{eqnarray}
\noindent which thus endow the spaces on the left with the structure of smooth Hilbert vector bundles. The same applies to $\mathbb L^2_{\e,\nu}, \mathbb H^1_{\e,\nu}$. 

\end{lm}

\begin{proof} This lemma is almost a definition of the bundle structure --- the only thing to verify is that the pointwise maps $v_\eta, w_\eta$ provide bounded linear isomorphisms between $H^{1,\pm}_{\e,\nu}(Y^\pm_{\e,\tau,\xi})$ and $H^{1,\pm}_{\e,\nu}(Y^\pm_{\e,\tau,0})$ and likewise for the $L^2$-versions. Since these spaces are equivalent to $L^{1,2}, rH^1_e$ and $L^2$ for every fixed $(\tau,\e,\xi)$, this follows exactly as in Lemma \ref{trivializationcompatible} (in fact, this bounded linear isomorphism is an isometry except for the volume form due to our choice of the pullback weights in Eqs. (\ref{Rxiweight})--(\ref{Phixiweight})). 
\end{proof}

\subsection{Concentrating Local Families}
\label{section8.2}
This subsection defines the universal Seiberg--Witten equations as sections of the vector bundles from Section \ref{section8.1}. The equations are viewed as deformation equations around a universal family of model solutions constructed from Definition \ref{notationhn}. 

\medskip 

\begin{defn} \label{universalfamilydef} For $(\e,\tau) \in (0,\e_0)\times (-\tau_0, \tau_0)$ and $\xi \in \mathcal E_\tau$, define the following.
\begin{enumerate}

\item[(A)] The {\bf universal family of concentrating local solutions} by

\be \left(\frac{\Phi^{(1)}_{\e,\tau,\xi}}{\e}, A^{(1)}_{\e,\tau,\xi}\right):= \underline \Upsilon^{-1}\left(\xi, \left(\frac{\Phi^{(1)}_{\e,\tau}}{\e}, A^{(1)}_{\e,\tau}\right))\right)\label{modelfamily}\ee

\noindent i.e. as the pullback by $\underline \Upsilon$ of the constant section at the model solutions $(\Phi^{(1)}_{\e,\tau}, A^{(1)}_{\e,\tau})$ of Theorem \ref{PartImain}. 
\smallskip 

\item[(B)] The universal families of de-singularized configurations and $\Z_2$-eigenvectors by
\be \left(\Phi^{(0)}_{\e,\tau,\xi}, A^{(0)}_{\e,\tau,\xi}\right):=\underline \Upsilon^{-1}\left(\xi, (\Phi^{(0)}_{\e,\tau}, A^{(0)}_{\e,\tau})\right) \hspace{1.5cm}  \left(\Phi^{}_{\tau,\xi}, A^{}_{\tau,\xi}\right):=\underline \Upsilon^{-1}\left(\xi, (\Phi^{}_\tau, A^{}_\tau)\right) \label{desingularizedfamilies}\ee 
using the pre-glued configurations $(\Phi^{(0)}_{\e,\tau}, A^{(0)}_{\e,\tau})$ (defined in Eq. \refeq{approximate1}) and the $\Z_2$-eigenvectors $(\Phi_\tau, A_\tau)$ respectively.  
\end{enumerate} 
\end{defn}

\medskip 

\begin{rem}\label{pullbackconnection}. It is easy to verify, via the definition of $\underline \Upsilon$ in the proof of Lemma \ref{SWtrivializations} that a connection can be pulled back equally well. By construction, the connections in (\refeq{modelfamily}) have curvature that is highly concentrated around the deformed curves $\mathcal Z_{\xi,\tau}$. 

\end{rem}

For each triple $(\e,\tau,\xi)$, there is an accompanying deformation equation of the Seiberg--Witten equations centered at the approximate solution (\refeq{modelfamily}). To define these, we define cutoff functions $\chi^\pm_{\e,\tau,\xi}:=\chi^\pm \circ  \underline F_\xi$ as the pullbacks of the cutoffs $\chi^\pm$ used for the central curves $\xi=0$ defined in Appendix \ref{appendixofnotation}. To simplify notation, the latter two subscripts are omitted when they are clear in context. The deformation equations become:

\be \text{SW}_{\e,\tau,\xi}(\ph^+,a^+, \psi^-,b^-):=  \text{SW}_\tau\left( \left(\frac{\Phi^{(1)}_{ \e,\tau,\xi}}{\e}, A^{(1)}_{\e,\tau,\xi}\right) \  + \   \chi^{+}_\e\cdot (\ph^+,a^+)   \ +  \   \chi^{-}_\e\cdot (\psi^-,b^-)\right) \label{universalSWpre}\ee

\noindent where $(\ph^+,a^+)\in H^{1,+}_{\e,\tau,\xi}(Y^+)$ and  $(\psi^-,b^-)\in H^{1,-}_{\e,\tau,\xi}(Y^-)$ are inside and outside configurations in in the fibers over $\xi\in \mathcal E_\tau$ respectively. Here $\text{SW}_\tau$ denotes the (extended, gauge-fixed) standard Seiberg--Witten equations (i.e. Definition \ref{SWdef} with gauge-fixing), using the parameter $p_\tau=(g_\tau, B_\tau)$. Note that the argument of $\text{SW}_\tau$ on the right side in (\refeq{universalSWpre}) is a {\it globally} $L^{1,2}$ configuration on $Y$ once the local pieces from $Y^\pm$ are pasted in using the cutoff functions. It therefore makes sense to require that this pasted configuration is an $L^{1,2}$-solution of the equations. 

To precise about the gauge-fixing in (\refeq{universalSWpre}),  $\text{SW}_\tau$ includes the gauge-fixing condition 
\be -d^\star \alpha   - i\frac{ \br i \ph^\#,  {\Phi_{\e,\tau,\xi}^{(0)}} \kt}{\e}=0, \label{familiesgauge}\ee
\noindent where $$(\ph^\#,\alpha)=\chi^{+}_\e\cdot (\ph^+,a^+)    +     \chi^{-}_\e\cdot (\psi^-,b^-) + \chi^+(\ph^{(1)}_{\e,\tau,\xi},  a^{(1)}_{\e,\tau,\xi}).$$ \noindent  Here,  $\chi^+(\ph^{(1)}_{\e,\tau,\xi},  a^{(1)}_{\e,\tau,\xi}) = (\Phi^{(1)}_{\e,\tau,\xi}, A^{(1)}_{\e,\tau,\xi})- (\Phi^{(0)}_{\e,\tau,\xi}, A^{(0)}_{\e,\tau,\xi})$ are the analogues of the correction terms from Theorem \ref{PartImain} for each $\xi,\tau$. For $\xi=0$, this is the gauge-fixing condition coincides with the one used in the proof of Theorem \ref{PartImain}. Notice that for all subsequent corrections, this gauge condition remains constant, i.e. it is {\it linear}, rather than updating the spinor $\Phi^{(0)}_{\e,\tau,\xi}$ in the second term as the iteration proceeds (which would constitute another non-linear term). 

 We can combine the $\xi$-parameterized family of deformation equations into a single universal equation. In the following, $p:\mathbb H^{1}_{\e,\nu}\to  \mathcal E_\tau$ denote the projection of the universal vector bundle from (\refeq{combinedH1def}) in Definition \ref{hilbertspaces}.

 \begin{defn}\label{universalSWdef}
For each $(\e,\tau)\in (0, \e_0)\times  (-\tau_0,\tau_0)$, the {\bf universal Seiberg--Witten equations} (resp. eigenvector equation) are the sections 

\begin{center}
\tikzset{node distance=2.0cm, auto}
\begin{tikzpicture}
\node(C)[yshift=-1.7cm]{$\mathbb H^1_{\e,\nu}( {\mathcal E_\tau})$};
\node(D)[above of=C]{$p^*\mathbb L^2_{\e,\nu}({\mathcal E_\tau})$};
\node(D')[xshift=1.4cm, yshift=-.5cm]{$\mathbb {SW}, \overline{\mathbb {SW}}$};
\draw[<-] (C) to node {$$} (D);
\draw[->] (C) edge[bend right=45]node[below] {$$} (D);
\end{tikzpicture}
\end{center}
defined by
\begin{eqnarray}
\mathbb {SW}(\xi, \ph^+,a^+, \psi^-,b^-) \ \ &:=& \text{SW}_{\e,\tau,\xi}(\ph^+, a^+,\psi^-,b^-) \label{universalSWdef1}\\
\overline{\mathbb {SW}}(\xi, \ph^+,a^+, \psi^-,b^-,\mu)&:=& \text{SW}_{\e,\tau,\xi}(\ph^+, a^+,\psi^-,b^-) -\mu \chi^-_\e\frac{\Phi_{\tau,\xi}}{\e}
\label{universalSWdef2}
\end{eqnarray}
\noindent where the latter version upgrades the domain $\mathbb H^1_{\e,\nu}\oplus \underline \R \to \mathcal E_\tau$ to also include a trivial summand whose component is denoted by $\mu \in \underline \R$. The extra term (\refeq{universalSWdef2}) compared to (\refeq{universalSWdef1}) is a cutoff version of the universal family of eigenvectors (defined on the right in \refeq{desingularizedfamilies}).
 
\end{defn}

We remark that, analogously to Lemma \ref{boundedUniversalDirac}, on the vector bundles over the restricted domain $\mathcal E^s_\tau$ for some higher-regularity $s\geq 5$, $\mathbb{SW}, \overline{\mathbb {SW}}$ are smooth sections. This follows from the same considerations as the proof of that lemma, after generalizing the expressions of Theorem \ref{nonlinearBG} to the Seiberg--Witten case. The latter is done in the upcoming Section \ref{section9}.

\begin{rem}
The purpose of universal eigenvector equations is to deal with the 1-dimensional obstruction to the linearized equations coming from the span of $\Phi_\tau$ itself (recall the discussion preceding Theorem \ref{SpectralCrossing}). The cutoff term pairs non-trivially with this obstruction in a bounded way, while still allowing freedom to correct the solution in a small neighborhood of $\mathcal Z_\tau$. Ultimately, the value of $\mu$ for the solution is used to define $\tau$ implicitly as a function of $\e$, giving the 1-parameter family of true solutions. This approach to gluing with such an obstruction is due to T. Walpuski in \cite{WalpuskiG2Gluing}. 
\end{rem}

To solve the Seiberg--Witten equations on $Y$, it suffices to solve the universal version for any fixed smooth parameter $\xi$. 

 \begin{cor}\label{solvinguniversalSW}
 Suppose that $\xi \in \mathcal E_\tau\cap C^\infty(\mathcal Z_\tau;N\mathcal Z_\tau)$ and that $f\in C^\infty(Y)$. Then, $${\mathbb {SW}_\tau}(\xi, \ph^+,a^+, \psi^-,b^-)=f \ \ \ \ \ \Leftrightarrow \ \ \ \ \ \text{SW}_\tau (\Phi, A)=f$$ \noindent where 
$ (\Phi, A)$ is the configuration on the right side of (\refeq{universalSWpre}), and $\text{SW}_\tau$ is the Seiberg--Witten equation using the parameters $(g_\tau, B_\tau)$. In particular, if $f=0$, then $(\Phi, A)$ is smooth solution of the Seiberg--Witten equations.

The equivalent statement holds for the eigenvector equations $\overline{\mathbb{SW}}$. 

 \end{cor} \begin{proof} The first statement is simply the Definition \ref{universalSWdef} of the left hand side. By Definition \ref{hilbertspaces}, $(\ph^+,a^+) \in H^{1,+}_{\e,\nu,\xi}(Y+)$ and $(\psi^-,b^-)\in H^{1,-}_{\e,\nu,\xi}(Y^-)$ imply that $(\Phi, A)\in L^{1,2}(Y)$, because the weighted norms are strictly larger. Since the background parameter $(g_\tau, B_\tau)$ is smooth, elliptic bootstrapping shows that $(\Phi, A)$ is also smooth. 
 
In the case of the eigenvector equation, $\chi^-\Phi_\tau\in C^\infty$ because $\Phi_\tau$ is smooth away from the singular set $\mathcal Z_\tau$ where $\chi^-=0$. Since $\xi$ is smooth, $\chi^-\Phi_{\tau,\xi}$ is smooth as well and elliptic bootstrapping applies again.  
 \end{proof}

\medskip

\subsection{Universal Linearization}
\label{section8.3} This section calculates of the derivative of the universal Seiberg--Witten equations with respect to the deformation parameter $\xi$. This follows from applying a version of Bourguignon-Gauduchon's formula (Theorem \ref{nonlinearBG}) to both the two-spinor Dirac operator, and the Hodge-de Rham operator $\bold d$. 

We first update our notation from Definition \ref{notationhn}. Using the bundle notation of Sections \ref{section8.1}--\ref{section8.2}, we now denote 
\begin{eqnarray}
h_0&=&(\mathcal Z_\tau, (-\ph^{(1)}_{\e,\tau},-a^{(1)}_{\e,\tau}), (0,0))  \label{h0seconddef}\\ 
h_1&=&(\mathcal Z_\tau, (0,0), (0,0)) \label{h1seconddef} 
\end{eqnarray}

\noindent where the latter four arguments are the $\mathbb H^{1,\pm}_{\e,\tau,\xi}$ components in order. Notice that, by Definition \ref{universalSWpre}, the universal Seiberg--Witten equations at the central curve $\mathcal Z_\tau$ with vanishing fiber components corresponds exactly to the model solutions (\refeq{approximate1}). Thus this notation is consistent with what was originally called $h_1$ in Notation \ref{notationhn}, and likewise for $h_0$ subtracting the corrections in (\refeq{approximate1}).

We consider the derivative at  $h_1=(\mathcal Z_\tau, 0,0)\in \mathbb H^1_{\e,\nu}$. Provided it is indeed differentiable, and the regularity works as expected, the (fiberwise component of) the linearization would extend to a map

\be \text{d}{\mathbb {SW}}_{h_1}: L^{2,2}(\mathcal Z_\tau ; \mathcal Z_\tau) \ \oplus \  H^{1,+}_{\e,\nu}(Y^+) \  \oplus  \ H^{1,-}_{\e,\nu}(Y^-) \  \lre \  L^2_{\e,\nu}(Y),\label{universalderiv2} \ee

\smallskip

\noindent where we have used the canonical splitting of $T_{h_1}\mathbb H^1_{\e,\nu}$ along the zero-section to write the domain in its horizontal and fiberwise tangent spaces. The next proposition shows that $\mathbb{SW}$ is indeed differentiable at $h_1$, and calculates the derivative in terms of derivatives of various operators. The subsequent Proposition \ref{deformationUSW} provides concrete expressions for these terms.

For the statement of the proposition, recall that 
\be  \underline g_{\xi,\tau} := \underline F_\xi^*g_\tau \hspace{2.5cm}\underline B_{\xi,\tau}=\underline F_\xi^* B_\tau \label{underlinedpullbacks}\ee

\noindent denote the pullbacks of the metric and connection in the parameter $p_\tau=(g_\tau, B_\tau)$. In the subsequent Proposition \ref{deformationUSW}, we will also use the derivatives 
\be \dot {\underline g}_{\xi,\tau} = \d{}{s}\Big |_{s=0} \underline {F}_{s\xi}^* g_\tau \hspace{2cm} \dot {\underline B}_{\xi,\tau} = \d{}{s}\Big |_{s=0} \underline {F}_{s\xi}^* B_\tau.\ee
\noindent along 1-parameter rays. Finally, the propositions and their proofs use the combined parallel transport map 
\be \underline {\frak T}_{\xi}=(\frak T_{g_\tau}^{\underline g_{\xi,\tau}}, \frak T_\Omega), \label{combinedparallel}\ee

\noindent formed as the direct sum of the parallel transport maps defined in Eq. (\refeq{paralleltransport}) using the metric (\refeq{underlinedpullbacks}), and $\frak T_\Omega$ from Definition \ref{SWtrivializationdef}. 

\begin{prop}\label{SWderivabstract} For $s\geq 5$, the universal Seiberg--Witten equations (Definition \ref{universalSWdef}) are differentiable at $h_1=(\mathcal Z_\tau, 0,0)\in \mathbb H^1_{\e,\nu}(\mathcal E^s_\tau)$. Moreover, in the local trivializations of Lemma \ref{SWtrivializations}, the derivative (\refeq{universalderiv2}) is given by 

\be \text{d}{\mathbb {SW}}_{h_1} (\xi, \ph^+, a^+, \psi^-,b^-)= \underline{\mathfrak B}^\text{SW}_{h_1}(\xi)  \ + \  \mathcal L_{h_1}\left( \chi_\e^{+}(\ph^+,a^+) \ + \ \chi_\e^{-}(\psi^-,b^-)\right)\label{6.6tocalculate}\ee

\noindent where 

\begin{itemize}

 \item  $\underline{\mathfrak B}^\text{SW}_{h_1}$ is defined by
   \be  \underline{\mathfrak B}^\text{SW}_{h_1}(\xi) = \d{}{s}\Big |_{s=0}\left(( \underline{\frak T}_{\xi}(s))^{-1}\circ \text{SW}_{p_\xi(s)}  \circ \underline{\frak T}_{\xi}(s)\right) \left(\frac{\Phi^{(1)}_{\e,\tau}}{\e}, A^{(1)}_{\e,\tau}\right) \label{621} \ee
 \noindent where $\underline {\frak T}_\xi(s)$ is the parallel transport map (\refeq{combinedparallel}), and $\text{SW}_{p_{\xi(s)}}$ denotes the Seiberg-Witten equations with parameter $p_\xi(s)= (g_{s\xi}, B_{s\xi})$.   
\item $\mathcal L_{h_1}$ is the linearized Seiberg-Witten equations at $(\Phi^{(1)}_{\e,\tau}, A^{(1)}_{\e,\tau})$ as in Lemma \ref{linearizedequations}. 
\noindent

\end{itemize}
\label{abstractfirstvariation}
\end{prop}

\begin{proof} (See also \cite[Prop. 5.5]{PartII}) For the duration of the proof, we suppress the dependence on $\tau,\e,\nu$ from the notation. Choose a path 
 \bea \gamma: (-s_0, s_0) &\to& \mathbb H^{1}_{}(\mathcal E)
 \\  s &\mapsto &(\mathcal Z_{\zeta(s)}, \frak y (s))\eea
\noindent  such that $\gamma(0)=h_1$, where $\zeta(s)=s\xi + O(s^2)$.  We may denote the combined vertical components in the bundle by $\frak y(s)=\underline{ \Upsilon}_{\zeta(s)}^{-1}\frak q(s)$, where  $\frak q(s)=(\ph^+_s, a^+_s, \psi^-_s, b^-_s)  \in H^{1,+} \oplus H^{1,-}$. Using Definition \ref{universalSWdef} and (\refeq{universalSWpre}), and then substituting Definition \ref{modelfamily}, the derivative (\refeq{6.6tocalculate}) is then given by 
\begin{eqnarray}  \d{}{s}\Big |_{s=0} \underline\Upsilon_{\zeta(s)} \circ {\mathbb {SW} }(\mathcal Z_{\zeta(s)}, \frak p(s))&=& \d{}{s}\Big |_{s=0}   \underline\Upsilon_{\zeta(s)}\circ \text{SW}\left(\left(\frac{\Phi^{(1)}_{\zeta(s)}}{\e}, A^{(1)}_{\zeta(s)}\right)+ \chi^\pm \frak y(s)\right)\nonumber \\ 
&=&  \d{}{s}\Big |_{s=0}  \underline \Upsilon_{\zeta(s)}\circ \text{SW}\circ \underline \Upsilon_{\zeta(s)}^{-1}\left(\left(\frac{\Phi^{(1)}}{\e}, A^{(1)}\right)+ \chi^{+}(\ph^+_s,a^+_s)+ \chi^{-}(\psi^-_s, b^-_s)\right)\nonumber \\ & & \label{tocalculate6.6}
\end{eqnarray}
\noindent where $ \underline\Upsilon$ is used to denote the trivialization of both $\mathbb H^1$ and $\mathbb L^2$.

(\refeq{tocalculate6.6}) appears as the rightmost vertical arrow in the commuting diagram below. The diagram decomposes $\underline{\Upsilon}=\underline{ \frak T}_{\zeta} \circ \mathfrak S \circ \underline {F}_\zeta^*$ as in Definitions \ref{trivializationdef} and \ref{SWtrivializationdef}. 
It also abbreviates $H^1=H^{1,+}\oplus H^{1,-}$,  $\Omega =( \Omega^0\oplus \Omega^1)(i\R)$, and $p_\zeta=(g_{\zeta(s)}, B_{\zeta(s)})$ and uses $S_E^h$ to denote the spinor bundle formed using the metric $h$.

\begin{center}
\tikzset{node distance=5.8cm, auto}
\begin{tikzpicture}
\node(A){$H^{1}_{}( S_E^{g_\zeta}\oplus \Omega)$};
\node(B)[above of=A,yshift=-3.5cm]{$L^2( S_E^{g_\zeta}\oplus \Omega)$};
\node(C)[right of=A]{$H^{1}(S_E^{g}\oplus \Omega)$};
\node(D)[right of=B]{$L^2(S_E^{g}\oplus \Omega)$};
\node(C')[left of=A]{$H^{1}_{}(S_E^{g}\oplus \Omega)$};
\node(D')[above of=C',yshift=-3.5cm]{$L^2(S_E^{g}\oplus \Omega)$};

\node(E')[below of=A,yshift=4.8cm]{$\begin{pmatrix}\text{varying }\underline  g_\zeta\\ \text{fixed }\mathcal Z\end{pmatrix}$};
\node(F')[left of=E']{$\begin{pmatrix}\text{varying }\mathcal Z_{\zeta}\\ \text{fixed }g\end{pmatrix}$};
\node(F')[right of=E']{$\begin{pmatrix}\text{fixed } S^{g}\\ \text{fixed }\mathcal Z\end{pmatrix}$};
\draw[->] (A) to node {$\text{SW}_{p_\zeta}$} (B);
\draw[<-] (A) to node {$(\underline{\frak T}_{\zeta})^{-1}$} (C);
\draw[<-] [swap](A) to node {$\frak S_{\zeta}^{}\circ \underline { F}_\zeta^*$} (C');
\draw[->] (C') to node {$\mathbb{SW} $} (D');
\draw[->] (D') to node {$\frak S_{\zeta}^{}\circ \underline {F}_\zeta^*$} (B);
\draw[<-] (B) to node {$(\underline{\frak T}_{\zeta})^{-1}$} (D);
\draw[->] [swap](C) to node {$\Upsilon_\zeta \text{SW}_\tau \Upsilon_\zeta^{-1} $} (D);
\end{tikzpicture}\end{center}

Expressing (\refeq{tocalculate6.6}) as a composition using the vertical middle arrow in the diagram, and writing the Seiberg-Witten equations near a configuration $(\Phi,A)$ as \be \text{SW}\left((\Phi,A)+(\ph,a)\right)= \text{SW}(\Phi,A)+\mathcal L_{(\Phi,A)}(\ph,a)+ Q(\ph,a),\label{deformationeq}\ee
the derivative is given by
\begin{eqnarray}
&=& \d{}{s}\Big |_{s=0} \Upsilon_{\zeta(s)}\circ \text{SW}\circ \Upsilon_{\zeta(s)}^{-1}\left(\left(\frac{\Phi^{(1)}}{\e}, A^{(1)}\right)+ \chi^{+}(\ph^+_s, a^+_s) + \chi^{-}(\psi^-_s,b^-_s)\right)\nonumber  \\ &=&\d{}{s}\Big |_{s=0} (\underline{\frak T}_{\zeta}(s))^{-1}\circ \text{SW}_{p_\zeta(s)}\circ (\underline{\frak T}_{\zeta}(s))\left(\left(\frac{\Phi^{(1)}}{\e}, A^{(1)}\right)+ \chi^{+}(\ph^+_s, a^+_s) + \chi^{-}(\psi^-_s,b^-_s)\right) \nonumber \\ 
&=& \d{}{s}\Big |_{s=0} (\underline{\frak T}_{\zeta}(s))^{-1}\circ \text{SW}_{p_\zeta(s)}\circ (\underline{\frak T}_{\zeta}(s))\left(\frac{\Phi^{(1)}}{\e}, A^{(1)}\right) \label{617} \\ 
& &+ \  \d{}{s}\Big |_{s=0}   (\underline{\frak T}_{\zeta}(s))^{-1}\circ (\mathcal L^{p_\zeta(s)}_{h_1(s)}+Q^{p_{\zeta}(s)})\circ (\underline{\frak T}_{\zeta}(s))\left( \chi^{+}(\ph^+_s, a^+_s) + \chi^{-}(\psi^-_s,b^-_s)\right) \label{618},
\end{eqnarray}

\noindent where the last equality is an instance of (\refeq{deformationeq}). Here, $\mathcal L_{h_1(s)}^{p_{\zeta(s)}}$ is the linearization of the Seiberg--Witten equations at $h_1(s)=\underline {\frak T}_\zeta(h_1)$ using the parameter $p_{\zeta(s)}$. 
 
Since $\zeta(s)=s\xi + O(s^2)$, (\refeq{617}) is by definition $\underline{\mathfrak B}_{h_1}(\xi)$ after dropping $O(s^2)$ terms. Differentiating (\refeq{618}) using the product rule, all terms vanish except those differentiating $(\ph^+_s, a^+_s,\psi^-_s,b^-_s)$ since $\frak y(0)=0$.  Because $\underline{\frak T}_{\zeta}(0)=\text{Id}$ and $Q$ is quadratic, what remains is simply $\mathcal L_{h_1}^{p_0}(\dot \ph, \dot a, \dot \psi, \dot b)$, giving the second bullet point.  

Finally, the next proposition, together with the arguments of Lemma (\refeq{boundedUniversalDirac}) show that this linearization extends to a bounded map into $L^2$ for $\xi \in \mathcal E^s_\tau$ for $s\geq 5$. 
\end{proof}

\medskip 

The next proposition gives a concrete formula for the term $\frak B^\text{SW}_{h_1}(\xi)$ using Theorem \ref{nonlinearBG}. 

\begin{prop}\label{deformationUSW} The term $\underline{\frak B}^\text{SW}_{h_1}(\xi)$ in (\refeq{6.6tocalculate}) is given by

\be  \underline{\mathfrak B}^\text{SW}_{h_1}(\xi)=\begin{pmatrix}\underline{ \mathcal B}_{\Phi^{(1)}}(\xi) \\ \underline{\mathfrak b}_{A^{(1)}}(\xi)+ \underline  \mu_{\Phi^{(1)}}(\xi)\end{pmatrix}\label{SWdeformationcomponents}\ee

\noindent where, abbreviating $\underline g_\xi=\underline g_{\xi,\tau}$, and $\dot {\underline g}_\xi=\d{}{s}\big |_{s=0} \underline g_{s\xi}$,

\begin{enumerate}
\item[(A)] $\underline{\mathcal B}_{\Phi^{(1)}}(\xi)$ is the metric variation of the Dirac operator from Theorem \ref{nonlinearBG}, 
 \be \underline{ \mathcal B}_{\Phi^{(1)}}( \xi)  =\left(-\frac{1}{2}\sum_{ij} \dot g_{\xi}(e_i,e_j) e^i .\nabla_j + \frac{1}{2} d \text{Tr}_{g_\tau}(\dot g_{\xi}).  +\frac{1}{2} \text{div}_{g_\tau}(\dot g_{\xi}). + \mathcal R(B_\tau, \xi). \right)\frac{\Phi^{(1)}_{\e,\tau} }{\e}\label{BPhi1eta}\ee
 \noindent where $\Phi^{(1)}=\Phi^{(1)}_{\e,\tau}$ are the model solutions as in (\refeq{modelfamily}), $.$ denotes the Clifford multiplication of $g_\tau$, and $\nabla_j$ is the covariant derivative on $S_E$ formed using the spin connection of $g_\tau$, $B_\tau$, and the model solution connection  $A^{(1)}_{\e,\tau}$.  \medskip

 \item[(B)] $\underline{\mathfrak b}_{A^{(1)}}(\xi)$ is the metric variation of the de-Rham operator $\bold d$ given by  
\be  {\bf \underline {\mathfrak  {b}} }_{A^{(1)}}(\xi)=  \left(-\frac{1}{2}\sum_{ij} \dot g_{\xi}(e_i,e_j)  \bold{c}(e^i)  \nabla^{\text{LC}}_j  + \frac{1}{2} \bold{c}( d \text{Tr}_{g_\tau}(\dot g_{\xi})) +\frac{1}{2} \bold{c}(\text{div}_{g_\tau}(\dot g_{\xi}))\label{deformingboldd}\right) A^{(1)}_{\e,\tau} \ee
\noindent where  $A^{(1)}_{\e,\tau}$ is the connection form of the model solutions (\refeq{modelfamily}) in the trivialization of Lemma \ref{Fermitrivialization}, $\bold c$ is the symbol of $\bold d$, and $\nabla^\text{LC}$ is the Levi-Civitas connection of $g_\tau$. 

\item[(C)] $ \underline \mu_{\Phi^{(1)}}(\xi)$ is the metric variation of the moment map given by  
 
 \be \frac{1}{2}\underline \mu_{\Phi^{(1)}}(\xi)= \left(  0 \ , \  -\frac{1}{2}  \sum_{jk} \  \frac{\br i \dot {\underline g}_\xi(e_j, e_k)e^j. \Phi^{(1)}_{\e,\tau} \ ,  \  \Phi^{(1)}_{\e,\tau}\kt}{\e^2}i e^k\right)\label{underlinemu}\ee
 
\noindent where $.$ and $\Phi^{(1)}_{\e,\tau}$ are as in Item (1). Here $j,k=1,2,3$ are indices and $i=\sqrt{-1}$. The zero is the vanishing $S_E$-component for this term.  
\end{enumerate}
\end{prop}

\begin{proof} (1) The metric variation formula of Bourguignon-Gauduchon (Theorem \ref{nonlinearBG}) applies equally well to twisted Dirac operators, provided the connection on the twisting bundle remains fixed.  The $U(1)$-connection in (\refeq{621}) (i.e. in the middle arrow of the diagram in the proof of Proposition \ref{SWderivabstract}) is the fixed connection $A^{(1)}_{\e,\tau}$ by Definition \ref{modelfamily} (since for $\xi\neq 0$ these connection are defined as the push-forward of this). The connection on $E$ differs from the fixed connection $B_\tau$ by a zeroth order (in both $\xi$ and $\Phi^{(1)}_{\e,\tau}$) term
$$B_\tau - \dot { \underline{B}}_{\xi,\tau}:= \mathcal R(B_\tau,\xi).$$ 
\noindent In fact, a quick calculation shows $\mathcal R(B_\tau, \xi)=\iota_{\xi} F_{B_\tau} $ is the contraction of the curvature with $\d{}{s}\Big |_{s=0} \underline {F}_{s\xi}$.  

(2) Let $L$ be the complex line bundle in Definition \ref{trivializationdef}. Let $A_\circ$ denote a smooth connection on $L$ that extends the product connection in the trivialization of Lemma \ref{Fermitrivialization}. Then  

$$\star_\xi F_{A^{(1)}}=\star_\xi F_{A_\circ} + \star_\xi d (A_\circ - A^{(1)}_{\e,\tau})$$

\noindent where $\star_\xi$ is the Hodge star of $\underline g_{\xi}$. Since $\underline g_{\xi}=g_\tau$ outside the neighborhood $N_{r_0}(\mathcal Z_\tau)$ because $\underline {F}_\xi=\text{Id}$ in that region, the derivative of the first term $\star_\xi F_{A_\circ}$ is zero in this region. On the other hand, $F_{A_\circ}=0$ inside $N_{r_0}(\mathcal Z_\tau)$, so this term vanishes there as well.  Consequently, the variation of the curvature (when supplemented with the $\Omega^0(i\R)$ component and gauge-fixing) reduces to the metric variation of $\bold d$. The variation of this Dirac-type operator follows equally well from (Theorem \ref{nonlinearBG}), with the additional simplification that the form bundle does not depend on the metric. 

(3) Let $\frak a_\xi$ be as in Theorem \ref{nonlinearBG} such that $\underline g_\xi=g_\tau( \frak a_\xi X,Y)$, then $\widetilde e_j=\frak a_\xi ^{-1/2}e_j$ is an orthonormal frame for $\underline g_{\xi}$ where $\{e_j\}$ is one of $g_\tau$. Expanding  the square root in Taylor series in the frame of $g_\tau$ and differentiating yields $\dot a_\xi(s)=-\tfrac{1}{2}\dot {\underline g}_{\xi,\tau}$, just as in the symbol term of (\refeq{nonlinearBG}). The variation formula then follows from the definition of $\mu$ in Eq. (\refeq{momentmap}). 
\end{proof}

\subsection{Non-Linear Terms}
\label{section8.4}

This section characterizes the non-linear terms in the universal Seiberg-Witten equations. The equations have quadratic non-linearities in fiber directions of $H^{1}_{\e,\nu}$---these simply being the non-linearities of the original Seiberg-Witten equations, but are quasi-linear in the deformation parameter $\xi$. There are also mixed terms quasi-linear in $\xi$ and linear or quadratic in the fiber directions. 

The universal Seiberg-Witten equations at a configuration $h\in  \mathbb H^1_\e \oplus \underline \R$ may be written 
\begin{eqnarray}
\mathbb {SW}\left( h \right)&=&\text{SW}\left(\frac{\Phi^{(1)}_{\e,\tau}}{\e}, A^{(1)}_{\e,\tau}\right)  \ + \  \text{d}{ \mathbb{SW}}_{h_1} (h)     \ +  \ \mathbb Q_{h_1}(h)\label{boldSWnonlinear}
\end{eqnarray}

\noindent where $\mathbb Q_{h_1}$ consists of the non-linear terms. The proposition uses the shorthand $h=(\xi,\ph,a)$, where $(\ph,a)=\chi^+(\ph^+,a^+)+\chi^-(\psi^-,b^-)$
for $(\ph^+, a^+,\psi,b)\in H^{1,+}_\e \oplus H^{1,-}_\e$ to simplify notation.

\begin{prop}
\label{nonlinearterms}
The non-linear term has the form 

\be
\mathbb Q_{h_1}(h) = Q_{\text{SW}}(\ph,a)  \ + \ Q_{\Phi^{(1)}} (\xi,\ph)  \ + \  Q_{A^{(1)}}(\xi, a)\ + \ Q_{a.\ph}(\xi, \ph,a) \ + \  Q_\mu(\xi, \ph,\ph) \label{nonlinearparade}
\ee

\noindent where 
\begin{enumerate}
\item $Q_{\text{SW}}$ is the standard non-linearity of the Seiberg--Witten equations given by $$Q_\text{SW}(a,\ph):=(a.\ph, \mu(\ph,\ph))$$ using Clifford multiplication with respect to $g_\tau$.  
\medskip 

\item $Q_{\Phi^{(1)}}$ is the non-linear portion of the metric variation of the Dirac operator $\slashed D_{A^{(1)}}$ on $S_E$ as in Lemma \ref{universalDnonlinear}. That is, 
\be Q_{\Phi^{(1)}}(\xi,\ph)=\underline{ \mathcal B}_\ph(\xi) \ \ + \ \ \mathfrak m_\Phi(\xi, \xi)   \ \ + \ \  \frak m_\ph (\xi,\xi)  \ \ + \ \ F_{\Phi + \ph}(\xi)  \ee

 \noindent taking $\Phi=\e^{-1}\Phi^{(1)}_{\e,\tau}$ and $\ph=\psi$ in the notation of that lemma.
 \medskip 
 
\item $Q_{A^{(1)}}$ is the non-linear portion of the metric variation of the de-Rham operator $\bold d$ exactly as in Item (2) above, but now taking $\Phi=A^{(1)}$ and $\psi=a$ and appropriately substituting ${\bf {cl}}$ for Clifford multiplication. 

\medskip 

\item $Q_{a.\ph}$ is the non-linearity arising from the metric variation on the first component of $Q_{\text{SW}}$, which has the form $$Q_{a.\ph}=F(\underline \chi'[\xi]) \cdot M_{a.\ph}(\xi)   $$
\noindent where $M_{a.\ph}(\xi)$ is a term of Type B in the sense of \eqref{typeB} taking $\psi=a.\ph$ with weight $w_B=1$, and $F$ is a $C^\infty(Y)$-linear combination of  $1,\underline{\chi}[\xi'],\underline{\del_a\chi}[\xi],$ and terms vanishing at least quadratically in these these. 

\medskip 

\item $Q_{\mu}$ is the non-linearity arising from the metric variation of the moment map. Schematically, it has the form 

$$-\frac{1}{2}\frac{\br  \dot  {\underline g}_\xi(e^j, e^k)ie^j (\Phi^{(1)}_{\e,\tau}+ \e \ph), \e\ph \kt}{\e^2} i e^k + \frac{\br q_{jk}(\xi)ie^j (\Phi^{(1)}_{\e,\tau}+ \e \ph), \Phi^{(1)}_{\e,\tau}+\e\ph \kt}{\e^2} i e^k  $$

\noindent where $q_{jk}$ vanishes at least quadratically in $\underline{\chi}[\xi'], \underline{\del_a\chi}[\xi]$.

\end{enumerate}
\end{prop}

\begin{proof}
(1) is immediate. (2)--(3) follow precisely as in Lemma \ref{universalDnonlinear} with the appropriate substitutions. (4)--(5) follow from using the expression for the pullback metric  (\refeq{pullbackmetric}) derived in the proof of Lemma \ref{universalDnonlinear}, and noting that Clifford multiplication and the $\mu$ involve only algebraic combinations of the components of the pullback metric used to form the tensors $\frak a, a$ preceding Theorem \ref{nonlinearBG}.

\end{proof}

 \section{Relating Deformation Operators}
 \label{section9}
 Section \ref{section6}, specifically Theorem \ref{deformationsmaina}, showed that the linearized operator $\text{d}\slashed{\mathbb D}$ including deformations of the singular set can sufficiently cancel the obstructions in the case of $\Z_2$-harmonic spinors (i.e. in the $\e=0$ setting). The goal of this section is to show that the same holds for the linearized operator $\text{d}{\mathbb{SW}}$ in the Seiberg--Witten setting. Specifically it is shown that, on the ``outside'' region $Y^-$, $\text{d}\mathbb{SW}$ (as given by Propositions \ref{SWderivabstract} and \ref{deformationUSW}) is a ``small'' perturbation of $\text{d}\slashed{\mathbb D}$ (defined in Theorem \ref{nonlinearBG} and Eq. \refeq{blockdecomp}), where ``small'' is given a more precise meaning in Definition  \ref{gluingpermissible} below.

 First, there is the obvious mismatch in these operators coming from scaling: the spinor in $\text{d}\mathbb{SW}$ is scaled by $\e^{-1}$, while the spinor in $\text{d}\slashed{\mathbb D}$ is the {\it normalized} eigenvector $\Phi_\tau$. To account for this, we introduce the following inverse normalization for the deformation. For $\eta \in L^{2,2}(\mathcal Z_\tau; N\mathcal Z_\tau)$, define $\xi$ by 
\be \boxed{\xi(t):=\e \eta(t).} \label{etaxi}\ee  
 
\noindent  With this normalization, we split the operator into the two regions $Y^\pm$ using the indicator functions $\bold 1^\pm$ as defined in the Appendix of Gluing Parameters in Section \ref{section2}. Thus, for $h_1=(\mathcal Z_\tau, 0,0)$ as in (\refeq{h1seconddef}) (cf. Notation \ref{notationhn}),  we write 
 
 \begin{eqnarray} \text{d}\mathbb {SW}_{h_1}(\xi , 0,0) &=&  \text{d}\mathbb {SW}_{h_1}(\xi , 0,0) \bold 1^+   \  + \   \text{d}\mathbb {SW}_{h_1}(\xi , 0,0) \bold 1^- \label{Xidef} \\
 &=: &  \left[ \text{d}\slashed{\mathbb D}_{h_\circ} (\eta, 0,0)\bold 1^+ \ + \  \Xi^+(\eta) \right] \ + \ \left[ \text{d} \slashed{\mathbb D}_{h_\circ} (\eta, 0,0)\bold 1^- \ + \  \Xi^-(\eta) \right], \label{Xipmdef}\end{eqnarray}
 
 \noindent where the second line is taken to be the definition of $\Xi^\pm$, where these are supported in the same regions as $\bold 1^\pm$. Here, $\text{d}_{h_\circ} \slashed{\mathbb D}$  denotes the linearization at $h_\circ=(\mathcal Z_\tau, A_\tau, \Phi_\tau)$. Note this is not a an element of the bundle $\mathbb H^1_{\e,\nu}$, but the configuration still makes sense as a non-smooth section of $S_E\to \Yminus \mathcal Z_\tau$ for a given $\eta$ (it {\it is} smooth on the support of $\bold 1^-$, just not at $\mathcal Z_\tau$). Observe that in progressing from (\refeq{Xidef}) to (\refeq{Xipmdef}) we have switched normalizations to replace $\xi$ by $\eta$ as defined above. 
 
  The goal of the present section, more specifically, is to provide precise bounds on $\Xi^\pm$ and their non-linear analogues. 

 \subsection{Two Deformation Operators}
 \label{section9.1}
 The operators $\Xi^\pm$ consist of a parade of terms formed from various functions of $\underline \chi[\eta']$ and its derivatives, given by Proposition \ref{deformationUSW}. Each such term will be bounded in terms of the $L^{s,2}$ norm of $\eta$ for some $s$, and certain powers of $\e^\beta$. We will ultimately show that each of these terms, for the various values of $s,\beta$ that occur obeys the following criteria. 
 
 \begin{defn} \label{gluingpermissible} An $\e$-parameterized family of linear operators $ M_\e: L^{s,2}(\mathcal Z_\tau; N\mathcal Z_\tau) \to L^2(N_{r_0}(\mathcal Z_0) ; S_E \oplus (\Omega^0\oplus \Omega^1))$ is said to be $\bold L_0$-{\bf permissible} for $L_0\in \N$ if it obeys a bound 
 
 $$\|M_\e(\eta)\|_{L^2(Y)} \leq C \e^{\beta} \|\eta\|_{s} $$
 
 \noindent with $\beta \geq \left(s-\tfrac12\right) \log(L_0)/  \log(\e^{-1}).$ 
 
 The family $M_\e$ is said to be {\bf gluing permissible} if it is $L_0$-permissible for $L_0=\e^{-1/2-2\gamma^+}$, and $\e^{1/3}M_\e$ is $L_0$-permissible for $L_0=\e^{-2/3}$.   
 
 \end{defn}
 
 In the above and what follows, the $L^{s,2}$ norm of $\eta$ is denoted simply by $\|\eta\|_s$. A quick unraveling of the definition shows that if $\eta$ is supported in Fourier modes with $|\ell|\leq L_0$, then being $L_0$ permissible means that 
 
 $$ C \e^{\beta} \|\eta\|_{L^{s,2}(\mathcal Z_\tau)} \leq C\e^{\beta} (L_0)^{s-1/2}\|\eta\|_{1/2} \leq C \e^{\beta -\left(s-\tfrac12\right) \log(L_0)/ \log(\e^{-1}) }\|\eta\|_{1/2}\leq C \|\eta\|_{1/2},$$
 
 \noindent uniformly in $\e$. In particular, if $L_0=\e^{-\alpha}$, then the constraint is $\beta \geq (s-\tfrac12)\alpha$. The criterion for being gluing permissible are chosen in hindsight, precisely because it is  ultimately this condition that is needed for the alternating iteration to converge.

Because of the exponential decay in Lemma \ref{exponentialdecay}, the (re-normalized) approximate solution $(\Phi_1, A_1)$ is a small perturbation of the original $\Z_2$-eigenvector $(\Phi_\tau, A_\tau)$ outside the invariant scale of $r=\e^{2/3}$. This leads to the following constraint on $\Xi^-$. 
\begin{lm}\label{relatingvariationsoutside}
Let  $\eta(t)=\e^{-1}\xi(t) \in C^\infty(\mathcal Z_\tau; N\mathcal Z_\tau)$ be a linearized deformation. There exists constants $C$ and $\gamma<<1$ such that
\be \|\Xi^-(\eta)\|_{L^2_{\e,\nu}}\leq C \e^{5/6-\gamma}  \|M^-_\e(\eta)\|_{L^2(Y)}  \label{correctiondeformation}\ee
\noindent where $M_\e^-(\eta)$  is a gluing-permissible term. More specifically (and slightly stronger), one has

\be \|\Xi^-(\eta)\|_{L^2_{\e,\nu}}\leq C \e^{11/12-\gamma}\|\eta\|_{3/2 + \underline\gamma-\nu} \label{correctiondeformation2}\ee
\noindent with $\underline \gamma$ as in Lemma \ref{deformationbounds}.    
\end{lm}

\begin{proof} By Theorem \ref{nonlinearBG} and Eq. (\ref{calBdef}), $\text{d}\slashed{\mathbb D}_{h_\circ}(\eta, 0)= \underline{ \mathcal B}_{\Phi_\tau}(\eta)$ and $\tfrac{1}{\e}\underline{\mathcal B}_{\Phi^{(1)}}(\xi)=\underline{\mathcal B}_{\Phi^{(1)}}(\eta)$. By Proposition \ref{deformationUSW}, we can therefore write $$ \Xi(\eta)=\text{d}{\mathbb {SW}}_{h_1}(\xi,0,0)-\text{d}\slashed{\mathbb D}_{h_\circ}(\eta, 0)=\begin{pmatrix}\underline {\mathcal B}_{\Phi^{(1)}}(\eta)-\underline {\mathcal B}_{\Phi_{\tau}}(\eta)  \\  \underline {\mathfrak b}_{A^{(1)}}(\eta) \end{pmatrix} \ + \  \underline\mu_{\Phi^{(1)}}(\eta).$$  
We first compare the spinor components. Recall $\ph^{(1)}_{\e,\tau}$ denotes the difference 
\be \frac{\Phi^{(1)}_{\e,\tau}}{\e}= \frac{\Phi_\tau}{\e} + \ph^{(1)}_{\e,\tau}, \label{perturbationbound} \ee

\noindent  which satisfies $\|\ph^{(1)}_{\e,\tau}\|_{H^{1,+}_{\e}}\leq \e^{-1/12-\gamma}$ on the support of $\bold 1^-$ by Theorem \ref{PartImain}(B) and the exponential bound on $\Phi^{(0)}_{\e,\tau}-\Phi_\tau$ from Lemma \ref{fiducialbounds}(B) (cf. the definition (\refeq{defpreglued})).

Investigating the first component, 
\begin{eqnarray}
 \mathcal B_{\Phi^{(1)}}(\xi) - \mathcal B_{\Phi_\tau}(\eta) & =& \e \mathcal B_{\ph^{(1)}}(\eta)\nonumber \\
&=& \e  \left(-\frac{1}{2}\sum_{ij} \dot g_{\eta}(e_i,e_j) e^i .\nabla_j + \frac{1}{2} d \text{Tr}_{g_\tau}(\dot g_{\eta}).  +\frac{1}{2} \text{div}_{g_\tau}(\dot g_{\eta}). + \mathcal R(B_\tau, \eta). \right)\ph^{(1)}_{\e,\tau}\label{Bphieta}
\end{eqnarray}

\noindent Each term in (\refeq{Bphieta}) is of Type B in the sense of Lemma \ref{deformationbounds}, each with weight $w_B \leq 2$. Applying Item (B) of that lemma with $\beta=\nu$ shows that 
\begin{eqnarray} \| \e \mathcal B_{\ph^{(1)}}(\eta )\bold 1^- \|_{L^{2,-}_{\e,\nu}}  &\leq&  C\e \|\eta\|_{3/2+\underline\gamma-\nu} \Big( \|\ph_1\|_{rH^1_e(Y^-)}\Big)  \\ &\leq &C\e^{11/12-\gamma} \|\eta\|_{3/2+\underline \gamma -\nu} \label{epBetabound}. \end{eqnarray}

\noindent as desired. To be completely precise, we note that the {\it proof} of that lemma shows that when considering the left-hand side only on the support of $\bold 1^-$, the spinor's norm on the right is only needed in the same region, since the estimate is local. In the final inequality, we have used that the $rH^1_e$ and $H^{1,+}_{\e}$ norms of section supported where $r\geq \e^{2/3}$ are comparable (uniformly). This establishes the desired bound for the spinor components.

For the $(\Omega^0\oplus \Omega^1)$ and $\mu_{\Phi^{(1)}}$ components, note that that the difference from the (flat) limiting connection $A_\tau$ and from $\Phi^{(1)}$ are exponentially small on $\text{supp}(\bold 1^-)$ i.e. 
 $$|A^{(1)}_{\e,\tau} - A_\tau|  \ + \ |\Phi^{(1)}_{\e,\tau}-\Phi_\tau| \leq C\e^{-3}\text{Exp}(-\tfrac{1}{\e^{\gamma}})$$

\noindent for $\gamma<<1$ by Corollary  \ref{corollaryexpdecay}.  Since $\underline{\mathfrak b}_{A_\tau}$ depends only on the $\Omega^0\oplus\Omega^1$-component,  this component satisfies (\refeq{correctiondeformation}) with an exponential factor in place of $\e^{11/12}$, which may be absorbed once $\e$ is sufficiently small.  For the $\mu$-term, note that Item(C) of Proposition \ref{deformationUSW} and the fact that $\mu$ is an off-diagonal pairing between $S^\text{Re}, S^\text{Im}$ by Item (2) of Lemma  \ref{subbundles} means that the $\mu$-term is bounded by the the product of the real and imaginary terms of $\Phi^{(1)}_{\e,\tau}$, the latter of which is exponentially small by Corollary \ref{corollaryexpdecay}, and Item (B) of Lemma \ref{fiducialbounds}, thus the same applies for this term. 
\end{proof}

\medskip 

The next lemma gives an analogous bound for the inside term $\Xi^+$. The situation on $Y^+$  stands in contrast to that of Lemma \ref{relatingvariationsoutside}: here, the two deformation operators bear no meaningful relation. However, since the inside region shrinks as $\e\to 0$, the norm of $\Xi^+$ shrinks sufficiently rapidly to be gluing permissible. Since the weight function $R_\e$ is almost constant (up to a factor of $\e^{-\gamma}$) on $\text{supp}(\bold 1^+)$, the lemma considers only the unweighted norms.

The proof utilizes the following re-scaling, which plays an essential role in the proof of Theorem \ref{PartImain} (see \cite[Sec. 5.3]{PartI}). There is a re-scaled coordinate $\rho=\kappa(t)\e^{2./3}$ where $\kappa(t)$ is the smooth function (\refeq{kappatau}), such that de-singularized solutions $(\Phi^{h_\e}_{\e,\tau}, A^{h_\e}_{\e,\tau})$ as in Definnition  \ref{desingularized} are given by 
\be (\Phi^{h_\e}_{\e,\tau}, A^{h_\e}_{\e,\tau})= (\e^{1/3}\Phi^H_\tau(\rho), A^H_\tau(\rho))\label{scaleinvariantdesingularized}\ee

\noindent where $\Phi^H_\tau, A^H_\tau$ are  smooth, $\e$-independent functions on $\mathcal Z_\tau \times \R^2$ in Fermi coordinates (\refeq{Fermicoords}). Moreover, $\Phi^H\sim \rho^{1/2}$ for $\rho>>1$, and $A^H=f(\rho)\left(\tfrac{dz}{z}-\tfrac{d\overline z}{\overline z}\right)$, where $f(\rho)$ vanishes to second order at $\rho=0$. See Section 4 of \cite{PartI} for detailed proofs.

 \begin{lm}\label{relatingvariationsinside}
Let  $\eta(t)=\e^{-1}\xi(t) \in C^\infty(\mathcal Z_\tau; N\mathcal Z_\tau)$ be a linearized deformation. There exist constants $C$ and $\gamma<<1$ such that
\begin{eqnarray} 
\|\Xi^+(\eta)\|_{L^2} &\leq &  C\e^{1/12-\gamma}\| M^+(\eta)\|_{L^2}
\end{eqnarray}
\noindent where $M_\e^+(\eta)$  is a gluing-permissible term. More specifically, one has 

\begin{eqnarray} 
\|\Xi^+(\eta)\|_{L^2} &\leq &  C \e^{-\gamma} \Big( \e^{1/3} \|\eta\|_{1,2} \ + \   \e^{11/12} \|\eta\|_{3/2+\underline \gamma ,2} \ + \  \ \e \|\eta\|_{2,2}  \ + \ \e^{19/12} \|\eta\|_{5/2 + \underline\gamma,2} \label{9.10} \Big).
\end{eqnarray}

\noindent  with $\underline \gamma$ as in Lemma \ref{deformationbounds}.    
\end{lm}

\begin{proof} 

By Proposition \ref{deformationUSW} and the triangle inequality, one has 

\be  \|\Xi^+(\eta)\|_{L^2}\leq  \|\text{d}\mathbb {SW}_{h_1}(\xi) \bold 1^+ \|_{L^2} \ + \ \|\text{d}\slashed {\mathbb {D}}_{h_\circ}(\xi) \bold 1^+ \|_{L^2} \label{dSWdD}\ee

\noindent and

\bea
\|\text{d}\mathbb {SW}_{h_1}(\xi) \bold 1^+ \|_{L^2} \leq \|\underbrace{ \underline{ \mathcal B}_{\Phi_{1}}(\eta)}_{\text{(I)}} \|_{L^2} \ + \  \e\|\underbrace{  \underline{\mathfrak b}_{A_{1}}(\eta)}_{\text{(II)}} \|_{L^2} \ + \  \e\|\underbrace{  \underline \mu_{\Phi_{1}}(\eta)}_{\text{(III)}} \|_{L^2}.
\eea

\noindent We begin with $\text{d}\mathbb {SW}_{h_1}(\xi)$, and proceed by bounding each term (I)--(III) separately. 

Beginning with (I), write  $(\Phi^{(1)}_{\e,\tau}, A^{(1)}_{\e,\tau})= (\Phi^{h_\e}_{\e,\tau}, A^{h_\e}_{\e,\tau}) + ( \ph^{(1)}_{\e,\tau}, a^{(1)}_{\e,\tau}) + O(\text{Exp}(c\e^{-\gamma}))$ where $( \ph^{(1)}_{\e,\tau}, a^{(1)}_{\e,\tau})$ are as in Item (B) of Theorem \ref{PartImain}. The exponentially small term arises from the difference between $\Phi^{h_\e}$ and $\Phi^{(0)}_{\e,\tau}$ as in Definition \ref{defpreglued}.   The term (I) is comprised of four subterms (Ia)--(Id) as in (\refeq{BPhi1eta}); for each of these there is a leading order part coming from $(\Phi^{h_\e}, A^{h_\e})$ and a perturbation coming from $(\ph^{(1)}, a^{(1)})$, both of which are denoted without the subscripts for the remainder of the proof. We begin with the leading order part of (Ia). Omitting indices and subscripts and writing $N^+=\text{supp}(\bold 1^+)$ for clarity, 
\bea
\|\dot {\underline g}_\eta . \nabla \Phi^{h_\e}\|^2_{L^2} &\leq &C \int_{N^+} \left(| \underline \chi(\eta')|^2 + | \underline {d\chi}(\eta)|^2  \right) |\nabla \Phi^{h_\e}|^2 r dr d\theta dt \\
&\leq&C \int_{N^+}  \left(| \underline \chi(\eta')|^2 + | \underline {d\chi}(\eta)|^2  \right) |\nabla_{\rho} \Phi^{h_\e}|^2 \rho d\rho d\theta dt\\
 &\leq& C \e^{2/3} \int_{\mathcal Z_\tau}  \left(| \underline \chi(\eta')|^2 + | \underline {d\chi}(\eta)|^2  \right)  \int_{\rho \leq \e^{-\gamma}}|\nabla_{\rho} \Phi^{H}|^2 \rho d\rho d\theta dt \\ &  \leq&  C\e^{2/3-\gamma} \|\eta\|_{1,2}^2 
 \eea  
 
 \noindent where we have changed variables to the rescaled coordinate $\rho$ (in both the volume and $\nabla$) and then substituted (\refeq{scaleinvariantdesingularized}). The last inequality follows from the fact that $\|\underline{d\chi}(\eta)\|_{L^2(S^1)} \leq C \|\eta'\|_{L^2(S^1)}$ (as in the proof of Lemma \ref{nonlineardeformationbounds}), and the fact that $\Phi^H\sim \rho^{1/2}$ for $\rho \geq \text{const}$. The same argument applies to the terms (Ib)--(Id), except there is an additional factor of $\e^{2/3}$ because there is no derivative to rescale, but the norm $\|\eta\|_{L^{2,2}}$ is needed. These give rise to the third term in (\refeq{9.10}). 
 
 Turning now to the leading order of term (II), which is again comprised of three sub-terms (IIa)--(IIc) as in (\ref{deformingboldd}), a similar rescaling argument applies to show 
  \bea
 \e^2 \|\dot {\underline g}_\eta  \nabla A^{h_\e}\|_{L^2}^2 &\leq & \e^2 \int_{N^+}  \left(| \underline \chi(\eta')|^2 + | \underline {d\chi}(\eta)|^2  \right) |\nabla_\rho f(\rho)\left(\tfrac{dz}{z}-\tfrac{d\overline z}{\overline z}\right)|^2 \rho d\rho d\theta dt\\
 &=& \e^2 \e^{-4/3} \int_{N^+}  \left(| \underline \chi(\eta')|^2 + | \underline {d\chi}(\eta)|^2  \right) |G(\rho)|^2 \rho d\rho d\theta dt \leq \e^{2/3}\|\eta\|^2_{1,2}
 \eea
 \noindent where we have substituted $\tfrac{1}{z}=\tfrac{\e^{2/3}}{\rho e^{i\theta}}$ (and likewise for $\overline z$), then used the fact that $f(\rho)$ vanishes to second order at $\rho=0$ to combine these into a smooth bounded function $G(\rho)$. The last inequality follows just as in term (I). Terms (IIb)--(IIc) proceed just as for (Ib)--(Id), making the same alterations. Thus term (II) is bounded by the first and third terms on the right in (\refeq{9.10}). 
 
 For the perturbation terms of  (I)--(II) arising from $(\ph^{(1)}, a^{(1)})$, note that for a perturbation by $\xi=\e \eta$ these perturbation terms pick up an extra factor of $\e$, precisely as in (\refeq{Bphieta}). These terms are then bounded by naively pulling out $C^0$ bound and using the $H^{1,+}$ norm to integrate. For instance, for terms (Ia) and (IIa), a factor of $\|\eta\|_{3/2+\underline \gamma,2}$ can be pulled out, reducing the integral to the $H^{1,+}_{\e}$-norm; for the other terms, the same applies to (Ib--Id) and (IIb)--(IIc) with $\|\eta\|_{5/2+\underline \gamma,2}$ (and the $R_\e$ weight in the norm gives an extra factor of $\e^{2/3-\gamma}$ for these).
 
 Finally, we turn to the term (III) arising from $\underline \mu$ as in (\refeq{underlinemu}). This term consists of terms 
 
 \be
\text{ (III)}= \e \cdot   \Big\br ig_\eta. \frac{\Phi^{h_\e}}{\e},\frac{\Phi^{h_\e}}{\e} \Big \kt \ + \ \e \cdot   \Big\br ig_\eta. \frac{\Phi^{h_\e}}{\e}, \ph^{(1)}\Big \kt + \e \cdot   \Big\br ig_\eta.  \ph^{(1)}, \ph^{(1)}\Big \kt \ + \ O(\text{Exp}(c\e^{-\gamma}))
 \ee
 
 \noindent where the exponentially small terms come from the difference from $|\Phi_\tau - \Phi^{(0)}|$. For the first term, pull out $\e \|\eta\|_{3/2+\underline\gamma}$ to leave the $L^2$ norm of $\e^{-1}\Phi^{h_\e}$, which the rescaling (\refeq{scaleinvariantdesingularized}) shows is bounded by $\e^{-\gamma}$. For the second, pull out the same factor, and then note that $\e^{-1}\Phi^{h_\e} \sim \e^{-2/3} \sim R_\e^{-1}$ in this region, so the remaining term inside the integral is bounded by the $H^{1,+}_{\e}$ norm of $\phi^{(1)}$. Finally, for the quadratic term in $\ph^{(1)}$, we have the interpolation inequality 

\be \|u \|_{L^4} \leq C \|u\|_{L^2}^{1/4} \|u\|_{L^{1,2}}^{3/4}.  \label{interpolation} \ee
 
 \noindent Since the $H^{1,+}_\e$ norm dominates both pieces, the bound by $\e^{11/12}\|\eta\|_{3/2+\underline\gamma}$ follows from the bound on the  $H^{1,+}_{\e}$ norm of $\ph^{(1)}$ from Theorem \ref{PartImain} (in fact, here with an additional factor of $\e^{1/6}\sim R_\e^{1/4}$ from the weight on the $L^2$ portion).

 It remains to show that the $\text{d}\slashed{\mathbb D}$ term of (\refeq{dSWdD}) obeys the same bounds. Recalling the expression (\refeq{calBdef}) (cf. Notation \ref{underlinenotation}), the triangle inequality yields
\smallskip 
$$
\|\text{d}\slashed{\mathbb {D}}_{h_\circ}(\eta , 0) \bold 1^{+} \|_{L^2} \leq \| (\dot {\underline g}_\eta)_{ij}e^i. \nabla \Phi_\tau \|_{L^2} + \|d \text{Tr}_{g_\tau}(\dot {\underline g}_\eta).\Phi_\tau\|_{L^2}+   \|\text{div}_{g_\tau}(\dot {\underline g}_\eta). \Phi_\tau\|_{L^2}+ \|\mathcal R(B_\tau,\eta).\Phi_\tau\|_{L^2}.
$$

\smallskip 
\noindent where summation is implicit in the first term. Since $\Phi_\tau$ is polyhomogeneous with leading order $r^{1/2}$ by Lemma \ref{asymptoticexpansion}, each of these terms is of Type A in the sense of Lemma \ref{deformationbounds}, with weight $w_A\leq 2$. Copying the proof of Item (A) of Lemma \ref{deformationbounds}, but now only integrating over the support of $\bold 1^+$ shows that 
\smallskip
$$ \| (\dot {\underline g}_\eta)_{ij}e^i. \nabla \Phi_\tau \|_{L^2} \leq  C\e^{1/3-\gamma}\|\eta\|_{1,2} \hspace{1.5cm}   \|d \text{Tr}_{g_\tau}(\dot {\underline g}_\eta).\Phi_\tau\|_{L^2}+   \|\text{div}_{g_\tau}(\dot {\underline g}_\eta). \Phi_\tau\|_{L^2}\leq  C\e^{1-\gamma}\|\eta\|_{2,2}.$$

\smallskip
\noindent as desired. The term $\mathcal R(\mathcal B_\tau,\eta)$ is subsumed by these because it has weight $w_A=1$. This shows the second term of (\refeq{dSWdD}) satisfies the desired bound.

\end{proof}

 \subsection{Non-Linear Bounds}
 \label{section9.3}
 This subsection bounds the non-linear terms in Proposition \ref{nonlinearterms}. There are once again a large number of terms, but similarly to Definition \ref{gluingpermissible}, all that ultimately matters is that the scaling of the powers of $\e$ is sufficient to counteract the higher-norms for certain values of $L_0$. We say a term $Q_\e$ is {\bf quadratically permissible} if it obeys 
 
 $$\|Q_\e(\eta,\ph,a)\| \leq C \e^{1/12-\gamma} M_e(\eta) \Big(\|\eta\|_{1/2,2} \ + \ \|(\ph,a)\|_{H^1_\e} \ + \ \|(\ph,a)\|^2_{H^1_\e}\Big)$$
 
 \noindent for some gluing permissible term $M_e$ (in the sense of Definition \ref{gluingpermissible}). We leave it to the reader to verify the arithmetic that each bound in the upcoming lemma states that the corresponding term is quadratically permissible.

The statement of the next lemma involves an auxiliary partition of unity defined as follows. With $\gamma<<1$ being the constant such that $\bold 1^+$ is the indicator function of $\{r\leq \e^{2/3-\gamma}\}$, let  $\zeta^\pm$ be a partition of unity consisting of two cutoff functions such that $|d\zeta^\pm|\leq Cr^{-1}$, and $\zeta^+=1$ where $\{r\leq \e^{2/3-2\gamma}\}$ while $\text{supp}(\zeta^+)\subseteq \{r\leq 2\e^{2/3-\gamma}\}$. The purpose of introducing $\zeta^+$ is that it is equal to $1$ on a neighborhood larger than the support of $\bold 1^+$ by a factor of $\e^{\gamma}$; this extra buffer zone allows the exponential decay of Lemma \ref{exponentialdecay} to take effect on $\text{supp}(\zeta^-)$ for configurations that decay away from the support of $\bold 1^+$ (see the proof of Corollary \ref{nonlinearboundscor} in Section \ref{section10}). 
\begin{lm} \label{nonlinearbounds} Let  $\mathbb Q=Q_\text{SW}+ Q_{\Phi^{(1)}}+ Q_{A^{(1)}}+Q_{a.\ph}+Q_\mu$ be the non-linear terms from Proposition \ref{nonlinearterms}. Then these satisfy the following bounds for constants $C$ and $\gamma<<1$ independent of $\e,\tau$. 
\smallskip 

\begin{enumerate}
\item[(A)] $Q_\text{SW}$ satisfies 
\bea
 \ \ \ \ \ \|Q_\text{SW}(\ph,a)\zeta^+\|\leq C \e^{1/3-\gamma} \|(\ph,a)\|_{H^1_\e}^2 \hspace{1.5cm} & &  \|Q^\text{Re}_\text{SW}(\ph,a)\zeta^-\|_{L^2}\leq C \e^{1/4} \|(\ph^\text{Im},a)\|^2_{H^1_\e}\\
 & & \|Q^\text{Im}_\text{SW}(\ph,a)\zeta^-\|_{L^2}\leq C \e^{1/4}\|\ph^\text{Re}\|_{H^1_\e} \|(\ph^\text{Im},a)\|_{H^1_\e}.
\eea
\smallskip

\item[(B)] Provided $\|\xi\|_{3/2+\underline \gamma,2}\leq 1$, then $ Q_{\Phi^{(1)}}, Q_{A^{(1)}}$ satisfy 
\bea \|Q_{\Phi^{(1)}}(\eta,\ph)\bold 1^+\|_{L^2}& \leq& C \e^{1-\gamma} \|\eta\|_{3/2+\underline\gamma,2}\left( \|\ph\|_{H^1_\e} \ + \ \e^{1/3} \|\eta\|_{1,2} + \e \|\eta\|_{2,2}  \right) + C \e^{5/3-\gamma}\|\eta\|_{5/2+\underline\gamma,2} \|\ph\|_{H^1_\e}  \\ \|Q_{\Phi^{(1)}}(\eta,\ph)\bold 1^-\|_{L^2} & \leq & C \left( \e\|\eta\|_{3/2+\underline \gamma,2} \|\ph\|_{H^1_\e} \ + \ \e \|\eta\|_{1/2,2} \|\eta\|_{3/2+\underline\gamma,2}  \right),\eea

\noindent and identically for $Q_{A^{(1)}}$ with $\|a\|_{H^1_\e}$ in place of $\|\ph\|_{H^1_\e}$. 

\smallskip
\item[(C)] Provided $\|\xi\|_{3/2+\underline \gamma,2}\leq 1$, then 
\bea
\|Q_{a.\ph}(\eta, a, \ph)\|_{L^2}&\leq &C \e\|\eta\|_{3/2+\underline\gamma,2} \|Q_\text{SW}(\ph,a)\|_{L^2}\\
\|Q_{\mu}(\eta,\ph)\|_{L^2} &\leq &C \e\|\eta\|_{3/2+\underline\gamma,2} \left(\|\ph\|_{H^1_e} + \|Q_\text{SW}(\ph,a)\|_{L^2} \right). 
\eea

\end{enumerate}
\end{lm}

\begin{proof}
(I) To bound $Q_{\text{SW}}$, we employ the interpolation inequality (\refeq{interpolation}), which holds on $Y$ independent of any of the context of Lemma \ref{relatingvariationsinside}. 
Since $|\zeta^\pm|^2 \leq |\zeta^\pm|\leq 1$, and $|d\zeta^\pm|\leq CR_\e^{-1}$ is bounded by the weight on the $L^2$-terms in the $H^1_\e$-norm, applying this shows, e.g. 
\bea
\|Q_{\text{SW}}(\zeta^+ \ph, \zeta^+a)\|_{L^2}&\leq &C \|\zeta^+ \ph\|_{L^2}^{1/4}   \ \|\zeta^+\ph\|^{3/4}_{L^{1,2}} \  \|\zeta^+ a\|^{1/4}_{L^2} \ \|  \zeta^+ a\|^{3/4}_{L^{1,2}} \\
&\leq & C \cdot \text{max} (\zeta^+R_\e^{1/2})\cdot \|(\ph,a)\|^2_{H^{1}_\e}\\
&\leq & C \e^{1/3-\gamma} \|(\ph,a)\|^2_{H^{1}_\e}. 
\eea 
\noindent because the $H^{1}_\e$-norm dominates the $L^2$-norm with an extra weight of $R_\e$. On the support of $\zeta^-$, the proof is the same but the fact that $Q_{\text{SW}}$ has at most one factor in $S^\text{Re}$ means the weight on the $H^{1,-}_\e$ norm compared to the $L^2$-norm gives an extra factor of $\e^{1/4}|\Phi_\tau|^{-1/4}R_\e^{1/4}\leq \e^{1/4}R_\e^{1/8}$ for  the $S^\text{Im}\oplus (\Omega^0\oplus \Omega^1)$ components, and an extra factor of $R_\e^{1/4}\leq C$ for the real components, since $R_\e\sim 1$ far from $\mathcal Z_\tau$. 

\medskip

(II) Follows from the same considerations as Lemmas \ref{relatingvariationsoutside} and  \ref{relatingvariationsinside} (with $(\ph,a)$ replacing $(\ph^{(1)}_{\e,\tau}, a^{(1)}_{\e,\tau})$ in  \ref{perturbationbound}), and employing Lemma \ref{nonlineardeformationbounds} in place of Lemma \ref{deformationbounds}, then invoking the assumption that $\|\xi\|_{3/2+\underline\gamma}<1$ at the end.

\medskip 

(III) Just as in the proof of Lemma \ref{universalDnonlinear}, terms of higher order than quadratic involve composition of the metric components with the diffeomorphism $\underline F_\xi$. Applying the bound (\refeq{TaylorPDEs}) for $s=3/2+\underline \gamma$ (in the case that assumption $\|\xi\|_{3/2+\underline \gamma,2}\leq 1$ holds) shows that these higher-order terms are bounded by a constant multiple of the quadratic terms in (II), and the proof then follows from the above. \end{proof}

 \section{Contraction Subspaces}
 \label{section10}

 With Sections \ref{section3}--\ref{section9} complete, the majority of the ingredients and estimates for the proof of Theorem \ref{maina} are in place. What remains is to combine these estimates to show a sufficient version of the alternating iteration scheme converges to the desired solutions. The next three sections set up and carry out a fixed-point argument that accomplishes this, with the goal of each section being as follows.
 
  (1) The current Section \ref{section10} defines Banach spaces $\mathcal H_{\e,\tau}, \mathfrak L_{\e,\tau}$ that serve as the global domain and codomain for the universal Seiberg--Witten equations. These are formed from various combinations of the spaces $\mathbb H^{1,\pm}_{\e,\nu}, \mathbb L^{2,\pm}_{\e,\nu}$ defined in Section \ref{section8}. (2) Section \ref{section11} constructs a non-linear approximate inverse $\mathbb A: \mathfrak L_{\e,\tau}\to \mathcal H_{\e,\tau}$  such that 
 
 \be \mathbb T= \text{Id}-\mathbb A\circ \overline{\mathbb{SW}}_\Lambda, \label{Tdef}\ee

 \noindent is a contraction on a neighborhood of the origin in $\mathcal H_{\e,\tau}$, where $\overline{\mathbb{SW}}_\Lambda=\overline{\mathbb{SW}}-\chi^- \e^{-1}{\Lambda(\tau)\Phi_\tau}$, . Specifically, $\mathbb A$ is formed as a composition of  three linear parametrices, to be denoted by $P_\xi, P^+, P^-$, one each for the three steps of the cyclic iteration described in Section \ref{section2.4}. The final Section \ref{section12} deals with the 1-dimensional obstruction coming from $\Phi_\tau$ and shows that (\refeq{Tdef}) has a fixed point which is the solutions sought in Theorem \ref{maina}.

 \subsection{The Support of Range and Obstruction Components}
 \label{section11pre}
Dealing with the loss of regularity of the deformation operator in Theorem \refeq{deformationsmaina} (which carries over to the Seiberg--Witten setting by Lemma \ref{relatingvariationsoutside}) requires careful exploitation of the link between tangential regularity and radial distance (recall the discussion in Section  \ref{section6.3}). The definition of the spaces $\mathcal H_{\e,\tau}$  is carefully crafted to incorporate deformations needed to cancel two types of obstruction terms that will appear in the gluing iteration. 

Recall that $\lambda^+=\e^{1/2}$ and that $\chi^+$ is a cutoff function whose derivative $d\chi^+$ is supported where $\lambda^+/4\leq r \leq \lambda^+/2$. Recall also that $S^\text{Re}\subseteq S_E$ is a well-defined subbundle on $\Yminus \mathcal Z_\tau$. 

\begin{defn} Let $\psi \in L^2(\Yminus \mathcal Z_\tau; S^\text{Re})$ be a spinor. We say 

\begin{enumerate}
\item[(A)] $\psi$ is {\bf supported outside} if $$\text{supp}(\psi) \subseteq \Yminus N_{\lambda^+/4}(\mathcal Z_\tau)$$

\noindent where $\lambda^+$ is as above. 
\item[(B)] $\psi$ is
 {\bf  supported on the neck} if $$\text{supp}(\psi) \subseteq Y^-_{\e,\tau}$$

\noindent where we recall that $Y^-_{\e,\tau}$ is defined by $r\geq \e^{2/3-\gamma^+}$. 

\end{enumerate}
\end{defn}

The following corollary restricts the Fourier modes in obstruction bundle $\text{\bf Ob}$ which spinors with such support project to. It is a direct application of Lemma \ref{fourierregimescokernel}. 

\begin{cor}  \label{obstructiondecaycor}Suppose that $\psi_1,\psi_2 \in L^2(\Yminus \mathcal Z_\tau; S^\text{Re})$ are supported outside and on the neck respectively. Then the projections $\text{ob}_\tau^{-1}\circ \Pi_\tau(\psi_j) \in L^2(\mathcal Z_\tau; \mathcal C_\tau)$ for $j=1,2$ obeys the following. 

\begin{enumerate}
\item[(A)] With $\gamma^+<<1$ as before, and $L_1=\e^{-1/2 -\gamma^+}$, then for any $M\in \N$ there exists a constant $C_M$ (uniform in $\e,\tau$) such that the projection of $\psi_1$ obeys

\begin{eqnarray} \|(1-\pi_{L_1})\circ \text{ob}_\tau^{-1} \circ \Pi_\tau (\psi_1)\|_{L^2(\mathcal Z_\tau)}& \leq &C_M \e^M \|\psi_1\|_{L^2(Y)} \label{7.1}
\end{eqnarray}

\noindent where $\pi_{L_1}$ is the projection to the Fourier modes with $|\ell|\leq L_1$. 
\item[(B)] With $L_2=\e^{-2/3}$, then for any $M\in \N$ there exists a constant $C_M$ (uniform in $\e,\tau$) such that the projection of $\psi_2$ obeys

\begin{eqnarray} \|(1-\pi_{L_2})\circ \text{ob}_\tau^{-1} \circ \Pi_\tau (\psi_2)\|_{L^2(\mathcal Z_\tau)}& \leq &C_M \e^M \|\psi_2\|_{L^2(Y)} \label{7.2}
\end{eqnarray}

\noindent where $\pi_{L_2}$ is the projection to the Fourier modes with $|\ell|\leq L_2$. 
\end{enumerate}
\end{cor}

\begin{proof} Let $\gamma_{\circ}$ be the value of $\gamma$ in the statement of Lemma \ref{fourierregimescokernel}. For (A), the assumptions of the lemma hold provided $$ \e^{\tfrac{1}{2}}\geq  \e^{\left(-\tfrac{1}{2}-\gamma^+\right)\big(\gamma_{\circ}-1\big)}=\e^{\left(\tfrac12 +\gamma^+ -\tfrac{\gamma_{\circ}}{2} -\gamma^+\gamma_{\circ}\right)}$$

\noindent which holds if we choose $\gamma_{\circ}=\gamma^+$, since the quadratic term is negligible. For (B), the same choice works. The conclusion then follows directly from that lemma. 
\end{proof}

\bigskip

Given the above, we define the following three Fourier regimes, which correspond to the three classes of obstruction vectors we will need to cancel during the gluing iteration.

\begin{defn}\label{Fourierregimes} For $\gamma^+<<1$ fixed as before, set \bea  L^\text{low}= \e^{-(1/2+\gamma^+)}  & \hspace{2cm} &L^\text{med}=\e^{-2/3}  \eea

\noindent and for $\ell \in \Z$ define 
\bea
{ \pi}^\text{low}(e^{i\ell t})&=&\begin{cases} e^{i\ell t} \hspace{.5cm}|\ell|\leq L^\text{low} \\ 0 \hspace{.6cm}  \ \  |\ell| > L^\text{low}\end{cases} \medskip   \\ \\  \pi^\text{med}(e^{i\ell t})&=&\begin{cases} e^{i\ell t} \hspace{.5cm}L^\text{low}<|\ell|\leq L^\text{med} \\ 0 \hspace{.8cm} |\ell| > L^\text{med}, |\ell|\leq L^\text{low}\end{cases}  \\  \\ \pi^\text{high}(e^{i\ell t})&=&(\text{Id}-\pi^\text{med}-\pi^\text{low})e^{i\ell t}.
\eea
\noindent 
We write e.g. $\Psi^\text{low}:= \pi^\text{low}\circ \text{ob}_\tau^{-1}(\Psi)$ as shorthand for the projections of an obstruction element $\Psi\in \text{\bf Ob}(\mathcal Z_\tau)$ to the corresponding Fourier regimes.  

\end{defn}
\bigskip

 When the linearization $\text{d}\slashed{\mathbb D}(\eta,0)$ in the direction of a deformation is used to cancel the obstruction components of an error term $\psi \in L^2(\Yminus \mathcal Z_\tau; S^\text{Re})$, the {\it range} components of the error term also grow. This is a consequence of the off-diagonal term in the block decomposition (\refeq{blockdecomp}). Thus solving the obstruction components updates the error term by 
 
 $$\psi \ \  \mapsto \ \  (1-\Pi_\tau)\psi + (1-\Pi_\tau)\text{d}\slashed{\mathbb D}(\eta,0).$$

\noindent One of the problem with achieving convergence of the gluing iteration is that the second term can {\it a priori} be much larger than the original error $\|\psi\|_{L^2}$. Moreover, as explained in Section  \ref{section7.4}, the projection $(1-\Pi_\tau)$ is non-local and disrupts the property that the error is cleanly supported where $d\chi^+\neq 0$. This, in turn, disrupts the property that there is decay across the neck region.

The following lemma addresses both these issues. First, it provides a key bound that, as applied in the gluing iteration, will show that the new term $(1-\Pi_\tau)\text{d}\slashed{\mathbb D}(\eta,0)$ does not meaningfully increase the total size of the error. Second, it shows that, while this term no longer has true support where $d\chi^+\neq 0$, it is still {\it effectively} supported in the same region, in the sense of Definition \ref{effectivesupport}. We emphasize here that the use of the tangential smoothing gauge is absolutely crucial in specifically this lemma -- for the gauge choice of Example  \ref{standarddeformations}, this estimate fails badly and prevents the iteration scheme from converging. 

Recall in the statement of the proposition that $h_\circ$ was defined in Eq. (\refeq{h0seconddef}) (cf. Notation \ref{notationhn}).

\begin{prop} \label{effectivesupportdD}Suppose that $\eta \in C^\infty(\mathcal Z_\tau, N\mathcal Z_\tau)$ satisfies the following property: there is an $M\in \R$ such that $\|\eta\|_{m+1/2, 2}\leq CM^{m}\|\eta\|_{1/2,2}$ for all $m>0$. Then deformation operator (\ref{calBdef}) of the universal Dirac operator satisfies 
\bea  \|  \text{d}\slashed{\mathbb D}_{h_\circ}(\eta,0)\|_{L^2_{0}}&\leq& C  \|\eta\|_{1/2,2} \\  \|  \text{d}\slashed{\mathbb D}_{h_\circ}(\eta,0)\|_{L^2_{-\nu}}&\leq& C  M^{\nu}\|\eta\|_{1/2,2} \eea
\noindent for any weights $\nu>0$.

 In particular, it is $\text{d}\slashed{\mathbb D}_{h_\circ}(\eta,0)$ effectively supported where $r=O(M^{-1})$ in the sense of Definition \ref{effectivesupport}. 
\end{prop}

\begin{proof}
By Lemma \ref{asymptoticexpansion}, each $\Phi_\tau$ is polyhomogeneous. After reducing $\tau_0$, we may assume (via the ``moreover'' statement in  Lemma \ref{asymptoticexpansion}), that the bounds (\refeq{polyhom2}) hold uniformly in $\tau$. Consequently, $\Phi_\tau$ obey the required bounds  (\refeq{typeA}) of Definition  \ref{typesAB} uniformly in $\tau$, so that each term (\refeq{calBdef}) of $\text{d}_{h_\circ}\slashed{\mathbb D}(\eta,0)$ is a term of Type A with weight $w_A\leq 2$ in the sense of Definition \ref{typesAB}. The conclusion follows directly from applying Item (A) of that Lemma \ref{deformationbounds} with $\beta=0$ and then $\beta=-\nu$ and invoking the assumption on $\eta$.  
\end{proof}

We remark that the property of being effectively supported is preserved under addition via the triangle inequality, and true support in a region is a particular instance of effective support. Thus if, modulo smaller error terms, we have $ \Pi_\tau \text{d}\slashed{\mathbb D}_{h_\circ}(\eta,0)=-\Pi_\tau(\psi)$ then 

$$ \psi + \text{d}\slashed{\mathbb D}_{h_\circ}(\eta,0)= (1-\Pi_\tau)\psi + (1-\Pi_\tau)\text{d}\slashed{\mathbb D}_{h_\circ}(\eta,0)$$

\noindent and the right hand side will obey an effective support bound if  $\psi$ and $\text{d}\slashed{\mathbb D}_{h_\circ}(\eta,0)$ do.

\subsection{Three Fourier Regimes}
\label{section10.1}

As explained in Section \ref{section2.5}, the gluing iteration combines two different methods of canceling the obstruction. 

These two methods were described in Sections \ref{section5.2} and \ref{section6.2} respectively. These defined two invertible elliptic operators 

\begin{eqnarray}
\underline T_{\Phi_\tau}: L^{1/2,2}(\mathcal Z_\tau; N\mathcal Z_\tau)&\lre& L^2(\mathcal Z_\tau; \mathcal C_\tau)\label{Tproxy}\\
\text{ob}_\tau^{-1}\slashed D_{A_\tau}  :  \ \ \ \   \ \ \ \ \ \  \ \mathcal X_\tau \  \ \ \ \ \ \ \ & \lre& L^2(\mathcal Z_\tau; \mathcal C_\tau)\label{Dproxy}.
\end{eqnarray}
\noindent both of which give a method of canceling the obstruction. (\refeq{Tproxy}) is an isomorphism by Proposition \ref{deformationsmainamd}, and (\refeq{Dproxy}) is by Corollary \ref{solvingcokernel}.  

Definition \ref{Fourierregimes} divides the codomain $L^2(\mathcal Z_\tau; \mathcal C_\tau)\simeq  \text{\bf Ob}_\tau^\perp$ (recall the latter is defined preceding Corollary \ref{solvingcokernel}) into three Fourier regimes; we now proceed to use the two operators above to define three similar regimes in the domains. These are combined into a single subspace 

$$\frak W_{\e,\tau}\subseteq L^{1/2,2}(\mathcal Z_\tau; N\mathcal Z_\tau) \oplus \mathcal X_\tau,$$

\noindent on which (a minor modification of) $\underline T_{\Phi_\tau}\oplus \text{ob}_\tau^{-1}\slashed D_{A_\tau}$ restricts to an isomorphism, and it is this subspace  $\frak W_{\e,\tau}$ that is ultimately incorporated into the space $\mathcal H_{\e,\tau}$ on which the contraction mapping  (\refeq{Tdef}) is defined.  In fact, we will see that $\frak W_{\e,\tau}\subseteq C^\infty(\mathcal Z_\tau; N\mathcal Z_\tau) \oplus \mathcal X_\tau$ consists only of smooth deformations. This justifies in hindsight our {\it a priori} assumptions of high regularity (in e.g.  the statements of  \ref{boundedUniversalDirac}  and Lemmas \ref{relatingvariationsoutside}, \ref{relatingvariationsinside}) on the linearized deformations of $\mathcal Z_\tau$.

\medskip 

The definition of $\frak W_{\e,\tau}$ first requires a small modification of the operator $\underline T_{\Phi_\tau}$. If this operator preserved Fourier modes, then the pre-images of the three Fourier regimes in Definition \ref{Fourierregimes} would also have Fourier modes supported in the same three regimes. Even when $\underline T_{\Phi_\tau}$ does not preserve Fourier modes, the tame estimates of Proposition \ref{quantitativehigherorder} (cf. Proposition \ref{deformationsmainamd}) show that solutions $\eta$ of $\underline T_{\Phi_\tau}(\eta)=f$ behave {\it as if} this were true\footnote{In the sense that adding $s$ derivatives increases the norm by at most a constant times $(L^\text{low})^s$, and likewise for $L^\text{med}$. }, provided the metric and spinor $\Phi_\tau$ obey the assumptions of that Proposition. We may force these assumptions to hold by truncating Fourier modes of the polyhomogeneous $\Phi_\tau$ and smooth parameter $(g_\tau, B_\tau)$.

In Fermi coordinates (Definition \ref{Fermicoords}) and the accompanying trivialization (Lemma \ref{Fermitrivialization}) on $N_{r_0}(\mathcal Z_\tau)$, smooth objects may be decomposed using Fourier modes in the $t$-direction, leading to families of Fourier series smoothly parameterized by the normal coordinates $(x,y)$. Since $\text{d}{\mathbb D}_{h_0}(\eta,0)$ is supported in $N_{r_0}(\mathcal Z_0)$, we may define a modified version of $\underline T_{\Phi_\tau}$ by restricting these Fourier modes. With $\pi^\text{low}$ the projection to the same range of Fourier modes as in Definition \ref{Fourierregimes}, we define

\smallskip 

\be  (g^\circ_\tau, B_\tau^\circ) =(\pi^\text{low}(g_\tau), \pi^\text{low}(B_\tau)) \hspace{2cm} \Phi_\tau^\circ := \pi^\text{low}(\Phi_\tau)\label{circdefinitions1}\ee

\noindent where $\pi^\text{low}$ applied for every fixed $(x,y)\in D_{r_0}$. These truncated structures give rise to a corresponding operator:

\begin{defn} \label{tametruncationdef} The {\bf tame truncation} of the deformation operator, denoted $\underline T_{\tau}^\circ$ is defined by 

$$\underline {T}_\tau^\circ := \text{ob}_\tau^{-1} \circ \Pi_\tau \circ \mathcal B_{\tau}^\circ$$ 

\noindent where

\be  \underline{ \mathcal B}_\tau^\circ:={\Bigg(}-\frac{1}{2}\sum_{ij}  {\dot{\underline g}_\xi^\circ}(e_i,e_j) e^i . \nabla^{g_0}_j  + \frac{1}{2} d \text{Tr}_{g_0}( {\dot{\underline g}_\xi^\circ}).  +\frac{1}{2} \text{div}_{g_0}( {\dot{\underline g}_\xi^\circ}). +  \mathcal R(B_0,  {\dot{\underline g}_\xi^\circ}). \Bigg)\Phi^\circ_\tau,  \label{truncatedB}\ee

\noindent with $\underline g_{\xi}^\circ:=\underline F_\xi^*(g_\tau^\circ)$, and $\dot{\underline g}_{\xi}^\circ:=\d{}{s}\big |_{s=0}\underline F_{s\xi}^*(g_\tau^\circ)$. We also define the {\bf tame truncation error} $\frak t^\circ_\tau$ as the difference

$$ \underline T_{\Phi_\tau}=:\underline T_\tau^\circ + \frak t^\circ_\tau$$

\noindent from the original operator.
\end{defn}

Note that this operator is precisely the analogue of $\underline T_{\Phi_\tau}$, but formed using the truncated data (\refeq{circdefinitions1}). (See also the expression \refeq{calBdef} and the discussion preceding Theorem {\ref{deformationsmaina}}). 

By construction, Corollary (\ref{quantitativehigherorder}) now applies to show the following. 

\begin{cor} \label{fraktcircbounds}For any $m\in \N$, there is a constant $C_m$ (uniform in $\e,\tau$) depending on $m$ such that the following hold. 

\begin{enumerate}
\item[(A)] The bound 
$$\|\frak t^\circ_\tau(\eta)\|_{2}\leq C_m \e^m \| \eta\|_{1/2,2}. $$ 

\noindent holds for $\frak t^\circ_\tau$. In particular, after possibly reducing $\e_0$, $\underline T_\tau^\circ$ is invertible.  

\item[(B)] The estimates (\refeq{strongellipticest}) hold uniformly in $\tau$. In particular, if 

$$\underline T_\tau^\circ(\eta^\text{low})=\Psi^\text{low} \hspace{1.5cm}\text{and}\hspace{1.5cm} \underline T_\tau^\circ(\eta^\text{med})=\Psi^\text{med}$$

\noindent are solutions for $\Psi^\text{low}, {\Psi}^\text{med}$ supported in the corresponding Fourier regime as in Definition \ref{Fourierregimes}, then 
\begin{eqnarray} \|\eta^\text{low}\|_{m+1/2,2} &\leq& C_m(L^\text{low})^{-m} \|\Psi^\text{low}\|_{L^2}   \\ \|\eta^{\text{med}}\|_{m+1/2,2} &\leq &C_m(L^\text{med})^{-m}\|\Psi^\text{med}\|_{L^2}\label{etalowmedest}\end{eqnarray}
\noindent uniformly in $\tau$ for any $m\geq 0$. 
\end{enumerate}

\smallskip

\end{cor}

\begin{proof}  Because $g_\tau, B_\tau, \Phi_\tau$ are smooth in the tangential directions, the Sobolev embeddings $C^3(S^1)\hookrightarrow H^4(S^1)$ applied for each fixed $(x,y)$ implies that 

$$\|\Phi_\tau - \Phi^\circ_\tau \|_{C^3}\leq C \e^{m + 2 } \|\Phi_\tau\|_{C^{2m + 10}} \hspace{2cm} \|g_\tau-g_\tau^\circ\|_{C^3} \ + \ \|B_\tau-B_\tau^\circ\|_{C^1}\leq C \e^{m+2} \|(g_\tau, B_\tau)\|_{C^{2m + 10}}$$

\noindent for any $m\in \N$. This follows because the difference $\Phi_\tau - \Phi^\circ_\tau$ is supported in Fourier modes above $L^\text{low}=\e^{-1/2-\gamma^+}$, thus each two derivatives brings out a factor of $\e^{1+2\gamma^+}\leq \e$. 

It is then easy to verify that  the $C^2$ norm of each term arising in the difference of the corresponding  pullback metrics ${\underline g}_{\xi,\tau}-{\underline g}_{\xi,\tau}^\circ= \underline F_\xi^*(g_\tau-g_\tau^\circ)$ obeys the same bounds (independent of $\xi \in \mathcal E_\tau$), and thus the $L^2$ norm of each term in the difference 

$$\underline {\mathcal B}_{\Phi_\tau}-\underline{\mathcal B}_\tau^\circ $$

\noindent does as well, where these are Notation \ref{underlinenotation} and (\refeq{truncatedB}) respectively. The bound on the operator norm of $\frak t_\tau^\circ: L^{1/2,2} \to L^2$ in Part (A) follows, since the projection $\Pi_\tau$ can only decrease the norm and $\text{ob}_{\tau}^{-1}$ is bounded. The second statement in Part (A) is then immediate from Neumann series.

For Part (B),  since $g_\tau, B_\tau, \Phi_\tau$ are smooth in the tangential directions, and $g_\tau^\circ, B_\tau^\circ, \Phi_\tau^\circ$ have only Fourier modes less than $L^\text{low}$, the hypotheses of Corollary \ref{strongellipticest} hold with $M=L^\text{low}$. 
\end{proof}

Using $\underline T_\tau^\circ$, we now make the following definition of the subspace $\mathfrak W_{\e,\tau}$.

\begin{defn} \label{frakWdef}Let $ \frak W_{\e,\tau}\subseteq L^{1/2,2}(\mathcal Z_\tau; N\mathcal Z_\tau) \ \oplus \  \mathcal X_\tau$ be the closed subspace given as the image of the following composition. 

    \begin{center}
\tikzset{node distance=5.6cm, auto}
\begin{tikzpicture}
\node(A){$L^{2}(\mathcal Z_\tau;\mathcal C_\tau)$};
\node(C)[right of=A]{$\begin{matrix} L^2(\mathcal Z_\tau; \mathcal C_\tau)\\ \oplus \\ L^2(\mathcal Z_\tau; \mathcal C_\tau)\end{matrix}$};
\node(C')[right of=C, yshift=0cm]{$\begin{matrix} L^{1/2,2}(\mathcal Z_\tau; N\mathcal Z_\tau)\\ \oplus \\ \mathcal X_\tau. \end{matrix} $};
\draw[->] (A) to node {$ \begin{pmatrix} \pi^\text{low} + \pi^\text{med} \\ \pi^\text{high}\end{pmatrix}$} (C);
\draw[->] (C) to node {$  \begin{pmatrix} \underline T_{\tau}^\circ \\ \text{ob}^{-1}_\tau \circ \slashed D_{A_\tau}\end{pmatrix}^{-1}$} (C');
\end{tikzpicture}\end{center}

\noindent  $\frak W_{\e,\tau}$ is equipped with the norm 

\smallskip 

\be \|(\eta, u)\|_{\frak W}:= \left( \|\eta^\text{low}\|_{1/2,2}^2  \ + \ \e^{-1/3}\|\eta^\text{med}\|_{1/2,2}^2  \ + \ \e^{-4/3} \|u\|_{H^1_e} \right)^{1/2}.\label{frakWnorm}\ee

\smallskip
\noindent where $\eta =\eta^\text{low}+ \eta^\text{med}$ is the decomposition of $\eta$ such that $T_\tau^\circ(\eta^\text{low}), T_\tau^\circ(\eta^\text{med})$ are the components of the image in the two corresponding Fourier regimes of the codomain of $T_\tau^\circ$ in Definition \ref{Fourierregimes}. Thus in this notation, $\eta^\text{low}\neq \pi^\text{low}(\eta)$ in general, since $T_\tau^\circ$ need not preserve Fourier modes.  
\end{defn}

Note several things about this definition: 

\begin{enumerate}
\item  By construction, $(\underline T_\tau^\circ, \text{ob}_\tau^{-1}\slashed D_{A_\tau}): \frak W_{\e,\tau}\to L^2(\mathcal Z_\tau; \mathcal C_\tau)\simeq \text{\bf Ob}_\tau^\perp$ is an isomorphism. Thus this map can be used to cancel obstruction elements by a joint combination of $(\eta,u) \in \frak W_{\e,\tau}$. 

\item Since the image in the top $L^2(\mathcal Z_\tau; \mathcal C_\tau)$ summand  in the middle of diagram consists of the span of Fourier modes $e^{i\ell t}$ with $|\ell|\leq L^\text{med}$, the intersection $\frak W_{\e,\tau}\cap L^{1/2,2}(\mathcal Z_\tau; N\mathcal Z_\tau)$ consists of solutions $\eta$ to 

$$\underline T_\tau^\circ(\eta)=\psi$$

\noindent with $\pi^\text{high}(\psi)=0$. By Item (B) of Corollary \ref{fraktcircbounds}, all such $\eta$ are smooth. Thus $$\frak W_{\e,\tau}\subseteq C^\infty(\mathcal Z_\tau; N\mathcal Z_\tau)\oplus \mathcal X_\tau$$

\noindent includes only smooth linearized deformations of $\mathcal Z_\tau$.

\item  The composition in the above diagram is {\it a priori} discontinuous in $\e$, since $\pi^\text{med},\pi^\text{high}$ jump when $\e^{-2/3}$ crosses an integer. We may amend these projections so that they are continuous in $\e$ by choosing an $\e$-dependent family of $U(2)$ matrices that interpolate between the two projections in the span of the $\ell$ and $\ell\pm 1$ Fourier modes for $ |\ell| + \tfrac14 \leq \e^{-1} \leq |\ell+1| - \tfrac14$ \footnote{adjusted with an appropriate factor of the length $|\mathcal Z_\tau|$.}.   

With this adjustment made, $\frak W_{\e,\tau}\to (0,\e_0)\times (-\tau_0, \tau_0)$ is a Banach vector bundle, because it is the image of a continuous, fiberwise bounded and injective map of vector bundles given by composition in the above diagram. The domain and codomain of this injection are vector bundles by Sections \ref{section5} -- \ref{section6}. 

\end{enumerate}

\bigskip

In light of the above constructions, we may extend our definition of the universal Seiberg--Witten equations to a map \be \overline{\mathbb{SW}}: \mathbb H^1_{\e,\nu}(\mathcal E_\tau \times \mathcal X_\tau)\to p_1^*\mathbb L^2_{\e,\nu}\label{SWwithmathcalX}\ee where $\mathcal X_\tau$ is the subspace from Lemma \ref{solvingcokernel}, and the bundle is the pullback of $\mathbb H^1_{\e,\nu}$ by the projection $\mathcal E_\tau \times \mathcal X_\tau \to \mathcal E_\tau$. The extended map is defined by replacing $\psi^-$ in (\refeq{universalSWpre}) by $\psi^- + u$ for $u \in \mathcal X_\tau$. This map factors through (\refeq{universalSWdef2}) since the inclusion $\mathcal X_\tau\hookrightarrow H^{1,-}_{\e,\nu}$ given by $u \mapsto \chi^- u$ is bounded (the weights are equivalent on $Y^-$ since it is a compact subset of $\Yminus \mathcal Z_\tau$). By restriction, we also have a map

$$\overline{\mathbb{SW}}: \mathbb H^1_{\e,\nu}( \frak W_{\e,\tau})\to p_1^*\mathbb L^2_{\e,\nu}$$

\noindent and bundles $\mathbb H^{1,\pm}_{\e,\nu}(\frak W_{\e,\tau}), \mathbb L^{2,\pm}_{\e,\nu}(\frak W_{\e,\tau})$ restricting those from Definition \ref{hilbertspaces}.

 \begin{rem}
 There is an important and delicate balance that must be struck in the choice of the three Fourier regimes, thus in the Definition \ref{frakWdef} of $\frak W_{\e,\tau}$. The higher the Fourier mode allow for the deformation $\eta$, the more extreme the loss of regularity of (\refeq{Tproxy}) becomes. There is an upper limit, above which the loss of regularity results in accumulating powers of $\e$ which cause the alternating iteration to fail to converge (in particular, the constant $M$ in Corollary  \ref{effectivesupportdD} becomes an uncontrollable power of $\e$).

    Conversely, there is a lower limit to the modes which may be solved for using (\refeq{Dproxy}). The solutions of (\refeq{Dproxy}) grow across the neck region, rather than decay, so their norms must be suppressed by an additional factor of precisely this growth. The estimates (\refeq{7.1}) and (\refeq{7.2}) provide such estimates, but only for spinors whose support is in the restricted regions therein. The radii of $Y^-, Y^+$ therefore place a lower bound on the obstruction modes that can be solved using the singular spinors in $\mathcal X_\tau$ \footnote{The inner radius of $Y^-$ is determined by the invariant scale of the model solutions, but the radius of $Y^+$ is ultimately an arbitrary choice. The tension between these two regimes of obstruction cancellation, however, cannot be eliminated simply by scaling these radii.}

    It should be regarded, perhaps, as a minor miracle that these upper and lower bounds can be satisfied simultaneously. Were the powers of $\e$ less fortuitous, restricting the ranges of Fourier modes solved by deformations and singular spinors respectively to the regimes that allowed convergence of the iteration would leave a gap in the spectrum of $\text{\bf Ob}_\tau$ that could not be cancelled. 
    
     \end{rem}
    
    \subsection{Contraction Subspaces}
\label{section10.2}

We now define the spaces $\mathcal H, \mathfrak L$ on which the approximate inverse $\mathbb A$ and the contraction $T$ as in (\refeq{Tdef}) are defined. 
    
These definitions involve attaching weights of various power of $\e$ to the norms of $\mathbb H^{1,\pm}_{\e,\nu}, \mathbb L^{2,\pm}_{\e,\nu}$. The reader is warned up front that, while there is some logic in choosing these weights (see Remark \ref{HLweightschoice}), they are ultimately chosen in hindsight, after some trial and error with the weights appearing in the gluing iteration. The goal is to construct these spaces so that the following two criteria are true: 

\begin{enumerate}
\item[(C1)] $\mathcal H\subseteq T\mathbb H^1_\e(\mathcal E_\tau\times\mathcal X_\tau)$ is a closed subspace (in a trivialization of the bundle) on which the linearization $\text{d}\overline{\mathbb {SW}}$ has Index 0 (cf. Subsection \ref{contractionsubspacessection}).
\item[(C2)] The map  (\refeq{Tdef}) is ultimately a contraction in this norm. 
\end{enumerate}

Recall that $\bold 1\pm$ denote the indicator functions of the regions $\{r \leq \e^{2/3-\gamma^+}\}, \{r \geq \e^{2/3-\gamma^+}\}$ respectively.

\begin{defn} \label{contractionspacesdef}Set \bea \mathcal H^+&:=& \left\{\mathcal L_{(\Phi^{(1)}, A^{(1)})}^{-1}(g\bold 1^+) \  \  \Big | \ \ g\in L^2(Y^+)\right\}\subseteq H^{1,+}_\e \\  \mathcal H^-&:= &\left\{\mathcal L_{(\Phi_\tau, A_\tau)}^{-1}((1-\Pi_\tau)g\bold 1^-) \  \  \Big | \ \ g\in L^2(Y^-)  \right\}\subseteq H^{1,-}_\e\eea
where $\mathcal H^-$ uses the solution in $H^{1,-}_\e$ that is $L^2$-orthgonal to $\Phi_\tau$ on $Y^-$. Equip these with the norms \begin{eqnarray} \|(\ph,a)\|_{\mathcal H^+}&=& \e^{-1/12-\gamma_{\mathcal L}} \|\mathcal L_{(\Phi^{(1)}, A^{(1)})}(\ph,a)\bold 1^+\|_{L^2} \label{H+norm} \\ \|(\psi,b)\|_{\mathcal H^-}&=&\left(\|\psi^\text{Re}\|^2_{rH^{1}_e} \ + \ \e^{\nu}\|\psi^\text{Re}\|^2_{rH^1_{-\nu}} \ + \ \e^{-1/2}\|(0,\psi^\text{Im}, b)\|_{H^{1,-}_\e} \right)^{1/2}. \label{H-norm}\end{eqnarray}

\noindent where $\nu=\nu^{-}=\tfrac12-\gamma^-$, and $\gamma_{\mathcal L}$ is as in Theorem \ref{Insideinvertibility}. 
\end{defn}

Note that in this definitions, $H^{1,\pm}_{\e}$ denotes the space with the weight $\nu=0$. Note also that the two linearization are taken at the model solutions $(\Phi^{(1)}_{\e,\tau}, A^{(1)}_{\e,\tau})$ from Theorem \ref{PartImain} on $Y^+$ and at the limiting eigenvector $(\Phi_\tau, A_\tau)$ on $Y^-$. These linearization are invertible by Theorem \ref{PartImain}(D) for $Y^+$, and by Lemma (\ref{solvingoutsideim}) and Definition (\ref{Obdefn}) for $Y^-$. Theorem \ref{PartImain}(D) implies that (\refeq{H+norm}) is equivalent to the $H^{1,+}_{\e}$ norm.

Using Lemma \ref{SWtrivializations}, there is a trivialization 

$$\underline{\Upsilon}:  \mathbb H^{1}_{\e,\nu}(\mathcal E_\tau \times \mathcal X_\tau) \simeq \Big (H^{1,+}_{\e,\nu} \oplus H^{1,-}_{\e,\nu} \Big )\times (\mathcal E_\tau \times \mathcal X_\tau).$$

\noindent In this trivialization, we define $\mathcal H_{\e,\tau}, \frak L_{\e,\tau}$ as follows. 
\begin{defn}\label{contractionspacesdefII}
Define
\bea \mathcal H_{\e,\tau}&:=& \mathcal H^+\oplus \mathcal H^- \oplus \frak W_{\e,\tau} \oplus \R \hspace{1.8cm} \frak L_{\e,\tau}:= L^2(Y)
\eea
\noindent and equip these with the norms
\begin{eqnarray}  \|(h^+, h^-, \eta,\mu)\|_{\mathcal H}&:=&\left( \|h^+\|_{\mathcal H^+}^2  \ + \  \|h^-\|_{\mathcal H^-}^2  \ + \  \|\eta\|_{\frak W}^2 \ + \  \e^{-2}|\mu|^2\right)^{1/2}\vspace{.5cm} \\ & & \nonumber \\ \medskip  \|\frak e\|_{\frak L}&:=&\Big(\e^{-2/12-2\gamma_{\mathcal L}}\|\frak e\bold 1^+\|^2_{L^2} \ +   \   \|\frak e^\text{Re}\bold 1^- \|^2_{L^2}  \ + \  \e^{\nu}  \|\frak e^\text{Re}\bold 1^-\|^2_{L^2_{-\nu}}\smallskip \nonumber \\ & & \   \   \  + \  \  \e^{-1/3} \|\frak e^\text{Im}\bold 1^-\|^2_{L^2} \  \ + \  \   \e^{-1/3}\|\pi^\text{med}\Pi_\tau(\frak e \bold 1^-)\|^2_{L^2} \Big)^{1/2} \label{frakLnorm}\end{eqnarray}
\noindent where $\frak e=(\frak e^\text{Re}, \frak e^\text{Im}) \in  L^2(S^\text{Re})\oplus L^2(S^\text{Im}\oplus \Omega^0\oplus \Omega^1)$.  Here $\nu=\nu^-=\tfrac12 - 10^{-6}$ is the outside weight. 
\end{defn}
\smallskip

These families of spaces form Banach vector bundles over pairs $(\e,\tau)$ for $\e,\tau$ sufficiently small (the norms on $\mathcal H^\pm$ are equivalent for different $\e,\tau$, though not uniformly, and $\frak W_{\e,\tau}$ is a vector bundle as following Definition \ref{frakWdef}).

\begin{rem}\label{HLweightschoice} As explained above, the definition of these norms requires some hindsight from attempting to do the gluing iteration. There are a few notable points, however. 

Weighted terms in the $\frak L$ norm dictate that the $\nu$-weighted term in the norms is larger (for norm $O(1)$) than the unweighted term by at most a factor of $\e^{-\nu/2}$, despite the fact that $\sup r^{-\nu}=\e^{-\nu(2/3-\gamma)}$ on $\text{supp}(\bold 1^-)$. This shows that configurations in $\mathfrak L$ that do not have $\e^{1/2}$-effective support receive a large penalty in their norm. In the same fashion, the powers of $\e^{-1/2}$ in (\refeq{H-norm}) impose a large penalty unless the $S^\text{Im}\oplus \Omega^0\oplus \Omega^1$ components are small compared to those in $S^\text{Re}$. This builds in the perspective that the imaginary and form components should be an almost negligible error in the iteration. 

In the same fashion, the factor of $\e^{-1/3}$ suppresses the Fourier modes in the medium range both in the domain (\refeq{frakWnorm}), and in the codomain (\refeq{frakLnorm})

\end{rem}

The following lemma shows the linearized Seiberg--Witten equations are bounded on the above spaces with operator norm independent of $\e$ (up to a factor of $\gamma)$. This assertion is non-trivial: it means that the new norms control the loss of regularity of the deformation operator. Indeed, if one simply uses the $L^{1/2,2}, H^{1,\pm}_\e$ and $L^2$-norms without the correct weights and powers of $\e$, the operator norm is only bounded by a constant times $\e^{-4/3}$ . Recall that $h_1$ is the model configuration defined in Eq. (\refeq{h1seconddef}).  
\begin{lm} \label{frakLbounded}
There is a constant $C$ independent of $\e,\tau$ such that the linearization $\text{d}\overline{\mathbb{SW}}_{h_1}: \mathcal H\to \frak L$ satisfies 
\be \|\text{d}\overline{\mathbb{SW}}_{h_1}(h)\|_{\frak L}\leq C\e^{-\gamma} \|h\|_{\mathcal H}\label{dSWbounded}\ee 
for some $\gamma<<1$. 
\end{lm}

\begin{proof}
By Proposition \ref{deformationUSW}, the linearization acting on $h=(\xi, \ph, a, \psi,b,\mu ,u)$ where $\xi=\e \eta$ may be written in the trivialization of Lemma \ref{SWtrivializations} as \begin{eqnarray}\text{d}\overline{\mathbb {SW}}_{h_1}(h)&=&\tfrac{1}{\e}\underline{\mathcal B}_{\Phi^{(1)}}(\xi)\  + \  \mathcal L_{(\Phi^{(1)}, A^{(1)})}(\frak p) \ + \ \mu \chi^- \frac{\Phi_{\tau}}{\e} \\
&=& \text{d}\slashed{\mathbb D}_{h_\circ}(\eta,0) + \Xi(\eta) + \mathcal L_{(\Phi^{(1)}, A^{(1)})}(\frak p) \ + \ \mu \chi^- \frac{\Phi_{\tau}}{\e} \label{fourterms}. \end{eqnarray}

\noindent where $\Xi$ is as in (\refeq{Xidef}), and $\frak p=\chi^+(\ph,a) + \chi^-(\psi+ u,b)$ where $(\ph,a)\in H^{1,+}_\e$, $(\psi,b)\in H^{1,-}_\e$ and $u^-\in \mathcal X_\tau$. We now bound four sub-terms independently. 
\medskip 

\begin{enumerate}

\item[(1)] {\underline {$H^{1,+}_\e $ Terms}}. Abbreviating $\mathcal L_{(\Phi_1, A_1)}=\mathcal L$, and using the definition (Definition \ref{contractionspacesdef}) of $\mathcal H^+$, 
\bea
\|\mathcal L(\chi^+(\ph,a))\|_\frak L&=& \e^{-2/12-2\gamma_{\mathcal L}}\|\mathcal L(\ph,a)\bold 1^+\|^2_{L^2}  +   \|d\chi^+\ph^\text{Re}\bold 1^-\|^2_{L^2} + \ \e^{\nu/2}\|d\chi^+\ph^\text{Re}\bold 1^-\|^2_{L^2_{-\nu}} \\   & &\ + \ \e^{-1/3}\|d\chi^+(\ph^\text{Im},a) \bold 1^-\|^2_{L^2} \ + \ \e^{-1/3} \|\pi^\text{med}\Pi_\tau(d\chi^+\ph^\text{Re}\bold 1^-)\|^2_{L^2} \label{outsideboundedL}\\  
&\overset{(?)}\leq & \|(\ph,a)\|_{\mathcal H^+} 
\eea
\noindent The first term of (\refeq{outsideboundedL}) is bounded by the definition (\refeq{H+norm}) of the $\mathcal H^+$-norm, the second is reduced to the case of the first by Lemma \ref{insidepolynomialdecay}; the third tern is identical to the second since $d\chi^+$ is supported where $r=O(\e^{1/2})$.

The fourth of  (\refeq{outsideboundedL}) is exponentially small by Corollary \ref{corollaryexpdecay} and Theorem \ref{PartImain} because the latter shows $\|-\|_{L^2}\leq \|-\|_{H^{1,+}_\e}\leq \|-\|_{\mathcal H^+}$. The fifth term of (\refeq{outsideboundedL}) is bounded by $O(\e^M)$ as a consequence of  Part (A) of Lemma \ref{obstructiondecaycor} again (note $\pi^\text{med}$ is the projection to {\it only} the modes in the medium range). 
\bigskip

\item[(2)] {\underline {$H^{1,-}_\e $ Terms}}.  Write $\mathcal L_{(\Phi^{(1)}, A^{(1)})}=\mathcal L_{(\Phi_\tau, A_\tau)} + K_1$. For the $(\psi,b)$ portion of $(\psi+u,b)$, the definition (\refeq{frakLnorm}) of the $\frak L$-norm shows

\begin{eqnarray}
\|\mathcal L_\tau(\chi^-(\psi,b))\|^2_\frak L&=& \e^{-2/12-2\gamma_{\mathcal L}} \|\slashed D_{A_\tau}(\chi^-\psi^\text{Re})\bold 1^+\|^2_{L^2}  +   \|\slashed D_{A_\tau}\psi^\text{Re}\bold \|^2_{L^2} + \ \e^{\nu}\|\slashed D_{A_\tau}\psi^\text{Re}\|^2_{L^2_{-\nu}} \\   & &\ + \ \e^{-1/3}\| \mathcal L^\text{Im}(\psi^\text{Im},b)\|^2_{L^2} \ + \ \e^{-1/3} \|\pi^\text{med}\Pi_\tau(\slashed D_{A_\tau}(\chi^-\psi^\text{Re})\bold 1^-)\|^2_{L^2} \label{outsideboundednessL}\\
&\overset{(?)}\leq &  C \|(\psi, b)\|^2_{\mathcal H^-} \nonumber
\end{eqnarray}

\medskip

\noindent where $\mathcal L^{\text{Im}}_{(\Phi_\tau, A_\tau)}$ is as in Lemma \ref{solvingoutsideim}. The second, third and fourth terms are bounded by the $\mathcal H^-$ norm simply by the boundedness of $\slashed D_{A_\tau}: r^{1+\nu}H^1_e\to r^\nu L^2$ and of $\mathcal L^\text{Im}$. 

Note that the first term of $\slashed D_{A_\tau}(\chi^-\psi^\text{Re})\bold 1^+$ is identically zero, because $\text{supp}(\chi^-)\cap \text{supp}(\bold 1^+)=\emptyset$.  For the projection term of (\refeq{outsideboundednessL}) involving $\pi^\text{med}$, note that since $\Pi_\tau(\slashed D_{A_\tau}\psi^\text{Re})=0$ by definition, then

\bea \Pi_\tau\Big([\chi^- \slashed D_{A_\tau}\psi^\text{Re} + d\chi^-\psi^\text{Re}] \bold 1^- \Big)&=& \Pi_\tau\Big( \chi^-\slashed D_{A_\tau}\psi^\text{Re} \ + \  (\chi^-\slashed D_{A_\tau}\psi^\text{Re})\bold 1^+\Big) \ + \ \Pi_\tau(d\chi^- \psi^\text{Re})\\
&=& \Pi_\tau\Big( -(1-\chi^-)\slashed D_{A_\tau}\psi^\text{Re} \ + \  (\chi^-\slashed D_{A_\tau}\psi^\text{Re})\bold 1^+\Big)  \ + \  \Pi_\tau(d\chi^- \psi^\text{Re}),
\eea

\noindent All these terms are exponentially small by Part (B) of Lemma \ref{obstructiondecaycor}. Finally, the additional weight of $\e^{-4/3}$ dominates the factor of $\e^{-1/3}$ on the $\pi^\text{med}$ term. This completes the bound, excluding the $K_1$ term. Lastly, $K_1$ is exponentially small on $\text{supp}(\chi^-)$ by Corollary  \ref{corollaryexpdecay}.

\bigskip 

\item[(3)] {\underline {$\mu$ Terms}}. The term $\chi^- \e^{-1}\mu \Phi_{\tau}$ is obviously bounded by $\|\mu\|_{\mathcal H}$ because the $\e^{-2}$ weight cancels the $\e$ in the denominator, and $\Phi_\tau\in L^2 \cap L^2_{-\nu}$ (note the power of $\e$ is {\it positive}, hence favorable, on the $L^{2}_{-\nu}$-term). By the $r^{1/2}$ asymptotics of $\Phi_\tau$, one has $|\Phi_\tau|\leq C \e^{1/3-\gamma}$ on $\text{supp}(\bold 1^+)\cap \text{supp}(\chi^-)$, which more than compensates for the $\e^{-1/12-\gamma_\mathcal L}$ weight in (\refeq{frakLnorm}). Finally, since $\pi^\text{med}(\Phi_\tau)=0$, $$\|\pi^\text{med}(\chi^-\Phi_\tau \bold 1^-)\|_{L^2}=\|\pi^\text{med}((1-\chi^-)\Phi_\tau)\|_{L^2} \leq C \e^{M}$$ 

\noindent using Part (B) of Lemma \ref{obstructiondecaycor} again. 
 
 \bigskip

\item[(4)] {\underline {Deformation Terms}}. Proceeding now to the terms involving $\eta=\e^{-1}\xi$, Corollary \ref{fraktcircbounds} shows that for $\eta \in \frak W$, 

\be \|\eta^\text{low}\|_{s+1/2,2}  \leq C\e^{s(-1/2-\gamma)}\|\eta^\text{low}\|_{1/2,2}\hspace{2cm}\|\eta^\text{med}\|_{s+1/2,2}  \leq C\e^{s(-2/3)}\|\eta^\text{med}\|_{1/2,2}.\label{cor116again}\ee

\noindent (Cf. Definition \ref{gluingpermissible}). We bound the terms on the supports of $\bold 1^\pm$ separately, beginning with the outside.

Substituting these into the conclusions of Proposition \ref{effectivesupportdD} applied with $\nu=0,\nu^-$ and Proposition \ref{relatingvariationsoutside} show that 
\begin{eqnarray} \|\text{d}\slashed{\mathbb D}_{h_\circ}(\eta,0) \bold 1^-\|_{L^2}  \  &\leq & C \|\eta\|_{1/2}  \label{unweighteddDbound}\\
 \|\text{d}\slashed{\mathbb D}_{h_\circ}(\eta,0)\bold 1^-\|_{L^2_{-\nu}}&\leq& C (\e^{-\nu(1/2+\gamma)}\|\eta^\text{low}\|_{1/2} + C \e^{-2\nu/3}  \|\eta^\text{med}\|_{1/2}) \leq C \e^{-\nu/2-\gamma}\|\eta\|_{\frak W} \label{weighteddDbound}\\
 \|\Xi^-(\eta)\|_{L^2}&\leq & C \e^{11/12-\gamma}( \e^{-1/2}\|\eta^{\text{low}}\|_{1/2} + \e^{-2/3} \|\eta^{\text{med}}\|_{1/2}) \leq C \e^{5/12-\gamma} \|\eta\|_{\frak W}. \label{Ximinusbound}\\
  \e^{\nu/2}\|\Xi^-(\eta)\|_{L^2_{-\nu}}&\leq& C \e^{4/12-\gamma}\|\eta\|_{\frak W}\label{weightedXiminus}.
 \end{eqnarray}
 
 \noindent The fourth of these bounds follows from the third, because for $\nu=\nu^-$,  $r^{-\nu}\leq \e^{-2\nu/3}\leq \e^{-\nu/2}\e^{-\nu/6}$ on $\text{supp}(\bold 1^-)$. The $\e^{-\nu/2}$ is moved to the left hand side, and $\tfrac{\nu}{6}=\frac{1}{12}-\gamma$, which yields (\refeq{weightedXiminus}).

 Next, for the inside terms, substituting the bounds (\refeq{cor116again}) into Proposition \ref{relatingvariationsinside} shows that 
 \begin{eqnarray}
 \|\text{d}\slashed{\mathbb D}_{h_0}(\eta,0)\bold 1^+ + \Xi^+(\eta)\|_{L^2}&\leq&  C \e^{-\gamma} \Big( \e^{1/3} \|\eta\|_{1} \ + \   \e^{11/12} \|\eta\|_{3/2+\underline \gamma } \ + \  \ \e \|\eta\|_{2}  \ + \ \e^{19/12} \|\eta\|_{5/2 + \underline\gamma} \Big)\nonumber \\
 &\leq & C \e^{1/12-\gamma}\|\eta^{\text{low}}\|_{1/2} + C\|\eta^\text{med}\|_{1/2} \\ &\leq& C \e^{1/12-\gamma} \|\eta\|_{\frak W} \label{Xiplusbound}. 
 \end{eqnarray}
 \noindent  Together, (\refeq{unweighteddDbound})--(\refeq{Xiplusbound}) show that all but the $\pi^\text{med}$-term in the $\frak L$-norm are bounded for $\text{d}{\mathbb{SW}}(\xi,0,0)=(\text{d}\slashed{\mathbb D}_{h_\circ}+\Xi)(\xi,0,0)$.

To complete the proof, we show the bound
 
 $$ \e^{-1/3}\|\pi^\text{med}\circ \Pi_\tau(\text{d}\slashed{\mathbb D}_{h_\circ}(\eta,0) \bold 1^- \ + \ \Xi^-(\eta) )\|^2_{L^2} \leq C \|\eta\|^2_{\frak W}$$

\noindent on the final projection term, which is slightly more involved. This is because $\eta$ contains modes in both the low and medium ranges, so the projection $\pi^\text{med}$ is not {\it a priori} small (e.g. we cannot apply Lemma \ref{obstructiondecaycor} as above). For the $\Xi^+$ term, the power of $\e$ in (\refeq{Ximinusbound}) is sufficient to overcome the additional factor of $\e^{-1/3}$. For the remaining term, we argue as follows.

Since $1=\bold 1^+ + \bold 1^-$, the triangle inequality shows 
\smallskip 
\be \hspace{2cm}  \left\|\pi^\text{med}\Pi_\tau \left(\text{d}\slashed {\mathbb D}_{h_\circ}(\eta,0)\bold 1^-\right)\right\|_{L^2}  \ \leq \  \left\|\pi^\text{med}\Pi_\tau\text{d}\slashed {\mathbb D}_{h_\circ}(\eta,0)\right\|_{L^2} \  + \  \left\|\pi^\text{med}\Pi_\tau \left(\text{d}\slashed {\mathbb D}_{h_\circ}(\eta,0)\bold 1^+\right)\right\|_{L^2}. \label{twodeformationterms}\ee
 \noindent For the first term of (\refeq{twodeformationterms}), 
 \bea \e^{-1/6}\|\pi^\text{med}\Pi_\tau\text{d}\slashed {\mathbb D}_{h_\circ}(\eta,0)\|_{L^2}&\leq&  \e^{-1/6}\| \pi^\text{med}\underline T_\tau^\circ (\eta)\|_{L^2} \ + \ \e^{-1/6}\|\frak t_\tau^\circ(\eta)\|_{L^2} \\ &\leq &C\|\eta^\text{med}\|_{\frak W} + O(\e^M)\|\eta\|_\frak W\eea
 by Definition (\refeq{frakWdef}) and Lemma \ref{fraktcircbounds}. For the second term of (\refeq{twodeformationterms}), Part (B) of Lemma \ref{obstructiondecaycor} applies once more to show this term is exponentially small.
\end{enumerate}
\end{proof}

 The following lemma translates the bounds on the nonlinear terms from Lemma \ref{nonlinearbounds} into bounds in the $\mathcal H, \frak L$-norms. Here, $h_N$ denotes an arbitrary point in $\mathcal H$ with certain hypotheses; eventually, this Corollary will be applied with $h_N$ being the approximate solution at the $N^{\text{th}}$ stage of the gluing iteration. 
 \begin{cor}\label{nonlinearboundscor}
 Suppose that $h_N\in \mathcal H$ satisfies $\|h_N\|_{\mathcal H}\leq C \e^{-1/20}$. Then there is a $C>0$ such that 
 
 \begin{enumerate}
 \item[(A)] $\|\mathbb Q(h)\|_{\frak L}\leq C \e^{2/12-\gamma} \|h\|_{\mathcal H}^2$,
 \item[(B)] $\|\text{d}\mathbb Q_{h_N}(h)\|_{\frak L}\leq C\e^{1/12-\gamma} \|h\|_{\mathcal H}$
 \item[(C)]  If $\|h_1\|_{\mathcal H},\|h_2\|_{\mathcal H}\leq C \e^{-1/20}$, then $\|\mathbb Q(h_1) -\mathbb Q(h_2)\|_{\frak L}\leq C \e^{1/12-\gamma}\|h_1 -h_2\|_{\mathcal H}$
 \end{enumerate}
 \noindent  hold uniformly in $\e,\tau$.

 \end{cor}
 
 \begin{proof} Let $h=(\xi, \ph , a, \psi, b, \mu,u)$. As in the proof of the previous Lemma \ref{HLweightschoice}, the extra weight on the $u$-term in the $\mathcal H$-norm shows that $\|\psi^\text{Re}+u\|_{rH^1}\leq 2\|h\|_{\mathcal H}$. It therefore suffices to prove the corollary with $\psi$ tacitly standing in for $\psi+u$. 
 
 \medskip 
 
 (A) We begin with the proof of Item (A). Let $\mathcal Q=Q_{\Phi} + Q_{A} + Q_{a.\ph} + Q_\mu$ be the latter four nonlinear terms in Proposition (\refeq{nonlinearterms}). Recall that Definition (\refeq{frakWnorm}) means $\e\|\eta\|_{3/2+\underline \gamma}\leq C \e^{1/2} \|\eta\|_{\frak W}$, and Definition \ref{contractionspacesdef} and Theorem \ref{PartImain} imply $\|(\ph,a)\|_{H^{1}_\e}\leq \|(\ph,a)\|_{\mathcal H}$. That the same holds for $(\psi,b)$ is immediate from Definition \ref{contractionspacesdef}. Item (II) and (III) of Lemma \ref{nonlinearbounds} therefore imply $$\|\mathcal Q(h)\|_{L^2}\leq C \e^{1/2-\gamma}\|h\|_{\mathcal H}^2.$$
 \noindent and because the weights in the $\frak L$-norm are larger than the $L^2$-norm by at most $\e^{-1/6}$, it follows that
 $$\|\mathcal Q(h)\|_{\frak L}\leq C \e^{3/12-\gamma}\|h\|_{\mathcal H}^2.$$
 
 Proceeding now to the final term $Q_{\text{SW}}$ of $\mathbb Q_{h_N}=\mathcal Q + Q_{\text{SW}}$, there are four sub-terms coming from $(\chi^+)^2Q_\text{SW}(\ph,a)$, $(\chi^-)^2Q_\text{SW}(\psi,b)$ and the cross-terms. Recall that Proposition \ref{nonlinearbounds} references a partition of unity $\zeta^\pm$, such that $\zeta^+=1$ on the support of $\bold 1^+$.

 We bound two pieces of each coming from the partition of unity $\zeta^\pm$ in Proposition \ref{nonlinearbounds}. By Part (A) of Proposition \ref{nonlinearbounds} and the above, 
$$\|Q_{\text{SW}}(\ph,a)\zeta^+\|_{\frak L}\leq \e^{-1/6}\|\zeta^+Q_{\text{SW}}(\ph,a)\zeta^+\|_{L^2}\leq C\e^{-1/6} \e^{4/12-\gamma} \|(\ph,a)\|_{H^{1,+}_\e}^2 \leq C \e^{2/12-\gamma}\|h\|_{\mathcal H}^2.$$
 \noindent There arguments for the other three terms are identical where $\zeta^+>0$, using that the $H^{1,\pm}_\e$ norms are comparable on the support of $\chi^-$. On the other hand where the second partition function $\zeta^->0$ is positive, the $S^\text{Re}\oplus \Omega$-components of $(\ph,a)$ are exponentially small by Lemma \ref{exponentialdecay} applied with compact sets $K_\e$ whose boundary lies halfway between $\text{supp}(\bold 1^+)$ and $\text{supp}(d\zeta^+)$. In this same region, the $S^\text{Im}\oplus \Omega$-components of $(\psi,b)$ are smaller by a factor of $\e^{1/6}$ by the weight in Definition \refeq{H-norm}. Thus in this region, the right colum of Lemma \ref{nonlinearbounds}  Part (A) provides the necessary bounds (in fact with $\e^{3/12-\gamma})$.

 (B) The derivative $\text{d}_{h_N}\mathbb Q(h)$ is given by the five terms of Proposition \ref{nonlinearterms} viewed as multilinear functions of the arguments, with precisely one argument being chosen from the components of $h$. The proof of the bound in (A) applies equally well in the bilinear setting to show that 
 $$\|\text{d}\mathbb Q_{h_N}(h)\|_{\mathfrak L}\leq C\e^{2/12-\gamma} \|h_N\|_{\mathcal H} \|h\|_{\mathcal H}$$
 \noindent and the assertion follows. More specifically, the requirement that $\|h_N\|\leq C \e^{-1/20}$, means $\e\|\eta_N\|_{3/2+\underline \gamma}\leq C$, hence the assumptions of Parts (B) and (C) of Lemma \ref{nonlinearbounds} are satisfied, and the terms of $\text{d}\mathbb Q_{h_N}$ multi-linear in $h_N$ can be bounded just as in that lemma.  
 
 (C) Follows from Item (B) applied to the family configurations $h_N=th_1 +(1-t)h_2$ and integration.  
  \end{proof}

 \section{The Alternating Iteration}
 \label{section11}
 
 This section combines everything done so far to perform the alternating iteration. As explained in Section \ref{section10}, we will construct a non-linear approximate inverse $\mathbb A: \mathfrak L_{\e,\tau}\to \mathcal H_{\e,\tau}$, so that Eq. (\refeq{Tdef}) is a contraction. More specifically, we will prove the following, for which we recall that $\overline{\mathbb{SW}}_\Lambda=\overline{\mathbb{SW}}-\chi^- \e^{-1}{\Lambda(\tau)\Phi_\tau}$ was defined following Eq. (\refeq{Tdef}).

\begin{prop}\label{fixedpoints} There exist $\e_0, \tau_0$ sufficiently small such that for $\e< \e_0$ and $\tau\in (-\tau_0, \tau_0)$, the following hold. There is a closed ball $\mathcal V_{\e,\tau}\subseteq \mathcal H_{\e,\tau}$ around $0$ such that for $h\in \mathcal V_{\e,\tau}$ and $N\in \N$ the map $\mathbb T$ defined in (\refeq{Tdef}) satisfies the following. 
\medskip
\begin{enumerate}
\item[(A)] The restriction $\mathbb T:\mathcal V_{\e,\tau}\to \mathcal V_{\e,\tau}$ is a $C^1$ family (in $\e,\tau$) of continuous functions of $h\in \mathcal V_{\e,\tau}$.
\medskip
\item[(B)] $\| \overline{\mathbb{SW}}_\Lambda(\mathbb T^N(h))\|_{\frak L}\leq \delta^N \|\overline{\mathbb{SW}}_\Lambda(h)\|_{\mathfrak L},$
\medskip 
\item[(C)] $\|\mathbb T(h_1) - \mathbb T(h_2)\|_\mathcal H \leq C\sqrt{\delta} \|h_1 - h_2\|_\mathcal H$ 

\end{enumerate}
\smallskip 

\noindent where $\delta = \e^{1/48}$. In particular, $T$ is a contraction for $\e$ sufficiently small.  
\end{prop}

\noindent Note the proposition implicitly uses the trivialization from Lemma (\refeq{SWtrivializations}) to conflate an open subset of $T\mathbb H^1_\e$ with $\mathbb H^1_\e$ so that $0\in \mathcal V_{\e,\tau}$ makes sense. The next three subsections each carry out one stage of the alternating iteration, by constructing the three parametrices $P_\xi, P^-, P^+$ respectively which are combined into  $\mathbb A$ (cf. Subsection \ref{contractionsubspacessection}). The proposition is proved in the final subsection, and Theorem \ref{maina} is deduced as a consequence in Section \ref{section12}.

  \subsection{The Deformation Step}
 \label{section11.1}
 
   This subsection constructs the deformation parametrix $P_\xi$, and establishes the first of the three induction steps in the cyclic iteration following  Eq. (\refeq{Fetadiffeo})). This step shows that a combination of deformation of $\mathcal Z_{\tau}$ and singular spinor in $\mathcal X_\tau$ can be jointly chosen to cancel the obstruction components of an error term, {\it without} the error term growing much larger. We emphasize once more that the tangential smoothing gauge (Section \ref{section6.3}) is essential in achieving the latter.

 The following proposition is applied to the error terms inductively, beginning with the error $\frak e_1$ of the initial approximate solutions $(\Phi^{(1)}_{\e,\tau}, A^{(1)}_{\e,\tau})$ in Theorem \ref{PartImain}. For the remainder of Section \ref{section11}, we fix, once and for all, a choice of $(\e,\tau)\in (0,\e_0)\times (-\tau_0, \tau_0)$ and omit this dependence from the subscripts in the notation where no confusion will arise. As before, $\e_0$ and $\tau_0$ are allowed to decrease a finite number of times of the course of the proofs.  

Let $\Pi^\perp: L^2\to \text{\bf Ob}^\perp(\mathcal Z_\tau)$ be the $L^2$-orthogonal projection where the latter is as defined preceding Corollary \ref{solvingcokernel}.  With this notation, the projection to the obstruction is written $\Pi_\tau=(\Pi^\perp, \pi_\tau)$ in the orthogonal decomposition in Definition \ref{Obdefn} where $\pi_\tau$ is the $L^2$-orthogonal projection to the span of $\Phi_\tau$. Define the deformation parametrix
\be  P_\xi: \frak L \lre \frak W\hspace{2cm} P_\xi :=  \left( \underline T_\tau^\circ,\slashed D_{A_\tau}\right)^{-1}\circ \Pi^\perp \circ \bold 1^-  \label{Pxi} \ee

\noindent where $(\underline T_\tau^\circ,\slashed D_{A_\tau})$ is the map from Definition \ref{frakWdef}, with the map denoted $\text{ob}$ now kept implicit in the notation.

\begin{prop} \label{deformationstep} $P_\xi$ is a linear operator uniformly bounded in $\e,\tau$ and satisfies the following property. If for $N\in \N$,  $h_N \in \mathcal H$ is a configuration with 

\begin{enumerate}
\item[(I)] $\|h_N\|_{\mathcal H}\leq C\e^{-\gamma}\|\frak e_1\|_\frak L$ 
\item[(II)] $\overline{\mathbb{SW}}_\Lambda (h_N)=\frak e_N$   \ (resp. $ \text{d}\overline{\mathbb{SW}}_{h_1} (h_N)=\frak e_N $) where $\|\frak e_N\|_\frak L\leq C \delta^{N-1}\|\frak e_1\|_\frak L$, 
\end{enumerate}
\noindent then the updated configuration \be h_N'=( \text{Id} - P_\xi \circ \overline{\mathbb{SW}}_\Lambda)h_N\label{hNprime}\ee
satisfies 
$$ \overline{\mathbb{SW}}_\Lambda (h'_N)=(1-\Pi)\frak e_N'  \ + \ \frak g'_N \ + \  \lambda'_N \e^{-1} \Phi_{\tau} \ + \ \frak e_{N+1}$$
(resp. the same for $\text{d}\overline{\mathbb{SW}}_{h_1}$), where $\frak e_N', \frak e_{N+1}\in \frak L$, and  $\lambda \in \R$ obey 
\begin{enumerate}
\item[(1)] $\|\frak e_N'\|_\frak L\leq C \e^{-\gamma} \|\frak e_N\|_{\frak L}$, \ and \  $\Pi(\frak e'_N)=0$. 
\item[(2)] $\|\frak g_N'\|_{\frak L}\leq C \e^{-\gamma}\|\frak e_N\|_\frak L$ \ and \ $\frak g_N'=\frak g_N'\bold 1^+$.
\item[(3)] $\|\frak e_{N+1}\|_{\frak L}\leq C \e^{1/48} \delta \|\frak e_N\|_\frak L$
\item[(4)] $|\lambda'_N |\leq C \e^{1-\gamma}\|\frak e_N\|_\frak L$.
\end{enumerate} 
Moreover, $h_N'$ continues to satisfy (I). 

\end{prop}

\begin{proof} The error term may be written $\frak e_N= e_N + f_N + g_N +\Psi_N$ where 
\bea
g_N &:=& \frak e_N \bold 1^+ \hspace{3.7cm} \  f_N:=\frak e_N^\text{Im}\bold 1^-  + \pi^\text{med}\Pi(\frak e^\text{Re}_N\bold 1^-) \\
\Psi_N &:= & \pi^\text{high}\Pi(\frak e^\text{Re}_N\bold 1^-) \hspace{2.5cm} e_N:=(1-\Pi)\frak e_N^\text{Re}\bold 1^- + \pi^\text{low}\Pi(\frak e_N^\text{Re} \bold 1^-) + \pi_\tau(\frak e_N^\text{Re} \bold 1^-).
\eea

\noindent These terms correspond to some of the term in the definition of the $\frak L$-norm in (\refeq{contractionspacesdef}).  That definition implies that $\|g_N\|_{L^2}\leq C \e^{1/12-\gamma}\|\frak e_N\|_\frak L$ \ , \ $\|f_N\|_{L^2}\leq C \e^{1/6}\|\frak e_N\|_\frak L$ and $\|\Pi(e_N + f_N + \Psi_N)\|_{L^2_{-\nu}} \leq C \e^{-\nu/2} \|\frak e_N\|_\frak L$. Moreover, Corollary  \ref{obstructiondecaycor}(B) applied with $M=10$ implies $\|\Psi_N\|_{L^2}\leq C \e^{10}\|\frak e_N\|_\frak L$.  

With $P_\xi$ as in (\refeq{Pxi}), set $$(\eta,u) :=P_\xi (\Pi^\perp(e_N + f_N + \Psi_N))$$ so that 
\begin{eqnarray}
\underline T^\circ(\eta^\text{low}) \ &=&\Pi^\perp(e_N)  \hspace{.1cm}\ =  \  \pi^\text{low}\Pi^\perp(\frak e^\text{Re}_N\bold 1^-)\label{Tetacancelling}\\
\underline T^\circ(\eta^\text{med})&=&\Pi^\perp(f_N)  \hspace{.1cm}  \ = \    \pi^\text{med}\Pi^\perp(\frak e^\text{Re}_N\bold 1^-)\label{Tetacancelling2}\\
\slashed D_{A_\tau}u&=& \Pi^\perp(\Psi_N)  \ = \ \pi^\text{high}\Pi^\perp(\frak e^\text{Re}_N\bold 1^-)\label{Diracucancelling}
\end{eqnarray}
\noindent where $\eta=\eta^\text{low}+\eta^\text{med}$. The above estimates and the uniform bounds on $(\underline T^\circ, \slashed D)^{-1}$ coming from Corollary \ref{quantitativehigherorder}  (the version in Corollary \ref{deformationsmainamd}) and Corollary \ref{solvingcokernel}  show that $\eta,u$ satisfy \be \|\eta\|_\frak W\leq C\|\frak e_N\|_\frak L  \hspace{2cm}\|u\|_\frak W\leq C \e^{8}\|\frak e_N\|_\frak L, \label{etaxiubounde8}\ee

\noindent where the $\frak W$ norm is as in Definition (\ref{frakWdef}).

We now proceed to calculate  $\overline{\mathbb{SW}}_\Lambda (h'_N)$, where $h_N'=h_N - (\eta , u)$ is as in (\refeq{hNprime}). First, with $\xi=\e\eta$ as previously, we compute: 
\begin{eqnarray}
\text{d}\overline{\mathbb {SW}}_{h_1}(\xi,0,0,0,0)&=&\text{d}\slashed{\mathbb D}_{h_\circ}(\eta,0) + \Xi^+(\eta) + \Xi^-(\eta)\nonumber  \\
&=& \underline T^\circ(\eta) + \frak t^\circ(\eta) + (1-\Pi)\text{d}\slashed{\mathbb D}_{h_\circ}(\eta) + \pi_\tau \text{d}\slashed{\mathbb D}_{h_\circ}(\eta) + \Xi^+(\eta) + \Xi^-(\eta)\label{Tcircerror}\\
\text{d}\overline{\mathbb {SW}}_{h_1}(0,0,0,0,u) &=& \slashed D_{A_\tau}(\chi^-u) + K_1\nonumber \\
&=& \slashed D_{A_\tau}u - (1-\chi^-)\slashed D_{A_\tau}u + d\chi^-.u + K_1(\chi^-u) \label{slashedDerror}
\end{eqnarray}
\noindent where we recall that $\Pi(\text{d}\slashed{\mathbb D}(\eta,0)= \underline T^\circ(\eta) + \frak t^\circ(\eta)$ by Definition \ref{tametruncationdef}, and $\Xi^\pm$ are as in Section \ref{section9}. Additionally, \be K_1=\mathcal L_{(\Phi^{(1)}, A^{(1)})} -\mathcal L_{(\Phi_\tau, A_\tau)} \label{K1defdeformation}\ee
denotes the difference in the linearizations at the model solution and the eigenvector, which is exponentially small by Corollary  \ref{corollaryexpdecay} (this was used in step (2) in the proof of Lemma \ref{dSWbounded}).

With these, we now compute: 
\begin{eqnarray}
 \overline{\mathbb{SW}}_\Lambda(h_N')&=&\overline{\mathbb{SW}}(h_1) +  \text{d}_{h_1}\mathbb{SW}(h_N)  \ -  \  \text{d}_{h_1}\mathbb{SW}(\xi, u)\ + \   \mathbb Q(h_N') \ - \ (\mu + \Lambda) \chi^-\frac{\Phi_\tau}{\e}\nonumber \\
 &=&   \ \overline{\mathbb{SW}}_\Lambda(h_N)  \  -  \  \text{d}_{h_1}\mathbb{SW}(\xi, u)\ + \ \mathbb Q(h'_N)-\mathbb Q(h_N)\nonumber \\
 &=& \frak e_N   \ - \  \underline T^\circ(\eta^\text{low}) \  - \  \underline T^\circ(\eta^\text{med})  \ - \  \slashed D_{A_\tau}u   \ - \  (1-\Pi)\text{d}\slashed{\mathbb D}_{h_0}(\eta) - \pi_\tau \text{d}\slashed{\mathbb D}_{h_\circ}(\eta)\label{1029} \\ & & \ + \ \Xi^+(\eta) \ + \   \text{\textcyr{Zh}}( \eta,u) \nonumber  \\
 &= & g_N \ + \ \Xi^+(\eta)  \ + \ \frak e^\text{Im}_N\bold 1^-   \ + \ (1-\Pi)(\frak e_N^\text{Re}\bold 1^-+\text{d}\slashed {\mathbb D}_{h_0}(\eta)) + \pi_\tau(\frak e_N^\text{Re} \bold 1^- + \text{d}\slashed{\mathbb D}_{h_\circ}(\eta))\nonumber
 \\ & & \ + \ \text{\textcyr{Zh}}( \eta ,u)\label{1030} \end{eqnarray}
 \noindent where Eq. (\refeq{1029}) is obtained by substituting (\refeq{Tcircerror}) and (\refeq{slashedDerror}) where the ``smaller'' terms are lumped into 
 \smallskip \be \text{\textcyr{Zh}}( \eta ,u):= \left(\mathbb Q(h'_N)-\mathbb Q(h_N)\right)  \ -  \ \Xi^-(\eta) \ - \ \frak t^\circ(\eta) \   +\ (1-\chi^-)\slashed D_{A_\tau}u  \ - \  d\chi^-.u  \ - \  K_1(\chi^-u)\label{Zh}\ee\smallskip
 \noindent and Eq. (\refeq{1030}) is the result of substituting Eqns. (\refeq{Tetacancelling})--(\refeq{Diracucancelling}) into Eq. (\refeq{1029}). Note, in particular, that the $\pi^\text{low},\pi^\text{med}, \pi^\text{high}$ terms from the original form of $\frak e_N$ have cancelled via our choice of $(\eta,u)$ in  Eq. (\refeq{Tetacancelling}--\refeq{Diracucancelling}). 
 
After this cancellation, we regroup the new error term (\refeq{1030}), as follows. Set 
 \bea
\frak e_N'&:=& \frak e^\text{Im}_N\bold 1^-   \ + \ (1-\Pi)(\frak e_N^\text{Re}\bold 1^-+\text{d}\slashed {\mathbb D}_{h_0}(\eta))\\
\frak g_N' &=& g_N \ + \ \Xi^+(\eta)  
\\
\lambda_{N}' &:=&  \e \br  \pi_\tau(\frak e_N^\text{Re} \bold 1^- + \text{d}\slashed{\mathbb D}_{h_0}(\eta)) \ , \ \Phi_\tau \kt_{L^2}\\
\frak e_{N+1}&:=& \text{\textcyr{Zh}}(\eta,u).
 \eea
 
 \noindent It now suffices to show that these satisfy the conclusions (1)--(4) of the proposition. 
 
 Beginning with (3), we have that 
 \be \|\text{\textcyr{Zh}}(\eta,u) \|_{\frak L}\leq C \e^{1/48}\delta \|\frak e_N\|_{\frak L}. \label{Zhbound}  \ee
 By the bound $\|u\|_\frak W\leq C \e^{8}\|\frak e_N\|_\frak L$ from Eq. (\refeq{etaxiubounde8}) above and Corollary \ref{fraktcircbounds}, all the terms of (\refeq{Zh}) except those involving $\mathbb Q, \Xi^+, \Xi^-$ are bounded by, say, $\e^{5}\|\frak e_N\|_\frak L$. Applying Corollary \ref{nonlinearbounds} to the $\mathbb Q$ terms, and repeating the argument that led to the bound  $\|\Xi^-(\eta)\|_{L^2} \leq C \e^{5/12-\gamma} \|\eta\|_{\frak W}$ in (\refeq{Ximinusbound}) during the proof of Lemma  (\ref{frakLbounded})  shows that 
 $\|\text{\textcyr{Zh}}(\eta,u) \|_{\frak L}\leq C \e^{1/12-\gamma} \|\frak e_N\|_{\frak L}$. Since $\gamma<<1$ and $\e^{1/48}\delta =\e^{1/24}$, (\refeq{Zhbound}) follows, which is conclusion (3). 
 
 (4) follows from the definition of $\lambda_N'$ above and Cauchy-Schwartz, since $\|\frak e_N^\text{Re}\bold 1^-\|_{L^2}\leq \|\frak e_N\|_\frak L$ and $$\|\text{d}\slashed{\mathbb D}_{h_\circ}(\eta)\|_{L^2}\leq C\|\eta\|_{1/2}\leq C \|\eta\|_\frak W$$
 
 \noindent by Proposition \ref{effectivesupportdD}. Combined with Eq. (\refeq{etaxiubounde8}), this is yields the conclusion (4), in fact without any factor of $\gamma$. 
 
 For (1)--(2), the fact the $\Pi(\frak e_N')=0$  and $\frak g'_N=\frak g_N'\bold 1^+$ are immediate from their definitions above (recall that the $\frak e^{\text{Im}}$ components are orthogonal to the obstruction by definition). It remains to show that the asserted bounds hold. To re-iterate the cancellation that led to (\refeq{1030}) beginning from the definition of $\frak e_N'$ above, 
 
  \bea 
  \frak e_N'&=& \frak e^\text{Im}_N\bold 1^-   \ + \ (1-\Pi)(\frak e_N^\text{Re}\bold 1^-+\text{d}\slashed {\mathbb D}_{h_0}(\eta))\\
  &=&  \frak e^\text{Im}_N\bold 1^-   \ + \ \frak e_N^\text{Re}\bold 1^- \ + \ \text{d}\slashed{\mathbb D}_{h_\circ}(\eta)   \ - \ \Pi(\frak e_N^\text{Re}\bold 1^-+\text{d}\slashed {\mathbb D}_{h_\circ}(\eta)) \\
  &=& \frak e_N^\text{Im}\bold 1^- \ + \  \frak e_N^\text{Re}\bold 1^-  \ + \ \text{d}\slashed{\mathbb D}_{h_\circ}(\eta) \  - \ \left( \frak t^\circ(\eta )  \ + \ \Psi_N \ + \ \lambda_N' \e^{-1}\Phi_\tau\right), 
  \eea

 \noindent where the last line uses the definition $\Pi(\text{d}\slashed {\mathbb D}_{h_\circ})=\underline T^\circ +\frak t^\circ$, and cancels the low and medium modes via the choice of $\eta$ in (\refeq{Tetacancelling2}). 
 
Because $\frak t^\circ(\eta), \Psi_N$ are $O(\e^8)$ these may be safely ignored as in the proof of (3). Since $\frak e^\text{Im}_N$ is unchanged by the cancellation, and $\pi^\text{med}\Pi(\frak e_N')=0$, it suffices to bound the three terms on the top line in the definition of the $\frak L$ norm (\ref{frakLnorm}). For this, one has, 
 \bea
 \| \text{d}\slashed{\mathbb D}_{h_0}(\eta) \bold 1^+\|_{L^2} \ &\leq&  \ \ \  C\e^{1/12-\gamma} \|\eta\|_\frak W \ \hspace{2.6cm} \leq \  C \e^{1/12-\gamma}\|\frak e_N\|_{\frak L}\\
  \|(\frak e_N^\text{Re} +  \ \text{d}\slashed{\mathbb D}_{h_0}(\eta)) \bold 1^-\|_{L^2}\ &\leq& \ \  \   C \|\frak e_N\|_\frak L  \ \  \ \  + \ \  \  \ C  \|\eta\|_\frak W \ \ \  \ \ \ \ \ \ \ \  \leq  \ \  C\|\frak e_N\|_{\frak L}\\
   \|(\frak e_N^\text{Re} +  \ \text{d}\slashed{\mathbb D}_{h_0}(\eta)) \bold 1^-\|_{L^2_{-\nu}}&\leq& \   C\e^{-\nu/2} \|\frak e_N\|_\frak L  \  +C \e^{-\nu/2-\gamma}\|\eta\|_{\frak W}  \ \ \  \leq  \ \  C\e^{-\gamma} \e^{-\nu/2}\|\frak e_N\|_{\frak L} \eea
 
 \noindent where in each line, the first inequality is the definition of the $\frak L$-norm is used in conjunction with the bounds (\refeq{Xiplusbound}), (\refeq{unweighteddDbound}), and (\refeq{weighteddDbound}) from the proof of Lemma \ref{frakLbounded}; the second inequality in each line substitutes the hypotheses and Eq. (\refeq{etaxiubounde8}). The first three lines show conclusion (1), excluding the $\lambda_N' \e^{-1}\Phi_\tau$ term. For this final rank 1 term, substituting the conclusion of (4) along with the fact that $\Phi_\tau \in L^2 \cap L^2_{-\nu}$ and is polyhomogeneous of growth $O(r^{1/2})$ as $r\to 0$ shows the same bounds on the $\frak L$ norm hold for this term. Conclusion (1) then follows.

 Conclusion (2) follows directly from the definition of $\frak g_N'$ and Lemma \ref{relatingvariationsinside}, which together show that
 
   $$  \| g_N + \Xi^+(\eta) \|_{L^2} \ \leq   C\e^{1/12-\gamma} \|\frak e_N\|_\frak L \ + \ C\e^{1/12-\gamma} \|\eta\|_\frak W \ \leq \  C \e^{1/12-\gamma}\|\frak e_N\|_{\frak L}.$$
   
   \noindent Reweighting the left-hand side to be the $\mathfrak L$-norm is precisely Conclusion (2).

 It remains to see that (I) and (II) hold. That (I) holds for $h_N'$ is immediate from the bounds $\|\eta\|_\frak W \ + \ \|u\|_\mathcal X\leq C \|\frak e_N\|_{\frak L}$ in Eq. (\refeq{etaxiubounde8}), and the fact that it holds for $h_N$, via the triangle inequality. The new error $\frak e_N'$ was defined precisely so that (II) is true. Finally, the proof of the (resp. $\text{d}\overline{\mathbb {SW}}$) statements is identical, omitting any mention of $\mathbb Q$ and $\Lambda$. 
 \end{proof}

 \subsection{The Outside Step}
 \label{section11.2}
 This subsection covers the second of the three stages of the cyclic iteration following (\refeq{Fetadiffeo}). Now that the leading error term $\frak e_N'$ is orthogonal to the obstruction (except the $1$-dimensional span of $\Pi(\Phi)$), solving in the outside can proceed using Lemma \ref{mappingpropertiesI} and Proposition \ref{solvingoutsideim}. 

Define the outside parametrix $P^-$ by 
\begin{eqnarray}  P^-: \frak L \lre \mathcal H^- \oplus \R \hspace{2cm} P^- &:=&  \left(\mathcal L^{-1}_{(\Phi_\tau, A_\tau)} (1-\Pi^\perp) \ , \ -\e \pi_\tau \right) \circ \bold 1^-  \label{P-}  \end{eqnarray}
\noindent where $\pi_\tau: \mathfrak L\to \R$ is understood to mean the coefficient of $\Phi_\tau$. Notice that the first component indeed lands in $\mathcal H^-$ by Definition (\refeq{contractionspacesdef}).

\begin{prop}\label{outsidestep}$P\!^-$ is a linear operator uniformly bounded in $\e,\tau$ and satisfies the following property. If for $N\in \N$,  $h'_N \in \mathcal H$ is a configuration satisfying the conclusions of Proposition \ref{deformationstep}, then the updated configuration 
\be h_N''= (\text{Id} - P^- \circ \overline{\mathbb{SW}}_\Lambda)h_N'\label{hNprimeprime}\ee
\noindent satisfies 
$$ \overline{\mathbb{SW}}_\Lambda (h''_N)=\frak e_N'' \bold 1^+ \ + \ \frak e_{N+1}$$

\noindent (resp. the same for $\text{d}_{h_1}\overline{\mathbb{SW}}$), where $\frak e_N'', \frak e_{N+1}\in \frak L$ obey 
\begin{enumerate}
\item[(1')] $\|\frak e_N''\|_\frak L\leq C \e^{-\gamma} \|\frak e_N\|_{\frak L}$
\item[(2')] Item (2) from Proposition \ref{deformationstep} continues to hold. 
\end{enumerate} 
Moreover, $h_N''$ continues to satisfy (I) from Proposition \ref{deformationstep}. 

\end{prop}
\begin{proof}

Beginning similarly to the proof of Proposition \ref{deformationstep}, write $\frak e_N' +\frak g_N'  = e'_N + f'_N + g'_N $, where 
\be g_N':=  \frak g_N'\bold 1^+ \hspace{2cm} f_N':=(\frak e'_N)^\text{Im} \bold 1^-\hspace{2cm} e_N' :=(\frak e_N')^\text{Re}.\label{efgprimed} \ee

\smallskip
\setlength{\lineskiplimit}{-100pt}
\noindent Conclusions (1)--(4) of Proposition \ref{deformationstep} mean that $\Pi(e_N')=0$, that $\|e'_N\|_{L^2_{-\nu}}\leq C \e^{-\gamma}\e^{-\nu/2}\| e'_N\|_\frak L$,  for both $\nu=0,\nu^-$, and that $\|f_N\|_{L^2}\leq C \e^{1/6} \|\frak e_N\|_\frak L$ and  \ $\|g_N\|_{L^2}\leq C \e^{1/12-\gamma}\|\frak e_N\|_\frak L$, where $\frak e_N$ is the original error from Proposition \ref{deformationstep}. 

\setlength{\lineskiplimit}{100pt}
Set $$(\psi,b,\mu):= P^-(e'_N + f'_N  + \lambda_N' \e^{-1}\Phi_\tau)$$ so that
\bea
\slashed D_{A_\tau}\psi^\text{Re}&=& e_N'\\
\mathcal L_{(\Phi_\tau, A_\tau)}^\text{Im}(\psi^\text{Im},b)&=& f_N'\\
\mu&=& -\e \pi(\lambda_N' \e^{-1}\Phi_\tau)=-\lambda_N'
\eea
\noindent where $\psi=(\psi^\text{Re}, \psi^\text{Im})\in S^\text{Re}\oplus S^\text{Im}$. 

We now show that $P^-$ is uniformly bounded. Lemma \ref{mappingpropertiesI} and Proposition \ref{solvingoutsideim} and the above bounds on $e'_N, f'_N$ show that these unique solutions (where $\br \psi^\text{Re}, \Phi_\tau\kt_{L^2}=0$), satisfy 
\begin{eqnarray}
\|\psi^\text{Re}\|_{rH^1_e}  & \leq & C \|e_N'\|_{L^2} \ \ \ \   \ \ \ \  \ \  \leq \   C \e^{-\gamma}\|\frak e_N\|_\frak L \label{psirebound} \\
\e^{\nu/2}\|\psi^\text{Re}\|_{r^{1+\nu}H^1_e}&\leq & C\e^{\nu/2}\|e_N'\|_{L^2_{-\nu}} \  \  \hspace{.18cm} \leq \ C \e^{-\gamma}\|\frak e_N\|_{\frak L}\label{psirealweightedbound} \\ 
 \e^{-1/6}\|(\psi^\text{Im},b)\|_{\mathcal H^{1,-}_\e} &\leq& C\e^{-1/6} \|f_N'\|_{L^2}     \  \ \  \leq  \ C \e^{-\gamma}\|\frak e_N\|_\frak L \label{psiimbound}\\
\e^{-1}|\mu |&\leq & C \|\lambda_N' \e^{-1}\Phi_\tau\|_{L^2}   \     \hspace{.28cm}\leq  \ C \e^{-\gamma}\|\frak e_N\|_\frak L \label{mubound}.
\end{eqnarray}

\noindent To explain further how to obtain the bounds in the middle column, the condition that $\br \psi^\text{Re}, \Phi_\tau\kt_{L^2}=0$ means that the $\nu=0$ version of Lemma \ref{mappingpropertiesI} applies without the projection term, which yields (\refeq{psirebound}). In turn,  (\refeq{psirealweightedbound}) follows from Lemma \ref{mappingpropertiesI} taking $\nu=\nu^-$.  (\refeq{psiimbound}) and (\refeq{mubound}) are immediate from Proposition \ref{solvingoutsideim} and the definition of $\mu$. The final column follows immediately from the bounds on (\refeq{efgprimed}) and conclusion (4) of Proposition \ref{deformationstep}. 

Uniform boundedness of $P^-$ is the statement that $\|\frak e_N' + \lambda_N' \e^{-1}\Phi_\tau\|_\frak L$ dominates each term in the middle column (since $P^-$ ignores $\frak g_N'$). Indeed, $\frak e_N'$ is, by its construction in the proof of Proposition \ref{deformationstep}, $L^2$-orthogonal to $\Phi_\tau$. Therefore, with $e_N'$ as defined in Eq. (\refeq{efgprimed}),
\begin{eqnarray}
\|e_N'\|_{L^2}  + \| \lambda_N' \e^{-1}\Phi_\tau\|_{L^2} &\leq& \ \ \ \  \ \  \|\frak e_N' + \lambda_N' \e^{-1}\Phi_\tau\|_{L^2}\hspace{3.0cm} \  \leq C \|\frak e_N' + \lambda_N' \e^{-1}\Phi_\tau\|_\frak L \ \  \ \ \ \ \ \label{L2orthogonalitybounds} \\
\e^{\nu/2}\|e_N' \bold 1^-\|_{L^2_{-\nu}}&\leq& \e^{\nu/2} \| (e_N' + \lambda_N' \e^{-1}\Phi_\tau) \bold 1^-\|_{L^2_{-\nu}}  + \|(\lambda_N' \e^{-1}\Phi_\tau) \bold 1^- \|_{L^2_{-\nu}}\leq C \|\frak e_N' + \lambda_N' \e^{-1}\Phi_\tau\|_\frak L. \nonumber
\end{eqnarray}
The first of these is simply orthogonality along with fact that the $\frak L$-norm dominates the $L^2$-norm. The second is the triangle inequality, then invoking the first one along with the fact that the $L^2_{-\nu}$ and $L^2$-norms are equivalent on the 1-dimensional span of $\Phi_\tau$. That $\e^{-1/6}\|f_N'\|_{L^2}$ in (\refeq{psiimbound}) is likewise bounded by the right side of (\refeq{L2orthogonalitybounds}) is immediate from Definition \ref{frakLnorm} of the $\frak L$-norm. This completes the claim that $P^-$ is uniformly bounded.

We now proceed to calculate  $\overline{\mathbb{SW}}_\Lambda (h''_N)$ where $h_N''=h'_N - (\psi,b , \mu)$ is as in (\refeq{hNprimeprime}). 
First, 
\begin{eqnarray} \text{d}\overline{\mathbb {SW}}_{h_1}(\psi,b,\mu) &=& \mathcal L_{(\Phi_\tau, A_\tau)} \chi^- (\psi,b) \ + \ K_{1}(\chi^-\psi,\chi^-b) \ - \ \mu \e^{-1}\chi^-\Phi_\tau  \nonumber \\
&=&  \slashed D_{A_\tau}\psi^\text{Re} \ + \ 
\mathcal L_{(\Phi_\tau, A_\tau)}^\text{Im}(\psi^\text{Im},b)  \ - \ \mu \e^{-1}\Phi_\tau \nonumber \\ &  &  \ + \  d\chi^-(\psi,b) \ + \ K_1(\chi^-\psi, \chi^-b)  \ + \ (1-\chi^-)\mu \e^{-1}\Phi_\tau, \label{outsidederivative}\end{eqnarray}

\noindent where $K_1$ is as in Eq. (\refeq{K1defdeformation}), and we have used the fact that $\chi^-=1$ on the support of $\bold 1^-$. Using the conclusion of Proposition \ref{deformationstep} and the definitions (\refeq{efgprimed}), then substituting (\refeq{outsidederivative}) yields
 \medskip 
\begin{eqnarray}
 \overline{\mathbb{SW}}_\Lambda(h_N'')&=& \text{d}_{h_1}\mathbb{SW}(h_N')  \ -  \  \text{d}_{h_1}\mathbb{SW}(\psi, b,\mu)\ + \   \mathbb Q(h_N'') \ - \ (\mu_N + \mu + \Lambda) \chi^-\frac{\Phi_\tau}{\e}\nonumber  \medskip \\
 &=&   \ \overline{\mathbb{SW}}_\Lambda(h_N')  \  -  \  \text{d}_{h_1}\mathbb{SW}(\psi,b, \mu)\ + \ \mathbb Q(h_N'')-\mathbb Q(h_N')\nonumber \medskip  \\
 &=&e_N'  \ + \ f_N' \ + \ g_N' \ + \   \lambda_N' \e^{-1}\Phi_\tau   \ - \   \slashed D_{A_\tau}\psi^\text{Re} \ - \ 
\mathcal L_{(\Phi_\tau, A_\tau)}^\text{Im}(\psi^\text{Im},b)  \ + \ \mu \e^{-1}\Phi_\tau \ \ \  \nonumber   \medskip  \\
& &  \ + \ \left(\mathbb Q(h_N'')-\mathbb Q(h_N') \right)  \ + \  d\chi^-(\psi,b) \ + \ K_1(\chi^-\psi,\chi^-b)  \ + \ (1-\chi^-)\mu \e^{-1}\Phi_\tau  \ + \ \frak e_{N+1}\nonumber \medskip  \\ 
 &= & g_N'  + \ \left(\mathbb Q(h_N'')-\mathbb Q(h_N') \right)  \ + \  d\chi^-(\psi,b) \ + \ K_1(\chi^-\psi,\chi^-b)  \ + \ (1-\chi^-)\mu \e^{-1}\Phi_\tau \ + \ \frak e_{N+1}. \nonumber
 \end{eqnarray}

\noindent Then (re)-define 
\begin{eqnarray}
\frak e_N''&:=&g_N' \ + \ d\chi^-(\psi,b) \ + \ (1-\chi^-)\mu \e^{-1}\Phi_\tau\label{eNprimeprime}\\
\frak e_{N+1}&\mapsto& \frak e_{N+1}  \ + \  \left(\mathbb Q(h_N')-\mathbb Q(h_N'') \right) \ + \ K_1(\chi^-\psi,\chi^-b).
\end{eqnarray}

\noindent Notice that there has again been a cancellation of the main error terms $e_N', f_N'$, and that $\frak e_N''$ now includes the leading-order alternating error term $d\chi^-(\psi,b)$, supported in the inside $Y^+$.  

To complete the proof of the proposition, we show $\frak e_N''$ and $\frak e_{N+1}$ satisfy conclusions (1)--(2). That (1) holds for the first term of $\frak e_N''$ in (\refeq{eNprimeprime}) follows from the triangle inequality. The bound on $g'_N$ from \refeq{efgprimed} is already as needed. The bounds on $e_N'$ listed following  Eq. (\refeq{efgprimed}) show precisely that $e_N'$ has $\e^{1/2}$-effective support (Definition \ref{effectivesupport}), hence Lemma  \refeq{polynomialdecayreal} applies to show that 
\be \|d\chi^-.\psi^\text{Re}\|_{L^2}\leq C \e^{1/12-\gamma} \| e_N'\|_{L^2}\leq C \e^{1/12-2\gamma} \|\frak e_N\|_{\frak L}. \label{alternatingerrorinside} \ee
\noindent Note that the $(1-\chi^-)\frak e_N'$ term in Lemma \ref{polynomialdecayreal} vanishes because $\chi^-=1$ on $\text{supp}(\bold 1^-)$.  (\refeq{alternatingerrorinside}) also holds for $(\psi^\text{Im},b)$ simply because of the $\e^{1/6}$ weight on these components in (\refeq{psiimbound}). Finally, for the third term of $\frak e_N''$ in (\refeq{eNprimeprime}) , notice that direct integration using the fact that $|\Phi_\tau|\leq C r^{1/2}$ shows that  $\|(1-\chi^-)\Phi_\tau\|_{L^2}\leq C \e^{1-\gamma}$, and (\refeq{mubound}) therefore shows the third term of $\frak e_N''$ is bounded by the right hand side of (\refeq{alternatingerrorinside}) with an extra factor of $\e^{11/12-\gamma}$. Combining these three shows that 

$$\|\frak e_N''\|_\frak L\leq C \e^{-\gamma}\|\frak e_N\|_\frak L$$

\noindent and $\frak e_N''=\frak e_N''\bold 1^+$ because $\chi^-=1$ outside $\text{supp}(\bold 1^+)$. 

Using the bounds (\refeq{psirebound})--(\refeq{psiimbound}), the terms involving $\mathbb Q$ can be bounded identically to in the proof of Proposition \ref{deformationstep}. As in the proof of Lemma \ref{dSWbounded}, $K_1$ is exponentially small, hence negligible. Conclusion (2) follows after increasing $\gamma$ slightly. 

Finally, that (I) continues to hold for $h_N''$ and of the (resp. $\text{d}\overline{\mathbb{SW}}_{h_1}$) statements follows identically to in Proposition \ref{deformationstep} using the uniform boundedness of $P^-$.  
\end{proof}

 \subsection{The Inside Step}
 
 \label{section11.3} The section completes the third stage of the cycle following (\refeq{Fetadiffeo}) by constructing $P^+$. Define 
 \begin{eqnarray}P^+: \frak L\lre \mathcal H^+ \hspace{2cm}P^+&:=& \mathcal L^{-1}_{(\Phi_1, A_1)}\bold 1^+\label{P+}.\end{eqnarray}
 \noindent Definition \ref{contractionspacesdef} ensures that $P^+$ lands in $\mathcal H^+$. 
 
 \begin{prop}\label{insidestep}$P^+$ is a linear operator uniformly bounded in $\e,\tau$ and satisfies the following property. If for $N\in \N$,  $h''_N \in \mathcal H$ is a configuration satisfying the conclusions of Proposition \ref{outsidestep}, then the updated configuration 
 \smallskip
\be h_{N+1}= (\text{Id} - P^- \circ \overline{\mathbb{SW}}_\Lambda)h_N''\label{hNprimeprimeprime}\ee

\smallskip
\noindent (resp. the same for $\text{d}\overline{\mathbb {SW}}_{h_1})$ satisfies the hypotheses of Proposition \ref{deformationstep} with $N'=N+1$. 

\end{prop} 
\begin{proof}
Uniform boundedness is immediate from the definition of the $\mathcal H^+$-norm in Definition \ref{contractionspacesdef}. Let $\frak e_N''$ be as in the conclusion of Proposition \ref{outsidestep}, so that we may write $$g_N'':=\frak e_N''\bold 1^+=\frak e_N''$$

\noindent where $\|g_N''\|_{L^2}\leq C \e^{1/12-\gamma}\|\frak e_N\|_{\frak L}$. 

Set $$(\ph,a):= P^+(g_N'')$$ so that $$\mathcal L_{(\Phi_1, A_1)}(\ph,a)=g_N'', $$
and \be \|(\ph,a)\|_{H^{1,+}_\e} \  \leq \  \|(\ph,a)\|_{\mathcal H^+}=C \e^{-1/12-\gamma} \|g_N''\| \ \leq  \ C \e^{-\gamma} \|\frak e_N\|_{\frak L}.\label{H+normbound}\ee

We now proceed to calculate  $\overline{\mathbb{SW}}_\Lambda (h_{N+1})$ where $h_{N+1}=h''_N - (\ph,a)$ where $h_N''$ is as in (\refeq{hNprimeprime}). 
\begin{eqnarray}
 \overline{\mathbb{SW}}_\Lambda(h_{N+1})&=& \text{d}_{h_1}\mathbb{SW}(h_N'')  \ -  \  \text{d}_{h_1}\mathbb{SW}(\ph, a)\ + \   \mathbb Q(h_{N+1}) \ - \ (\mu_N + \mu + \Lambda) \chi^-\frac{\Phi_\tau}{\e}\nonumber \\
 &=&   \ \overline{\mathbb{SW}}_\Lambda(h_N')  \  -  \  \text{d}_{h_1}\mathbb{SW}(\psi,b, \mu)\ + \ \mathbb Q(h_{N+1})-\mathbb Q(h_{N}'')\nonumber \\
 &=&g_N'' \ - \ \mathcal L_{(\Phi_1, A_1)}(\ph,a)   \ + \ \left(\mathbb Q(h_{N+1})-\mathbb Q(h_N'') \right)  \ + \  d\chi^+(\ph,a) \ + \  \frak e_{N+1}\nonumber \\
  &=& \left(\mathbb Q(h_{N+1})-\mathbb Q(h_N'') \right)  \ + \  d\chi^+(\ph,a) \ + \  \frak e_{N+1} \label{1048}
 \end{eqnarray}
 
 \noindent Here, the main error $g_N''$ has been cancelled, and the alternating error from $d\chi^+$ has now been shifted back to the outside region. 
 
 Re-defining $\frak e_{N+1}$ to include all three terms of (\refeq{1048}), we claim that it satisfies (II) of Proposition \ref{deformationstep} with $N'=N+1$. In fact, one has the slightly stronger bound \be \|\frak e_{N+1}\|_{\frak L}\leq C \e^{1/48-\gamma} \delta \|\frak e_N\|_{\frak L}.\label{extra196}\ee
 
 \noindent Indeed, the terms involving $\mathbb Q$ in (\refeq{1048}) may be bounded as in the proof of Proposition \ref{deformationstep}. Then, since $d\chi^+$ is supported where $r=O(\e^{1/2})$, Lemma \ref{insidepolynomialdecay} implies
 \bea \|d\chi^+(\ph,a)\|_{\frak L} &\leq & \|d\chi^+(\ph,a)\|_{L^2} + \e^{\nu} \|d\chi^+(\ph,a)\|_{L^2_{-\nu}}  \\ &\leq& C\|d\chi^+(\ph,a)\|_{L^2}\leq C  \e^{-1/24-\gamma} \|g_N''\|_{L^2} \leq C \e^{1/24-\gamma}\|\frak e_N\|_{\frak L}.\eea 
\noindent Finally, the original $\frak e_{N+1}$ in (\refeq{1048}) already satisfies (\refeq{extra196}) by virtue of (2') in Proposition \ref{outsidestep}.  (\refeq{extra196}) follows. 

It remains to close the induction. Over the course of the proofs of (a single cycle of) Propositions \ref{deformationstep}, \ref{outsidestep}, and \ref{insidestep}, the constants $C, \gamma$ have been increased a finite number of times. Let $C_0$ being the original constant in Item (II) of Proposition \ref{deformationstep} and $C_1, \gamma_1$ be the final versions of the constants appearing in (\refeq{extra196}). Since $\gamma_1<<1$ still, we may assume that $$C_1 \e^{1/48-\gamma_1}\leq C_0$$

\noindent once $\e$ is sufficiently small, which reduces (\refeq{extra196}) to hypothesis (II) of Proposition \ref{deformationstep} with $N'=N+1$ as desired. The proof of (I) and of the (resp. $\text{d}\overline{\mathbb {SW}}_{h_1}$) statements again follows as in Propositions \ref{deformationstep} and \ref{outsidestep} using (\refeq{H+normbound}). 
\end{proof}

 \subsection{Proof of Proposition \ref{fixedpoints}}
 \label{section11.4}

 Analogously to (\refeq{boldAprelimdef}), define $\mathbb A$ and $\mathbb P_1=\text{d}\mathbb A$ by 
 \begin{eqnarray} \mathbb A &= &P_\xi\  +  \  P^- (\text{Id} - \  \overline{\mathbb {SW}}_\Lambda P_\xi)  \  \ + \ \  P^+ \left(\text{Id} -  \ \overline{\mathbb {SW}}_\Lambda P_\xi  \ - \ \ \overline{\mathbb{SW}}_\Lambda P^- (\text{Id} \  - \ \overline{\mathbb {SW}}_\Lambda  P_\xi) \right),\label{mathbbAdef}\\
\mathbb P_1&=& P_\xi\  +  \  P^- (\text{Id} - \text{d}\overline{\mathbb {SW}}_{h_1} P_\xi)  \ + \ P^+ \left(\text{Id} -\text{d}\overline{\mathbb {SW}}_{h_1} P_\xi - \ \text{d}\overline{\mathbb {SW}}_{h_1}  P^- (\text{Id} -\text{d}\overline{\mathbb {SW}}_{h_1} P_\xi) \right) \label{mathbbPdef} \end{eqnarray}
 
 \noindent So that (analogously to \ref{nestedT}), 
 \begin{eqnarray} \mathbb T=\text{Id}-\mathbb A\circ \  \overline{\mathbb {SW}}_\Lambda  \   \ &=& (\text{Id} - P^+ \ \overline{\mathbb {SW}}_\Lambda \ ) (\text{Id} - P^- \overline{\mathbb {SW}}_\Lambda \ ) (\text{Id}-P_\xi  \ \overline{\mathbb {SW}}_\Lambda \ ),  \label{Titeration}  \\
 \text{d}\mathbb T= \text{Id}-\mathbb P_1\circ \text{d}\overline{\mathbb {SW}}_{h_1} &=& (\text{Id} - P^+\text{d}\overline{\mathbb {SW}}_{h_1}) (\text{Id} - P^- \text{d}\overline{\mathbb {SW}}_{h_1}) (\text{Id}-P_\xi \text{d}\overline{\mathbb {SW}}_{h_1})\label{dTiteration}. \end{eqnarray}

\noindent Thus applying $T$ carries out one complete cycle of the iteration following Eq. (\refeq{Fetadiffeo}), and $\text{d}T$ one complete cycle of the linearized iteration (the resp. statements in Propositions \ref{deformationstep}--\ref{insidestep}). 

The next lemma justifies the definitions of the spaces $\mathcal H, \mathfrak L$ by showing that the linearization is uniformly invertible on these (up to an error of $\e^{-\gamma}$). 
\begin{lm} \label{dSWinvertible}The linearization $\text{d}\overline{\mathbb {SW}}_{h_1}: \mathcal H\to \frak L$ is invertible, and there is a constant $C$ independent of $\e,\gamma$ such that 
\be \|h\|_{\mathcal H}\leq C \e^{-\gamma} \|\text{d}\overline{\mathbb {SW}}_{h_1}(h)\|_\frak L \label{dSWbound}\ee

\noindent holds. 
\end{lm}
\begin{proof} Let $\mathbb P_0:\frak L\to \mathcal H$ be defined by 
$$\mathbb P_0:= P^+ \oplus \begin{pmatrix}P_\xi & 0 \\ *& P^-\end{pmatrix}$$

\noindent where $* =-P^- (1-\Pi^\perp)\text{d}\slashed{\mathbb D}_{h_\circ}P_\xi$. Since the map $\beta:\frak L(Y)\to \frak L(Y^+)\oplus \frak L(Y^-)$ defined by $\beta(\frak e):=(\frak e\bold 1^+, \frak e\bold 1^-)$, which is built into the definitions of $P_\xi, P^\pm$ is an isomorphism, and so is each of $P_\xi, P^-, P^+$, it follows (since $*$ is bounded by Proposition \ref{effectivesupportdD}) that $\mathbb P_0$ is an isomorphism hence Fredholm with index 0. 

Then, the calculations (\refeq{Tcircerror}),  (\refeq{slashedDerror}), and  (\refeq{outsidederivative}) and the subsequent bounds (e.g. \refeq{Zhbound}) in the proofs of Propositions \ref{deformationstep} and \ref{outsidestep} show that
\begin{eqnarray}
\mathbb P_1 -\mathbb P_0&=& - (P^+ +P^- )(\Xi + \frak t^\circ)P_\xi   \ - \  P^+(d\chi^-P^-(\text{Id}-\text{d}\overline{\mathbb{SW}}_{h_1} P_\xi))  \ + \  * \nonumber \\  & &- (P^+ + P^-) O(\e^8).  \label{1129}
\end{eqnarray}
\noindent where $\star$ consists of terms involving the 1-dimensional span of $\Phi_\tau$, and the $O(\e^8)$ accounts for terms involving $u$ and $K_1$. The terms in the top line of (\refeq{1129}) are compact, because they factor through either the compact inclusion $\mathcal H^{1,-}\hookrightarrow L^2$ or through the finite-dimensional spaces spanned by $\eta \in \frak W$ and $\Phi_\tau$ respectively. 

The uniform boundedness of $P_\xi, P^\pm$ in Propositions  \ref{deformationstep}, \ref{insidestep}, \ref{outsidestep}, together with the boundedness of $*$ from Proposition \ref{effectivesupportdD} show that  both  $\|\mathbb P_0\|, \|\mathbb P_0^{-1}\|\leq C \e^{-\gamma}$. Once $\e$ is sufficiently small, the $O(\e^8)$ terms may be safely ignored. It follows that $\mathbb P_1$ is Fredholm of index 0 once $\e$ is sufficiently small. 

The same argument shows that $\mathbb P_N$ defined by
$$ \text{Id}-\mathbb P_N\circ \text{d}\overline{\mathbb {SW}}_{h_1} = \text{d}\mathbb T^N$$
\noindent is likewise Fredholm of Index 0. By construction (cf. Eq. \eqref{nestedT}), $\mathbb P_N$ applies $N$ stages of the linearized alternating iteration from Propositions \ref{deformationstep}--\ref{insidestep} (i.e. the resp. $\text{d}\overline{\mathbb {SW}}_{h_1}$ statements) starting from $\frak e_1= \text{Id}-\text{d}\overline{\mathbb {SW}}_{h_1}\mathbb P_0(\frak e_0)$ where $\frak e_0=\text{d}\overline{\mathbb {SW}}_{h_1}(h)$. We conclude that $$\text{d}\overline{\mathbb {SW}}_{h_1}\mathbb P_N=\text{Id} + O(\delta^{N-1}).$$
\noindent We conclude that $\text{d}\overline{\mathbb {SW}}_{h_1}$ is surjective, hence an isomorphism by the Fredholm index. The bound (\refeq{dSWbound}) follows from the initial bound $\|\mathbb P_0(\frak e_0)\|\leq C \e^{-\gamma}$ above,  and statements (I) in Propositions  \ref{deformationstep}--\ref{insidestep} with $h_N$ being the correction from $\mathbb P_N-\mathbb P_0$ (equivalently, one can sum the geometric series in $\delta$ to give a bound on the limit of the sequence of $h_N$ by a constant times the norm of the initial guess). 
\end{proof}

\begin{proof}[Proof of Proposition \ref{fixedpoints}] Let $\mathcal V=B_r(0)\subseteq \mathcal H$ of radius $r=\e^{-1/20}$, and let $\mathbb T$ be given by (\refeq{Tdef}) with $\mathbb A$ defined by (\refeq{mathbbAdef}) (equivalently, $\mathbb T$ is given by Eq. \refeq{Titeration}). 

(A) is deduced assuming (C) as follows. Let $h \in \mathcal V$, then since $\|\overline{\mathbb {SW}}_\Lambda(0)\|_\frak L \leq C \e^{-1/24-\gamma}$ by Theorem \ref{PartImain}, 
$$\|\mathbb T(h)\|_\mathcal H  \leq \|\mathbb T(h) -\mathbb T(0)\|_\mathcal H + \|\mathbb T(0)\|_\mathcal H \leq C\sqrt{\delta} \|h\|_\mathcal H+ C\e^{-1/24-2\gamma}\leq r,$$
\noindent where the bound on $\mathbb T(0)$ is a consequence of (I) in Propositions \ref{deformationstep}--\ref{insidestep} with $\frak e_1=\overline{\mathbb {SW}}_\Lambda(0)$. Thus $\mathbb T:\mathcal V\to \mathcal V$ preserves $\mathcal V$. Continuity of $\mathbb T$ is immediate from (C). $C^1$ dependence on $(\e,\tau)$ is immediate from the $C^1$ dependence of the Seiberg--Witten equations on $p_\tau$, and the $C^1$ dependence of $(\Phi^{(1)}, A^{(1)})$ and $(\Phi_\tau, A_\tau)$ and the linearized equations at these and their inverses used to construct $\mathbb T$.  

(B) By (\refeq{Titeration}), applying $T$ constitutes a full cycle of the three-stage iteration carried out by Propositions \ref{deformationstep}--\ref{insidestep}. The conclusion follows from applying these three propositions successively. 

(C) Let $\frak q$ be such that 
\be \mathbb T(h)=\text{d} \mathbb T(h)+ \frak q(h).\label{frakqdef}\ee
where $\text{d}\mathbb T$ is as in (\refeq{dTiteration}). The same argument as (B), now using the (resp. $\text{d}\overline{\mathbb {SW}}_{h_1}$) statements in Propositions \ref{deformationstep}--\ref{insidestep} shows that 
\be \|\text{d}\overline{\mathbb {SW}}_{h_1}(\text{d}\mathbb T(h))\|_{\frak L} \leq C \delta \|\text{d}\overline{\mathbb {SW}}_{h_1}(h)\|_\frak L.\label{dTreduction}\ee

\noindent Therefore since $\text{d}T$ is linear, Lemmas \ref{dSWbounded} and \ref{dSWinvertible} shows 
\bea
\|\mathbb T(h) - \mathbb T(\widetilde h)\|_{\mathcal H} &\leq& C \e^{-\gamma} \|\text{d}\overline{\mathbb {SW}}_{h_1}(\mathbb T(h)-\mathbb T(\widetilde h))\|_\frak L  \\
&\leq &  C \e^{-\gamma} \|\text{d}\overline{\mathbb {SW}}_{h_1}(\text{d}\mathbb T(h)-\text{d}\mathbb T(\widetilde h))\|_\frak L \ + \ C \e^{-\gamma} \|\text{d}\overline{\mathbb {SW}}_{h_1}(\frak q(h)-\frak q(\widetilde  h))\|_\frak L \\
&\leq & C \e^{-\gamma} \delta \|\text{d}\overline{\mathbb {SW}}_{h_1}(h-\widetilde h)\|_\frak L \ + \ C \e^{-\gamma} \|\text{d}\overline{\mathbb {SW}}_{h_1}(\frak q(h)-\frak q(\widetilde  h))\|_\frak L\\
&\leq &  C \e^{-2\gamma} \delta \|h -\widetilde  h\|_{\mathcal H}  \ + \  C \e^{-2\gamma} \|\frak q(h)-\frak q(\widetilde h)\|_{\mathcal H}. 
\eea

\noindent where we have substituted (\refeq{dTreduction}) in the third line. 

Since $C\e^{-2\gamma}\delta <\sqrt{\delta}$ once $\e$ is sufficiently, small, the following bound completes (C): 
\be \|\frak q(h)-\frak q(\widetilde h)\|_{\mathcal H}\leq C \e^{1/12-6\gamma}\|h -\widetilde  h\|_{\mathcal H}.\label{1134}\ee

\noindent To prove (\refeq{1134}), we calculate $\frak q=\mathbb T-\text{d}\mathbb T$ by expanding (\refeq{Titeration}) and (\refeq{dTiteration}), with $\overline{\mathbb {SW}}_\Lambda= \text{d}\overline{\mathbb {SW}}_{h_1} + \mathbb Q +\Lambda$  where the latter is shorthand for $\Lambda = \Lambda(\tau)\e^{-1}\chi^-\Phi_\tau$. Writing $\mathbb L= \text{d}\overline{\mathbb {SW}}_{h_1},$ the difference becomes
\bea
\frak q&=& P^-\mathbb Q  \ + \ P_\xi \mathbb Q \ + \  P^+\mathbb Q \ + \  P^-\mathbb QP_\xi  \mathbb L  \ + \  P^-\mathbb QP_\xi\mathbb Q \ + \ P^- \mathbb LP_\xi\mathbb Q \\
& +&  P^+ \mathbb Q P^- \mathbb L \  + \  P^+\mathbb QP_\xi \mathbb L  \ + \  P^+\mathbb Q P^-\mathbb Q  \ + \ P^+\mathbb Q P_\xi \mathbb Q \ + \ P^+\mathbb L P^-\mathbb Q  \ + \ P^+\mathbb LP_\xi \mathbb Q\\
& +  & P^+ \mathbb LP^- \mathbb Q P_\xi  \mathbb L \ + \  P^+  \mathbb LP^-  \mathbb L P_\xi \mathbb Q  + P^+ \mathbb Q P^-\mathbb L P_\xi \mathbb L \ + \  P^+\mathbb LP^-\mathbb Q P_\xi \mathbb Q  \\
& + &   P^+\mathbb Q P^-\mathbb Q P_\xi \mathbb L \ + \ P^+\mathbb Q P^- \mathbb L P_\xi \mathbb Q \ + \  P^+\mathbb Q P^-\mathbb Q P_\xi \mathbb Q  \\
 &  + & \frak q_\Lambda
\eea
\noindent where $\frak q_\Lambda$ is the same collection of terms replacing  each instance of $\mathbb Q$ with $\Lambda$. The above expression shows that all possible combinations of up to 3 compositions of $\mathbb Q$ with the parametrices occur, but each term except $\frak q_\Lambda$ has {\it at least a single instance of $\mathbb Q$}. Since $\Lambda$ is constant, $q_\Lambda(h_1)-q_\Lambda(h_2)=0$. Because each remaining term has at least one factor of $\mathbb Q$, 
 (\refeq{1134}) now follows from the boundedness of $P_\xi, P^\pm$ in Propositions \ref{deformationstep}--\ref{insidestep}, the boundedness of $ \mathbb L$ from Lemma \ref{dSWbounded}, along with Items (A) and (C) of Corollary \ref{nonlinearboundscor}. \end{proof}

 \section{Gluing}
\label{section12}

\subsection{Glued Configurations} 
\label{section12.1}
Proposition \ref{fixedpoints} and the Banach fixed-point theorem (with $C^1$ dependence on parameters) immediately imply the following. 

\begin{cor} There exist $\e_0, \tau_0$ sufficiently small such that for every pair $(\e,\tau) \in (-\e_0, \e_0) \times (-\tau_0, \tau_0)$, there exist tuples $(\mathcal Z, \Phi, A ,\mu)$ depending in a $C^1$ way on $\e,\tau$ where
\begin{enumerate}
\item[(1)]  $\mathcal Z(\e,\tau)=\mathcal Z_{\tau, \xi(\e,\tau)}$ is the singular set arising from a linearized deformation $\xi(\e,\tau) \in C^\infty(\mathcal Z_\tau, N\mathcal Z_\tau)$,  

\item[(2)] $\mu(\e,\tau)\in \R$, and 

\item[(3)] $(\Phi(\e,\tau), A(\e,\tau)) \in C^\infty(Y; S_E\oplus \Omega)$
\end{enumerate}
that satisfy 
\be  \text{SW}\left(\frac{\Phi(\e,\tau)}{\e}, A(\e,\tau)\right) =(\Lambda(\tau) + \mu(\e,\tau) \chi^-_{\e} \cdot \frac{ \Phi_{\tau, \xi(\e,\tau)}}{\e}\label{gluedconfigurations} \ee
\noindent where $\Phi_{\tau, \xi(\e,\tau)}$ is as in Definition \ref{universalfamilydef}, and $\chi^-_\e=\chi^-_{\e,\tau,\xi(\e,\tau)}$ cutoff function (\refeq{universalSWpre}). 
\end{cor}
\begin{proof}
Part (C) of Proposition \ref{fixedpoints} and the Banach fixed-point theorem give a  $C^1$ family of fixed points $h=(\xi, \ph, a , \psi, b, \mu , u)\in \mathcal H_{\e,\tau}$. By (B) of Proposition \ref{fixedpoints}, these satisfy, $\overline{\mathbb{SW}}_\Lambda(h)=0$. Setting

$$(\Phi(\e,\tau), A(\e,\tau))=(\Phi^{(1)}_{\e,\tau,\xi}, A^{(1)}_{\e,\tau,\xi} ) \ + \ \chi_\e^+(\ph, a) \ + \ \chi_\e^-(\psi +u, b),  $$

\medskip 
\noindent where $\xi=\xi(\e,\tau)$, the conclusion follows Definition (\refeq{universalSWdef2}) of $\overline{\mathbb {SW}}$ (cf. the definition of $\overline{\mathbb {SW}}_{\Lambda}$ following Eq. (\refeq{Tdef})). The deformation $\xi(\e,\tau) \in C^\infty(\mathcal Z_\tau;  N\mathcal Z_\tau)$ is smooth by construction (recall (2) following Definition \ref{frakWdef}). (3) holds by Corollary \ref{solvinguniversalSW} with right-hand side $g =\e^{-1}\chi_\e^-\cdot  \Lambda(\tau)\Phi_{\tau,\xi(\e,\tau)}$, which satisfies $g\in C^\infty(Y)$  as in the proof of that corollary.  
\end{proof}

\smallskip

 \subsection{The One-Dimensional Obstruction}
\label{section12.2}
The configurations (\refeq{gluedconfigurations}) solve the Seiberg--Witten equations if and only if 
\be \Lambda(\tau) + \mu(\e,\tau)=0 \label{1dimob}\ee
is satisfied. The next lemma shows that the assumption of transverse spectral crossing (Definition \ref{spectralcrossingdef}) means the condition (\refeq{1dimob}) defines $\tau$ implicitly as a function of $\e$.

\begin{lm} \label{smoothdependence} The solutions $(\Phi(\e,\tau), A(\e,\tau), \mu(\e,\tau))$ from Proposition \ref{gluedconfigurations} depend in a $C^1$ way on $(\e,\tau)\in (0,\e_0)\times (-\tau_0, \tau_0)$. Moreover, 

\be |\mu(\e,\tau)| \ + \ |\del_\tau \mu(\e,\tau)| \leq  C \e^{2/3}. \label{iftbounds} \ee

\noindent holds uniformly. 
 \end{lm}
 
 \begin{proof} The standard proof of the Banach fixed point theorem with smooth dependence on parameters shows that a smooth family of fixed points $h_\tau$ of $\mathbb T$ as in Proposition \ref{fixedpoints} satisfy $$\| \del_\tau h(\tau)\|_{\mathcal H}\leq C\|(\del_\tau \mathbb T_\tau)h(\tau)\|_{\mathcal H}.$$
 
 \noindent Because of the weight on $\mu$ in the definition (\refeq{contractionspacesdefII}) of the $\mathcal H_{\e,\tau}$-norm, it therefore suffices to show that \be \|(\del_\tau \mathbb T_\tau)h(\tau)\|_{\mathcal H}\leq C \e^{-1/3}. \label{toshow}\ee

 By Theorem \ref{PartImain} Item (I), Lemma \ref{solvingoutsideim}, the expression Proposition \ref{deformationUSW} (using Theorem \ref{PartImain}(B)), and the proofs of Lemmas \ref{relatingvariationsoutside}, \ref{relatingvariationsinside}, and Proposition \refeq{effectivesupportdD} -- which may be repeated equally well with $(\del_\tau\Phi^{(1)}_{\e,\tau}, \del_\tau A^{(1)}_{\e,\tau})$ -- show that 
 
 $$\|\del_\tau \mathcal L_{(\Phi^{(1)},A^{(1)})}^+(\ph,a)\|_{L^2}  \ + \  \|\del_\tau \mathcal L_{(\Phi^{(1)},A^{(1)})}^-(\psi,b)\|_{L^2} +  \| \del_\tau \text{d}\overline{\mathbb {SW}}_{h_1}(\xi)\|_{L^2} \leq C\e^{-\gamma} \|h\|_{\mathcal H} $$
  are uniformly bounded as maps the $H^{1,\pm}_\e\to L^2$ and $\frak W\to L^2$. To verify this, simply note that all the bounds in these propositions are ultimately a consequence of the polyhomogeneous asymptotics of $\Phi_\tau$ and applications of Lemma \ref{deformationbounds}. Since the power appearing in the expansions of $\Phi_\tau$ are independent of $\tau$, $\del_\tau \Phi_\tau$ has an identical expansion, thus the hypotheses of Lemma \ref{deformationbounds} hold equally well for $\del_\tau \Phi_\tau$ in each instance. A similar argument applies to $\del_\tau \mathbb Q$.

  Together withTheorem \ref{PartImain}(A) (the ``moreover'' statement), these show that the operator norm of  $\del_\tau \overline{\mathbb {SW}}_\Lambda$ a map $\mathcal H\to \mathfrak L$ is bounded by $C\e^{-\gamma}$ at $h\in \mathcal C$.  Differentiating $\text{Id}=P_\tau^{-1}P_\tau$ for all three parametrices and using the bounds from Propositions \ref{deformationstep}--\ref{insidestep} yields bounds on $\del_\tau P_\xi,\del_\tau P^\pm$ by at most $C \e^{-1/12-\gamma}$. Using the product rule on (\refeq{Titeration}) and combining these yields (\refeq{toshow}). 
 \end{proof}

The proof of the following lemma is essentially identical to the treatment of the similar 1-dimensional obstruction in  \cite[Eq. (10.6)]{WalpuskiG2Gluing}.
\begin{lm} \label{1dimIFT}If the family of parameters $p_\tau=(g_\tau, B_\tau)$ has a transverse spectral crossing, then (\refeq{1dimob}) implicitly defines a function $\tau(\e)$ so that 
$$\Lambda(\tau(\e)) + \mu(\e,\tau(\e))=0$$

\noindent for $\e\in (0, \e_0)$, and either $\tau(\e)>0$ or $\tau(\e)<0$. 
\end{lm}

\begin{proof} 
The assumption that $\tau=0$ is a transverse spectral crossing means that $\Lambda(0)=0$ and $\dot \Lambda(0)\neq 0$. By the inverse function theorem, there is an inverse $\Lambda^{-1}$ defined on an open neighborhood of $\tau =0$. Set $\Gamma(\e) =\Lambda \circ \tau(\e)$, so that (\refeq{1dimob}) becomes the condition that 
$$\Gamma(\e) + \mu (\e, \Lambda^{-1}\circ \Gamma(\e))=0. $$

\noindent This is an equation for a real number $\Gamma$ depending in a $C^1$ way on the single parameter $\e$. Becuase $\Gamma(0)=0$, (\refeq{iftbounds}) implies that this equation can be solved for $\Gamma \in \R$ using the Inverse Function Theorem with $C^1$ dependence on the parameter $\e$, after possibly one final reduction of $\e_0$. For the solution $\Gamma(\e)$, the corresponding function $\tau(\e)=\Lambda^{-1}\circ \Gamma(\e)$ solves (\refeq{1dimob}). 

To see the sign of $\tau(\e)$, we expand (\refeq{1dimob}) in series. From the proof of Proposition \ref{fixedpoints} implies that $\mu$ is a sum $\mu(\e,\tau)=\mu_1(\e,\tau) +\delta \mu_2(\e,\tau)+ \delta^2\mu_3(\e,\tau)+\ldots$, where $\delta=\e^{1/48}$ as in Proposition \ref{fixedpoints}. Meanwhile $\Lambda(\tau)$ is smooth and may be expanded in Taylor series. 
\bea
0&=& \Lambda(\tau) + \mu(\e,\tau)\\
&=& \dot \Lambda(0)\tau + \e^{2/3}\overline \mu_1(\e,0) + O(\tau^2) + O(\e^{2/3}\delta) + O(\e^{2/3}\tau),
\eea
where $\mu_i = \e^{2/3}\overline \mu_i$. It follows that for $\e$ sufficiently small, $\tau$ has the opposite sign of $\frac{\overline \mu_1(\e,0)}{\dot \Lambda(0)}$. 
\end{proof}

\begin{proof} [Proof of Theorem \ref{maina}] Part (A) follows directly from Lemma \ref{reconstructingspinc}. Along the family of parameters $\tau(\e)$ satisfying (\refeq{1dimob}) constructed in Lemma \ref{1dimIFT}, the glued configurations of Proposition \ref{gluedconfigurations} satisfy the (extended) Seiberg--Witten equations. Integrating by parts shows the $0$-form component $a_0$ vanishes, and setting $(\Psi_\e, A_\e)=(\e^{-1}{\Phi(\e,\tau(\e))} \ , \ A(\e,\tau(\e)))$ yields the solutions in Part (B). 

The glued configurations have $\|\Psi_\e\|_{L^2}=\e^{-1} + O(\e^{-1/24-\gamma})$, and this norm depends in $C^1$ way on $\e$. By re-parameterizing, it may be assumed that $\|\e \Psi_\e\|_{L^2}=1$ as in (\refeq{normalization}). As $\e\to 0$, Theorem \ref{HWcompactness} shows that after renormalization, $(\Psi_\e, A_\e)$ converges in $C^\infty_{loc}$ to a $\Z_2$-harmonic spinor for the parameter $p_0=(g_0, B_0)$ and that $\|\e\Psi_\e\|\to |\Phi_0|$ in $C^{0,\alpha}$. Since $(\mathcal Z_0, A_0, \Phi_0)$ is regular, it is the unique $\Z_2$-harmonic spinor for this parameter and must therefore be the limit. This establishes Part (C). 
\end{proof}

\begin{rem}\label{signofcrossing} In potential applications to wall-crossing formulas, it may be of interest to extract the precise sign of $\tau(\e)$ for small $\e$, as was done in \cite{DWDeformations} when $\mathcal Z_0=\emptyset$. 
Recalling the construction of $\mu$ from the proof of Proposition \ref{outsidestep} with $N=1$, the leading order obstruction in the proof of Lemma \ref{1dimIFT} is
\be \e^{2/3}\overline \mu_1(\e,0)= \e \br \text{d}\slashed{\mathbb D}_{h_\circ}(\eta_1,0) \ + \ d\chi^+.\ph^{(1)}_{\e,\tau}, \Phi_0\kt, \ee
\noindent where $\ph^{(1)}_{\e,\tau}$ is as in Item (II) of Theorem \ref{PartImain} for $\tau=0$. In fact, since $\Phi_0=O(r^{1/2})$, simple integration shows that the second term is smaller by a factor of $\e^{1/2}$ than the bound on $\text{d}\slashed{\mathbb D}_{h_\circ}(\eta_1,0) $ obtained in the proof of Proposition \ref{deformationstep}. This suggests that the sign can, in principle, be calculated only from the sign of the inner product $\br \text{d}\slashed{\mathbb D}_{h_\circ}(\eta_1,0),  \Phi_0\kt$. This cannot, however, be proved definitively without knowing bounds on the relative sizes of such term are optimal.

\end{rem}

\appendix 

\section*{Glossary of Notation}
\label{glossary}
In this appendix, we collect the various notation used throughout the article. With the exception of notation used only within the scope of a particular subsection or proof, we list all notation below with links to the corresponding definitions, in alphabetical orders of Latin, Greek, and other symbols respectively. Notation not appearing in this list is either local to a subsection, or is a gluing parameter defined in Appendix \ref{appendixofnotation}.
\bigskip

\begin{enumerate} 
\item[$A \ $]  used to denote an extended connection in a Seiberg--Witten configuration, defined preceding \refeq{extended}).
\smallskip 
\item[$A_0$]  the flat connection on $\ell \to \Yminus \mathcal Z_0$ in the initial $\Z_2$-harmonic spinors. Assumed to obey the hypothesis of Theorem \ref{maina}. 
\smallskip
\item[$\mathbb A$] the non-linear approximate inverse of $\overline{\mathbb SW}_\Lambda$, defined in Eq. \refeq{mathbbAdef} (cf. Subsection \ref{contractionsubspacessection}).  
\smallskip
\item[$\mathcal A_{U(1)}$] the space of $U(1)$ connections on $\text{det}(S)$ (of {\it a priori} undefined regularity).
\smallskip
\item[$a_{g_0}^{g_s}, \frak a$] tensors depending on two metrics appearing in the Bourguignon-Gauduchon formula (defined preceding Theorem \ref{bourguignon}). 
\smallskip
\item[$B$] the smooth background $SU(2)$ connection on $E$, part of the parameter pairs $p=(g,B)$. 
\smallskip
\item[$\mathcal B_{\Phi_\tau}(\eta)$] the partial derivative of $\slashed{\mathbb D}$ in the direction of a deformation $\eta$ of the singular set, defined in Eq. (\refeq{calBdef}).
\smallskip
 \item[$\underline {\mathcal B}_{\Phi_\tau}(\eta)$] the partial derivative in the deformation direction in the tangential smoothing gauge, defined in Notation \ref{underlinenotation}. \smallskip
 \item[$\underline{\mathcal B}_{\Phi^{(1)}}(\xi)$] spinor component of the partial derivative of the universal Seiberg--Witten equations at the model solutions in the direction of a deformation, defined in Proposition \ref{deformationUSW}. 
 \smallskip
\item[$\underline{\mathcal B}^\circ_{\tau}(\eta)$] the tame truncation of the partial derivative in the direction of a deformation $\eta$, defined in Definition \ref{tametruncationdef}. 
 \smallskip
  \item[$\underline{\mathfrak b}_{A^{(1)}}(\xi)$] term of the universal Seiberg--Witten equations at the model solutions in the direction of a deformation coming from the deformation of the connection, defined in Proposition \ref{deformationUSW}. 
 \smallskip
\item[$\mathcal C_\tau$] the (complex rank 1) Calder\'on bundle over $\mathcal Z_\tau$, defined following Eq. (\refeq{calderonbunddef}).
\smallskip
\item[$c(t),d(t)$] leading coefficients of the $\Z_2$-harmonic eigenvectors in the polyhomogeneous expansions of Lemma \ref{asymptoticexpansion}. 
\smallskip 
\item[$\slashed D_A$] the Dirac operator at a smooth $U(1)$-connection on $Y$. Depends implicitly on $p=(g,B)$. 
\smallskip 
\item[$\slashed D_{A_0}, \slashed D_{A_\tau}$] the singular Dirac operators on $\Yminus \mathcal Z_0, \Yminus \mathcal Z_\tau$ respectively. Defined in Eq. (\refeq{1.1}) and Eq. (\refeq{DA0preserves}). 
\smallskip
\item[$\slashed{\mathbb D}$] the universal Dirac operator, defined in Eq. (\refeq{universalDiracdef}) and more precisely in Definition \ref{universalDirac})
\smallskip
\item[$(\text{d}\slashed{\mathbb D})_{(\mathcal Z,\Phi)}$] the linearization of universal Dirac operator at a pair $(\mathcal Z,\Psi)$, defined in Lemma \ref{boundedUniversalDirac}.
\smallskip
\item[$ \text{d}\slashed{\mathbb D}_{h_\circ} $] used as shorthand for the derivative of the universal Dirac equation at $(\mathcal Z_\tau, \Phi_\tau)$. 

\smallskip
\item[$\bold d$] the Hodge de-Rham operator on $(\Omega^0\oplus \Omega^1)(i\R)$ given by $\begin{pmatrix} 0 & -d^\star \\ -d  & \star d \end{pmatrix}$, defined in Lemma \ref{linearizedequations}.
\smallskip

\smallskip 
\item[$E$] a fixed (trivial) auxiliary $SU(2)$ bundle over $Y$.
\smallskip 
\item[$\mathcal E_\tau$] the domain of the chart $\text{Exp}_\tau$, defined in Eq. (\refeq{Expdef}). 
\smallskip
\item[$\mathcal E_\tau^s$] the intersection $\mathcal E_\tau\cap L^{s,2}(\mathcal Z_\tau; N\mathcal Z_\tau)$ of links near $\mathcal Z_\tau$ of regularity $s$, defined preceding Lemma  \ref{trivializationcompatible}. 
\smallskip
\item[$\text{Exp}_\tau$] the exponential chart on the space of embeddings centered at $\mathcal Z_\tau$, defined in  (\refeq{Expdef})
\smallskip
\item[$e^j$] used to denote a local orthonormal frame of $T^*Y$ in a given metric. 
\smallskip 
\item[$e^{i\ell t}$] Fourier modes of sections of the Calder\'on bundle(s) $\mathcal C_\tau\to\mathcal Z_\tau$. 
\smallskip 
\item[$e_N, f_N, g_N, \lambda_N, \Psi_N$] used to denote sub-error terms in the alternating iteration in Section \ref{section11}.
 \smallskip
\item[$\frak e_N, \frak e_N', \frak e_N''$] used to denote the main error terms in each of the three successive stages of the alternating iteration in Section \ref{section11}. 
\smallskip
\item[$F_A$] the $\Omega^2(Y; i\R)$-valued curvature of a $U(1)$-connection $A$.
\smallskip 
\item[$\mathbb F_\tau$] an admissible family of diffeomorphisms, defined in Definition \ref{admissibledef}
\smallskip
\item[$\underline {\mathbb F}_\tau$] the tangentially smoothing admissible family, defined in Definition \ref{underlineFdef}. 
\smallskip
\item[$\mathcal G$] the gauge group $C^\infty(Y; U(1))$. 
\smallskip 
\item[$g$] used to denote a general Riemannian metric on the closed manifold $Y$. 
\smallskip
\item[$ (g^\circ_\tau, B_\tau^\circ) $] the tame truncations of the parameter pair $(g_\tau, B_\tau)$, defined in Eq. \eqref{circdefinitions1}.  
\smallskip
\item[$g_0, g_\tau$] the central metric and family of metrics in Theorem \ref{maina}. 
\smallskip
\item[$g_\eta, g_{\eta,\tau}$] the pullback metric $F_\eta^*(g_\tau)$ (or $g_{\tau,\eta})$ when ambiguity may arise, defined in Notation \ref{Notation6.1}. $\tau$ is omitted from subscript when clear from context.  
\smallskip
\item[$ \underline g_\eta,\underline g_{\eta,\tau},$] the pullback metrics $\underline F_\eta^*(g_\tau)$in the tangential smoothing gauge, defined in Notation \ref{underlinenotation}. $\tau$ is omitted from subscript when clear from context.  
\smallskip

\item[$r^{1+\nu}H^1_e$] edge Sobolev space of regularity $s=1$, defined in Definition \ref{rH1edef}.
\smallskip 
\item[$ rH^1_\perp$] subspace of $rH^1_e$ given by the $L^2$-orthogonal complement of $\Phi_\tau$, defined in Eq. (\refeq{rH1perp}).
\smallskip 
\item[$H^s_b$] boundary Sobolev spaces with of regularity $s$, defined following Proposition \ref{cokerpropertiesII}. 
\smallskip 
\item[$H^{1,+}_{\e,\nu}$] the inside weighted Sobolev space of regularity $1$, defined in Definitions \ref{insidenorms}, \ref{outsidespaces}.
\smallskip
\item[$H^{1,-}_{\e,\nu}$] the outside weighted Sobolev space of regularity $1$, defined in Definitions \ref{outsidenorms}, \ref{insidespaces}.
\smallskip
\item[$H^{1,\pm}_{\e,\tau,\xi}$] fibers of the universal configuration space $\mathbb H^{1,\pm}_{\e,\nu}(\mathcal E_\tau)$, defined preceding Definition \ref{hilbertspaces}. 
\smallskip
\item[$\mathbb H$] the quaternions.
\smallskip 
\item[$\mathbb H^1_e(\mathcal E_\tau)$]  the Hilbert vector bundle of spinors defined over $\mathcal E_\tau$ with fiber $rH^1_e$ (defined in Eq. \refeq{Hilbertfamilies}). 
\smallskip 
\item[$\mathbb H^{1,\pm}_{\e,\tau}(\mathcal E_\tau)$] universal inside and outside configuration space over singular sets $\mathcal E_\tau$ with weight $\nu$ and parameter $\e$ and regularity $1$, Definition \ref{hilbertspaces}. 
\smallskip
\item[$\mathcal H^\pm$] inside and outside weighted subspaces on which the alternating iteration is a contraction mapping, defined in Definition \ref{contractionspacesdef}. 
\smallskip
\item[$\mathcal H_{\e,\tau}$] the global domain on which the alternating iteration is a contraction mapping, defined in \ref{contractionspacesdefII}. 
\smallskip
\item[$h_0, h_1$] shorthand notation for pre-glued configurations, model solutions and subsequent iterates, defined in Definition \ref{notationhn} and (\refeq{h0seconddef}--\refeq{h1seconddef}).
\smallskip
\item[$h_\circ$] used as shorthand for the singular configuration $(\Phi_\tau, A_\tau)$ on $\Yminus \mathcal Z_\tau$, introduced following \eqref{Xipmdef}.
\smallskip
\item[$h_N, h_N', h_N''$] used to denote the approximation in each of the three successive stages of the alternating iteration in Section \ref{section11}. 
\smallskip
\item[$J$] an almost complex structure on $S|_{\Yminus \mathcal Z_\tau}$, defined preceding Eq. (\refeq{realimdecomp}).
\smallskip 
\item[$j$] the almost-complex structure on $E$, arising from the right action of $j\in \mathbb H$ in local trivializations,  defined preceding Eq. (\refeq{realimdecomp}).
\smallskip 
\item[$L_0$] used to denote Fourier mode cut-offs 
\smallskip
\item[$L^\text{low}, L^\text{med}$] used to denote Fourier mode cut-off of the low and medium mode regimes, defined in Definition \ref{Fourierregimes}. 
\smallskip
\item[$L^2(X;V)$] the space of square-integrable sections of a vector bundle $V$ over domain $X$.
\smallskip 

\item[$L^{2,+}_{\e,\nu}$] the inside weighted Sobolev space of regularity $0$, defined in Definition \ref{insidenorms}, \ref{outsidespaces}.
\smallskip 
\item[$L^{2,\pm}_{\e,\tau,\xi}$] fibers of the universal configuration space $\mathbb L^{2,\pm}_{\e,\nu}(\mathcal E_\tau)$, defined preceding Definition \ref{hilbertspaces}. 
\smallskip 
\item[$L^{2,-}_{\e,\nu}$] the outside weighted Sobolev space of regularity $0$, defined in Definition \ref{outsidenorms}, \ref{insidespaces}.
\smallskip 

\item[$\mathbb L^2(\mathcal E_\tau)$]  the Hilbert vector bundle of spinors defined over $\mathcal E_\tau$ with fiber $L^2$ (defined in Eq. \refeq{Hilbertfamilies}). 
\smallskip
\item[$\mathbb L^{2}_{\e,\tau}(\mathcal E_\tau)$] universal inside and outside configuration space over singular sets $\mathcal E_\tau$ with weight $\nu$ and parameter $\e$ and regularity $0$, Definition \ref{hilbertspaces}.  
\smallskip 
\item[$\mathfrak L_{\e,\tau}$] the global codomain in which the error terms of the alternating iteration live, defined in \ref{contractionspacesdefII}.  
\smallskip
\item[$\mathcal L_{(\Phi,A)}$] the extended, gauge-fixed linearized Seiberg--Witten equations at a smooth configuration $(\Phi,A)$ on $Y$. Implicitly includes the renormalization by $\e^{-1}$ of the spinor. Defined in Lemma \ref{linearizedequations}. 
\smallskip 
\item[$\mathcal L_{(\Phi_\tau, A_\tau)}$] the singular linearization of the Seiberg--Witten equation at a $\Z_2$-harmonic eigenvector, defined in Lemma \ref{linearizedequationsdef}.
\smallskip 
\item[$\mathcal L_{(\Phi_\tau, A_\tau)}^\text{Im}$] the lower $2\times 2$ block operator in the linearized Seiberg--Witten equation at a $\Z_2$-harmonic eigenvector, defined in Eq.  \ref{Limdef} (cf. Lemma \ref{linearizedequationsdef}).
\smallskip
\item[$\ell$] used to denote a real line bundle over $\Yminus \mathcal Z_\tau$. 
\smallskip 
\item[$\ell \in \Z$] used to index Fourier modes in the directions tangential to $\mathcal Z_\tau$ in Fermi coordinates. Runs over the set $2\pi \Z / |\mathcal Z_\tau|$ in the neighborhood of each component of $\mathcal Z_\tau$. 
\smallskip 
\item[$N_{r_0}(\mathcal Z_\tau)$] the image of the Fermi coordinate chart in Definition \ref{Fermicoords}. 
\smallskip 
\item[$\text{\bf Ob}, \text{\bf Ob}_\tau$]  the obstruction bundle over $\tau\in (-\tau_0, \tau_0)$ and the obstruction space for a fixed $\tau$ defined in Definition (\ref{Obdefn}).
\smallskip 
\item[$\text{\bf Ob}^\perp$] The codimension 1 subspace of $\text{\bf Ob}_\tau$ that is $L^2$-orthogonal to $\Phi_\tau$, defined preceding Corollary \ref{solvingcokernel}.  
\smallskip 
\item[$\text{\bf ob}_\tau$] bounded linear isomorphism $L^2(\mathcal Z_\tau; \mathcal C_\tau)\to \text{\bf Ob}_\tau$, defined in Proposition \ref{cokerpropertiesII}.
\smallskip 
\item[$\text{ ob}_\tau$] infinite-dimensional component of the obstruction map $\text{\bf ob}_\tau=(\text{ob}_\tau, \iota_\tau)$ defined in Proposition \ref{cokerpropertiesII}.
\smallskip
\item[$P_\xi$] the parametrix of the ``deformation'' stage of the alternating iteration, defined in Eq. (\refeq{Pxi}).
\smallskip
\item[$P^-$] the parametrix of the ``outside'' stage of the alternating iteration, defined in Eq. (\refeq{P-}).
\smallskip
\item[$P^+$] the parametrix of the ``outside'' stage of the alternating iteration, defined in Eq. (\refeq{P+}).
\smallskip
\item[$\mathbb P$] the nested approximate parametrix of $\text{d}\overline{\mathbb {SW}}_\Lambda$, defined in Eq. \refeq{mathbbPdef}.  
\smallskip
\item[$p$] $=(g,B)$ a smooth parameter pair of a metric $g\in \mathscr{M}et(Y)$ and connection $B\in \mathcal A_{SU(2)}(E)$. 
\smallskip 
\item[$Q_{h_0}$] non-linear terms in the universal Dirac operator, defined in Lemma \ref{universalDnonlinear}. 
\smallskip
\item[$Q_\text{SW},,...,Q_\mu$] used with various subscripts to denote different types of non-linear terms comprising $\mathbb Q_{h_1}$, defined in (\refeq{nonlinearterms}). 
\smallskip
\item[$\mathbb Q_{h_1}$] the non-linear terms of the universal Seiberg--Witten equations at the model solutions $(\Phi^{(1)}_{\e,\tau},A^{(1)}_{\e,\tau})$, defined in Eq. (\refeq{boldSWnonlinear}).
\smallskip
\item[$R_0$] the constant used in the definition (\refeq{R0def}) of the tangential smoothing gauge. Assumed to be large enough that the operator in Proposition \ref{deformationsmainamd} is invertible. 
\smallskip
\item[$R_{\e,\tau}(r)$] weight function in Sobolev norms, defined preceding Definition \ref{insidenorms}.
\smallskip
\item[$r$] exponential distance to $\mathcal Z_\tau$ in Fermi coordinates (Definition \ref{Fermicoords}). In an abuse of notation, also used to denote the weights in the Sobolev spaces $r^{1+\nu}H^1_e$, $r^\nu L^2$. 
\smallskip 
\item[$r_0$] the radius of the Fermi coordinate chart in Definition \refeq{Fermicoords}. 
\smallskip 
\item[$S$] a $\text{Spin}^c$ structure, often the one defined by Lemma \ref{reconstructingspinc}. 
\smallskip 
\item[$S_0$] the spinor bundle of a spin structure on $Y$, defined preceding Eq. (\refeq{1.1})
\smallskip 
\item[$S_E$] the two-spinor bundle $S\otimes_\R E$ over $Y$. 
\smallskip
\item[$S^\text{Re}, S^\text{Im}$] the real and imaginary subbundles over $\Yminus\mathcal Z_\tau$, defined by Eq. (\refeq{realimdecomp}).
\smallskip
\item[$\text{SW}_\tau(\Phi, A)$] used as shorthand for the Seiberg--Witten equations (\refeq{SW1} --\refeq{SW2}) with parameter $p_\tau$ on the configuration $(\Phi,A)$.
\smallskip
\item[$\mathbb{SW}, \overline{\mathbb{SW}}$] the universal Seiberg--Witten equations and universal Seiberg--Witten eigenvector equation, defined in Definition \ref{universalSWdef}. Definition extended over the base $\mathcal E_\tau \times \mathcal X_\tau$ in in Eq. \eqref{SWwithmathcalX}.
\smallskip
\item[$\overline{\mathbb {SW}}_\Lambda$] the universal Seiberg--Witten eigenvector equations with constant term subtracted, $\overline{\mathbb{SW}}_\Lambda=\overline{\mathbb{SW}}-\chi^- \e^{-1}{\Lambda(\tau)\Phi_\tau}$, defined following Eq. (\refeq{Tdef}). 
\smallskip
\item[$ T_{\Phi_\tau}$] the deformation operator associated to the admissible family of Example \ref{standardfamily}, defined in Eq. \ref{deformation}.
\smallskip
\item[$\underline T_{\Phi_\tau}$] the deformation operator in the tangential smoothing gauge, defined in Notation \ref{underlinenotation}. 
\smallskip
\item[$\underline T_\tau^\circ$] the tame truncation of the deformation operator, defined in Definition \ref{tametruncationdef}. 
\smallskip
\item[$\mathbb T, \text{d}\mathbb T$] the non-linear and linear contraction contraction mappings on $\mathcal H_{\e,\tau}$ induced by the alternating iteration defined in Eq. (\refeq{Titeration}), (\refeq{dTiteration}). See also (\refeq{Tdef}). 
\smallskip
\item[$\mathcal T_{\Phi_\tau}$] the zeroth order operator appearing in the coordinate expression for $T_{\Phi_\tau}$, defined in Eq. \refeq{calTdef}.
\smallskip
\item[$\frak t_\tau^\circ$] the error caused by the tame truncation of the deformation operator, defined in Definition \ref{tametruncationdef}. 
\smallskip
\item[$\frak T_s, \frak T_E, \frak T_{g_\tau}^{g_\eta}$] parallel transport maps on the pullbacks of $S_0, L\otimes E$, and $S\otimes E$ to $X$ respectively (defined in \refeq{spinorparalleltrans}, \refeq{paralleltransport}).
\smallskip
\item[$(t,x,y)$] cylindrical Fermi coordinates on $N_{r_0}(\mathcal Z_\tau)$, defined in Definition \ref{Fermicoords}
\smallskip
\item[$\mathcal V_{\e,\tau}$] a closed ball centered at $0\in \mathcal H_{e,\tau}$ on which $\mathbb T$ is a contraction, appears in Proposition \ref{fixedpoints}.  
\smallskip
\item[$\frak W_{\e,\tau}$] the joint subspace used to cancel deformations, defined in Definition \ref{frakWdef}. Concrete manifestation of the space from Hypothesis \ref{hyp4}.  
\smallskip
\item[$X$] the metric cylinder $([0,1]\times Y,ds^2+g_s)$ for a family of metric $g_s$ on $Y$ used for parallel transport (defined in \refeq{metriccylinder}).
\smallskip
\item[$\mathcal X_\tau$] the space of singular spinors such that $\slashed D_{A_\tau}: \mathcal X_\tau \to \text{\bf Ob}^\perp_\tau$ is surjective, defined in Corollary \ref{solvingcokernel}. 
\item[$Y$] a closed, oriented 3-manifold fixed throughout. 
\smallskip
\item[$Y^\pm_{\e,\tau}, Y^\pm_{\e,\tau, \xi}$] inside an outside regions for gluing parameter $\e,\tau$ and deformation parameter $\xi\in \Gamma(\mathcal Z_\tau; N\mathcal Z_\tau)$, defined in Eq. \eqref{lambdadef}, and Definitions \ref{outsidenorms},  \ref{hilbertspaces}, often denoted $Y^\pm$ as in Appendix \ref{appendixofnotation}.  
\smallskip
\item[$\mathcal Z_0$] the singular set of the initial $\Z_2$-harmonic spinor in Theorem \ref{maina}.
\smallskip 
\item[$\mathcal Z_\eta$] the graph link $F_\eta[\mathcal Z_\tau]$ corresponding to $\eta \in \mathcal E_\tau$ ($\mathcal Z_{\eta,\tau}$ when ambiguity may arise), defined in Notation \ref{Notation6.1}.
\smallskip 
\item[$\mathcal Z_\tau$] the smooth singular sets of the family of $\Z_2$-eigenvectors in Theorem \ref{SpectralCrossing}.
\smallskip 
\item[$|\mathcal Z_\tau|$] the length of (a component of) the singular set. 
\smallskip 
\item[$z$] complex coordinate in cylindrical Fermi coordinates $z=x+iy$, defined in Definition \ref{Fermicoords}.

\item[$\Gamma(S_E)$] used to denote sections of $S_E$ or other vector bundles of unspecified or smooth regularity. 
\smallskip 
\item[$\gamma(e^j)$] Clifford multiplication as a map $\gamma: (\Omega^0\oplus \Omega^1)(i\R)\to \text{End}(S_E)$. 
\smallskip 
\item[$\gamma$] see Important Remark \ref{gammarem}. 
\smallskip 
\item[$\gamma^\pm$] the small numbers $\gamma^\pm = 10^{-6}$. 
\smallskip 
\item[$\gamma_{\mathcal L}$] the small number $\tfrac{2}{3}\left(\tfrac{1}{4}-\nu\right)+ \gamma^+ \nu$, see Theorems \ref{PartImain}, \ref{Insideinvertibility}
\smallskip 
\item[$\e$] the $L^2$-norm parameter in Theorem \ref{maina}, defined preceding Eqns. (\refeq{buSW1}--\refeq{buSW3}). Assumed to be in the range $(0, \e_0)$. 
\smallskip 
\item[$\zeta_\ell$] the exponentially decaying correction in the obstruction basis spinor $\Psi_\ell$, defined in Proposition \ref{cokerpropertiesII}
\smallskip
\item[$\eta$] used to denote a normal section in $\Gamma(\mathcal Z_\tau; N\mathcal Z_\tau)$ corresponding to a deformation of $\mathcal Z_\tau$ via $\text{Exp}_\tau$. 
\smallskip
\item[$\eta^\text{low}, \eta^\text{med}$] used to denote components of a deformation $\eta$ supported in the low and medium Fourier mode regimes, Definition \ref{Fourierregimes}. 
\smallskip
\item[$\theta$] the angular coordinate in cylindrical Fermi coordinates $(t,r,\theta)$, defined in Definition \ref{Fermicoords}. 
\smallskip 
\item[$\iota_\tau$] the 1-dimensional component of the obstruction map $\text{\bf ob}_\tau=(\text{ob}_\tau, \iota_\tau)$ defined in Proposition \ref{cokerpropertiesII}.
\smallskip 
\item[$\Lambda_{-1}, \Lambda_0, \Lambda_1$] closed isotropic, Lagrangian, and coisotropic subspaces of boundary values used to define boundary conditions, defined preceding Definition   \ref{mixedbddef}. 
\smallskip
\item[$\Lambda_\tau$] the eigenvalue of the $\Z_2$-harmonic eigenvectors $(\mathcal Z_\tau, A_\tau, \Phi_\tau)$, defined in Definition \ref{SpectralCrossing}. Assumed to be transverse (Definition  \ref{spectralcrossingdef}).
\smallskip
\item[$\mu(\Phi,\Phi)$] the (extended) moment map acting on a spinor $\Phi$, defined following Eq. (\refeq{gaugefixing}). See also (\refeq{momentmap}). 
\smallskip
\item[$\mu$] used to denote the coordinate in the $\R$-component of $\mathcal H_{\e,\tau}$.

\smallskip
\item[$\underline{\mu}_{\Phi^{(1)}}(\xi)$] term of the universal Seiberg--Witten equations at the model solutions in the direction of a deformation coming from the deformation of the moment map, defined in Proposition \ref{deformationUSW}. 
\smallskip
\item[$\nu$] used to denote weights in function spaces. 
\smallskip 
\item[$\nu^+$] the fixed inside weight $\frac{1}{4}-10^{-6}$. 
\smallskip 
\item[$\nu^-$] the fixed outside weight $\frac{1}{2}-10^{-6}$. 
\smallskip
\item[$\Xi^\pm$] perturbation operators of the deformation operator on the support of $\bold 1^\pm$, defined in Eq. (\refeq{Xipmdef}).
\smallskip
\item[$\xi$] used to denote a renormalized deformation $\e \eta$ for $\eta\in \Gamma(\mathcal Z_\tau; N\mathcal Z_\tau)$, defined in  Eq. \ref{etaxi}. 
\smallskip
\item[$\xi_\ell$] the error term in the obstruction basis spinor $\Psi_\ell$, defined in Proposition \ref{cokerpropertiesII}.
\smallskip 
\item[$ \Pi_\tau$] the $L^2$-orthogonal projections to $\text{\bf Ob}(\mathcal Z_\tau)$.
\smallskip 
\item[$ \Pi_\tau^\perp$] the $L^2$-orthogonal projection to the codimension 1 subspace $\text{\bf Ob}(\mathcal Z_\tau)^\perp \subseteq \text{\bf Ob}(\mathcal Z_\tau)$ defined preceding Corollary \ref{solvingcokernel}.  
\item[$(1-\Pi_\tau)$] $L^2$-orthogonal projections to the $L^2$-orthogonal complement of $\text{\bf Ob}(\mathcal Z_\tau)$.
\smallskip
\item[$\pi^\text{low}, \pi^\text{med}, \pi^\text{high}$] used to denote the $L^2$-orthogonal projections to the three Fourier mode regimes, Definition \ref{Fourierregimes}.  
\smallskip
\item[$\pi_\tau$] the $L^2$-orthogonal projection to the 1-dimensional span of $\Phi_\tau\in \text{\bf Ob}(\mathcal Z_\tau)$. Defined preceding Proposition \ref{deformationstep}. 
\smallskip
\item[$\widetilde\pi_\tau$] the $L^2$-orthogonal projection to the 1-dimensional span of $\Phi_\tau \in rH^1_e$, defined in Lemma \ref{mappingpropertiesI}. 
\smallskip
\item[$\sigma$] the real structure on $S_E |_{\Yminus \mathcal Z_\tau}$, defined by Eq. (\refeq{realimdecomp}).
\smallskip 
\item[$\tau$] parameterizes the family $p_\tau=(g_\tau, B_\tau)$ of metrics and background connections in Theorem \ref{maina}. 
\smallskip 
\item[$\Upsilon_{\mathbb F}$] the trivialization of the Banach vector bundles $\mathbb H^1_e, \mathbb L^2$ associated to an admissible family $\mathbb F$, (Definition \ref{trivializationdef}). 
\smallskip
\item[$\underline \Upsilon_{\mathbb F_\tau}$]  the trivializations associated to the tangential smoothing gauge, defined in Notation \ref{underlinenotation}, 
\smallskip
\item[$\underline{\Upsilon}$] the trivialization of Seiberg--Witten configurations induced by the tangential smoothing gauge, defined in Definition \ref{SWtrivializationdef} and Lemma \ref{SWtrivializations}.
\smallskip
\item[$\Upsilon_{\text{\bf Ob}}$] the trivialization of the obstruction bundle, defined in Proposition \ref{cokerpropertiesI}.
\smallskip
\item[$\Phi$] used to denote a general renormalized smooth spinor in $\Gamma(Y; S_E)$, so that $\|\Phi\|_{L^2}=O(1)$. 
\smallskip 
\item[$\Phi_0,\Phi_\tau$] used to denote the initial $\Z_2$-harmonic spinor and eigenvectors in Theorem \ref{maina}. 
\smallskip
\item[$ \Phi_\tau^\circ$] the tame truncations of the $\Z_2$-harmonic eigenvector $\Phi_\tau$, defined in Eq. \eqref{circdefinitions1}. 
\smallskip
\item[$(\Phi_\tau^\bullet, A^\bullet_\tau)$] the leading-order truncations of the $\Z_2$-harmonic eigenvectors, defined in Eq. (\refeq{bulletedversion}).
\smallskip
\item[$(\Phi^{h_\e}_{\e,\tau}, A^{h_\e}_{\e,\tau})$] the de-singularized configurations, defined on $Y^+$ in Definition \ref{desingularized}.
\smallskip
\item[$ (\Phi^{(0)}_{\e,\tau}, A^{(0)}_{\e,\tau})$] the pre-glued configurations on $Y$, defined in Definition \ref{notationhn}. 
\smallskip
\item[$ (\Phi^{(1)}_{\e,\tau}, A^{(1)}_{\e,\tau})$] the model solutions on $Y$ for parameters $\e,\tau$, defined in Definition \ref{notationhn}. 
\smallskip
\item[$ (\Phi^{(1)}_{\e,\tau.\xi}, A^{(1)}_{\e,\tau,\xi})$] the universal concentrating family on $Y$ for parameters $\e,\tau,\xi$, defined in Definition \ref{universalfamilydef}. 
\smallskip
\item[$\underline \chi_\ell$] the family of smoothing operators used to define the tangential smoothing gauge, defined in (\refeq{chiderivfamily}).  
\smallskip
\item[$\Psi$] used to denote a general smooth spinor in $\Gamma(Y; S_E)$ prior to renormalization, so that $\|\Psi\|_{L^2}=O(\e^{-1})$. 
\smallskip 
\item[$\Psi_\ell^\circ$] the model singular $\Z_2$-harmonic spinors defined in (\refeq{Psilmodel}).
\smallskip 

\item[$\Psi_\ell$] the basis $\text{\bf ob}_\tau(e^{i\ell t})$ for $\ell \in \Z$ of the image of $\text{ ob}_\tau$ (which has corank 1 in $\text{\bf Ob}(\mathcal Z_\tau)$), defined in Proposition \ref{cokerpropertiesII}.
\smallskip 
\item[$\star$] the Hodge star operator of the metric (denoted e.g. $\star_{g_\tau}$ when ambiguity may arise).

\end{enumerate}

\section{ Proof of Proposition  \ref{deformationsmainamd}}
\label{appendixA}
The purpose of this appendix is to prove Proposition \ref{deformationsmainamd}. The proof is a straightforward extension of the proof of Theorem \ref{deformationsmaina} from \cite[Sec. 6]{PartII}, and we begin by recalling this briefly. 

By Proposition \ref{cokerpropertiesII}(A), the composition $T_{\Phi_\tau}(\eta)=\text{ob}_\tau^{-1} \circ \Pi_\tau \circ \mathcal B_{\Phi_\tau}(\eta)$ is calculated by the sequence of inner products. 

\begin{eqnarray} T_{\Phi_\tau}(\eta) &=&\sum_{\ell \in \Z} \br  \mathcal B_{\Phi_\tau}(\eta) , \Psi_\ell\kt \cdot e^{i\ell t} \label{A1}\\
&=& \sum_{\ell \in \Z} \left\br -\frac{1}{2}\sum_{ij} \dot g_{\eta}(e_i,e_j) e^i . \nabla^{g_\tau}_j  + \frac{1}{2} d \text{Tr}_{g_\tau}(\dot g_{\eta}).  +\frac{1}{2} \text{div}_{g_\tau}(\dot g_{\eta}). +  \mathcal R(B_\tau, \dot g_{\eta}). \Phi_\tau , \Psi_\ell\right\kt \cdot e^{i\ell t} \nonumber
\end{eqnarray}

\noindent where $\dot g_\eta =\d{}{s}|_{s=0}g_{s\eta}$, and $\Psi_\ell$ is as in Proposition \ref{cokerpropertiesII}(B). The proof of Theorem \ref{deformationsmaina} proceeds by direct calculation of this sequence of inner products: the leading order term in the theorem is the model operator when $\Psi_\ell=\Psi_\ell^\circ$ and the metric is a product in a neighborhood of $\mathcal Z_\tau$, and deviations from this model case lead to the additional compact operator denoted $K_\tau$ in the theorem statement.

The extension of the proof to Proposition \ref{deformationsmainamd} is analogous, but replaces the metric $\dot g_\eta$ in the expression (\refeq{A1}) by $\dot {\underline g}_\eta$ (as in Notation \ref{underlinenotation}). Proposition \ref{deformationsmainamd}  asserts that once $R_0$ as in (\refeq{R0def}) is chosen sufficiently large, the difference between these two sequences of inner products in small in terms of $R_0$. The idea behind this is straightforward: up to a small error term, $\Psi_\ell$ decays exponentially by Proposition \ref{cokerpropertiesII}(B) with $1/e$ length $O(|\ell|^{-1})$. By construction, $\dot g_\eta =\dot {\underline g}_\eta$ are equal up to radius $r=R_0|\ell|^{-1}$ in the $\ell^{th}$ Fourier mode, thus the difference between each term in the two sequences of inner products is smaller than the norm of the spinor outside $R_0$ times the $1/e$ length.

\begin{proof}[Proof of Proposition \ref{deformationsmainamd}] We fix the following notation. Let $\Psi_\ell=\chi \Psi_\ell^\circ + \zeta_\ell + \xi_\ell$ be the decomposition as in Proposition \ref{cokerpropertiesII}(B), and let $\mathcal B_{\Phi_\tau}, \underline{\mathcal B}_{\Phi_\tau}$ be the two versions of (\refeq{A1}) formed with $g_{s\eta}=F_{s\eta}^*g_\tau$ and $\underline g_{s\eta}=\underline F_{s\eta}^* g_\tau$ respectively. Finally, let $R_0>1$ be as in the statement of the proposition (cf. (\refeq{R0def})), and let $R_1=R_0/2$.

Let $\Phi_{\tau}^{R_1}, g_\tau^{R_1}, B_{\tau}^{R_1}$ be the truncations of the eigenvector, metric, and connection in tangential Fermi coordinates around $\mathcal Z_\tau$ to modes $\ell$ with $|\ell|\leq R_1$. By the Sobolev embedding $C^1(S^1)\hookrightarrow H^2(S^1)$ for applied to $S^1\times \{(x,y)\}$ for every fixed $(x,y)$ in Fermi coordinates, these obey 

\begin{eqnarray} \|\Phi_\tau - \Phi^{R_1}_\tau \|_{C^3}&\leq& C_mR_1^{-m-2}\|\Phi_\tau\|_{C^{2m + 10}}  \label{A2}\\
 \|g_\tau-g_\tau^{R_1}\|_{C^3} \ + \ \|B_\tau-B_\tau^{R_1}\|_{C^1}&\leq& C_m R_1^{-m-2} \|(g_\tau, B_\tau)\|_{C^{2m + 10}} \nonumber\end{eqnarray}

\noindent for any $m\in \N$, and in particular for some choice say $m=12$. Let $\mathcal B_{\Phi_\tau}^{R_1}\underline{\mathcal B}_{\Phi_\tau}^{R_1}$ be the corresponding versions of the operator $\underline {\mathcal B}_{\Phi_\tau}$ as in (\refeq{A1}) formed using the mode-truncated objects $\Phi^{R_1}_\tau, B^{R_1}_\tau$ and $g_{\eta}^{R_1}:=F_\eta^*(g^{R_1})$ and  $\underline g_{\eta}^{R_1}:=\underline F_\eta^*(g^{R_1})$ respectively. The bounds (\refeq{A2}) imply the corresponding bound 

$$\|\mathcal B_{\Phi_\tau}(\eta)-\mathcal B_{\Phi_\tau}^{R_1}(\eta)\|_{L^2(Y)} \leq C R_1^{-m}\|\eta\|_{L^{2,2}(\mathcal Z_\tau)},$$

\noindent and likewise for $\underline{\mathcal B}_{\Phi_\tau}(\eta)-\underline{\mathcal B}_{\Phi_\tau}^{R_1}(\eta)$ (cf. the similar proof of Corollary \ref{fraktcircbounds} for details). By the triangle inequality, 

\bea \|\mathcal B_{\Phi_\tau}(\eta)-\underline{\mathcal B}_{\Phi_\tau}(\eta)\|_{L^2(Y)} & \leq & \|\mathcal B_{\Phi_\tau}(\eta)-\mathcal B_{\Phi_\tau}^{R_1}(\eta)\|_{L^2(Y)}  \ + \ \|\mathcal B^{R_1}_{\Phi_\tau}(\eta)-\underline{\mathcal B}^{R_1}_{\Phi_\tau}(\eta)\|_{L^2(Y)}  \ + \ \|\underline {\mathcal B}^{R_1}_{\Phi_\tau}(\eta)-\underline{\mathcal B}_{\Phi_\tau}(\eta)\|_{L^2(Y)} \\ 
&\leq & 2C R_1^{-m} \|\eta\|_{L^{2,2}(\mathcal Z_\tau)} \ + \ \|\mathcal B^{R_1}_{\Phi_\tau}(\eta)-\underline{\mathcal B}^{R_1}_{\Phi_\tau}(\eta)\|_{L^2(Y)} \\
&\leq & C R_0^{-M}\|\eta_{L^{2,2}(\mathcal Z_\tau)}\| \ + \ \|\mathcal B^{R_1}_{\Phi_\tau}(\eta)-\underline{\mathcal B}^{R_1}_{\Phi_\tau}(\eta)\|_{L^2(Y)} \eea

\noindent for $M=12$, thus it suffices to prove the proposition for $\mathcal B^{R_1}_{\Phi_\tau}$. 

Consider first a single Fourier mode $\eta=\eta_k e^{ik t}$. Observe that, due to the cutoff $\chi_{r_0}$ defined in (\refeq{chiderivfamily}),  $F_\eta=\underline F_\eta$ when $R_0|k|\geq r_0/2$, i.e. when $|k|\leq 2R_0/r_0$. Consequently, it suffices to show the bound for $\eta$ having only Fourier modes above this range. 

With $\Psi_\ell = \chi \Psi_\ell + \zeta_\ell + \xi_\ell$, we begin with the inner products with the first two terms. Note that by the restriction of the modes of the operator coefficients, \be \mathcal B_{\Phi_\tau}^{R_1}(e^{ikt})-\underline{\mathcal B}_{\Phi_\tau}^{R_1}(e^{ikt}) \in \text{Span}\big\{e^{i\ell t} \ \big | \ |\ell-k| \geq |k|/2\big\}\label{fourierrestriction}\ee

\noindent because both operators have only coefficients in the $[k-R_1,  k+R_1]$ range, and $|k|\geq 2R_0/r_0$ implies $R_1\leq |k|/2$ for $r_0<1/2$. Since $\chi\Psi_\ell + \zeta_\ell$ has supported in Fourier modes in the range $[\ell-|\ell|/2, \ell + |\ell|/2]$, by Proposition \ref{cokerpropertiesII}(B), the only non-zero products arise for $|k|/3 \leq |\ell| \leq 3|k|$. Thus 
\noindent 
\bea 
\left\|\sum_{\ell\in \Z}\left\br \mathcal B_{\Phi_\tau}^{R_1}(e^{ikt})-\underline{\mathcal B}_{\Phi_\tau}^{R_1}(e^{ikt})  \ , \ \chi \Psi_\ell + \zeta_\ell \right\kt \cdot e^{i\ell t} \right\|^2_{L^2}&\leq& \sum_{|k|/3\leq |\ell| \leq 3|k|} \left| \left\br \mathcal B_{\Phi_\tau}^{R_1}(e^{ikt})-\underline{\mathcal B}_{\Phi_\tau}^{R_1}(e^{ikt})  \ , \ \chi \Psi_\ell + \zeta_\ell \right\kt_{L^2} \right |^2  \\ 
&\leq & C\|e^{ikt}\|^2_{L^{2,2}(\mathcal Z_\tau)} \cdot \Big\| \chi \Psi_\ell + \zeta_\ell\Big\|^2_{\Yminus N_k}\eea

\noindent where $N_k=\{ (t,r,\theta) \ | \ r \leq \frac{R_0}{3|k|}\}$. And in this region, the exponential decay in the expression (\refeq{Psilmodel}) and Proposition \ref{cokerpropertiesII}(B) implies

$$ \Big\| \chi \Psi_\ell + \zeta_\ell\Big\|_{\Yminus N_k}\leq C \text{Exp}\left( -|\ell|c_1 \cdot \frac{R_0}{2|k|}\right) \leq C \text{Exp}(-R_0/c).$$

\noindent where $c=c_1/2$. Summing over Fourier modes and using Plancharel's theorem completes the bound for the first two terms.

For the contribution for the final $\xi_\ell$ term, note that the Fourier mode restriction (\refeq{fourierrestriction}) shows that $\nabla_t\sim |k|$ on the difference term, thus integrating by parts,

\bea
|k|^{2M} \Big |\left\br \mathcal B_{\Phi_\tau}^{R_1}(e^{ikt})-\underline{\mathcal B}_{\Phi_\tau}^{R_1}(e^{ikt})  \ , \ \xi_\ell \right\kt\Big |^2 &\leq &\Big |\left\br \nabla_t^M\left( \mathcal B_{\Phi_\tau}^{R_1}(e^{ikt})-\underline{\mathcal B}_{\Phi_\tau}^{R_1}(e^{ikt})\right)  \ , \ \xi_\ell \right\kt\Big | \\
&\leq &\Big |\left\br\mathcal B_{\Phi_\tau}^{R_1}(e^{ikt})-\underline{\mathcal B}_{\Phi_\tau}^{R_1}(e^{ikt})  \ , \ \nabla_t^{M} \xi_\ell \right\kt\Big | \leq C_M |\ell|^{-2} \\ 
\eea

\noindent using Proposition \ref{cokerpropertiesII}(B.iii) which implies that $\|\nabla_t^M \xi_\ell\|\leq C|\ell|^{N-M}$ for any $N\geq M$. In particular for $N=M+2$. Using the summability over $\ell$, and then summing over $k$ using the fact that we have reduced to the range $|k|\geq 2R_0/r_0$ shows this term is bounded by $CR_0^{-M}$, completing the proposition.  
\end{proof}

\section{ Twisted APS Boundary Conditions}
\label{appendixB}

This appendix provides details of the mixed APS boundary and projection conditions (Definition \ref{mixedbddef}) used in the proofs of Theorem \ref{Insideinvertibility} and Proposition \ref{outboundaryconditionsdef}. The first of these two results is precisely \cite[Thm 7.1]{PartI}, and the latter is a slight extension using the techniques established there. The reader is referred to \cite[Sec. 7.1--7.3]{PartII} for further details. 

It suffices to work in Fermi coordinates of radius $r_0$ and the induced trivializations of $S_E, \Omega^1(i\R)$ as in Definition \ref{Fermitrivialization}, because both boundaries $\del Y^\pm_{\e,\tau}$ are contained in such a coordinate chart $N_{r_0}(\mathcal Z_\tau)$ once $\e_0$ is sufficiently small. In this trivialization, we may write spinors in the form $\ph=(\alpha, \beta)$ where $\alpha, \beta \in \Gamma(N_{r_0}(\mathcal Z_\tau); \mathbb H)$ are $\mathbb H\simeq \C^2$-valued spinors. The Dirac operator in the product metric with the product connection is \be \slashed D_{\circ} \begin{pmatrix}\alpha \\ \beta \end{pmatrix}=\begin{pmatrix}
i\del_t & -2\del \\ 2\delbar & -i\del _t 
\end{pmatrix} \begin{pmatrix}\alpha \\ \beta \end{pmatrix}, \label{DiracIntbyParts}\ee

\noindent and obeys the integration by parts formula, 

\be \int_{N_{r_0}(\mathcal Z_\tau)}  \ \br \slashed D_{\circ} \ph, \psi \kt  \ dV   \ - \ \br \ph,  \slashed D_{\circ} \psi \kt  \ dV  = \Omega_{\del Y^+}(\ph,\psi)\label{bigOmegadY}\ee

\noindent by Definition \ref{symplecticdY}. The same holds for $\slashed D_{A}$ for any smooth connection $A$, since Clifford multiplication $\Omega^1(i\R): S_E \to S_E$ is self-adjoint, and likewise for the extension including $0$-forms. It follows that if $\Lambda_0\subseteq L^{1/2,2}(Y^+; \C^2\otimes_\C \mathbb H)$ is Lagrangian with respect to $\Omega_{\del Y^+}$, then $\slashed D_A$ subject to the boundary condition that $\Pi_{\Lambda_0}\ph|_{\del Y+} =0$ is self-adjoint, thus Fredholm of index 0 by standard elliptic theory. The same holds for isotropic or coisotropic subspaces $\Lambda_{\pm 1}$ as in Definition \ref{mixedbddef}, now with Fredholm index $\text{dim}(\Lambda_1/\Lambda_0)$ or $-\text{dim}(\Lambda_0/\Lambda_{-1})$ respectively.

One natural choice of boundary condition (and thus of corresponding subspace $\Lambda_i$) are Atiyah-Patodi-Singer or {\bf {APS boundary conditions}} \cite[Sec. 17]{KM, APS1}. These boundary conditions are defined in terms of the spectrum $\text{Spec}(\slashed D_{\del Y^+})$ of the induced Dirac operator on the boundary. In this case, 
$$\Lambda= \left\{ \phi \in L^{1/2,2}(\del Y^+; \C^2 \otimes_\C \mathbb H) \ \Big | \ \Pi_{L}(\ph)=0 \right\}, $$

\noindent where $\Pi_{L}$ is the spectral projection to eigenspaces with eigenvalue $\lambda \leq L$. It is understood in this definition that we take the intersection of the subspace defined by the $L^2$ spectrum with $L^{1/2,2}$. In this case, choice $L<0$ result in an isotropic subspace, thus negative Fredholm index and vice-versa for $L>0$. 

In the analysis of elliptic edge operators, one often treats these operators as families of differential operators parameterized by the tangential variables \cite{MazzeoEdgeOperators}. The same viewpoint offers a more useful boundary condition in this setting. Since $\del Y^\pm\simeq T^2$, we may decompose a spinor first in Fourier modes in the $\theta$ direction, then in the tangential direction to write $$\ph \big |_{\del Y^\pm}= \begin{pmatrix}\alpha(t,\theta) \\ \beta(t,\theta) \end{pmatrix}= \sum_{k\in \Z}  \begin{pmatrix}\alpha_k(t) \\ \beta_k(t) \end{pmatrix}e^{ik\theta}=\sum_{k,\ell \in \Z}  \begin{pmatrix}\alpha_{k\ell} \\ \beta _{k\ell}\end{pmatrix} e^{ik\theta}e^{i\ell t},$$

\noindent and define spectral boundary conditions using only the spectral decomposition in the $\theta$ directions. 

\begin{defn}  \label{normalAPSlag}The {\bf normal APS Lagrangian} subspace is the one defined in Fourier modes above by the requirement that 

\bea \Lambda_0^{\text{Norm}, +}&=&\left\{  \begin{pmatrix}\alpha(t,\theta) \\ \beta(t,\theta)\end{pmatrix}  \in L^{1/2,2}( \del Y^+; S_E)  \ \Big | \  \begin{matrix} \alpha_{k}(t)=0 \text{ for } k <0 \\ \beta_{k}(t)=0 \text{ for } k >0 \\ \end{matrix}\right\} \\ 
 \Lambda_0^{\text{Norm}, -}&=&\left\{  \begin{pmatrix}\alpha(t,\theta) \\ \beta(t,\theta)\end{pmatrix}  \in L^{1/2,2}( \del Y^+; S_E)  \ \Big | \  \begin{matrix} \alpha_{k}(t)=0 \text{ for } k \geq 0 \\ \beta_{k}(t)=0 \text{ for } k \leq 0 \\ \end{matrix}\right\}.\eea

\noindent Thus the restriction of a spinor obeying the corresponding boundary condition $\Pi_{\Lambda_0}^{\text{Norm},+}(\ph)=0$ has allowed Fourier modes visualized as follows: 
\end{defn}
\begin{eqnarray*}
 \hspace{1.75cm}  \underline{k=-1}  \  \   \ &\underline{ k =0}& \ \ \ \underline{k=1} \medskip  \\  
\ldots  \alpha_{-2}(t)  \  \ \   \ \ \ \ \alpha_{-1}(t)     \   \ \ \ &0&   \ \ \ \  \  \ 0  \  \   \ \  \ \ \ \ \ \ \ 0 \ldots \\
\ldots  0  \   \ \ \  \ \ \ \ \   \ \ \ \ 0 \ \ \ \  \     \   \ \ \ &0 & \ \ \ \  \beta_{1}(t)  \  \  \ \ \ \  \beta_{2}(t) \ldots \\ 
\end{eqnarray*}
\label{NBV3d}

\noindent and the reverse for $\Pi_{\Lambda_0}^\text{Norm,-}(\ph)=0$.

\begin{lm} \label{normalBC}The normal APS Lagrangian defines self-adjoint boundary conditions on $Y^\pm$. In particular, the corresponding boundary value problem 

$$\slashed D_A: \{\ph \in L^{1,2}(Y^\pm; S_E) \ | \ \Pi_{\Lambda_0}^{\text{Norm},\pm}(\ph)=0\} \to L^2(Y^\pm; S_E)$$

\noindent is Fredholm of Index 0. 
\end{lm}
\begin{proof} In the case of $Y^+$, this follows immediately from integration by parts using the expression  (\refeq{DiracIntbyParts}) and the integration by parts formulae

\begin{eqnarray}
\int_{D^2} \br -2\del \beta, \alpha \kt_\C + \br \beta, -2\delbar \alpha\kt_\C  \ dV &=&i \int_{\del D} \br -\beta,\alpha \kt_\C  e^{-i\theta}d\theta \label{delbyparts}\\
\int_{D^2} \br 2\delbar \alpha, \beta\kt_\C + \br \alpha, 2\del \beta\kt_\C  \ dV &=& -i\int_{\del D} \br \alpha,\beta\kt_\C e^{i\theta}d\theta \label{delbarbyparts}
\end{eqnarray} 
\noindent for $\del,\delbar$, written using the Hermitian inner product $\br -, - \kt_\C$. See \cite[Sec. 6]{PartI} for additional details. The perturbation arising from the connection is compact and self-adjoint, thus does not affect this calculation. For $Y^-$ the proof is identical, but integrating over $Y^-$ the boundary terms above reverse sign. 
\end{proof}

More generally, we can define a Lagrangian subspace $\Lambda_0$ such that for each Fourier mode $k$ in the $\theta$-direction, the allowed functions of $t$ are defined {\it either} pointwise {\it or} by a collection of Fourier modes $\ell$ in the $t$-direction. We restrict the discussion to the relevant cases for Theorem \ref{Insideinvertibility} and Lemma \ref{outboundaryconditionsdef}, though the construction also applies in greater generality. Let $E_{-1,0}\simeq \C^4 \to \mathcal Z_\tau$ denote the span of the $(\alpha_{-1}e^{-i\theta}, \beta_0)$ Fourier modes in $\theta$, where $\alpha_{-1}, \beta_0\in \mathbb C^2$, viewed as a bundle of complex rank $4$ over $\mathcal Z_\tau$. The boundary restrictions of spinors in these modes satisfying the boundary condition of Definition \ref{normalAPSlag} span the space  

\be L^{1/2,2}(\mathcal Z_\tau; E_{-1,0}) \supseteq \left\{\begin{pmatrix} \alpha_{-1}(t) \\ 0 \end{pmatrix} \  \Big | \ \alpha_{-1}(t)\in L^{1/2,2}(\mathcal Z_\tau; \C^2)\right\}= \Pi_{\Lambda_0^{\text{Norm},+}}^{-1}(0)\cap \Gamma(\mathcal Z_\tau; E_{-1,0})\label{constsubbundle}\ee

\noindent on $Y^+$, and span the similar space where only  $\beta_0(t)$ is non-zero on $Y^-$. Thus one can view the boundary condition in these modes as saying the restriction lies pointwise in a half-dimensional subbundle $V\subseteq  E_{-1,0}$, which is Lagrangian with respect to an appropriate symplectic form. 

More specifically, the symplectic form is as follows. The almost-complex structure 
\be \mathcal J=  E_{-1,0}\to E_{-1,0} \hspace{2cm}\mathcal J\begin{pmatrix}\alpha \\ \beta \end{pmatrix}=\begin{pmatrix} -\beta \\ \alpha \end{pmatrix}\label{mathcalJdef}\ee
\noindent induces a pointwise symplectic structure $\omega_E(\ph,\psi)=\br \ph, \mathcal J\psi\kt$ on $E_{-1,0}$. Restricted to $\Gamma(\mathcal Z_{\tau}; E_{-1,0})$, the symplectic form (\refeq{bigOmegadY}) reduces to $\Omega_{\del Y^+}(\ph,\psi)=\int_{\del Y^+}\omega_E(\ph,\psi) r dt$. Note that $\mathcal J$ does not coincide with either of the almost-complex structures $j,J$ used in Eq. (\refeq{realimdecomp}).

\begin{defn}  \label{permissiblelagrangian} A subbundle $V_t \subseteq E_{-1,0}\to \mathcal Z_\tau$ is a {\bf permissible bundle of Lagrangians} if 
\begin{enumerate}
\item[(A)] $V_t$ is Lagrangian with respect to $\omega_E$ for all $t\in \mathcal Z_\tau$, 
\item[(B)] $V_t$ is homotopic to the constant subbundle in Eq. (\refeq{constsubbundle}) through subbundles obeying (A). 
\end{enumerate}

\noindent Such a permissible bundle defines a Lagrangian $\Lambda_{V_t}\subseteq L^{1/2,2}(\del Y^\pm; S_E)$ by 

$$\Lambda_{V_t}\cap \Gamma(\mathcal Z_\tau; E_{-1,0})= L^{1/2,2}(\mathcal Z_\tau; V_t) \hspace{2cm} \Lambda_{V_t}\cap \Gamma(\mathcal Z_\tau; E_{-1,0}^\perp)=\Lambda^{\text{Norm}, \pm}_0 \cap\Gamma(\mathcal Z_\tau; E_{-1,0}^\perp),$$

\noindent and thus a {\bf twisted APS boundary condition} $\Pi_{\Lambda_{V_t}}=0$. 
\end{defn}

 \cite[Lem. 7.5]{PartI} shows the following: 
 
 \begin{lm}  \label{twistedAPS}Suppose that $V_t \to \mathcal Z_\tau$ is a permissible bundle of Lagrangians. Then the boundary value problem 
 \be \slashed D_A: \{\ph \in L^{1,2}(Y^\pm; S_E) \ | \ \Pi_{\Lambda_{V_t}}(\ph)=0\} \to L^2(Y^\pm; S_E)\ee
\noindent defined by the accompanying twisted APS boundary condition is Fredholm of Index 0. 
 \end{lm}
 \begin{proof} Integration by parts and Young's inequality show that condition (A) from Definition \ref{permissiblelagrangian} implies that $\slashed D_\circ$ with this boundary condition has finite-dimensional kernel and closed range. Integration by parts again shows that the cokernel is given by solutions of the Dirac equation satisfying the boundary condition defined by the permissible bundle of Lagrangians $V_t^\perp$. Thus the boundary value problem is Fredholm, and condition (B) ensures the index is zero by Lemma \ref{normalBC}. Altering the connection is a compact perturbation so does not affect the conclusions. See  \cite[Lem. 7.5]{PartI}  for details. 
 \end{proof}

The difficulty that necessitates the use of these rather intricate boundary conditions arises from the off-diagonal term in the linearization 

\be \mathcal L_{(\Phi^{(1)}, A^{(1)})}(\ph,a)=\begin{pmatrix}\slashed D_{A^{(1)}} & \gamma(\_)\tfrac{\Phi^{(1)}}{\e} \\ \tfrac{\mu(\_,\Phi^{(1)})}{\e} & \bold d \end{pmatrix}\begin{pmatrix}\ph \\ a \end{pmatrix},\label{linearization2}\ee

\noindent calculated in Lemma \ref{linearizedequationsdef}. The topological twist of the real-line bundle $\ell$ is retained at the boundary $\del Y^+$ even after de-singularization the $\Z_2$-harmonic eigenvector to the model solution in Theorem \ref{Insideinvertibility}. In particular, one has $\Phi^{(1)}\approx r^{1/2}(c(t), d(t)e^{-i\theta})$ on $\del Y\pm$, where $c(t),d(t)$ are the leading coefficients in the expansion of Lemma \ref{asymptoticexpansion}. The twist $e^{-i\theta}$ which is the remnant of the real line bundle $\ell$ leads to boundary terms when integrating (\refeq{linearization2}) by parts using any n\"aive choice of boundary conditions. Although these boundary terms are compact, the factor of $\e^{-1}$ makes it challenging to show the operator is sufficiently invertible for any boundary condition when they are present.  The solution in \cite{PartI} is to use twisted APS boundary conditions defined in terms of the leading coefficients $(c(t), d(t))$ which precisely allow these terms to cancel. 

To define these boundary conditions, we first write the $(\Omega^0\oplus \Omega^1)(i\R)$ components in complex notations as follows: set 

 \be \zeta=(a_0 + ia_t)dz, \hspace{2cm} \omega = (a_y-ia_x)d\overline z\hspace{2cm}\label{formiso}\ee

\noindent where $a=(ia_0, ia_xdx+ ia_ydy + ia_zdz)$. In these complex coordinates, one has 

$$ \gamma(pdz + qd\overline z)=\begin{pmatrix} ip & -\overline q \\ -q & i\overline p\end{pmatrix} \hspace{2cm} \bold d \begin{pmatrix} \zeta \\ \omega
\end{pmatrix}=\begin{pmatrix} -i\del_t & 2\del \\ -2\delbar & i\del_t
\end{pmatrix}  \begin{pmatrix} \zeta \\ \omega
\end{pmatrix}$$

\noindent where $\gamma$ is Clifford multiplication (composed (\refeq{formiso})). In particular, $\bold d$ has the same coordinate expression as $-\slashed D_\circ$. Everything in the preceding discussion about about boundary values applies equally well to boundary value problems for $\bold d$, now with $\C^2$-valued forms in place of $\C^4$-valued spinors.

We now prove Proposition \ref{outboundaryconditionsdef}. First, extend the leading-order configurations (\refeq{bulletedversion}) to $Y^+$ by 
$$ (\Phi^\star_\tau,A^\star_\tau):=  \zeta^+(\Phi^\bullet_\tau,A^\bullet_\tau) + (1-\zeta^+)(\Phi_\tau, A_\tau),$$

\noindent where $\zeta^+$ is a cut-off function equal to $1$ on $\del Y^-$, and vanishing for $r\geq 2\text{dist}(\del Y^-, \mathcal Z_\tau)$. Next, since the lemma deals only with the $S^\text{Im}$ components, and $S\simeq S^\text{Im}$ via $\Psi \mapsto \frac{1}{2}(\Psi -\sigma \Psi)$ as in (\refeq{realimdecomp}), it suffices to work in a trivialization $S^\text{Im}\simeq \underline{\C}^2$ on $Y^-\cap N_{r_0}(\mathcal Z_\tau)$, in which 

$$\Phi^\star_\tau \Big |_{\del Y^-}=\begin{pmatrix} c(t) \\ d(t)e^{-i\theta}
\end{pmatrix}r^{1/2}$$

\noindent for $c(t),d(t)$ the leading coefficients of the expansion in Lemma \ref{asymptoticexpansion}. Analogously to (\refeq{constsubbundle}), there is a subspace $E_{-1,0}\simeq \underline \C^2$ now denoting the span of $\C$-valued functions $\alpha_{-1}(t), \beta_0(t)$. We then set \be \label{Vtdef}V_t=\left\{\begin{pmatrix} \alpha_{-1}(t) \\ \beta_0(t)
\end{pmatrix} \ \Big | \  \beta_0(t)\overline d(t) \ + \ \overline{\alpha}_{-1}(t) c(t)=0\right\}\subseteq E_{-1,0}.\ee

\noindent The non-degeneracy assumption on $\Phi_\tau$ (Definition \ref{regulardef}) implies that $|c(t)|^2 + |d(t)|^2\neq 0$, thus (\refeq{Vtdef}) defines a 2-dimensional real subspace of $E_{-1,0}$.\begin{proof}[Proof of Proposition \ref{outboundaryconditionsdef}] Define a Lagrangian $\Lambda_0$ on configurations $(\alpha,\beta, \zeta, \omega)\in \C^4\simeq S^\text{Re}\oplus (\Omega^0\oplus \Omega^1)$ (using the isomorphism (\refeq{formiso}) above) as follows. $\Lambda_0$ is defined by the property that the boundary condition $\Pi_{\Lambda_0}(\alpha,\beta, \zeta, \omega)=0$ means the configuration's allowed Fourier modes on $\del Y^-$ are given by 

\begin{eqnarray}
 \hspace{1.75cm}  \underline{k=-1}  \  \   \ &\underline{ k =0}& \ \ \ \underline{k=1} \medskip   \nonumber \\  
\ldots  0  \  \ \   \ \ \  \boxed{ \alpha_{-1}(t)}       \ \ \ &\alpha_0(t)&   \ \   \  \ \alpha_1(t)  \  \   \ \ \ \ \alpha_2(t) \ldots \nonumber  \\
\ldots  \beta_{-2}(t)    \ \ \   \ \ \ \ \beta_{-1}(t) \   \ \ \ &\boxed{\beta_0(t)} & \ \ \ \ 0   \  \  \ \ \ \ \ \   \ \ \ \  0 \ldots  \label{Fouriertable}\\ 
\ldots  0  \  \ \    \ \ \ \ \ \ \  0   \ \ \ \ \      \ \ \ &\zeta_0(t)&   \ \   \  \ \zeta_1(t)  \  \   \ \ \ \ \zeta_2(t) \ldots \nonumber \\
\smallskip
\ldots  \omega_{-2}(t)    \ \ \   \ \ \ \ \omega_{-1}(t) \   \ \ \ &{\omega_0(t)} & \ \ \ \ 0   \  \  \ \  \ \ \ \  \ \ \ \  0 \ldots \nonumber  \\ \nonumber\\
\text{where }&   \boxed{\alpha_{-1}(t) + \beta_0(t) \in V_t} & 
\end{eqnarray}

Inspection of (\refeq{Vtdef}) shows $V_t$ is spanned by complex linear combinations of $(d(t), -c(t))$. The definition of $\mathcal J$ in (\refeq{mathcalJdef}) shows that $$\left\br a\begin{pmatrix} d(t) \\ -c(t)\end{pmatrix} \ , \ b\mathcal J \begin{pmatrix} d(t) \\ -c(t)\end{pmatrix} \right\kt=0,$$

\noindent for all $a,b\in \C$, thus $V_t$ is indeed Lagrangian. By a homotopy of the pair $(c(t),d(t))$ through functions obeying the non-degeneracy condition, $V_t$ is homotopic to the constant bundle of Lagrangians as in Eq. (\refeq{constsubbundle}) defined by the pair $(c(t), d(t))=(0,1)$. $V_t$ is therefore a permissible bundle of Lagrangians. The linearization at $(\Phi_\tau^\star, A_\tau^\star)$ 

\be \mathcal L^\text{Im}_{\star}(\ph^\text{im}, a):= \begin{pmatrix}   \slashed D_{A^\star_\tau} & \gamma(\_)\frac{\Phi^\star_\tau}{\e} \\  \frac{\mu(\_, \Phi^\star_\tau)}{\e}. & \bold d \end{pmatrix} \begin{pmatrix}\ph^\text{im}\\ a \end{pmatrix}, \ee

\noindent subject to these boundary conditions is Fredholm of Index 0, by applying Lemma \ref{twistedAPS}) to $\slashed D_{A_\tau^\star}$, Lemma  \ref{normalBC} to $\bold d$, and using that the off-diagonal terms are compact. This proves Part (B) of the Lemma. 

To ensure the boundary terms indeed vanish, we calculate the boundary terms. Integration by parts using $\mathcal L_\star^\text{Im}$ shows (see \cite[Lem. 6.24]{PartI} for details) that for weight $\nu=0$,

\be
\label{boundarytermsB}\|\mathcal L_\star^\text{Im}(\ph^\text{im} ,a)\|^2_{L^2}=\|(\ph^\text{im} ,a)\|^2_{H^{1,-}_\e}  \ + \  \frac{1}{\e}\br (\ph^\text{im} ,a), \frak B(\ph^\text{im} , a)  \kt \ + \  \Big (\text{Boundary terms}. \Big ) \ee

\noindent as in the proof of Lemma \ref{solvingoutsideim}, where the boundary terms are given (using (\refeq{delbyparts}--\refeq{delbarbyparts}) by 
\bea
 \br -2\del \beta , i \zeta \alpha^\star_\tau\kt_\C  \ - \   \br \beta , -2\del (i \zeta \alpha^\star_\tau)\kt_\C  &= &i\br  -\beta ,\zeta \alpha_\tau^\star\kt_\C  \ \cdot \ e^{-i\theta} \\
 \br -2\del \beta , i \zeta \alpha^\star_\tau\kt_\C  \ - \   \br \beta , -2\del ( -\overline \omega \beta^\star_\tau e^{-i\theta})\kt_\C &= &i\br  -\beta , -\overline \omega \beta^\star_\tau e^{-i\theta}\kt_\C  \ \cdot \  e^{-i\theta} \\
\br 2\delbar \alpha , \omega \alpha_\tau^\star  \kt_\C \ - \ \br \alpha, 2\delbar(\omega \alpha_\tau^*) \kt_\C&= &-i\br  \alpha , \omega \alpha_\tau^\star\kt_\C   \ \cdot \  e^{-i\theta} \\
\br 2\delbar \alpha , i \overline \zeta \beta_\tau^\star e^{-i\theta} \kt_\C \ - \ \br \alpha, 2\delbar( i \overline \zeta \beta_\tau^\star e^{-i\theta}) \kt_\C &= &-i\br \alpha , i \overline \zeta \beta^\star_\tau e^{-i\theta}\kt_\C  \ \cdot \  e^{-i\theta}. 
\eea

\noindent where $\alpha_\tau^*=c(t), \beta_\tau^\star=d(t)$. Using the table above to investigate which Fourier modes lead to non-zero inner products, one sees that the only overlapping terms in the inner products on the right arise in the $\alpha_{-1}(t), \beta_0(t)$ Fourier modes. Thus the boundary terms in (\refeq{boundarytermsB}) are given in the real inner product by 

\bea 
 \Big (\text{Boundary terms}. \Big ) & = &\text{Re}\Big[ i\br  -\beta ,\zeta \alpha_\tau^\star\kt_\C  \ \cdot \ e^{-i\theta}  \ + \ i\br  -\beta , -\overline \omega \beta^\star_\tau e^{-i\theta}\kt_\C  \ \cdot \  e^{-i\theta} \\ & &  \ + \ i\br  \alpha , \omega \alpha_\tau^\star\kt_\C   \ \cdot \  e^{-i\theta} \ + \ -i\br \alpha , i \overline \zeta \beta^\star_\tau e^{-i\theta}\kt_\C  \ \cdot \  e^{-i\theta}\Big]\\
 &=&\text{Re}\Big[ i\beta_0(t)\omega_0(t)\overline d(t)  -i \alpha_{-1}(t)\overline \omega_0(t)\overline c(t)\Big]\\
 &=& \text{Re}\Big[ i \omega_0(t) \cdot \big(\beta_0(t)\overline d(t)  + \overline{ \alpha}_{-1}(t) c(t) \big)\Big]\\
 &=& 0,
\eea
precisely by the definition of $V_t$ in Eq. (\refeq{Vtdef}). In passing from the second to third lines above, the second term was conjugated. This proves Part (B) of the lemma, in the case of linearizing at $(\Phi_\tau^\star, A_\tau^\star)$. Because $(\Phi_\tau, A_\tau)-(\Phi_\tau^\star, A_\tau^\star)=O(r^{3/2})$ differ by lower order terms where $r\leq 2 r^{2/3-\gamma^+}$, the linearization at the eigenvector differs only by a $O(\e^{2/3-\gamma^+})$ in the $H^{1,-}_{\e,0}$ norm. This completes the proof of (A) in the case that $\nu=0$. For $\nu\neq 0$, the weight $R_{\e,\tau}(r)$ is constant on the boundary and is simply carried along for the entire proof.

\end{proof}

The proof of Theorem \ref{PartImain} follows a similar scheme. We offer the following remark, and refer the reader to \cite[Sec. 6--7]{PartI} for details. 

\begin{rem} The boundary conditions on $\mathcal L_{(\Phi^{(1)}, A^{(1)})}$ on $Y^+$ used in the proof of Theorem \ref{PartImain} are similar to those in Lemma \ref{outboundaryconditionsdef}. The allowed Fourier modes in Table (\refeq{Fouriertable}) are replaced by the adjoint boundary condition with allowed modes 

\begin{eqnarray*}
 \hspace{1.75cm}  \underline{k=-1}  \  \   \ &\underline{ k =0}& \ \ \ \underline{k=1} \medskip  \\  
\ldots  \alpha_{-2}(t)  \  \ \   \ \ \  \boxed{ \alpha_{-1}(t)}       \ \ \ &0 &   \   \  \  \ \ \  0 \  \  \ \ \    \  \ \ \ \ 0  \ \ldots \\
\ldots  0 \ \     \ \ \   \ \ \ \  \ 0 \  \ \ \    \ \ \ &\boxed{\beta_0(t)} & \ \ \  \ \   \beta_1(t)    \  \ \ \  \  \beta_2(t) \ldots \\ 
\smallskip
\smallskip
\ldots  \zeta_{-2}(t)    \ \ \   \ \ \ \ \zeta_{-1}(t) \   \ \ \ & \ 0 \  & \ \ \  \ \ \ 0   \  \  \ \  \ \ \ \ \  \  \ \ \ \  0 \ldots \\ 
\ldots  0  \  \ \    \ \ \ \ \ \ \  0   \ \ \ \ \      \ \ \ & \  0 \ &   \ \   \   \ \ \omega_1(t)  \  \ \   \ \ \ \ \omega_2(t) \ldots \\
\end{eqnarray*}
\end{rem}

\noindent and the boxed modes in the rank 8 real bundle $E_{-1,0}\to \mathcal Z_\tau$ are constrained so that 

\be 0=\mu_\C(\alpha,\beta)=  b_1 \overline \alpha_1^\bullet + \overline a_1 \beta_1^\bullet + b_2 \overline \alpha_2^\bullet + \overline a_2 \beta_2^\bullet.\label{BC2}\ee 

\noindent where the $\C^2=\mathbb H$-valued spinors are written in the trivialization as
 $$  \begin{matrix}\alpha_{-1}  =a_1 \otimes 1 + a_2 \otimes j   \\  \ \beta_{0} \  =b_1 \otimes 1 + b_2 \otimes j \end{matrix} \hspace{2cm}\Phi_\tau^\bullet=\begin{pmatrix} \alpha_1^\bullet \\ \beta_1^\bullet\end{pmatrix}\otimes 1 + \begin{pmatrix} \alpha_2^\bullet \\ \beta_2^\bullet\end{pmatrix}\otimes j. $$
\noindent In this case, a similar integration by parts argument shows that the boundary terms vanish (cf. \cite[Lem. 6.24]{PartI}). This defines the Lagrangian subspace used in the mixed APS boundary and orthogonality conditions of Theorem \ref{Insideinvertibility}.   

The setting of Theorem \ref{Insideinvertibility} is more challenging because the projection conditions must also be non-trivial. With only the APS boundary conditions defined above, smoothings of the singular harmonic spinors (\refeq{Psilmodel}) violate uniform elliptic estimates for small values of the Fourier index $\ell$. The mixed APS boundary and projection conditions in Theorem \ref{Insideinvertibility} amend this by requiring that the $L^2$ projection with the low modes $|\ell|\leq O(\e^{-1/2})$ of these smoothed singular spinors be zero, at the cost of relaxing the boundary condition (\refeq{BC2}) to be free in precisely the same number $O(\e^{-1/2})$ of the low modes of sections of $V_t$. See \cite[Fig. 2]{PartI} for an illustration. 
\smallskip


{\small

\medskip

\bibliographystyle{amsalpha}
{\small 
\bibliography{Bib_revised}}

@article{uhlenbeckremovable,
  title="{\emph{\it{Removable singularities in Yang-Mills fields}}}",
  author={Uhlenbeck, Karen K},
  journal={Communications in Mathematical Physics},
  volume={83},
  pages={11--29},
  year={1982},
  publisher={Springer}
}

@article{uhlenbeckcompactness,
  title="{\emph{\it {Connections with $L^p$ bounds on curvature}}}",
  author={Uhlenbeck, Karen K},
  journal={Communications in Mathematical Physics},
  volume={83},
  pages={31--42},
  year={1982},
  publisher={Springer}
}

@article{taubesgluing,
  title="{\emph{\it{Self-dual Yang-Mills connections on non-self-dual 4-manifolds}}}",
  author={Taubes, Clifford Henry},
  journal={Journal of Differential Geometry},
  volume={17},
  number={1},
  pages={139--170},
  year={1982},
  publisher={Lehigh University}
}

@inproceedings{APS1,
  title="{\emph{\it{Spectral asymmetry and Riemannian geometry. I}}}",
  author={Atiyah, Michael F and Patodi, Vijay K and Singer, Isadore M},
  booktitle={Mathematical Proceedings of the Cambridge Philosophical Society},
  volume={77},
  number={1},
  pages={43--69},
  year={1975},
  organization={Cambridge University Press}
}

@book{DK,
  title="{\emph{\it {The Geometry of Four-Manifolds}}}",
  author={Donaldson, Simon Kirwan and Kronheimer, Peter B},
  year={1997},
  publisher={Oxford university press}
}

@article{BW25,
  title="{\emph{\it{Dirac operators twisted by ramified Euclidean line bundles}}}",
  author={Bera, Gorapada and Walpuski, Thomas},
  journal={arXiv preprint arXiv:2503.01392},
  year={2025}
}

@book{SchwartzAlternating,
  title="{\emph{\it{Domain decomposition methods-algorithms and theory}}}",
  author={Toselli, Andrea and Widlund, Olof},
  volume={34},
  year={2004},
  publisher={Springer Science \& Business Media}
}

@book{FrU12,
  key = {FrU12},
  title="{\emph{\it{Instantons and Four-manifolds}}}",
  author={Freed, Daniel S and Uhlenbeck, Karen K},
  volume={1},
  year={2012},
  publisher={Springer Science \& Business Media},
}

@inproceedings{ebin1970manifold,
  title="{\emph{\it{The manifold of Riemannian metrics}}}",
  author={Ebin, David G},
  booktitle={Proc. Sympos. Pure Math},
  volume={15},
  number={Part 2},
  pages={11--40},
  year={1970}
}

@article{Palais,
  title="{\emph{\it{Foundations of global non-linear analysis}}}",
  author={Palais, Richard S},
  journal={(No Title)},
  year={1968}
}

@misc{TomsNotes,
  title="{\emph{\it{Lectures Notes on Math 18.937}}}",
  author={Mrowka, Tomasz and Wang, Donghao},
  year={2019},
  publisher={MIT},
  url={hi},
}

@misc{WalpuskiNotes,
  title="{\emph{\it{Lectures on Generalised Seiberg--Witten Equations}}}",
  author={Walpuski, Thomas},
  year={2023},
  publisher={Frontiers in Geometry and Topology, Proceedings of Symposia in Pure Mathematics},
}

@misc{DoanThesis,
  title="{\emph{\it{Monopoles and Fueter sections on Three-manifolds}}}",
  author={Doan, Aleksander},
  year={2019},
  publisher={State University of New York at Stony Brook}
}

@article{JoyceAssociatives,
  title="{\emph{\it{Conjectures on counting associative 3-folds in $G_2 $-manifolds}}}",
  author={Joyce, Dominic},
  journal={arXiv preprint},
  pages={arXiv:1610.09836},
  year={2016},
}

@article{Haydys2008nonlinear,
  title="{\emph{\it{Nonlinear Dirac operator and Quaternionic Analysis}}}",
  author={Haydys, Andriy},
  journal={Communications in Mathematical Physics},
  volume={281},
  pages={251--261},
  year={2008},
  publisher={Springer},
}

@article{donaldsonfibrations,
  title="{\emph{\it{Adiabatic limits of co-associative Kovalev--Lefschetz fibrations}}}",
  author={Donaldson, Simon},
  journal={Algebra, Geometry, and Physics in the 21st Century: Kontsevich Festschrift},
  pages={1--29},
  year={2017},
  publisher={Springer}
}

@article{yuanqiatiyahclass,
  title="{\emph{\it{Atiyah classes and the essential obstructions in deforming a singular $G_2$-instanton}}}",
  author={Wang, Yuanqi},
  journal={Mathematische Zeitschrift},
  volume={300},
  number={3},
  pages={2997--3021},
  year={2022},
  publisher={Springer}
}

@article{DonaldsonAlternating,
  title="{\emph{\it{Connections, cohomology and the intersection forms of 4-manifolds}}}",
  author={Donaldson, Simon K},
  journal={Journal of Differential Geometry},
  volume={24},
  number={3},
  pages={275--341},
  year={1986},
  publisher={Lehigh University}
}

@article{DonaldsonGluing,
      title="{\emph{\it{Calabi-Yau metrics on Kummer surfaces as a model glueing problem}}}", 
      author={Simon Donaldson},
        journal={arXiv preprint},
  	pages={arXiv:1007.4218},
  	year={2016},
}

@book{MSBig,
  title="{\emph{\it{J-holomorphic curves and symplectic topology}}}",
  author={McDuff, Dusa and Salamon, Dietmar},
  volume={52},
  year={2012},
  publisher={American Mathematical Soc.}
}

@book{TaylorPDEsIII,
  title="{\emph{\it{J-holomorphic curves and symplectic topology}}}",
  author={McDuff, Dusa and Salamon, Dietmar},
  volume={52},
  year={2012},
  publisher={American Mathematical Soc.}
}

@article{PacardRitore,
  title="{\emph{\it{From constant mean curvature hypersurfaces to the gradient theory of phase transitions}}}",
  author={Pacard, Frank and Ritore, Manuel},
  journal={Journal of Differential Geometry},
  volume={64},
  number={3},
  pages={359--423},
  year={2003},
  publisher={Lehigh University}
}

@article{RyosukeThesis,
  title="{\emph{\it {The Moduli Space of $S^1$-type Zero Loci for $\mathbb{Z}/2$-Harmonic Spinors in Dimension 3}}}",
  author={Takahashi, Ryosuke},
  journal={arXiv Preprint},
  pages={arXiv 1503.00767},
  year={2015},
}

@article{DonaldsonMultivalued,
  title="{\emph{\it {Deformations of Multivalued Harmonic Functions}}}",
  author={Donaldson, Simon},
  journal={The Quarterly Journal of Mathematics},
  volume={72},
  number={1-2},
  pages={199--235},
  year={2021},
  publisher={Oxford University Press UK}
}

@article{SiqiSlag,
  title="{\emph{\it {The Branched Deformations of the Special Lagrangian Submanifolds}}}",
  author={He, Siqi},
  journal={arXiv Preprint},
  volume={},
  pages={arXiv 2202.12282},
  year={2022},
}

@article{HitchinHarmonicSpinors,
  title="{\emph{\it{Harmonic Spinors}}}",
  author={Hitchin, Nigel},
  journal={Advances in Mathematics},
  volume={14},
  number={1},
  pages={1--55},
  year={1974},
  publisher={Academic Press}
}

@article{Lichnerowicz, 
 title="{\emph{\it{Spineurs Harmoniques}}}",
 author={Lichnerowicz, A.},
 journal={C.R. Acad. Sci. Paris, Series A--B},
 number={257},
 pages={7--9},
 year={1963},
 }

@article{AS1963,
  title="{\emph{\it{The index of elliptic operators on compact manifolds}}}",
  author={Atiyah, Michael F and Singer, Isadore M},
  journal={Bulletin of the American Mathematical Society},
  volume={69},
  number={3},
  pages={422--433},
  year={1963}
}

@incollection{grieser2001basics,
  title={Basics of the b-calculus},
  author={Grieser, Daniel},
  booktitle={Approaches to singular analysis},
  pages={30--84},
  year={2001},
  publisher={Springer}
}

@article{MazzeoHaydysTakahashiExamples,
  title="{\emph{\it{New examples of $\mathbb Z/2$ harmonic 1-forms and their deformations}}}",
  author={Haydys, Andriy and Mazzeo, Rafe and Takahashi, Ryosuke},
  journal={arXiv preprint},
  pages={arXiv:2307.06227},
  year={2023}
}

@article{MazzeoHaydysTakahashi,
  title="{\emph{\it{An index theorem for $\mathbb Z/2$-harmonic spinors branching along a graph}}}",
  author={Haydys, Andriy and Mazzeo, Rafe and Takahashi, Ryosuke},
  journal={arXiv preprint},
  pages={arXiv:2310.15295},
  year={2023}
}

@article{GregSiqi,
  title="{\emph{\it{$\mathbb Z_2$-Harmonic Spinors on Connect Sums}}}",
  author={He, Siqi and Parker, Gregory},
    journal={arXiv preprint},
  pages={arXiv:2407.10922},
  year={2024}
}

@article{PartII,
  title="{\emph{\it{Deformations of $\mathbb Z_2 $-Harmonic Spinors on 3-Manifolds}}}",
  author={Parker, Gregory J},
  journal={To appear in Geometric and Functional Analysis},
  pages={ arXiv:2301.06245},
  year={2026}
}

@article{YMHGluing,
  title="{\emph{\it{Entire solutions to 4 dimensional Ginzburg--Landau equations and codimension 2 minimal submanifolds}}}",
  author={Badran, Marco and del Pino, Manuel},
  journal={Advances in Mathematics},
  volume={435},
  pages={109365},
  year={2023},
  publisher={Elsevier}
}

@article{ConcentratingDirac,
  title="{\emph{\it{Concentrating Dirac Operators and Generalized Seiberg-Witten Equations}}}",
  author={Parker, Gregory J},
  journal={To appear in Math Research Letters},
  pages={ arXiv:2307.00694},
  year={2026}
}

@article{MinhCompactness,
  title="{\emph{\it{The three-dimensional Seiberg-Witten equations for $3/2$-spinors: a compactness theorem}}}",
  author={Sadegh, Ahmad Reza Haj Saeedi and Nguyen, Minh Lam},
  journal={arXiv preprint},
  pages={ arXiv:2311.11902},
  year={2023}
}

@article{DWExistence,
  title="{\emph{\it {On the Existence of Harmonic  {$\mathbb{Z}_2$}-Spinors}}}",
  author={Doan, Aleksander and Walpuski, Thomas},
  journal={Journal of Differential Geometry},
  volume={117},
  number={3},
  pages={395--449},
  year={2021},
  publisher={Lehigh University}
}

@article{ZhangRectifiability,
  title="{\emph{\it {Rectifiability and {Minkowski} Bounds for the Zero Loci of {$\mathbb{Z}_2$}-Harmonic Spinors in Dimension 4}}}",
  author={Zhang, Boyu},
  journal={arXiv Preprint (to appear in Comm. Analysis and Geom.)},
  pages={arXiv 2202.12282},
  year={2017}
}

@article{TaubesZeroLoci,
  title="{\emph{\it {The Zero Loci of {$\mathbb{Z}/2$}-Harmonic Spinors in Dimension 2, 3 and 4}}}",
  author={Taubes, Clifford Henry},
  journal={arXiv Preprint},
  pages={arXiv:1407.6206},
  year={2014}
}

@article{HamiltonNashMoser,
  title="{\emph{\it{The Inverse Function Theorem of {Nash} and {Moser}}}}",
  author={Hamilton, Richard S},
  journal={Bulletin (New Series) of the American Mathematical Society},
  volume={7},
  number={1},
  pages={65--222},
  year={1982},
  publisher={American Mathematical Society}
}

@article{SiqiZ3,
  title="{\emph{\it {Existence of Nondegenerate {$\mathbb{Z}_2$}-Harmonic 1-forms via {$\mathbb{Z}_3$} Symmetry}}}",
  author={He, Siqi},
  journal={arXiv Preprint},
  pages={arXiv:2202.12283},
  year={2022}
}

@article{HaydysBlowupsets,
  title="{\emph{\it {The Infinitesimal Multiplicities and Orientations of the Blow-up set of the {Seiberg-Witten} Equation with Multiple Spinors}}}",
  author={Haydys, Andriy},
  journal={Advances in Mathematics},
  volume={5},
  pages={193-218},
  year={2019}
}

@article{WalpuskiZhangCompactness,
author = {Walpuski, Thomas and Zhang, Boyu},
year = {2021},
pages = {239-397},
title = "{\emph{\it {On the Compactness Problem for a Family of Generalized Seiberg-Witten Equations in Dimension 3}}}",
volume = {170(17)},
journal = {Duke Mathematical Journal},
doi = {10.4310/CJM.2013.v1.n2.a2}
}

@article{Taubes3dSL2C,
author = {Taubes, Clifford Henry},
year = {2013},
month = {01},
pages = {239-397},
title = "{\emph{\it {$\text{PSL}(2; \mathbb{C})$ Connections on 3-manifolds with {$L^2$} Bounds on Curvature}}}",
volume = {1},
journal = {Cambridge Journal of Mathematics},
doi = {10.4310/CJM.2013.v1.n2.a2}
}

@article{HWCompactness,
  title="{\emph{\it {A Compactness Theorem for the {Seiberg--Witten} Equation with Multiple Spinors in Dimension Three}}}",
  author={Haydys, Andriy and Walpuski, Thomas},
  journal={Geometric and Functional Analysis},
  volume={25},
  number={6},
  pages={1799--1821},
  year={2015},
  publisher={Springer}
}

@article{PartI,
  title="{\emph{\it {Concentrating Local Solutions of the Two-Spinor Seiberg-Witten Equations on 3-Manifolds}}}",
  author={Parker, Gregory J.},
   journal={To appear in Asian Journal of Mathematics},
   pages={arXiv:2210.08148},
  year={2026}
}

@article{Dashen1,
  title="{\emph{\it {A construction of non-degenerate $\mathbb Z_2$-harmonic functions on $\mathbb R^n$}}}",
  author={Yan, Dashen},
   journal={arXiv Preprint},
   pages={arXiv:2503.19286},
  year={2025}
}

@article{Dashen2,
  title="{\emph{\it {On non-degenerate $\mathbb Z_2$-harmonic 1-forms with shrinking branching sets}}}",
  author={Yan, Dashen},
   journal={arXiv Preprint},
   pages={arXiv:2510.07678},
  year={2025}
}

@article{MWWW,
  title="{\emph{\it {Ends of the Moduli space of {Higgs} Bundles}}}",
  author={Mazzeo, Rafe and Swoboda, Jan and Weiss, Hartmut and Witt, Frederik},
  journal={Duke Mathematical Journal},
  volume={165},
  number={12},
  pages={2227--2271},
  year={2016},
  publisher={Duke University Press}
}

@article{TaubesVW,
  title="{\emph{\it {The Behavior of Sequences of Solutions to the {Vafa-Witten} Equations}}}",
  author={Taubes, Clifford Henry},
  journal={arXiv Preprint},
  pages={ arXiv:1702.04610},
  year={2017}
}

@article{TaubesKWNahmPole,
  title="{\emph{\it{Sequences of {Nahm pole} Solutions to the {$SU(2)$} {Kapustin-Witten} Equations}}}",
  author={Taubes, Clifford Henry},
  journal={arXiv Preprint},
  pages={ arXiv:1805.02773},
  year={2018}
}

@article{Taubes4dSL2C,
  title="{\emph{\it {Compactness Theorems for  {$SL(2,{\mathbb C})$} Generalizations of the 4-dimensional Anti-Self Dual Equations}}}",
  author={Taubes, Clifford Henry},
  journal={arXiv Preprint},
  pages={ arXiv:1307.6447},
  year={2013}
}

@article{TaubesU1SW,
title ="{\emph{\it  {On the Behavior of Sequences of Solutions to {$U(1)$}-{Seiberg-Witten} Systems in Dimension 4}}}",
author = {Taubes, Clifford Henry},  
journal={arXiv Preprint},
pages = {arXiv:1610.07163},
year = {2016},
}

@article{TaubesKW2,
  title="{\emph{\it{The $\mathbb R$-invariant solutions to the Kapustin--Witten equations on $(0,\infty)\times \mathbb R^2
  \times \mathbb R$ with generalized Nahm pole asymptotics}}}",
  author={Taubes, Clifford Henry},
  journal={arXiv preprint arXiv:1903.03539},
  year={2019}
}

@article{TaubesNonlinearDirac,
  title="{\emph{\it{Nonlinear Generalizations of a 3-Manifold's Dirac Operator}}}",
  author={Taubes, Clifford Henry},
  journal={AMS IP Studies in Advanced Mathematics},
  volume={13},
  pages={475--486},
  year={1999},
  publisher={Providence, RI; American Mathematical Society; 1999}
}

@article{WalpuskiG2Gluing,
  title="{\emph{\it{$G_2$-Instantons, Associative Submanifolds and Fueter Sections}}}",
  author={Walpuski, Thomas},
  journal={Communications in Analysis and Geometry},
  volume={25},
  number={4},
  pages={847--893},
  year={2017},
  publisher={Springer}
}

@article{DWAssociatives,
  title="{\emph{\it {On Counting Associative Submanifolds and {Seiberg-Witten} Monopoles}}}",
  author={Doan, Aleksander and Walpuski, Thomas},
  journal={Pure and Applied Mathematics Quarterly},
  volume={15},
  number={4},
  pages={1047--1133},
  year={2019}
}

@article{DonaldsonSegal,
    author = {Donaldson, Simon and Segal, Ed},
    title = "{\emph{\it{Gauge Theory in Higher Dimensions, II}}}",
    journal={Surveys in
differential geometry Volume XVI. Geometry of special holonomy and related topics.},
volume={16},
pages={1-41},
    year = {2011}
}

@article{HaydysG2SW,
  author = {Haydys, Andriy},
  title ="{\emph{\it{G2 Instantons and the {Seiberg-Witten} Monopoles}}}",
  publisher = {arXiv Preprint},
  pages={arXiv: 1703.06329},
  year = {2017}
}

@article{DWDeformations,
  title="{\emph{\it {Deformation Theory of the Blown-up {Seiberg--Witten} Equation in Dimension Three}}}",
  author={Doan, Aleksander and Walpuski, Thomas},
  journal={Selecta Mathematica},
  volume={26},
  number={3},
  pages={1--48},
  year={2020},
  publisher={Springer}
}

@article{HaydysCorrespondence,
  title="{\emph{\it {Gauge theory, Calibrated Geometry and Harmonic Spinors}}}",
  author={Haydys, Andriy},
  journal={Journal of the London Mathematical Society},
  volume={86},
  number={2},
  pages={482--498},
  year={2012},
  publisher={Oxford University Press}
}

@article{MazzeoEdgeOperators,
  title="{\emph{\it {Elliptic Theory of Differential Edge Operators I}}}",
  author={Mazzeo, Rafe},
  journal={Communications in Partial Differential Equations},
  volume={16},
  number={10},
  pages={1615--1664},
  year={1991},
  publisher={Taylor \& Francis}
}

@article{MazzeoEdgeOperatorsII,
  title="{\emph{\it{Elliptic Theory of Differential Edge Operators, {II}: {Boundary} Value Problems}}}",
  author={Mazzeo, Rafe and Vertman, Boris},
  journal={Indiana University Mathematics Journal},
  pages={1911--1955},
  year={2014},
  publisher={JSTOR}
}

@article{Bourguignon,
  title="{\emph{\it{Spineurs, Op{\'e}rateurs de {Dirac} et Variations de M{\'e}triques}}}",
  author={Bourguignon, Jean-Pierre and Gauduchon, Paul},
  journal={Communications in Mathematical Physics},
  volume={144},
  number={3},
  pages={581--599},
  year={1992},
  publisher={Springer}
}

@book{KM,
  title="{\emph{\it {Monopoles and Three-Manifolds}}}",
  author={Kronheimer, Peter B and Mrowka, Tomasz},
  volume={10},
  year={2007},
  publisher={Cambridge University Press Cambridge}
}

@article{DiracVariationNonlinear,
  title="{\emph{\it{A Universal Spinor Bundle and the {Einstein--Dirac--Maxwell} Equation as a Variational Theory}}}",
  author={M{\"u}ller, Olaf and Nowaczyk, Nikolai},
  journal={Letters in Mathematical Physics},
  volume={107},
  number={5},
  pages={933--961},
  year={2017},
  publisher={Springer}
}

@article{VWOriginalPaper,
  title={\emph{{\it A Strong Coupling Test of S-Duality}}},
  author={Vafa, Cumrun and Witten, Edward},
  journal={Nuclear Physics B},
  volume={431},
  number={1-2},
  pages={3--77},
  year={1994},
  publisher={Elsevier}
}

@article{HaydysFSCategory,
  title="{\emph{{\it Fukaya-Seidel Category and Gauge Theory}}}",
  author={Haydys, Andriy},
  journal={J. Symplectic Geom.},
  volume={13},
  number={1},
  pages={151-207},
  year={2015},
}

@article{WittenFivebranesKnots,
  title={\emph{{\it Fivebranes and Knots}}},
  author={Witten, Edward},
  journal={Quantum Topology},
  volume={3},
  number={1},
  pages={1--137},
  year={2011}
}

@article{WittenKhovanovGaugeTheory,
  title={\emph{{\it Khovanov Homology and Gauge Theory}}},
  author={Witten, Edward},
  journal={Proceedings of the Freedman Fest},
  volume={18},
  pages={291--308},
  year={2012},
  publisher={Geometry \& Topology Publications Coventry}
}

@book{MorganSW,
  title={\emph{{\it The {Seiberg-Witten} Equations and Applications to the Topology of Smooth Four-manifolds}}},
  author={Morgan, John W},
  volume={44},
  year={1996},
  publisher={Princeton University Press}
}

@article{TaubesWu,
  author = {Taubes, Clifford Henry and Wu, Yingying},
 title = "{\emph{\it{Examples of Singularity Models of $\Z/2$-harmonic 1-forms and Spinors in Dimension 3}}}",
 journal={arXiv Preprint},
 pages={arXiv 2001.00227},
 year={2020}
 }
\end{document}



\end{document}